\documentclass[a4paper,11pt]{amsart}
\usepackage[margin=3cm, marginpar=2.5cm]{geometry}
\usepackage{graphicx}
\usepackage{amsfonts}
\usepackage{epsf}
\usepackage{amssymb}
\usepackage{amsmath}
\usepackage{amscd}
\usepackage{amsthm}
\usepackage{tikz}
\usetikzlibrary{
  cd,
  calc,
  positioning,
  arrows,
  decorations.pathreplacing,
  decorations.markings
}
\usepackage{subcaption}
\usepackage{pdfpages}
\usepackage{fancyhdr}
\usepackage{setspace}
\usepackage{hyperref}
\usepackage[all]{xy}
\usetikzlibrary{matrix}
\usepackage{verbatim}
\usepackage{enumerate}
\usepackage{mathrsfs}
\usepackage{pinlabel}
\usepackage{lscape}
\usepackage{color}

\usepackage[colorinlistoftodos]{todonotes}

\theoremstyle{plain}
\newtheorem{te}{Theorem}[section]
\newtheorem{prop}[te]{Proposition}
\newtheorem{co}[te]{Corollary}
\newtheorem{lemma}[te]{Lemma}

\theoremstyle{definition}
\newtheorem{de}[te]{Definition}
\newtheorem{facts}[te]{Facts}
\newtheorem{q}[te]{Question}

\theoremstyle{remark}
\newtheorem{re}[te]{Remark}

\newcommand{\Q}{\mathbb{Q}}
\newcommand{\N}{\mathbb{N}}
\newcommand{\R}{\mathbb{R}}
\newcommand{\Z}{\mathbb{Z}}

\newcommand{\CP}{\mathbb{CP}}
\newcommand{\CPbar}{\overline{\CP}\,\!^2}
\newcommand{\acc}{\mathcal{S}}
\newcommand{\Gc}{\mathcal{G}}
\newcommand{\Lc}{\mathcal{L}}
\newcommand{\Rc}{\mathcal{R}}

\newcommand{\spinc}{spin$^c$ }

\newcommand{\unknot}{\mathcal{O}}

\renewcommand{\epsilon}{\varepsilon}

\DeclareMathOperator{\rk}{rk}
\DeclareMathOperator{\Tor}{Tor}

\newcommand{\bb}{\mathbf{b}}
\newcommand{\cc}{\mathbf{c}}

\title{surgeries on torus knots, rational balls, and cabling}
\author{Paolo Aceto}
\address{Mathematical Institute, University of Oxford, Oxford, United Kingdom}
\email{paoloaceto@gmail.com}

\author{Marco Golla}
\address{CNRS and Laboratoire de math\'ematiques J. Leray, University of Nantes, Nantes, France}
\email{marco.golla@univ-nantes.fr}

\author{Kyle Larson}
\address{Department of Mathematics, University of Georgia, Athens, Georgia U.S.A.}
\email{kyle.larson@uga.edu}

\author{Ana G. Lecuona}
\address{School of mathematics and statistics, University of Glasgow, Glasgow, United Kingdom}
\email{ana.lecuona@glasgow.ac.uk}

\begin{document}
\begin{abstract}
We classify which positive integral surgeries on positive torus knots bound rational homology balls.
Additionally, for a given knot $K$ we consider which cables $K_{p,q}$ admit integral surgeries that bound rational homology balls.
For such cables, let ${\acc}(K)$ be the set of corresponding rational numbers $\frac{q}{p}$. We show that ${\acc}(K)$ is bounded for each $K$.
Moreover, if $n$-surgery on $K$ bounds a rational homology ball then $n$ is an accumulation point for ${\acc}(K)$.
\end{abstract}

\maketitle
\tableofcontents

\begin{section}{Introduction}
In the Kirby problem list Casson asks~\cite[Problem~4.5]{Kirbylist}: \emph{which rational homology $3$-spheres bound rational homology $4$-balls}?
While it is hard to imagine a complete answer to such a question, in the past two decades much progress has been made by restricting to special families of 3-manifolds, notably Seifert fibered spaces.
In~\cite{Lisca-ribbon} Lisca provided a complete list of which lens spaces bound rational homology balls. The proof relies on an obstruction based on Donaldson's diagonalization theorem, and its implementation requires a careful analysis of the integral lattices associated with certain definite 4-manifolds bounded by lens spaces.
Following his work, several authors extended this approach, sometimes in conjunction with obstructions from Heegaard Floer homology, to various families of Seifert fibered spaces (see, for example,~\cite{Lecuona} and references therein).

In~\cite{AcetoGolla} the first and the second author considered the question of which Dehn surgeries on a given knot bound rational homology balls.
From this perspective, Lisca's work provides a complete answer for the unknot.
In~\cite{AcetoGolla} the authors determined which positive integral surgeries on positive torus knots $T_{p,q}$ with $q\equiv \pm1\pmod p$ bound rational homology balls, and gave restrictions for surgeries along arbitrary torus knots to bound rational homology balls.
Note that such surgeries are Seifert fibered spaces.
The restriction $q\equiv \pm1\pmod p$ is unnatural and was primarily a device to simplify the technical analysis involving certain integral lattices.

In this paper we remove the restriction on the pair $(p,q)$ and give a complete classification of which positive integral surgeries on positive torus knots bound rational homology balls. We will define three sets of triples of positive integers, $\Gc$, $\Rc$, and $\Lc$ below. We fix the convention that our torus knots $T_{p,q}$ have $p<q$.

\begin{te}\label{maintheorem}
Let $S^3_n(T_{p,q})$ be the result of $n$-surgery on the positive torus knot $T_{p,q}$, with $n\in\N$. Then $S^3_n(T_{p,q})$ bounds a rational homology ball if and only if the triple $(p,q;n)$ belongs to a set of triples $\Gc \cup \Rc \cup \Lc$.

Triples in $\Gc$ correspond to surgeries that are Seifert fibered $3$--manifolds with three singular fibers, those in $\Rc$ correspond to connected sums of two lens spaces, and those in $\Lc$ correspond to lens space surgeries.
\end{te}

We now define the three sets $\Gc$, $\Rc$, and $\Lc$.
Recall the definition of the Fibonacci sequence $F_{k+1}=F_{k}+F_{k-1}$, with $F_0=0$ and $F_1=1$.
We also define three auxiliary sequences\footnote{These sequences are in fact well-known: $R_k$ is the $2k^{\rm th}$ half-companion Pell number; $S_k$ is the $(2k-1)^{\rm th}$ Pell number; $2T_k$ is the $2k^{\rm th}$ Pell number.}
$\{R_k\}$, $\{S_k\}$, and $\{T_k\}$. Note that we use the same recursive relation and only change the initial values.
\[
\left\{
\begin{array}{l}
R_0 = 1,\\
R_1 = 3,\\
R_{k+1} = 6R_k - R_{k-1};
\end{array}
\right.
\quad
\left\{
\begin{array}{l}
S_0 = 1,\\
S_1 = 1,\\
S_{k+1} = 6S_k - S_{k-1};
\end{array}
\right.
\quad
\left\{
\begin{array}{l}
T_0 = 0,\\
T_1 = 1,\\
T_{k+1} = 6T_k - T_{k-1}.
\end{array}
\right.
\]

The set $\Gc$ contains the following four families and two exceptional cases:
\begin{enumerate}
   \item $(p,q;n)=(p,p+1;p^2)$ for some $p\geq 2$;
   \item $(p,q;n)=(p,p+1;(p+1)^2)$ for some $p\geq 2$;
   \item $(p,q;n)=(p,4p\pm1;(2p)^2)$ for some $p\geq 2$;
   \item $(p,q;n)=(R_k,R_{k+1};R_kR_{k+1}-2)$ for some $k\geq 1$;
   \item $(p,q;n)\in\{(3,22;64),(6,43;256)\}$.
\end{enumerate}

The set $\Rc$ comprises the following seven families and two exceptional cases:
\begin{enumerate}\setcounter{enumi}{5}
   \item $(p,q;n)=(r^2,(r+1)^2;r^2(r+1)^2)$ for some $r\geq2$;
   \item $(p,q;n)=(r^2,(2r\pm 1)^2;r^2(2r\pm 1)^2)$ for some $r\equiv 2\pmod 4$ and $r\geq 2$;
   \item $(p,q;n)=(r^2,(2r\pm 2)^2;r^2(2r\pm 2)^2)$ for some $r\equiv \mp 3\pmod 8$ and $r\geq 5$;
   \item $(p,q;n)=(F_k^2,F_{k+1}^2;F_k^2F_{k+1}^2)$ for some $k\geq 3$;
   \item $(p,q;n)=(F_{2k-1}^2,F_{2k+1}^2;F_{2k-1}^2F_{2k+1}^2)$ for some $k\geq 2$;
   \item $(p,q;n)=(S_k^2,4T_k^2;4T_k^2S_k^2)$ for some $k\geq 2$;
   \item $(p,q;n)=(4T_k^2,S_{k+1}^2;S_{k+1}^24T_k^2)$ for some $k\geq 2$;
   \item $(p,q;n)\in\{(9^2,14^2;9^2\cdot14^2),(11^2,14^2;11^2\cdot14^2)\}$;
\end{enumerate}

Finally, the set $\Lc$ comprises the following five families:
\begin{enumerate}\setcounter{enumi}{13}
   \item $(p,q;n)=(2r+1,2r+3;4(r+1)^2)$ for some $r\geq 1$;
   \item $(p,q;n)=(F_{2k},F_{2k+2};F_{2k}F_{2k+2}+1)$ for some $k\geq 1$;
   \item $(p,q;n)=(F_{2k+1},F_{2k+3};F_{2k+1}F_{2k+3}-1)$ for some $k\geq 1$;
   \item $(p,q;n)=(F_{2k+1},F_{2k+5};F_{2k+1}F_{2k+5}-1)$ for some $k\geq 1$;
   \item $(p,q;n)=(S_{k+1},S_{k+2};S_{k+1}S_{k+2}-1)$ for some $k\geq 1$.
\end{enumerate}

Some of these families were already known. Indeed, families (1)--(3) and the two exceptional cases in (5) correspond to torus knots with $q\equiv \pm1\pmod p$ and hence come from~\cite{AcetoGolla}. Families (2), (10), (18), the two sporadic cases in (5), and the sub-family with $q=4p-1$ in (3) arise in the study of rational cuspidal curves in $\CP^2$. Indeed, in~\cite{55letters}, Fern\'andez de Bobadilla, Luengo, Melle-Hern\'andez, and N\'emethi classify all rational complex curves in $\CP^2$ with one singularity of topological type $x^p = y^q$ (with $\gcd(p,q) = 1$); a regular neighborhood of such a curve is the trace of $d^2$--surgery along $T_{p,q}$, where $d$ is the degree of the curve, and its complement is a rational homology ball bounded by the surgery (with its orientation reversed).
Some lens spaces and connected sums of lens spaces also appear in complex geometry, when classifying certain singular surfaces which admit a smoothing to $\CP^2$~\cite{HackingProkhorov}, and in the corresponding topological version of smooth embeddings of rational homology balls bounded by lens spaces~\cite{Owens-embeddings}.
In the context of the theory of surface singularities, Stipsicz, Szab\'o and Wahl~\cite{StipsiczSzaboWahl} studied the existence of symplectic rational homology ball fillings of canonical contact structures on links of singularities. There is some overlap between the two resulting lists, also in the case of Seifert fibered spaces: for instance, family (5) also appears in their list.

Surgeries along the torus knot $T_{p,q}$ are reducible if the surgery coefficient is $pq$, lens spaces if the surgery coefficient is $pq\pm 1$, and three-legged Seifert fibered spaces otherwise~\cite{Moser}. We recall that Lisca has classified which lens spaces~\cite{Lisca-ribbon} and connected sums of lens spaces~\cite{Lisca-sums} bound rational homology balls. The analysis for the corresponding torus knot surgeries, which yields the two sets $\Rc$ and $\Lc$ amounts to matching two different lists: the list of all connected sums of lens spaces arising as surgeries along torus knots~\cite{Moser} and the one coming from Lisca's work.
This is mostly an arithmetic argument and is carried out in Appendix~\ref{s:appendix}. Therefore the essential part of the proof of Theorem~\ref{maintheorem} consists in showing that if $q\not\equiv \pm1\pmod p$ and the surgery coefficient $n$ is not in $\{pq,pq\pm1\}$, then $S^3_{n}(T_{p,q})$ bounds a rational homology ball if and only if $(p,q;n)$ belongs to family $(4)$.

Our result is obtained using a mixture of Heegaard Floer homology and lattice embeddings via Donaldson's Theorem A~\cite{Donaldson}. In this paper, the input from correction terms in Heegaard Floer homology is very limited, consisting mostly of an extension of~\cite[Theorem~1.3]{AcetoGolla}, while the bulk of the proof is a completely novel analysis of \emph{obstructions} to embeddability of lattices. The lattices under consideration are intersection lattices of plumbed 4--manifolds bounded by small Seifert spaces. It is worth stressing here that the umbrella term of ``analysis of embeddability of lattices'' encompasses a wide range of strikingly different ideas and techniques. In his pioneering work Lisca~\cite{Lisca-ribbon,Lisca-sums} studied the embeddability of all lattices arising from linear graphs into diagonal lattices of the same rank. 
Some years later, Greene~\cite{Greene-Berge} studying the Berge conjecture analyzed 
again lattices arising from linear graphs, but this time embedded into codimension-1 diagonal lattices and subject to a further condition arising from Heegaard Floer homology. Needless to say that Greene's and Lisca's work and approaches to the embeddability of their respective lattices are completely different from each other. Both of these groundbreaking works focused on embeddability while our approach is an intricate \emph{obstruction} to the embeddability of a subfamily of lattices associated to small Seifert spaces. The lattice embedding results in this paper shed some further light into the still open question of which are the small Seifert spaces whose associated lattices embed in diagonal ones of the same rank. There is no known or proposed strategy to tackle this very interesting problem.

We can collect some of the information contained in Theorem~\ref{maintheorem} by defining the following set of rational numbers
\[
{\acc}:=\left\{\frac{q}{p} \ | \ S^3_{n}(T_{p,q}) \ \text{bounds a rational ball for some} \ n\in\N\right\}.
\]
Each infinite family in Theorem~\ref{maintheorem} determines an accumulation point for ${\acc}$.

\begin{co}\label{accumulationpoints}
Let $\phi=\frac{1}{2}(1+\sqrt{5})$ be the golden ratio and let $\psi=3+2\sqrt{2}$. 
The set of accumulation points for ${\acc}$ is $\{1,\phi^2,4,\psi,\phi^4\}$.
Moreover ${\acc}\subset[1,\frac{22}{3}]$. 
\end{co}

Note that the integral accumulation points for ${\acc}$ are precisely the positive integral surgeries on the unknot which bound rational homology balls.
Since torus knots are cables of the unknot this observation suggests the following definition. Let $K_{p,q}$ be the $(p,q)$-cable\footnote{We adopt the convention that the $K_{p,q}$ lives on the bounday of a tubular neighbourhood of $K$, in the homology class of the curve $p\lambda + q\mu$, where $\lambda$ and $\mu$ are the Seifert longitute and meridian of $K$, respectively. We can and will always assume $p>0$; since $K_{1,q} = K$ for every $q$, we exclude all these cases.} of a knot $K$ and define:
\[
{\acc}(K):=\left\{\frac{q}{p} \mid p > 1, q > 0, S^3_{n}(K_{p,q}) \ \text{bounds a rational ball for some} \ n\in\N\right\}.
\]
We have ${\acc}={\acc}(\unknot) \cap [1,\infty)$, where $\unknot$ is the unknot: that is, we think of $\acc$ as the set of cabling slopes $q/p > 1$ for which the corresponding cable of the unknot has a surgery that bounds a rational homology ball. In fact, $\acc(\unknot) = \acc \cup \{r^{-1} \mid r \in \acc\} \cup \{\frac1p \mid p \ge 2\}$, where the last family comes from the fact that $\unknot_{p,1} = \unknot$ for every $p$.

With this notation in place we can state the following theorem, which explains and generalizes the appearance of $1$ and $4$ as accumulation points for ${\acc}$.

\begin{te}\label{rationalcob}

Let $K \subset S^3$ be a knot and $n$ an integer. For every integer $m>1$ there exists a rational homology cobordism between $S^3_n(K)$ and $S^3_{m^2n}(K_{m,mn \pm 1})$.
In particular, if $S^3_n(K)$ bounds a rational homology ball then $n$ is a two-sided accumulation point for ${\acc}(K)$.
\end{te}

Indeed, by work of Lisca~\cite{Lisca-ribbon}, $1$ and $4$ are the only positive integers $n$ for which $n$--surgery on the unknot bounds a rational homology ball. We have not been able to establish a similar relationship with rational slopes; however, we observe that $S^3_{p^2/{kp-1}}(\unknot)$ bounds a rational homology ball for every $p$ prime and $k$ coprime with $p$, and so the set of slopes $r$ such that $S^3_r(\unknot)$ bounds a rational homology ball is actually dense in $\R$.

The appearance of the numbers $\phi^2$, $\phi^4$, and $\psi$ is quite mysterious; their square roots $\phi$, $\phi^2 = \frac{3+\sqrt{5}}2$, and $\sqrt\psi = 1+\sqrt{2}$ (which is also known as the silver ratio), are among the numbers that are worst approximated by rational numbers (this dates back to work of Markov~\cite{Markov} and Hurwitz~\cite{Hurwitz}, see~\cite[Theorem~193]{HardyWright}).

As stated in Corollary~\ref{accumulationpoints} we have ${\acc}\subset[1,\frac{22}{3}]$. The fact that ${\acc}$ is bounded was first shown in~\cite[Theorem~1.3]{AcetoGolla} using tools from Heegaard Floer homology. Our next result generalizes this property.

\begin{te}\label{bound}
For every knot $K \subset S^3$ the set ${\acc}(K)$ is bounded.
\end{te}

We will give concrete bounds on $\acc(K)$ in Theorems~\ref{t:p<Nq} and~\ref{t:q<Np}; these bounds will be expressed in terms of the concordance invariant $\nu^+$ defined by Hom and Wu using Heegaard Floer homology~\cite{HomWu}. In particular, we will show that if $\nu^+(K) > 0$, then $\acc(K)$ is also bounded away from $0$.

In a different direction, we look at fusion numbers and at applications to the slice-ribbon conjecture for Montesinos knots and links.
The \emph{fusion number} of a ribbon knot $K$ is the minimal number of bands needed to build a ribbon disc for $K$.
Recall that every small Seifert fibered space is the double cover of $S^3$ branched over a Montesinos knot or a link; in our case, $S^3_n(T_{p,q})$ is a knot if $n$ is odd, and a $2$--component link if $n$ is even. 
We will show the following.

\begin{prop}\label{p:ribbon}
Let $M(r_1,r_2,r_3)\subset S^{3}$ be a Montesinos link whose branched double cover is $S^3_n(T_{p,q})$ for $(p,q;n) \in \Gc$. Then $M(r_1,r_2,r_3)$ bounds a surface (with no closed components) of Euler characteristic $1$ properly embedded in $B^{4}$. Moreover, if $M(r_1,r_2,r_3)$ is a knot, then it is ribbon and has fusion number $1$.
\end{prop}

In particular, if $M \subset S^3$ is a Montesinos knot whose branched double cover is $S^3_n(T_{p,q})$ for some positive integers $p$, $q$, $n$ with $\gcd(p,q) = 1$, then $M$ is slice if and only if it is ribbon, which in turn happens if and only if $(p,q;n)\in \Gc$. The proposition above parallels Lisca's statement for $2$--bridge links~\cite{Lisca-ribbon}; however, note that in our case the restriction that the branched double cover of $M$ is $S^3_n(T_{p,q})$ for some $p,q,n$ is quite strong.

\subsection{Comments and further questions}

In this paper we only consider \emph{positive} integral surgeries on \emph{positive} torus knots. Since bounding a rational homology ball is a property which is independent 
of the orientation, the only other case to consider is that of \emph{negative} integral surgeries on \emph{positive} torus knots.
Since $S^3_{-n}(T_{p,q})=-S^3_{n}(-T_{p,q})$ and $\nu^+(-T_{p,q})=0$, it follows from~\cite[Theorem~5.1]{AcetoGolla} that if $S^3_{-n}(T_{p,q})$ bounds a rational homology ball then $n\in\{1,4\}$.
However, even with this restriction our techniques fail to obstruct the existence of rational homology balls and standard constructions (of rational balls) 
do not seem to apply. Two examples are representative of this issue. First, the manifold $S^3_{-1}(T_{2,3})$, which is the Brieskorn sphere $\Sigma(2,3,7)$, is known
to bound a rational homology ball~\cite{FintushelStern}, but the construction of the rational homology ball is highly nontrivial. Indeed, since the Rokhlin invariant of $\Sigma(2,3,7)$ is non-trivial, any handle decomposition of such a rational homology ball must contain a 3--handle.
All the rational homology balls we construct comprise solely 1-- and 2--handles.
Secondly, it is remarkable that (to the best of our knowledge) it is still unknown whether or not the manifold $S^3_{-4}(T_{2,3})$ bounds a rational homology ball.
We also note that the techniques employed in this paper do not obstruct $S^3_{-1}(T_{p,q})$ or $S^3_{-4}(T_{p,q})$ from bounding a rational homology ball: indeed, the negative definite plumbing corresponding to these surgeries is a blow up of the surgery trace, and therefore its intersection form is $\langle -1\rangle^{\oplus n}$ or $\langle -4 \rangle\oplus \langle -1\rangle^{\oplus n-1}$, both of which embed in a diagonal lattice in codimension 0.

While the proof of Theorem~\ref{maintheorem} requires a significant effort with lattice embeddings, numerical evidence suggests that $S^3_n(T_{p,q})$ bounds a rational homology ball if and only if it is unobstructed by Heegaard Floer correction terms.
Lens spaces exhibit the same behavior: they bound a rational homology ball if and only if they pass the lattice obstruction (in their case, with both orientations)~\cite{Lisca-ribbon}, and this happens if and only if they pass the correction terms criterion~\cite{Greene}.
The common point between lens spaces and Seifert fibered rational homology spheres is that they both bound \emph{sharp} 4--manifolds~\cite{OzsvathSzabo-plumbed}.
It is therefore natural to ask the following question.

\begin{q}\label{q:dinvariants-converse}
Does $S^3_n(T_{p,q})$ bound a rational homology ball if and only if it has $\sqrt{n}$ vanishing correction terms?
\end{q}

Since correction terms of $S^3_n(T_{p,q})$ are computed in terms of the semigroup generated by $p$ and $q$, i.e. the set $\{hp + kq \mid h,k \in \Z_\ge 0\}$, arithmetics might play a non-trivial role in answering Question~\ref{q:dinvariants-converse}.
Also, as mentioned above, it is intriguing to understand whether non-integral accumulation points for ${\acc}(K)$ carry topological significance. In the case of $\acc(\unknot)$, we have no explanation for the appearance of $\phi^2$, $\phi^4$, and $\psi$, but it might be related to measures of approximability by rational numbers.

Since Theorem~\ref{bound} tells us that $\acc(K)$ is bounded, it is natural to ask about its cardinality.

\begin{q}
Can ${\acc}(K)$ be finite? Can it be empty?
\end{q}

In Theorem~\ref{rationalcob} we saw that if $S^3_n(K)$ bounds a rational homology ball then $n$ is a two-sided integral accumulation point for ${\acc}(K)$. It is natural to ask about the converse.

\begin{q}
Is there a knot $K$ such that ${\acc}(K)$ has a one-sided integral accumulation point?
Suppose $n$ is a two-sided integral accumulation point for ${\acc}(K)$. Does $S^3_n(K)$ bound a rational ball?
\end{q}

Note that Theorem~\ref{maintheorem} implies that the answer is affirmative if $K$ is the unknot.

Finally, as mentioned above, some of the families and exceptional cases listed above arise as boundary of knot traces embedded in $\CP^2$. For some of the remaining families, namely (1) and the rest of (3), it is easy to verify that the corresponding knot traces embed in $\CP^2$ (Propositions~\ref{embedding-fam12} and~\ref{embedding-fam3}).
For all other cases, the following argument shows that all their knot traces embed in a homotopy $\CP^2$: since all our rational balls, as well as Lisca's rational balls, are constructed only using handles of index at most 2, if we glue the knot trace $X_n(T_{p,q})$ to the corresponding rational ball $B_{p,q;n}$ turned upside down, we have a handle decomposition of a 4--manifold $Z_{p,q;n}$ without 1--handles. In particular, $Z_{p,q;n}$ is simply connected, it has Euler characteristic $\chi(Z_{p,q;n}) = \chi(X_n(T_{p,q})) + \chi(B_{p,q;n}) = 3$, and $b^+_2(Z_{p,q;n}) \ge 1$ (since it contains the trace of a positive surgery), and therefore it is a homotopy $\CP^2$.
We believe that the answer to the following question is affirmative.

\begin{q}
Does $S^3_n(T_{p,q})$ bound a rational homology ball if and only if $X_n(T_{p,q})$ embeds in $\CP^2$?
\end{q}

\subsection*{Outline of the paper}
We start in Section~\ref{accpoints} with the proof of Theorem~\ref{rationalcob}. This result is independent of the rest of the paper. Section~\ref{surgtolatt} associates to the manifolds $S^{3}_{n}(T_{p,q})$ two families of plumbing graphs (Proposition~\ref{seifertplumbing}). The question of whether or not $S^{3}_{n}(T_{p,q})$ bounds a rational homology ball will be addressed via a careful analysis of a lattice embedding problem concerning the aforementioned graphs. The answer to the embedding problem can be found in Proposition~\ref{p:mainA}. Section~\ref{surgtolatt} culminates with the proof of our main theorem, modulo the work in Section~\ref{latticeemb}. In Section~\ref{exceptional} we construct rational homology balls associated to the graphs in Proposition~\ref{p:mainA} which admit an embedding (Proposition~\ref{bounding}). The next section, Section~\ref{cablesHF}, collects some results on the more general problem of cables of arbitrary knots bounding rational homology balls, and proves Theorem~\ref{bound}.
The longest section in this article, Section~\ref{latticeemb}, is devoted to the highly technical proof of Proposition~\ref{p:mainA}. Finally, for completeness, we have added an appendix where we determine which lens spaces (and sums thereof) that bound rational homology balls can be obtained as surgeries on torus knots.

\subsection*{Acknowledgements} We would like to thank Frank Swenton for developing the Kirby calculator software (KLO)~\cite{KLO}, which helped us in several proofs. We would also like to thank Brendan Owens for many interesting comments. The third author was partially supported by NKFIH Grant K112735 and by the ERC Advanced Grant LDTBud.
\end{section}


\section{Accumulation points}\label{accpoints}
Our study of which surgeries on torus knots bound rational homology balls has led to the discovery of an intriguing relationship between surgery coefficients and cables of knots. Indeed, as stated in Theorem~\ref{rationalcob}, if $K$ is a knot in $S^3$ and $K_{p,q}$ denotes its $(p,q)$-cable, the set of rational points
\[
{\acc}(K):=\left\{\frac{q}{p} \mid p > 1, q > 0, S^3_{n}(K_{p,q}) \ \text{bounds a rational ball for some} \ n\in\N\right\}.
\]
has accumulation points corresponding to all integers $m$ for which $S^3_{m}(K)$ bounds a rational homology ball.
This follows from the fact that for each knot $K$ and all integers $m > 1$ and $n$, we can construct a rational homology cobordism from $S^3_n(K)$ to $S^3_{m^2n}(K_{m,mn + 1})$. A generalization of this fact when $n=0$ appeared previously in~\cite[Theorem~2.1 and Corollary~2.3]{CFHH}.

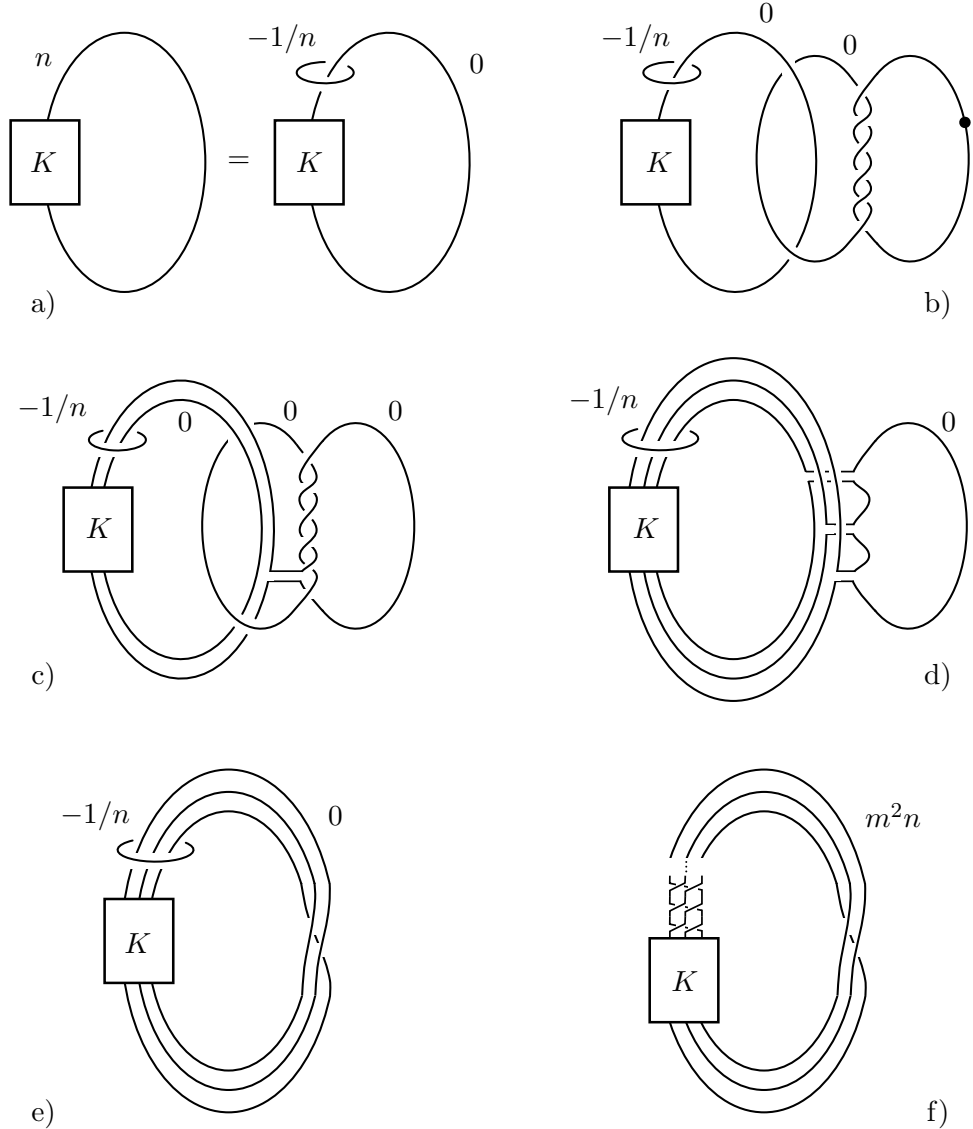
\begin{figure}[!htbp]
\labellist
\pinlabel $K$ at 16 447
\pinlabel a) at 16 380
\pinlabel $n$ at 16 495
\pinlabel $=$ at 107 447
\pinlabel $K$ at 140 447
\pinlabel $-1/n$ at 127 505
\pinlabel $0$ at 218 493
\pinlabel $-1/n$ at 292 505
\pinlabel $K$ at 303 447
\pinlabel b) at 433 380
\pinlabel $0$ at 353 517
\pinlabel $0$ at 392 502
\pinlabel $K$ at 42 275
\pinlabel c) at 16 205
\pinlabel $-1/n$ at 20 331
\pinlabel $0$ at 82 325
\pinlabel $0$ at 131 330
\pinlabel $0$ at 180 330
\pinlabel $K$ at 297 275
\pinlabel d) at 433 205
\pinlabel $-1/n$ at 277 335
\pinlabel $0$ at 438 325
\pinlabel $K$ at 60 80
\pinlabel e) at 16 0
\pinlabel $-1/n$ at 40 140
\pinlabel $0$ at 152 140
\pinlabel f) at 433 0
\pinlabel $K$ at 315 62
\pinlabel $m^2n$ at 412 140
\endlabellist
\centering
\includegraphics[scale=.80]{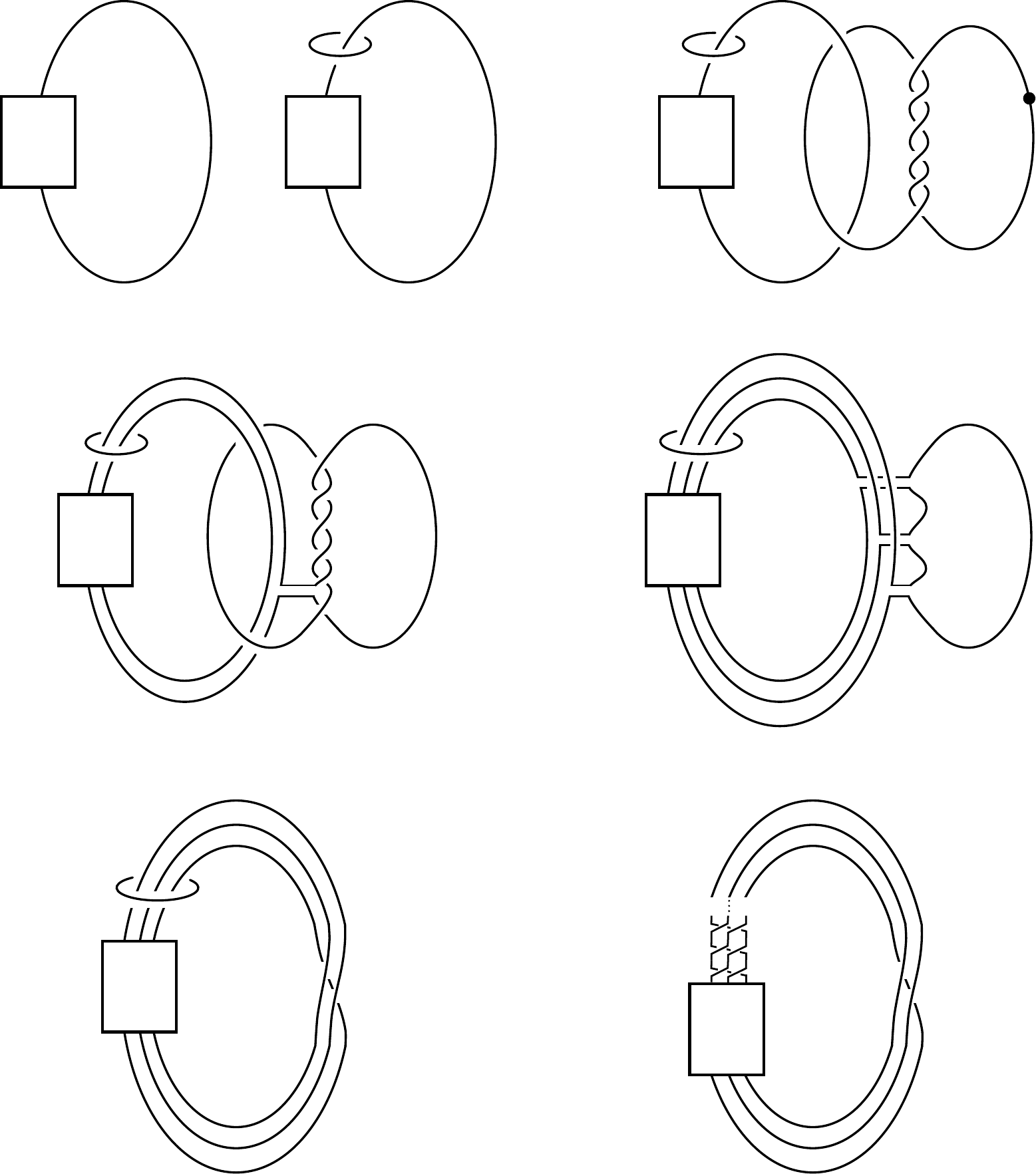}
\caption{A rational homology cobordism.}\label{f:cabling}
\end{figure}

\begin{proof}[Proof of Theorem~\ref{rationalcob}]
The second assertion is immediate from the first: since $S^3_{m^2n}(K_{m,mn+1})$ is rational homology cobordant to $S^3_n(K)$, if the latter bounds a rational homology ball, then so does the former. (Concretely: one such ball is obtained gluing any rational ball whose boundary is $S^3_n(K)$ with the rational cobordism of the first part of the statement.)

To prove the first assertion, we first reduce to the case of $S^3_{nm^2}(K_{m,mn+1})$. To this end, suppose that the statement holds with the $+$ sign for every $K$, $m$, and $n$ as in the statement. Then, $S^3_{nm^2}(K_{m,mn-1}) = -S^3_{-nm^2}(-K_{m,mn+1})$, where $-J$ denotes the mirror of $J$, and the latter is rational homology cobordant to $-S^3_{-n}(-K) = S^3_n(K)$ by hypothesis.

We now proceed to build the rational cobordism from $S^3_n(K)$ to $S^3_{nm^2}(K_{m,mn+1})$, using one 1-handle and one 2-handle.

A surgery diagram for $S^3_n(K)$ is given in Figure \ref{f:cabling}(a), where the right and left pictures are related by a slam dunk, or, equivalently, a Rolfsen twist. We build a rational
homology cobordism by attaching a 1-handle and a 2-handle to $S^3_n(K) \times I$ as in Figure \ref{f:cabling}(b). Note that in the pictures we draw the case $m=3$;
for the general case the 2-handle in Figure \ref{f:cabling}(b) will run over the 1-handle $m$ times.
(The 2-handle is 0-framed; its attaching curve and the dotted circle representing the 1-handle form a $(2,2m)$-torus link.)

Now to see the 3-manifold resulting from the cobordism we surger the 1-handle to a 0-framed 2-handle, and slide this new 2-handle $m$ times over the knot $K$. 
The first such slide is indicated in Figure \ref{f:cabling}(c). After $m$ slides the middle 2-handle will be a 0-framed meridian of $K$, and by slam dunking we can remove both of these components 
from the surgery diagram, resulting in Figure \ref{f:cabling}(d). An isotopy gives Figure \ref{f:cabling}(e), where in the general case there will be $m$ strands with a $1/m$ twist in them. Lastly we perform a Rolfsen 
twist on the unknotted component with surgery coefficient $-1/n$ to remove this component from the diagram. This puts $n$ full twists in the $m$ strands and changes the surgery coefficient
from 0 to $m^2n$, giving Figure \ref{f:cabling}(f). In Figure \ref{f:cabling}(f) we recognize the resulting knot to be $K_{m,mn+1}$, completing the proof.
\end{proof}

In~\cite{AL} the first and third authors studied the natural homomorphism $\psi: \Theta^3_\mathbb{Z} \rightarrow \Theta^3_\mathbb{Q}$
from the integral homology cobordism group to the rational homology cobordism group, and showed that no lens space (or connected sum of lens spaces) represents a nontrivial element in the image.
In fact, in general it appears difficult to produce families of interesting examples in the image.
Here we observe that Theorem~\ref{rationalcob} gives a nice way to find rational homology spheres that do belong to the image of $\psi$.

\begin{co}
For any knot $K \subset S^3$ and integer $m>1$, the manifold $S^3_{\pm m^2}(K_{m,\pm m\pm 1})$ belongs to the image of $\psi$.
\end{co}

\begin{proof}
By Theorem~\ref{rationalcob}, $S^3_{\pm m^2}(K_{m,\pm m\pm 1})$ is rational homology cobordant to $S^3_{\pm 1}(K)$, an integral homology sphere.
\end{proof}
\begin{section}{From surgeries to lattices}\label{surgtolatt}
The aim of this section is to reduce Theorem \ref{maintheorem} to a series of technical statements regarding integral lattices.
Recall that surgeries on torus knots are Seifert fibered spaces. More precisely, by \cite[Lemma~4.4]{OwensStrle} (see also~\cite{Moser}), we have
\[
S^3_{n}(T_{p,q})=Y\left(2;\frac{p}{q^*},\frac{q}{p^*},\frac{pq-n}{pq-n-1}\right),
\]
where $q^*$ is uniquely determined by $1\leq q^*\leq p-1$ and $qq^*\equiv 1 \pmod p$ (similarly for $p^*$ where we exchange the roles of $p$ and $q$). Up to orientation reversal, any such manifold arises as the boundary of a positive definite plumbed $4$-manifold. Here our convention for three-legged Seifert fibered spaces is the one described in Figure~\ref{f:SFS-convention}.

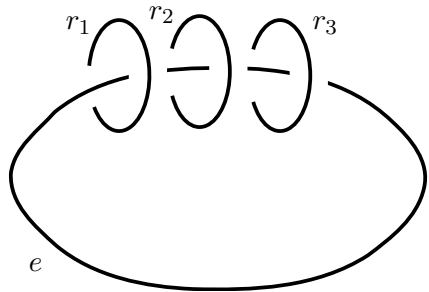
\begin{figure}
\labellist
\pinlabel $e$ at 10 10
\pinlabel $r_1$ at 26 100
\pinlabel $r_2$ at 57 104
\pinlabel $r_3$ at 118 100
\endlabellist
\includegraphics{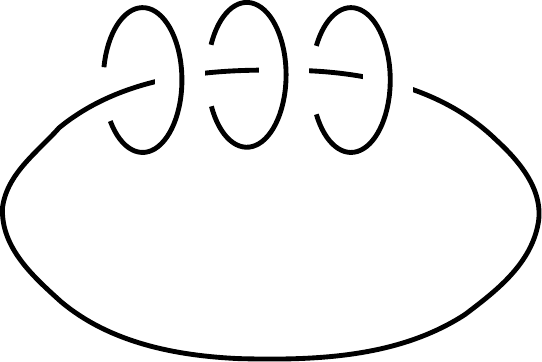}
\caption{A surgery description of the Seifert fibered space $Y(e;r_1,r_2,r_3)$.}\label{f:SFS-convention}
\end{figure}

We describe these 4-manifolds in the next proposition. The easy proof is left as an exercise for the reader.
We refer to \cite[Section~2]{AcetoGolla} for more details on plumbed 4-manifolds.

\begin{prop}\label{seifertplumbing}
Let $n\in\N\setminus\{pq-1,pq,pq+1\}$. Then, the manifold $S^3_{n}(T_{p,q})$ (up to possibly reversing orientation) can be described as the boundary of the positive definite plumbed $4$-manifold associated to one of the following graphs. 
\begin{itemize}
\item If $1\leq n\leq pq-2$ the plumbed 4-manifold whose boundary is $S^3_{n}(T_{p,q})$ is
associated to the graph:
 \[
  \begin{tikzpicture}[xscale=1,yscale=-0.5]
  	\node (A2_-1) at (-2,3) {$\Gamma_{1}:$};
	\node (A0_4) at (4, 0) {$b_1$};
    \node (A0_6) at (6, 0) {$b_m$};
    \node (A0_7) at (7, 0) {$k+1$};
    \node (A1_4) at (4, 1) {$\bullet$};
    \node (A1_5) at (5, 1) {$\dots$};
    \node (A1_6) at (6, 1) {$\bullet$};
    \node (A1_7) at (7, 1) {$\bullet$};
    \node (A2_0) at (0, 2) {$a_n$};
    \node (A2_2) at (2, 2) {$a_1$};
    \node (A2_3) at (3, 2) {$2$};
    \node (A3_0) at (0, 3) {$\bullet$};
    \node (A3_1) at (1, 3) {$\dots$};
    \node (A3_2) at (2, 3) {$\bullet$};
    \node (A3_3) at (3, 3) {$\bullet$};
    \node (A4_4) at (4, 4) {$2$};
    \node (A4_6) at (6, 4) {$2$};
    \node (A5_4) at (4, 5) {$\bullet$};
    \node (A5_5) at (5, 5) {$\dots$};
    \node (A5_6) at (6, 5) {$\bullet$};
\draw[decorate,decoration={brace,amplitude=7pt,mirror},xshift=0.4pt,yshift=-0.3pt](3.9,5.3) -- (6.1,5.3) node[black,midway,yshift=-0.5cm] {\footnotesize $N$};
    \path (A1_4) edge [-] node [auto] {$\scriptstyle{}$} (A1_5);
    \path (A3_0) edge [-] node [auto] {$\scriptstyle{}$} (A3_1);
    \path (A1_6) edge [-] node [auto] {$\scriptstyle{}$} (A1_7);
    \path (A1_5) edge [-] node [auto] {$\scriptstyle{}$} (A1_6);
    \path (A3_3) edge [-] node [auto] {$\scriptstyle{}$} (A1_4);
    \path (A5_4) edge [-] node [auto] {$\scriptstyle{}$} (A5_5);
    \path (A3_2) edge [-] node [auto] {$\scriptstyle{}$} (A3_3);
    \path (A3_3) edge [-] node [auto] {$\scriptstyle{}$} (A5_4);
    \path (A3_1) edge [-] node [auto] {$\scriptstyle{}$} (A3_2);
    \path (A5_5) edge [-] node [auto] {$\scriptstyle{}$} (A5_6);
   \end{tikzpicture}
  \]

\item If $ n\geq pq+2$ the plumbed 4-manifold whose boundary is $-S^3_{n}(T_{p,q})$ is associated to the graph:
 \[
  \begin{tikzpicture}[xscale=1,yscale=-0.5]
    \node at (-2,3) {$\Gamma_{2}:$};
    \node (A0_4) at (4, 0) {$a_1$};
    \node (A0_6) at (6, 0) {$a_n+1$};
    \node (A0_7) at (7, 0) {$2$};
    \node (A0_9) at (9,0)  {$2$};
    \node (A1_4) at (4, 1) {$\bullet$};
    \node (A1_5) at (5, 1) {$\dots$};
    \node (A1_6) at (6, 1) {$\bullet$};
    \node (A1_8) at (8, 1) {$\dots$};
    \node (A1_9) at (9, 1) {$\bullet$};
    \node (A2_0) at (0, 2) {$b_m$};
    \node (A1_7) at (7, 1) {$\bullet$};
    \node (A1_6) at (6, 1) {$\bullet$};
    \node (A2_2) at (2, 2) {$b_1$};
    \node (A2_3) at (3, 2) {$2$};
    \node (A3_0) at (0, 3) {$\bullet$};
    \node (A3_1) at (1, 3) {$\dots$};
    \node (A3_2) at (2, 3) {$\bullet$};
    \node (A3_3) at (3, 3) {$\bullet$};
    \node (A4_4) at (4, 4) {$2$};
    \node (A4_6) at (6, 4) {$2$};
    \node (A5_4) at (4, 5) {$\bullet$};
    \node (A5_5) at (5, 5) {$\dots$};
    \node (A5_6) at (6, 5) {$\bullet$};\draw[decorate,decoration={brace,amplitude=7pt,mirror},xshift=0.4pt,yshift=-0.3pt](6.9,1.3) -- (9.1,1.3) node[black,midway,yshift=-0.5cm] {\footnotesize $k-1$};
\draw[decorate,decoration={brace,amplitude=7pt,mirror},xshift=0.4pt,yshift=-0.3pt](3.9,5.3) -- (6.1,5.3) node[black,midway,yshift=-0.5cm] {{\footnotesize $N$}};
    \path (A1_4) edge [-] node [auto] {$\scriptstyle{}$} (A1_5);
    \path (A1_6) edge [-] node [auto] {$\scriptstyle{}$} (A1_7);
    \path (A1_8) edge [-] node [auto] {$\scriptstyle{}$} (A1_9);
    \path (A1_7) edge [-] node [auto] {$\scriptstyle{}$} (A1_8);
    \path (A3_0) edge [-] node [auto] {$\scriptstyle{}$} (A3_1);
    \path (A1_5) edge [-] node [auto] {$\scriptstyle{}$} (A1_6);
    \path (A3_3) edge [-] node [auto] {$\scriptstyle{}$} (A1_4);
    \path (A5_4) edge [-] node [auto] {$\scriptstyle{}$} (A5_5);
    \path (A3_2) edge [-] node [auto] {$\scriptstyle{}$} (A3_3);
    \path (A3_3) edge [-] node [auto] {$\scriptstyle{}$} (A5_4);
    \path (A3_1) edge [-] node [auto] {$\scriptstyle{}$} (A3_2);
    \path (A5_5) edge [-] node [auto] {$\scriptstyle{}$} (A5_6);
  \end{tikzpicture}
\]
\end{itemize}
where the numbers $k$ and $N$ and the strings of integers $(a_1,\dots,a_n)$, $(b_1,\dots,b_m)$ are uniquely determined by the following: 
\begin{itemize}
 \item $k>0$ satisfies $q=kp+r$ with $1\leq r\leq p-1$;
 \item $N=|pq-n|-1 \ge 1$;
 \item $[b_m,\dots,b_1]^-:=b_m-\frac{1}{b_{m-1}-\frac{1}{\dots}}=\frac{p}{p-r} \ \ ; \ \ b_i\geq 2 \ \ \text{for each} \ i$;
 \item $[a_n,\dots,a_1]^-:=a_n-\frac{1}{a_{n-1}-\frac{1}{\dots}}=\frac{p}{r} \ \ ; \ \ a_i\geq 2 \ \ \text{for each} \ i$.
\end{itemize} 
 \end{prop}
An \emph{integral lattice} is a pair $(G,Q)$ where $G$ is a finitely generated, free, Abelian group and $Q:G \times G\rightarrow\Z$ a symmetric bilinear form. A morphism of integral lattices is a homomorphism of Abelian groups which preserves the bilinear form.
An \emph{embedding} of integral lattices is an injective morphism. We denote by $(\Z^n,\mathrm{Id})$ the standard positive definite lattice of rank $n$. To a given 4-manifold $X$ we can associate 
an integral lattice $(H_2(X;\Z)/\Tor,Q_X)$ where $Q_X$ is the intersection form defined on the second homology group with integral coefficients. If a 4-manifold arises from a plumbing graph $\Gamma$ we denote the associated integral 
lattice by $(\Z\Gamma,Q_{\Gamma})$.

Our main ingredient for the proof of Theorem \ref{maintheorem} is an obstruction based on Donaldson's diagonalization theorem \cite{Donaldson, Donaldson2}. This obstruction relates the existence of a rational ball with morphisms between integral lattices.

\begin{prop}\label{Donaldsonobstruction}
 Let $Y$ be a rational homology sphere. Let $X$ be a smooth 4-manifold with positive-definite intersection form such that $\partial X=Y$. If $Y$ bounds a rational homology ball then there exists an embedding of integral lattices
\[
 (H_2(X;\Z)/\Tor,Q_X)\longrightarrow (\Z^m,\mathrm{Id})
\]
 with $m=\rk(H_2(X;\Z))$.
\end{prop}
With the notation from Proposition \ref{seifertplumbing} we can specialize the above statement to surgeries on torus knots. 
\begin{co}\label{Donaldsonobstructionforus}
 Let $n\in\N\setminus\{pq-1,pq,pq+1\}$. Suppose that $S^3_{n}(T_{p,q})$ bounds a rational homology ball. Then, there exists an embedding of integral lattices
\[
 (\Z\Gamma_i,Q_{\Gamma_i})\longrightarrow (\Z^m,\mathrm{Id})
\]
 where $m$ is the number of vertices of $\Gamma_i$ and the value of $i\in\{1,2\}$ is determined as in Proposition \ref{seifertplumbing}.
\end{co}
In order to state our main technical result, we need to introduce some notation. In both $\Gamma_1$ and $\Gamma_2$ the strings $(a_1,\dots,a_n)$ and $(b_1,\dots,b_m)$ satisfy  
\begin{equation*}
\frac1{[a_n,\dots,a_1]^-}+\frac1{[b_m,\dots,b_1]^-}=1.
\end{equation*}
We refer to such strings as \emph{complementary}. We will refer to the bottom right 2-chain in $\Gamma_1$ and $\Gamma_2$ as the $2$-leg, and to the other two linear connected components obtained after deleting the trivalent vertex as the qc-legs (where qc stands for \emph{quasi complementary}). The next proposition is the key ingredient in the proof of our main theorem, Theorem~\ref{maintheorem}.

\begin{prop}\label{p:mainA}
Let $\Gamma_1$ and $\Gamma_2$ be as in Proposition \ref{seifertplumbing}. Suppose further that $1\leq k \leq 7$ and that neither of the strings $(a_1,\dots,a_n)$ and $(b_1,\dots,b_m)$ is a $2$-chain. Finally, in the case of the graphs $\Gamma_{1}$ assume additionally that if $k=7$ then $N\geq 4$. Under these constraints,
\begin{enumerate}
\item The only lattices of the form $(\Z\Gamma_{1},Q_{\Gamma_{1}})$ which admit an embedding into the standard lattice of the same rank $(\Z^{\rk(\Gamma_{1})},\mathrm{Id})$ belong to the subfamily
\[
  \begin{tikzpicture}[xscale=1,yscale=-0.5]
	\node (A0_4) at (4, 0) {$3$};
	\node (A0_4) at (5, 0) {$6$};
    \node (A0_6) at (7, 0) {$6$};
    \node (A0_7) at (8, 0) {$5+1$};
    \node (A1_4) at (4, 1) {$\bullet$};
    \node (A1_5) at (5, 1) {$\bullet$};
    \node (A1_6) at (6, 1) {$\dots$};
    \node (A1_7) at (7, 1) {$\bullet$};
    \node (A1_8) at (8, 1) {$\bullet$};
    \node (A2_0) at (0, 2) {$2$};
    \node (A2_2) at (1, 2) {$3$};
    \node (A2_3) at (2, 2) {$2$};
    \node at (3, 2) {$2$};
    \node  at (-2, 2) {$2$};
    \node  at (-1, 2) {$2$};
    \node  at (0, 2) {$2$};
    \node  at (-4, 2) {$2$};
    \node (A3_-3) at (-4, 3) {$\bullet$};
    \node (A3_-2) at (-3, 3) {$\dots$};
    \node (A3_-1) at (-2, 3) {$\bullet$};
    \node (A3_0) at (-1, 3) {$\bullet$};
    \node (A3_1) at (0, 3) {$\bullet$};
    \node (A3_2) at (1, 3) {$\bullet$};
    \node (A) at (2, 3) {$\bullet$};
    \node (A3_3) at (3, 3) {$\bullet$};
    \node (A4_4) at (4, 4) {$2$};
    \node (A5_4) at (4, 5) {$\bullet$};
    \draw[decorate,decoration={brace,amplitude=7pt,mirror},xshift=0.4pt,yshift=-0.3pt](4.9,1.3) -- (7.1,1.3) node[black,midway,yshift=-0.5cm] {\footnotesize $r$};
   \draw[decorate,decoration={brace,amplitude=7pt,mirror},xshift=0.4pt,yshift=-0.3pt](-2.1,3.3) -- (1.1,3.3) node[black,midway,yshift=-0.5cm] {\footnotesize $\mathcal{P}$}; 
    \path (A1_4) edge [-] node [auto] {$\scriptstyle{}$} (A1_5);
    \path (A3_0) edge [-] node [auto] {$\scriptstyle{}$} (A3_1);
    \path (A1_6) edge [-] node [auto] {$\scriptstyle{}$} (A1_7);
    \path (A1_7) edge [-] node [auto] {$\scriptstyle{}$} (A1_8);
    \path (A1_5) edge [-] node [auto] {$\scriptstyle{}$} (A1_6);
    \path (A3_3) edge [-] node [auto] {$\scriptstyle{}$} (A1_4);
    \path (A3_3) edge [-] node [auto] {$\scriptstyle{}$} (A5_4);
    \path (A3_1) edge [-] node [auto] {$\scriptstyle{}$} (A3_0);
    \path (A3_1) edge [-] node [auto] {$\scriptstyle{}$} (A3_2);
    \path (A3_2) edge [-] node [auto] {$\scriptstyle{}$} (A);
    \path (A) edge [-] node [auto] {$\scriptstyle{}$} (A3_3);
    \path (A3_0) edge [-] node [auto] {$\scriptstyle{}$} (A3_-1);
    \path (A3_-1) edge [-] node [auto] {$\scriptstyle{}$} (A3_-2);
    \path (A3_-2) edge [-] node [auto] {$\scriptstyle{}$} (A3_-3);
   \end{tikzpicture}
  \]
where on the left qc-leg the pattern $\mathcal{P} := [2,2,2,3]$ is repeated $r>0$ times between two weight $2$ vertices.
\item The lattice $(\Z\Gamma_{2},Q_{\Gamma_{2}})$ admits no embedding into $(\Z^{\rk(\Gamma_{2})},\mathrm{Id})$.
\end{enumerate}
\end{prop}

We are now ready to present the proof of Theorem~\ref{maintheorem}. Several key facts in it will be stated and proved in subsequent sections.
 
\begin{proof}[Proof of Theorem~\ref{maintheorem}]
We want to determine which manifolds $S^3_n(T_{p,q})$ bound rational homology balls. We distinguish two main cases. In what follows, let $k = \lfloor \frac qp \rfloor$, i.e. $q = kp + r$ for some $1 \le r \le p-1$.
\begin{enumerate}
\item First, let us assume that $S^3_n(T_{p,q})$ is Seifert fibered with three exceptional fibers, i.e.\ that $n\in \N\setminus\{pq,pq\pm 1\}$~\cite{Moser}. 
	\begin{itemize}
	\item If $q\equiv \pm1 \bmod p$ then we can directly apply \cite[Theorem 1.4]{AcetoGolla}. This accounts for the families (1),(2),(3) and (5) in the statement. 
	\item If $q\not\equiv \pm1 \bmod p$ we argue as follows. On the one hand, it follows easily from Proposition~\ref{seifertplumbing} that, 
since $q\not\equiv \pm1 \bmod p$, neither $(a_1,\dots,a_n)$ nor $(b_1,\dots,b_m)$ is a $2$-chain.
Moreover, by Proposition~\ref{pruning1}, we know that in this case, if $S^3_n(T_{p,q})$ is the boundary of a rational homology ball, then $1\leq k \leq 7$.
Furthermore, Proposition~\ref{pruning2} guarantees that, if $S^3_n(T_{p,q})$ is associated to a graph in the family $\Gamma_{1}$ via Proposition~\ref{seifertplumbing} and bounds a rational homology ball, then its parameters do not satisfy $k=7$ and $N\in\{1,2,3\}$.

We want to use Corollary~\ref{Donaldsonobstructionforus} as an obstruction to $S^3_n(T_{p,q})$ bounding a rational homology ball. By the above remarks, it suffices to study the cases in the statement of Proposition~\ref{p:mainA}. By Corollary~\ref{Donaldsonobstructionforus}, the only manifolds $S^3_n(T_{p,q})$ which in this case might bound rational homology balls are the ones in the family in the first statement of Proposition~\ref{p:mainA}. In Proposition~\ref{bounding} we construct explicit rational homology balls bounding them. This is family (4) in the statement of Theorem~\ref{maintheorem}.
\end{itemize}
\item Now assume that $S^3_n(T_{p,q})$ is either a lens space or a connected sum of two lens spaces, i.e.\ that $n\in \{pq,pq\pm 1\}$. In this case Lisca's work (see~\cite{Lisca-ribbon} and~\cite{Lisca-sums}) provides us with a complete list of which lens spaces or sums of lens spaces bound rational balls. The next step is then to identify which manifolds in Lisca's lists arise as Dehn surgery on torus knots. This work is carried out in the appendix (see Proposition~\ref{reduciblebounding} and Proposition~\ref{lensspacesbounding}).
From this analysis we obtain the remaining families stated in Theorem \ref{maintheorem}.\qedhere
\end{enumerate}
\end{proof}
\end{section}
\section{Explicit rational homology balls, ribbon disks, and trace embeddings}\label{exceptional}

In this section we will exhibit rational homology balls corresponding to the triples in $\Gc$, i.e. those which yield three-legged Seifert fibered spaces. We show that these rational homology balls arise as double covers of the 4--ball branched over slice surfaces for Montesinos links, and, for triples in families (1)--(3) we show that the corresponding trace surgeries embed in $\CP^2$.

To fix the notation, we write $m(K)$ for the mirror of the knot $K \subset S^3$.

\subsection{Families (1)--(3)}

We begin with families (1) and (2). Recall that we denote with $X_n(K)$ the trace of the surgery along the knot $K$, i.e. the 4--manifold obtained by attaching an $n$--framed 2--handle to $B^4$ along $K$.
Here, for convenience, we temporarily drop our convention that in $T_{p,q}$ we always have $p < q$, so that we write $T_{p,p-1}$ instead of $T_{p-1,p}$

\begin{prop}\label{embedding-fam12}
The knot trace $X_{p^2}(T_{p,p\pm 1})$ embeds in $\CP^2$. Moreover, we can choose the embedding so that the complement is a rational homology ball built with one handle of each index 0, 1, and 2.
\end{prop}

\begin{proof}
For convenience, call $X = X_{p^2}(T_{p,p\pm 1})$ and $DX$ its double, viewed as the union of $-X$ and $X$ turned upside down. We note (but do not use) that $DX$ is diffeomorphic to $S^2 \times S^2$ if $p$ is even, and to $\CP^2 \# \CPbar$ if $p$ is odd.

$DX$ has a handle decomposition with two 2--handles, one attached along $K = m(T_{p,p\pm 1})$ with framing $-p^2$, and the other attached along a meridian of $K$ with framing $0$. This is what is depicted on the left in Figure~\ref{f:families12}, where we view $K$ as the closure of the $p$--braid $(\sigma_1\cdots\sigma_{p-1})^{-(p\pm 1)}$.

\begin{figure}
\labellist
\pinlabel $-\frac{p\pm 1}{p}$ at 24 59
\pinlabel $0$ at 68 0
\pinlabel $-p^2$ at 59 117
\pinlabel $\mp\frac{1}{p}$ at 137 59
\pinlabel $+1$ at 148 0
\pinlabel $0$ at 197 0
\pinlabel $0$ at 186 117
\pinlabel $\mp\frac{1}{p}$ at 265 59
\pinlabel $0$ at 325 0
\pinlabel $+1$ at 276 0
\endlabellist
\includegraphics{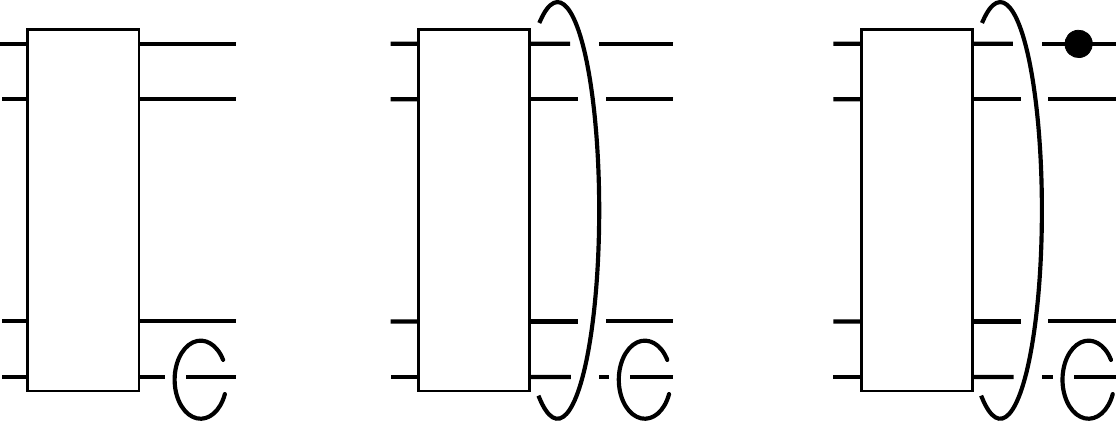}
\caption{Kirby calculus for embedding knot traces in $\CP^2$ for families (1) and (2). In all figures there is a 4--handle.}\label{f:families12}
\end{figure}
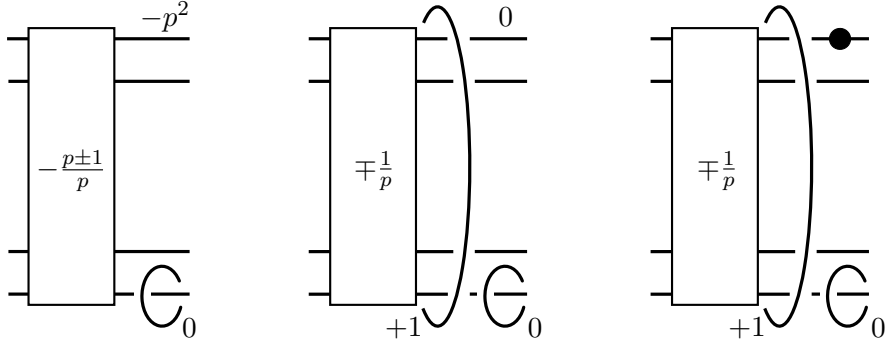

We perform a blow-up on $DX$ along the axis of the braid, and we obtain the handle decomposition in the middle of Figure~\ref{f:families12}. Note that the knot $K$ becomes unknotted and gets framing $0$; therefore, we can do a zero-dot surgery, and replace it with a carved 1--handle, as in Figure~\ref{f:families12} on the right. We also observe that since we only slid the $2$--handle with attaching curve $K$ along the $+1$--framed unknot, the new 4--manifold is the union of $-X\#\CP^2$ and $X$ along their boundary; in particular, the union of the $0$--framed $2$--handle attached along the meridian of $K$ and the $4$--handle is still diffeomorphic to $X$.

After sliding off the axis of the braid from the 1--handle, we can cancel the 1--handle with the meridian of the corresponding dotted curve, and we obtain a handle decomposition of $\CP^2$. Indeed, since the meridian of the dotted curve is 0--framed, the other 2--handle remains a $+1$--framed unknot. We did not need to do this explicitly, at the cost of using the Property P (by Property P, if a \emph{closed} 4-manifold admits a handle decomposition with a single 2--handle and no 1--handles or 3--handles, then it must be $\pm \CP^2$).

Therefore, the right-most picture shows a rational homology ball (the $0$--handle, the $1$--handle, and the $+1$--framed $2$--handle) embedded in $\CP^2$; its complement is the union of the other $2$--handle and the $4$--handle, which, as observed above, is diffeomorphic to $X$. This proves the claim.
\end{proof}

The argument for family (3) is quite similar, so we just outline it.

\begin{prop}\label{embedding-fam3}
The knot trace $X_{4p^2}(T_{p,4p\pm 1})$ embeds in $\CP^2$. Moreover, we can choose the embedding so that the complement is a rational homology ball built with one handle of each index 0, 1, and 2.
\end{prop}

\begin{proof}
We refer to Figure~\ref{f:family3}. On the left we have the double $DX$ of $X = X_{4p^2}(T_{p,4p\pm 1})$, where $X$ is the union of the $0$--framed $2$--handle and the $4$--handle.

\begin{figure}
\labellist
\pinlabel {\tiny $-\frac{4p\pm1}p$} at 24 95
\pinlabel {\tiny $0$} at 50 31
\pinlabel {\tiny $4p^2$} at 65 155
\pinlabel {\tiny $\mp\frac{1}p$} at 122 95 
\pinlabel {\tiny $0$} at 220 31
\pinlabel {\tiny $+1$} at 144 30
\pinlabel {\tiny $+1$} at 160 30
\pinlabel {\tiny $+1$} at 176 30
\pinlabel {\tiny $+1$} at 192 30
\pinlabel {\tiny $\mp\frac{1}p$} at 280 95 
\pinlabel {\tiny $0$} at 330 31
\pinlabel {\tiny $+1$} at 316 155
\pinlabel {\tiny $+2$} at 288 35
\pinlabel {\tiny $+2$} at 288 21
\pinlabel {\tiny $+2$} at 288 7
\pinlabel {\tiny $\mp\frac{1}p$} at 393 95 
\pinlabel {\tiny $0$} at 450 31
\pinlabel {\tiny $+1$} at 430 155
\pinlabel {\tiny $+2$} at 430 35
\pinlabel {\tiny $+2$} at 430 20
\pinlabel {\tiny $+2$} at 430 5
\pinlabel {\tiny $+1$} at 402 9
\pinlabel {\tiny $+0$} at 383 10
\pinlabel {\tiny $\mp\frac{1}p$} at 506 95 
\pinlabel {\tiny $0$} at 563 31
\pinlabel {\tiny $+1$} at 543 155
\pinlabel {\tiny $-3$} at 500 14
\endlabellist
\includegraphics[width=\textwidth]{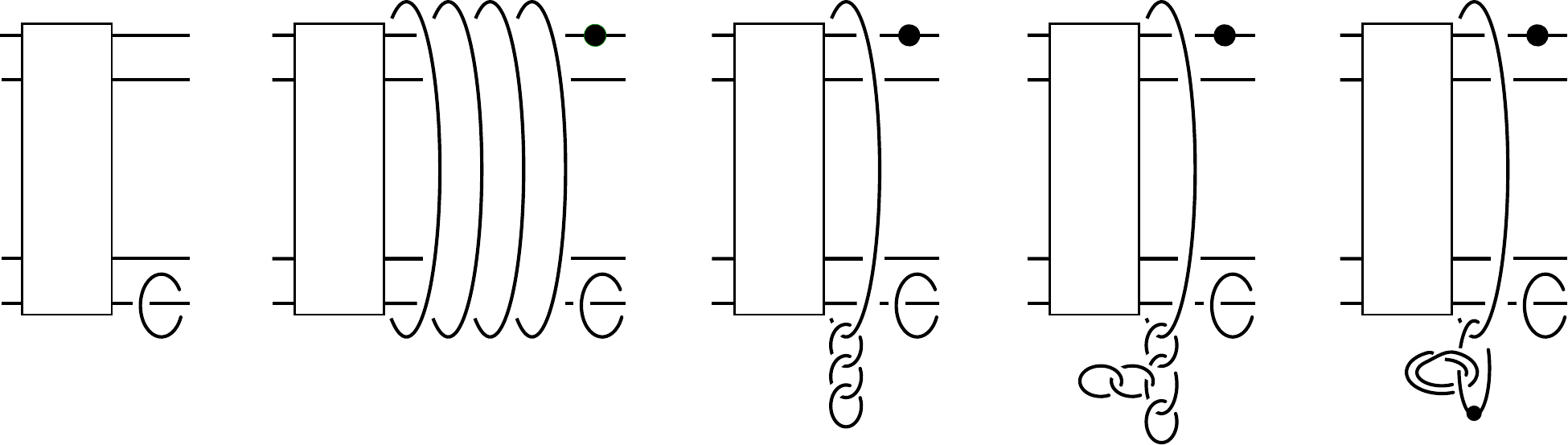}
\caption{Kirby calculus for embedding knot traces in $\CP^2$, for family (3). In all figures there is a 4--handle.}\label{f:family3}
\end{figure}
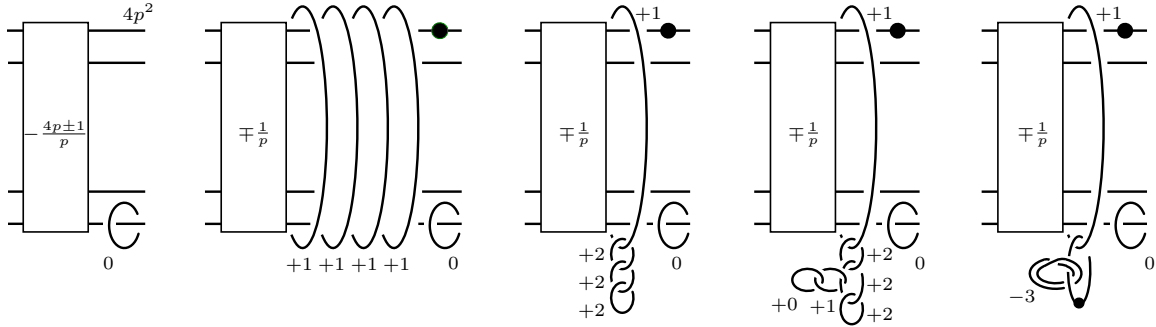

We blow up four times along the axis of the braid, obtaining a decomposition of $DX$ as $(-X \# 4\CP) \cup X$, and then we do a zero-dot surgery along the $0$--framed $2$--handle that corresponded to $m(T_{p,4p\pm1})$ (Figure~\ref{f:family3}(b)). We then slide the $+1$--framed $2$--handles one off the other (Figure~\ref{f:family3}(c)). We then add two extra $2$--handles, which preserve the decomposition of a $4$--manifold as $X'\cup X$ (Figure~\ref{f:family3}(c)), and we blow down the $+1$-framed unknotted attaching curves, to get to Figure~\ref{f:family3}(e). This exhibits a handle decomposition of a $4$--manifold as $X \cup Q$, where $Q$ is the rational homology ball constructed with the two $1$--handles, and two of the $2$--handles.

We have a cancelling pair of a $1$-- and a $2$--handle in $Q$; so $Q$ has a handle decomposition with only one $1$--handle and one $2$--handle. Moreover, either by using the property P or by handle-sliding and explicitly cancelling the remaining $1$--handle with the $2$--handle in $X$, we see that the $4$--manifold in Figure~\ref{f:family3}(e) is indeed $\CP^2$.
\end{proof}

\subsection{Family (4)}

In this section we will show that the manifolds $S^{3}_{n}(T_{p,q})$ obtained as surgeries on torus knots that correspond to the graphs in Proposition~\ref{p:mainA} (via Proposition~\ref{seifertplumbing}) do bound rational homology balls constructed with exactly one 1--handle and one 2--handle. We do \emph{not} verify that gluing the 2--ball and the trace of the cobordism we obtain an embedding in $\CP^2$.

The graphs under consideration are of the form
\[
  \begin{tikzpicture}[xscale=1,yscale=-0.5]
	\node (A0_4) at (4, 0) {$3$};
	\node (A0_4) at (5, 0) {$6$};
    \node (A0_6) at (7, 0) {$6$};
    \node (A0_7) at (8, 0) {$5+1$};
    \node (A1_4) at (4, 1) {$\bullet$};
    \node (A1_5) at (5, 1) {$\bullet$};
    \node (A1_6) at (6, 1) {$\dots$};
    \node (A1_7) at (7, 1) {$\bullet$};
    \node (A1_8) at (8, 1) {$\bullet$};
    \node (A2_0) at (0, 2) {$2$};
    \node (A2_2) at (1, 2) {$3$};
    \node (A2_3) at (2, 2) {$2$};
    \node at (3, 2) {$2$};
    \node  at (-2, 2) {$2$};
    \node  at (-1, 2) {$2$};
    \node  at (0, 2) {$2$};
    \node  at (-4, 2) {$2$};
    \node  at (-5, 3) {$\Lambda_{s}=$};
    \node (A3_-3) at (-4, 3) {$\bullet$};
    \node (A3_-2) at (-3, 3) {$\dots$};
    \node (A3_-1) at (-2, 3) {$\bullet$};
    \node (A3_0) at (-1, 3) {$\bullet$};
    \node (A3_1) at (0, 3) {$\bullet$};
    \node (A3_2) at (1, 3) {$\bullet$};
    \node (A) at (2, 3) {$\bullet$};
    \node (A3_3) at (3, 3) {$\bullet$};
    \node (A4_4) at (4, 4) {$2$};
    \node (A5_4) at (4, 5) {$\bullet$};
    \draw[decorate,decoration={brace,amplitude=7pt,mirror},xshift=0.4pt,yshift=-0.3pt](4.9,1.3) -- (7.1,1.3) node[black,midway,yshift=-0.5cm] {\footnotesize $s$};
   \draw[decorate,decoration={brace,amplitude=7pt,mirror},xshift=0.4pt,yshift=-0.3pt](-2.1,3.3) -- (1.1,3.3) node[black,midway,yshift=-0.5cm] {\footnotesize repeated $s$ times}; 
    \path (A1_4) edge [-] node [auto] {$\scriptstyle{}$} (A1_5);
    \path (A3_0) edge [-] node [auto] {$\scriptstyle{}$} (A3_1);
    \path (A1_6) edge [-] node [auto] {$\scriptstyle{}$} (A1_7);
    \path (A1_7) edge [-] node [auto] {$\scriptstyle{}$} (A1_8);
    \path (A1_5) edge [-] node [auto] {$\scriptstyle{}$} (A1_6);
    \path (A3_3) edge [-] node [auto] {$\scriptstyle{}$} (A1_4);
    \path (A3_3) edge [-] node [auto] {$\scriptstyle{}$} (A5_4);
    \path (A3_1) edge [-] node [auto] {$\scriptstyle{}$} (A3_0);
    \path (A3_1) edge [-] node [auto] {$\scriptstyle{}$} (A3_2);
    \path (A3_2) edge [-] node [auto] {$\scriptstyle{}$} (A);
    \path (A) edge [-] node [auto] {$\scriptstyle{}$} (A3_3);
    \path (A3_0) edge [-] node [auto] {$\scriptstyle{}$} (A3_-1);
    \path (A3_-1) edge [-] node [auto] {$\scriptstyle{}$} (A3_-2);
    \path (A3_-2) edge [-] node [auto] {$\scriptstyle{}$} (A3_-3);
   \end{tikzpicture}
  \]
and the corresponding parameters $n,p,q$ that describe the boundaries of the plumbed 4-manifolds can be defined recursively 
using the auxiliary sequence defined in the introduction
$\{R_s\}$ as
\[
\left\{
\begin{array}{l}
R_0 = 1,\\
R_1 = 3,\\
R_{s+1} = 6R_s - R_{s-1}.
\end{array}
\right.
\]

We now verify that this is the correct plumbing. Let $p = R_s$ and $q = R_{s+1}$. Since the surgery coefficient is $n=R_sR_{s+1}-2$, the corresponding graph is of type 1, with 2-chain of length $N = pq-n-1 = 1$.

Let us look at the two qc-legs. The sequence $R_s$ is easily checked to be increasing, and since $R_{s+1} = 6R_{s} - R_{s-1}$, we have that $q = 5p + r$, where $0 < r = p-R_{s-1} < p$, so that $k=5$. We want to compute the sequence $(b_1, \dots, b_m)$ first: we know that $[b_m, \dots, b_1]^- = \frac{p}{p-r}$; however, from the computation above, this is
\[
\frac{p}{p-r} = \frac{R_{s}}{R_{s-1}} = \frac{6R_{s-1}-R_{s-2}}{R_{s-1}} = 6 - \left(\frac{R_{s-1}}{R_{s-2}}\right)^{-1}.
\]
An easy induction shows that the corresponding $b_j$ is $6$ as long as $R_{j-1} > 1$, when $b_j = R_1 = 3$. This verifies that the top right leg is the same as in $\Lambda_s$.
Now we know that $(a_1,\dots,a_n)$ is dual to $(b_1,\dots,b_n)$, and Riemenschneier's point rule (that we recall in the Section~\ref{ss:riemen} below) readily recovers $(a_1,\dots,a_n)$ as claimed.

Indeed, from Proposition~\ref{seifertplumbing} we obtain that the boundary of the 4-manifold associated to $\Lambda_{s} $ is $S^{3}_{R_{s}R_{s+1}-2}(T_{R_{s},R_{s+1}})$, the fourth family in Theorem~\ref{maintheorem}.
\begin{prop}\label{bounding}
$S^{3}_{R_{s}R_{s+1}-2}(T_{R_{s},R_{s+1}})$ bounds a rational homology ball for each $s\geq0$.
\end{prop}
\begin{proof}
We show that $S^{3}_{R_{s}R_{s+1}-2}(T_{R_{s},R_{s+1}})$ bounds a rational homology ball by exhibiting an integral surgery from this manifold to $S^1 \times S^2$. Such a surgery corresponds to attaching a 2--handle to $\partial(S^1 \times B^3)$, and since the resulting 3-manifold boundary is a rational homology sphere, it is easy to check that the 4-manifold must be a rational homology ball. After representing the graph $\Lambda_s$ as a surgery diagram, the gray curve in Figure~\ref{f:exceptional_family1} shows an integral surgery to $S^1 \times S^2$. To see this, simply perform three blow-downs to arrive at Figure~\ref{f:exceptional_family2}. Note that after these blow-downs, the left qc-leg in the graph is completely unchanged, the central vertex now has weight 1, the short leg with one weight 2 vertex is no longer there and the top qc-leg has lost the first 3 weight vertex and now the former second vertex with weight 6 is linked to the old central vertex and has now weight 3.

It is immediate to check that the surgery coefficients on the strings of unknots on either side of the component with framing 1 in Figure~\ref{f:exceptional_family2} are complementary. This in turn implies that Figure~\ref{f:exceptional_family2} is a surgery diagram for $S^1 \times S^2$, since there is a sequence of blowdowns that terminates with a single 0--framed unknot (see \cite[Proposition 2.4]{Paolo}).
\end{proof}

\begin{re}
Note that it is possible, in principle, to construct the required rational homology ball from the surgery cobordism exhibited in the proof: just attach a $0$--framed $2$--handle along the meridian of the surgery curve and follow the same moves as in the previous proof. This gives an explicit handle decomposition of the ball described above.
\end{re}

\begin{figure}[h]
\labellist
\pinlabel $2$ at 130 20
\pinlabel $1$ at 350 20
\pinlabel $6$ at 335 185
\pinlabel $3$ at 260 185
\pinlabel $2$ at 185 185
\pinlabel $2$ at 110 185
\endlabellist
\centering
\includegraphics[scale=.70]{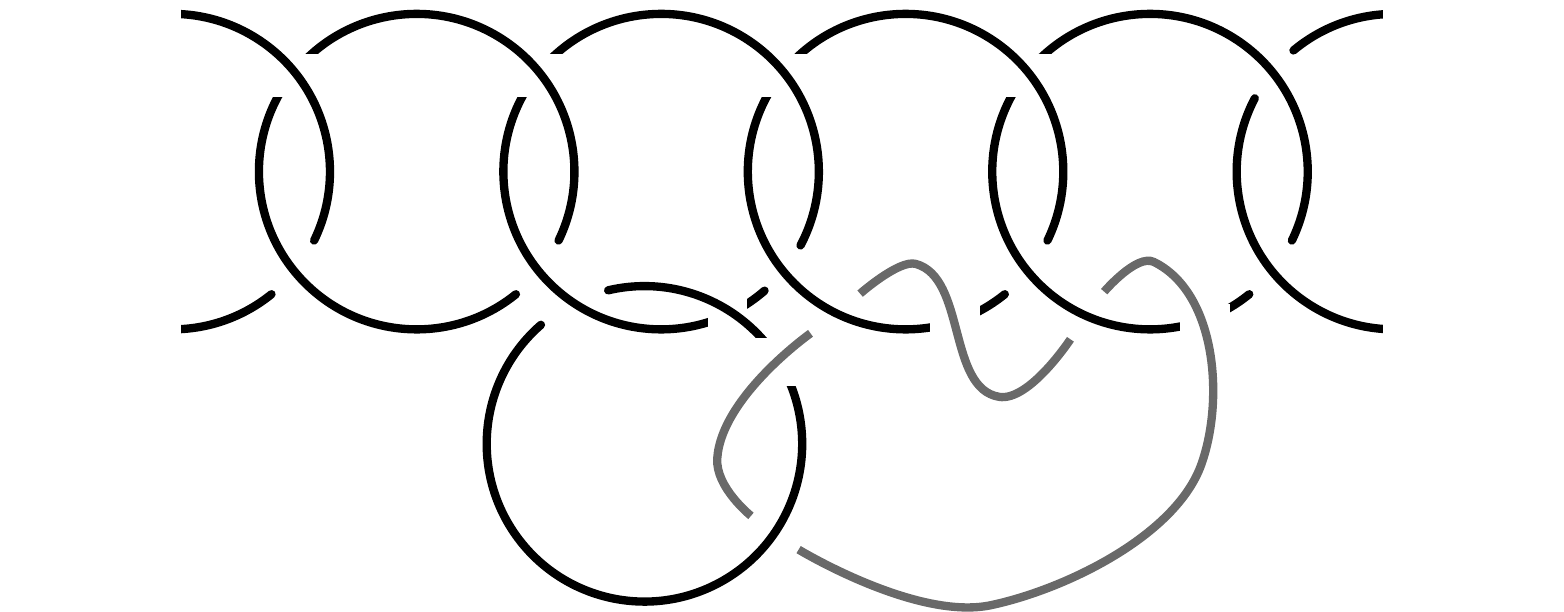}
\caption{A surgery from $\Lambda_s$ to $S^1 \times S^2$.}\label{f:exceptional_family1}
\end{figure}

\begin{figure}[h]
\labellist
\pinlabel $2$ at 120 120
\pinlabel $1$ at 190 120
\pinlabel $3$ at 255 120
\endlabellist
\centering
\includegraphics[scale=.70]{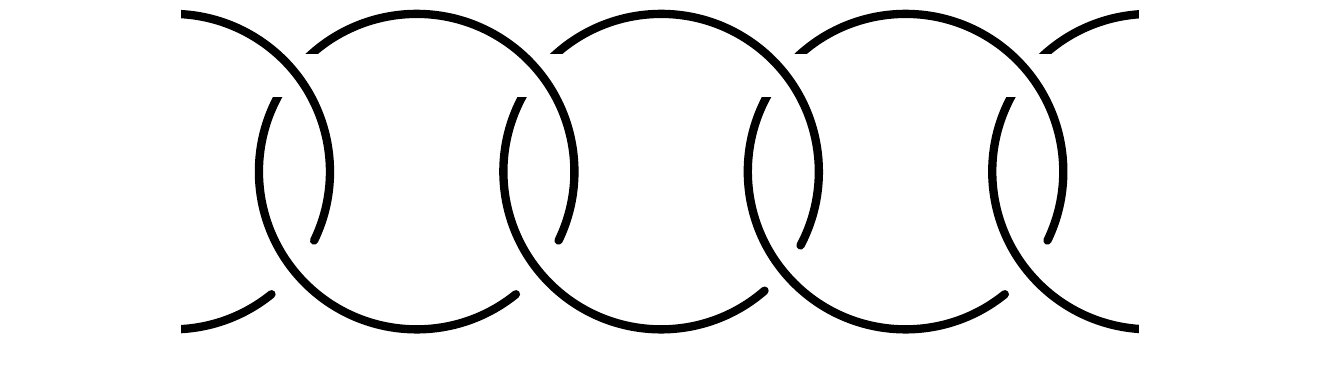}
\caption{A surgery diagram for $S^1 \times S^2$.}
\label{f:exceptional_family2}
\end{figure}

\subsection{The two sporadic cases in (5)}

Similar to what has been done in the previous subsection, we exhibit surgeries along knots in $S^3_{64}(T_{3,22})$ and $S^3_{256}(T_{6,43})$ yielding $S^1\times S^2$. The proof is analogous to the one in the previous subsection, so we omit it, and rather just display the two handle attachments. We also observe that we already know from~\cite{55letters} that the corresponding knot traces embed in $\CP^2$. We do \emph{not} verify that the rational homology balls that we obtain from the next proposition can be glued to the traces to form $\CP^2$, but we believe this to be the case.

\begin{prop}
$S^3_{64}(T_{3,22})$ and $S^3_{256}(T_{6,43})$ both bound rational homology balls constructed with one $0$--handle, one $1$--handle and one $2$--handle.
\end{prop}

\begin{proof}
The two corresponding handle attachments are displayed in Figures~\ref{f:322attack} and~\ref{f:643attack}. In both figures, the bracketed components represent the corresponding 3--manifold as boundary of a positive definite plumbing, and the unbracketed 2--handle gives the cobordism.
\end{proof}

\begin{figure}
\begin{subfigure}[t]{0.45\textwidth}
\includegraphics[width=\textwidth]{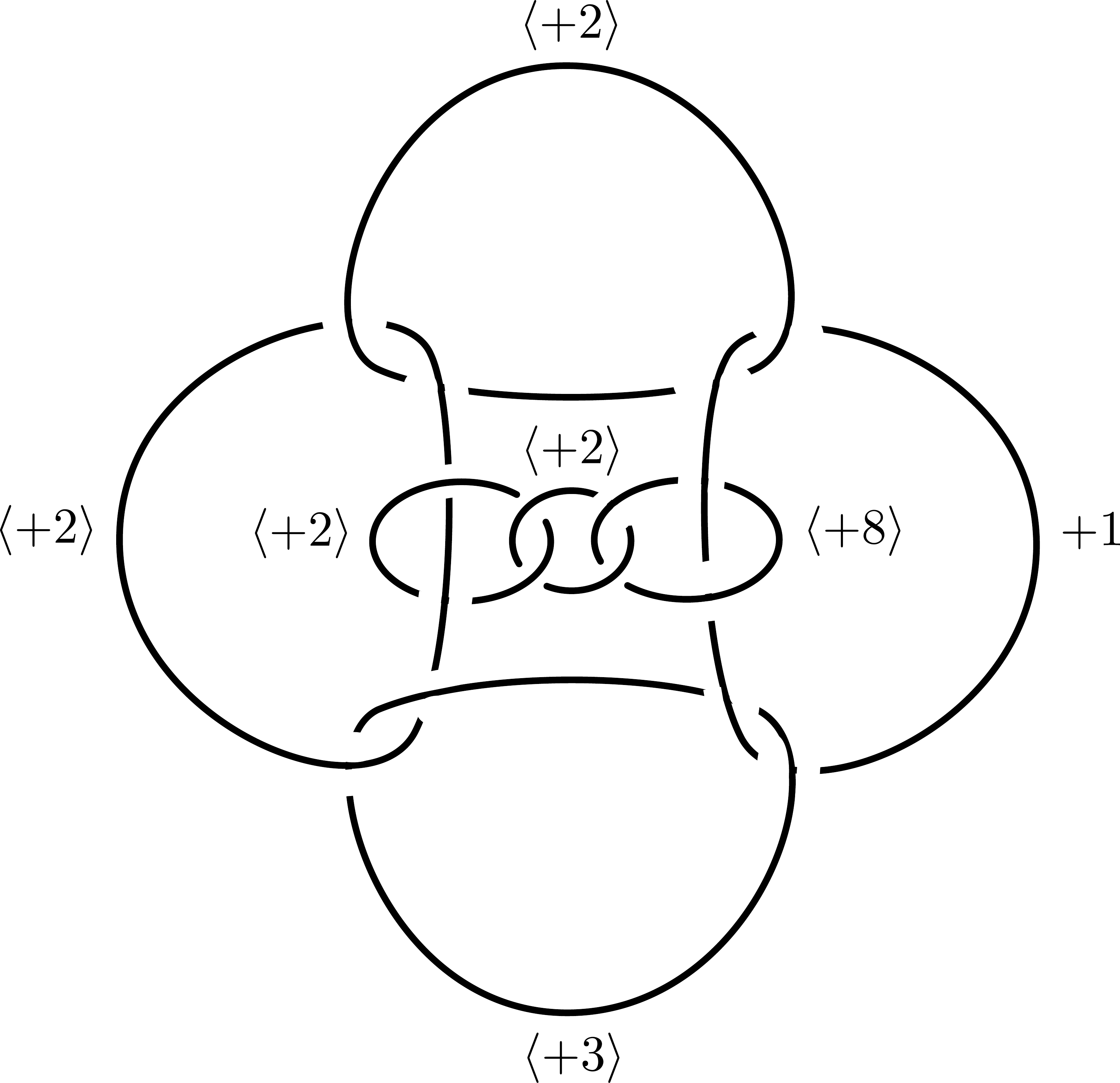}
\caption{The surgery cobordism from $S^3_{64}(T_{3,22})$ to $S^1\times S^2$.}
\label{f:322attack}
\end{subfigure}
\hspace{0.05\textwidth}
\begin{subfigure}[t]{0.45\textwidth}
\includegraphics[width=\textwidth]{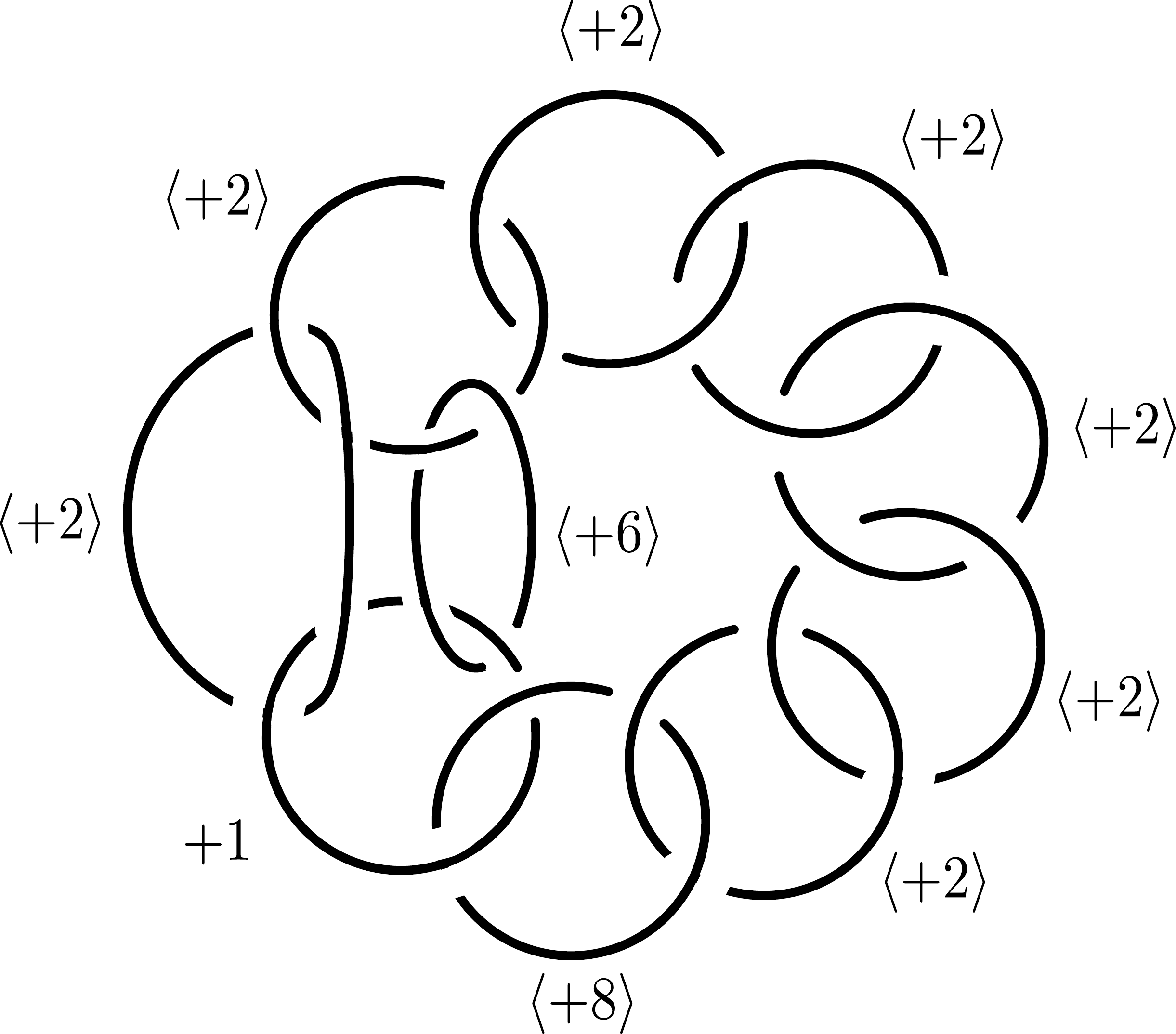}
\caption{The surgery cobordism from $S^3_{256}(T_{6,43})$ to $S^1\times S^2$.}
\label{f:643attack}
\end{subfigure}
\caption{}
\end{figure}

\subsection{Relationship to Montesinos links}

Each Seifert fibered rational homology 3--sphere is the double branched cover of a Montesinos link in $S^3$~\cite{montesinos}. The Seifert fibered spaces in Theorem~\ref{maintheorem}, that is, the surgeries on torus knots determined by the triples in $\mathcal G$, all have the following property, which we will show below: the rational homology balls they bound are double covers of $B^{4}$ branched over surfaces of Euler characteristic one. These surfaces are properly embedded in $B^{4}$ and their boundaries are the above mentioned Montesinos links.

If the Montesinos link is a knot, the surface bounded in $B^{4}$ is a disc and the knot is said to be slice. There is a vast literature on the subject, probably starting with the work of Casson and Harer~\cite{CassonHarer} and followed by a more systematic study of links in the Montesinos family bounding surfaces of Euler characteristic $1$ by several authors (see for example~\cite{Lisca-ribbon, Lisca-sums, greene-jabuka, 5pretzel, Lecuona, lec-pretzel, miller-pretzel, owens-chi}).

In general a Seifert space does not determine a unique Montesinos link. Indeed, a single Seifert space might be the double cover of $S^{3}$ branched over non isotopic mutant Montesinos links~\cite{bedient}. However, in the case of three-legged Seifert fibered spaces, the ones appearing in Theorem~\ref{maintheorem}, the correspondence between the double covers and the Montesinos links is one-to-one~\cite[Theorem~12.29]{BurdeZ}. In light of this, we will denote with $M(r_1,r_2,r_3)$ the Montesinos link whose branched double cover is the Seifert fibered space $Y(0;r_1,r_2,r_3)$.

We now want to prove Proposition~\ref{p:ribbon}, asserting that if $M(r_1,r_2,r_3)$ has branched double cover $S^3_n(T_{p,q})$ for some $p,q,n$, then it is slice if and only if it ribbon, and this happens if and only if $(p,q;n) \in \Gc$. Moreover, if $n$ is odd, then $M(r_1,r_2,r_3)$ is a knot with fusion number $1$.

%

\begin{proof}
We will argue family by family. We start with the first two families.

It is easy to verify that $S^3_{p^2}(T_{p,p+1}) = Y(0;p,-p,-p-1)$ and $S^3_{(p+1)^2}(T_{p,p+1}) = Y(0;p,-p-1,p+1)$. Notice that the corresponding families of Montesinos links will be either knots, when the surgery coefficient is odd, or two-component links, when the surgery coefficient is even. In this case, since the parameters are all integers, the Montesinos links belong to the well-studied families of pretzel links. The pretzel knots under consideration have all precisely one even parameter and the fact that they are ribbon was established in~\cite{lec-pretzel}. They have all fusion number one, as can be extrapolated from Figure~\ref{f.pretzels}. In the link case, the construction of a non-orientable surface of Euler characteristic one is also sketched in Figure~\ref{f.pretzels}.

\figure
\begingroup%
  \makeatletter%
  \providecommand\color[2][]{%
    \errmessage{(Inkscape) Color is used for the text in Inkscape, but the package 'color.sty' is not loaded}%
    \renewcommand\color[2][]{}%
  }%
  \providecommand\transparent[1]{%
    \errmessage{(Inkscape) Transparency is used (non-zero) for the text in Inkscape, but the package 'transparent.sty' is not loaded}%
    \renewcommand\transparent[1]{}%
  }%
  \providecommand\rotatebox[2]{#2}%
  \newcommand*\fsize{\dimexpr\f@size pt\relax}%
  \newcommand*\lineheight[1]{\fontsize{\fsize}{#1\fsize}\selectfont}%
  \ifx\svgwidth\undefined%
    \setlength{\unitlength}{293.14782571bp}%
    \ifx\svgscale\undefined%
      \relax%
    \else%
      \setlength{\unitlength}{\unitlength * \real{\svgscale}}%
    \fi%
  \else%
    \setlength{\unitlength}{\svgwidth}%
  \fi%
  \global\let\svgwidth\undefined%
  \global\let\svgscale\undefined%
  \makeatother%
  \begin{picture}(1,0.60975395)%
    \lineheight{1}%
    \setlength\tabcolsep{0pt}%
    \put(0,0){\includegraphics[width=\unitlength,page=1]{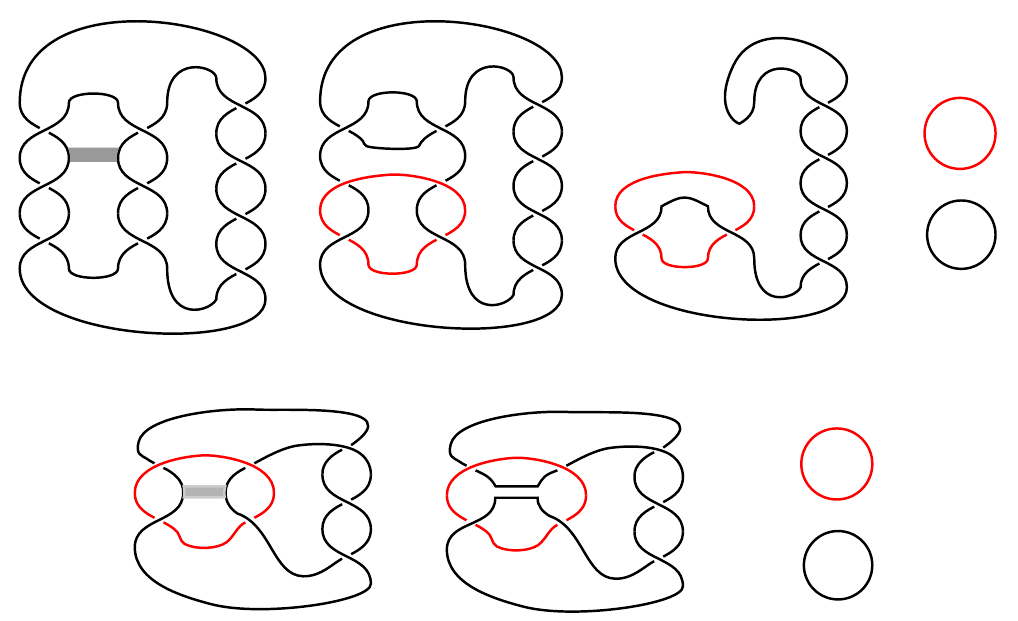}}%
  \end{picture}%
\endgroup%

\caption{On the top row we have the pretzel knot $M(3,-3,-4)$ with a band. The next diagrams show the effect of the band move, describing a ribbon disk of fusion number 1. It is clear that the same band move will produce the same result for all knots $M(p,-p,-p-1)$ with $p$ odd. In the second row we have an example of the link case, precisely $M(2,-2,-3)$. This time the band move describes the union of a disk and a M\"obius band and again it can be generalized for arbitrary even $p$. The knots and links in the family $M(p,-p-1,p+1)$ are dealt with in complete analogy.}
\label{f.pretzels}
\endfigure

Likewise, we have that $S^3_{4p^2}(T_{p,4p-1}) = Y(0;p,-p,\frac{4p-1}{4})$ and that $S^3_{4p^2}(T_{p,4p+1}) = Y(0;p,-p,-\frac{4p+1}4)$. 
These Seifert manifolds can be described respectively as the boundaries of the plumbing graphs:
 \[
  \begin{tikzpicture}[xscale=1.6,yscale=-0.4]
	\node (A0_4) at (4, 0) {$p$};
    \node (A1_4) at (4, 1) {$\bullet$};
    \node (A2_0) at (1, 2) {$4$};
    \node (A2_2) at (2, 2) {$p$};
    \node (A2_3) at (3, 2) {$0$};
    \node (A3_1) at (1, 3) {$\bullet$};
    \node (A3_2) at (2, 3) {$\bullet$};
    \node (A3_3) at (3, 3) {$\bullet$};
    \node (A4_4) at (4, 4) {$-p$};
    \node (A4) at (4, 5) {$\bullet$};
    \node at (5,3) {and};
    \node (B0_4) at (9, 0) {$p$};
    \node (B1_4) at (9, 1) {$\bullet$};
    \node (B2_0) at (6, 2) {$4$};
    \node (B2_2) at (7, 2) {$-p$};
    \node (B2_3) at (8, 2) {$0$};
    \node (B3_1) at (6, 3) {$\bullet$};
    \node (B3_2) at (7, 3) {$\bullet$};
    \node (B3_3) at (8, 3) {$\bullet$};
    \node (B4_4) at (9, 4) {$-p$};
    \node (B4) at (9, 5) {$\bullet$};
    \path (A3_3) edge [-] node [auto] {$\scriptstyle{}$} (A4);
    \path (A3_2) edge [-] node [auto] {$\scriptstyle{}$} (A3_3);
    \path (A3_3) edge [-] node [auto] {$\scriptstyle{}$} (A1_4);
    \path (A3_1) edge [-] node [auto] {$\scriptstyle{}$} (A3_2);
    \path (B3_3) edge [-] node [auto] {$\scriptstyle{}$} (B4);
    \path (B3_2) edge [-] node [auto] {$\scriptstyle{}$} (B3_3);
    \path (B3_3) edge [-] node [auto] {$\scriptstyle{}$} (B1_4);
    \path (B3_1) edge [-] node [auto] {$\scriptstyle{}$} (B3_2);
   \end{tikzpicture}
  \]
In Figure~\ref{f.montesinoslinks} the reader can find the diagrams of the corresponding families of Montesinos links. In the same figure, through band moves, the surfaces of Euler characteristic one bounded by these links are described. How to obtain a projection of the Montesinos link from the above plumbing graphs can be found in \cite[Figure~3]{Lecuona}.

\begin{figure}
\includegraphics[scale=1]{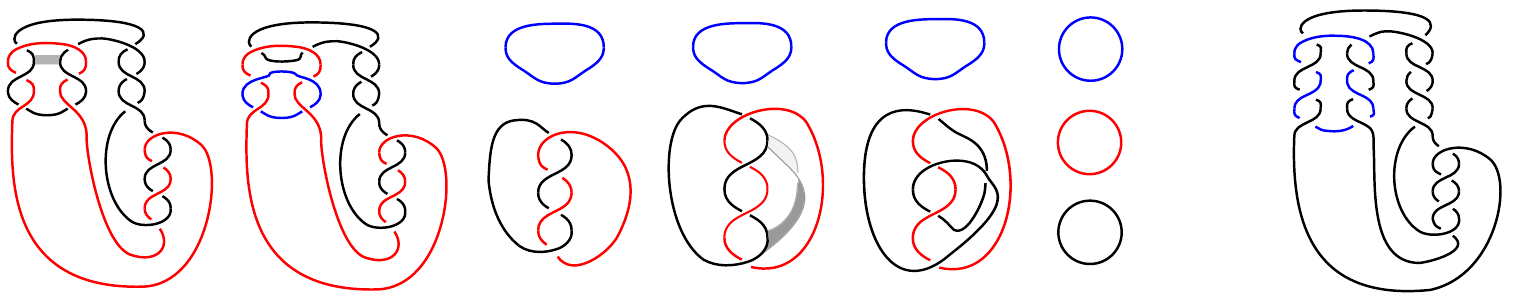}
\caption{The left and right most diagrams are the Montesinos links $M(3,-3,-\frac{13}{4})$ and $M(4,-4,-\frac{17}{4})$. In general, the diagram of the Montesinos link $M(p,-p,-\frac{4p+1}{4})$ will have 3 top columns of $p,-p$ and $-p$ half crossings and the third column will be connected to a set of 4 half crossings. The band move depicted in the top left diagram yields the third diagram independently of the value $p$. The analogous move on the right most diagram yields also the third diagram from the left. The two band moves describe the union of a disk and a M\"oebius band. The only difference in the diagrams of the links  $M(p,-p,\frac{4p-1}{4})$ is that there the 3 top columns have $p,-p$ and $p$ half crossings. The same band moves apply to describe the Euler characteristic one surfaces.}
\label{f.montesinoslinks}
\end{figure}

Finally, for each member of the remaining family, in which all corresponding Montesinos links are knots, and each of the two sporadic cases, for which we are dealing with two component links, we have shown in the previous subsections that there is a 2--handle attachment from their branched double covers $Y$ to $S^1\times S^2$. In Figure~\ref{f.eqsurgeries} we show that these handle attachments can be done in an equivariant fashion with respect to the covering involution. It follows that in $S^{3}$ these handle attachments correspond to band moves in the link diagrams (see~\cite[Section~3]{Lecuona} for details). Since $S^{1}\times S^{2}$ as branched cover of $S^{3}$ has branching set the unlink of two components, it follows that the result of the band moves on the diagrams is necessarily the unlink of two components. If we started with a knot, this band move describes a ribbon disk with fusion number 1; on the other hand, if we had started with a two component link, the band move shows how to construct a M\"obius band and a disk, which is a surface of Euler characteristic one with two boundary components. 
\end{proof}

\figure
\begingroup%
  \makeatletter%
  \providecommand\color[2][]{%
    \errmessage{(Inkscape) Color is used for the text in Inkscape, but the package 'color.sty' is not loaded}%
    \renewcommand\color[2][]{}%
  }%
  \providecommand\transparent[1]{%
    \errmessage{(Inkscape) Transparency is used (non-zero) for the text in Inkscape, but the package 'transparent.sty' is not loaded}%
    \renewcommand\transparent[1]{}%
  }%
  \providecommand\rotatebox[2]{#2}%
  \newcommand*\fsize{\dimexpr\f@size pt\relax}%
  \newcommand*\lineheight[1]{\fontsize{\fsize}{#1\fsize}\selectfont}%
  \ifx\svgwidth\undefined%
    \setlength{\unitlength}{416.95745392bp}%
    \ifx\svgscale\undefined%
      \relax%
    \else%
      \setlength{\unitlength}{\unitlength * \real{\svgscale}}%
    \fi%
  \else%
    \setlength{\unitlength}{\svgwidth}%
  \fi%
  \global\let\svgwidth\undefined%
  \global\let\svgscale\undefined%
  \makeatother%
  \begin{picture}(1,0.63656202)%
    \lineheight{1}%
    \setlength\tabcolsep{0pt}%
    \put(0,0){\includegraphics[width=\unitlength,page=1]{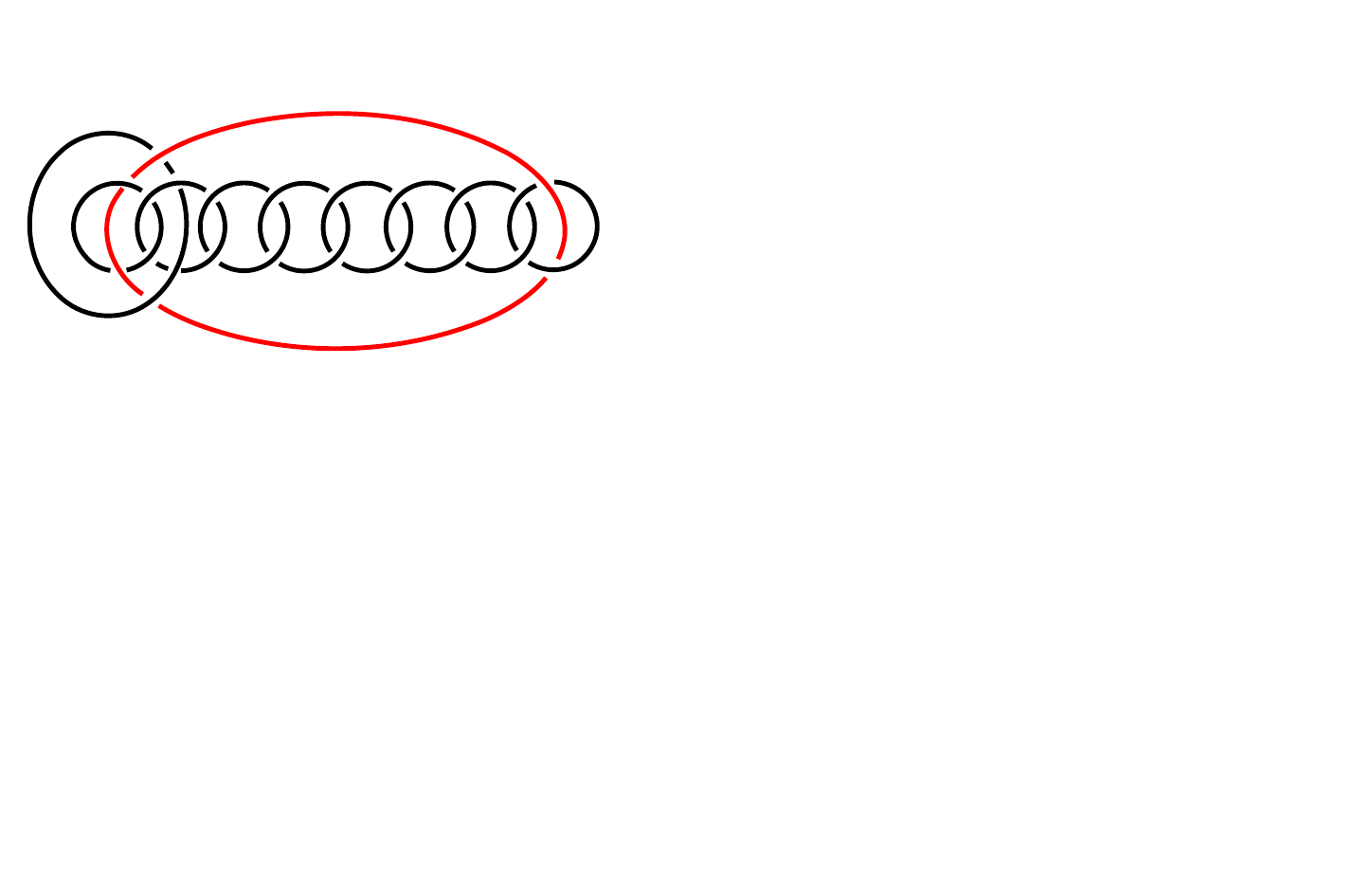}}%
    \put(0.05205946,0.50515866){\color[rgb]{0,0,0}\makebox(0,0)[lt]{\lineheight{1.25}\smash{\begin{tabular}[t]{l}$6$\end{tabular}}}}%
    \put(0.06461037,0.54459138){\color[rgb]{0,0,0}\makebox(0,0)[lt]{\lineheight{1.25}\smash{\begin{tabular}[t]{l}$2$\end{tabular}}}}%
    \put(0.13476907,0.41781072){\color[rgb]{0,0,0}\makebox(0,0)[lt]{\lineheight{1.25}\smash{\begin{tabular}[t]{l}$2$\end{tabular}}}}%
    \put(0.17154582,0.41611334){\color[rgb]{0,0,0}\makebox(0,0)[lt]{\lineheight{1.25}\smash{\begin{tabular}[t]{l}$2$\end{tabular}}}}%
    \put(0.2100199,0.41611334){\color[rgb]{0,0,0}\makebox(0,0)[lt]{\lineheight{1.25}\smash{\begin{tabular}[t]{l}$2$\end{tabular}}}}%
    \put(0.2535517,0.4131534){\color[rgb]{0,0,0}\makebox(0,0)[lt]{\lineheight{1.25}\smash{\begin{tabular}[t]{l}$2$\end{tabular}}}}%
    \put(0.30175752,0.41552974){\color[rgb]{0,0,0}\makebox(0,0)[lt]{\lineheight{1.25}\smash{\begin{tabular}[t]{l}$2$\end{tabular}}}}%
    \put(0.34588985,0.41451133){\color[rgb]{0,0,0}\makebox(0,0)[lt]{\lineheight{1.25}\smash{\begin{tabular}[t]{l}$2$\end{tabular}}}}%
    \put(0.42091412,0.42141401){\color[rgb]{0,0,0}\makebox(0,0)[lt]{\lineheight{1.25}\smash{\begin{tabular}[t]{l}$8$\end{tabular}}}}%
    \put(0.37930356,0.53030148){\color[rgb]{0,0,0}\makebox(0,0)[lt]{\lineheight{1.25}\smash{\begin{tabular}[t]{l}$1$\end{tabular}}}}%
    \put(0,0){\includegraphics[width=\unitlength,page=2]{eqsurgeries.pdf}}%
    \put(0.20658972,0.5823234){\color[rgb]{0,0,0}\makebox(0,0)[lt]{\lineheight{1.25}\smash{\begin{tabular}[t]{l}$S^3_{256}(T_{6,43})$\end{tabular}}}}%
    \put(0,0){\includegraphics[width=\unitlength,page=3]{eqsurgeries.pdf}}%
    \put(0.58393774,0.52256866){\color[rgb]{0,0,0}\makebox(0,0)[lt]{\lineheight{1.25}\smash{\begin{tabular}[t]{l}$3$\end{tabular}}}}%
    \put(0.55188141,0.56570096){\color[rgb]{0,0,0}\makebox(0,0)[lt]{\lineheight{1.25}\smash{\begin{tabular}[t]{l}$2$\end{tabular}}}}%
    \put(0.70458983,0.40369888){\color[rgb]{0,0,0}\makebox(0,0)[lt]{\lineheight{1.25}\smash{\begin{tabular}[t]{l}$2$\end{tabular}}}}%
    \put(0.7582377,0.40122289){\color[rgb]{0,0,0}\makebox(0,0)[lt]{\lineheight{1.25}\smash{\begin{tabular}[t]{l}$2$\end{tabular}}}}%
    \put(0.81436152,0.40122289){\color[rgb]{0,0,0}\makebox(0,0)[lt]{\lineheight{1.25}\smash{\begin{tabular}[t]{l}$2$\end{tabular}}}}%
    \put(0.94142046,0.43064753){\color[rgb]{0,0,0}\makebox(0,0)[lt]{\lineheight{1.25}\smash{\begin{tabular}[t]{l}$8$\end{tabular}}}}%
    \put(0.86856962,0.55293552){\color[rgb]{0,0,0}\makebox(0,0)[lt]{\lineheight{1.25}\smash{\begin{tabular}[t]{l}$1$\end{tabular}}}}%
    \put(0,0){\includegraphics[width=\unitlength,page=4]{eqsurgeries.pdf}}%
    \put(0.70964115,0.5945424){\color[rgb]{0,0,0}\makebox(0,0)[lt]{\lineheight{1.25}\smash{\begin{tabular}[t]{l}$S^3_{64}(T_{3,22})$\end{tabular}}}}%
    \put(0,0){\includegraphics[width=\unitlength,page=5]{eqsurgeries.pdf}}%
    \put(0.40972956,0.22845119){\color[rgb]{0,0,0}\makebox(0,0)[lt]{\lineheight{1.25}\smash{\begin{tabular}[t]{l}$2$\end{tabular}}}}%
    \put(0.44534477,0.28757101){\color[rgb]{0,0,0}\makebox(0,0)[lt]{\lineheight{1.25}\smash{\begin{tabular}[t]{l}$1$\end{tabular}}}}%
    \put(0.70086099,0.25576039){\color[rgb]{0,0,0}\makebox(0,0)[lt]{\lineheight{1.25}\smash{\begin{tabular}[t]{l}$6$\end{tabular}}}}%
    \put(0.5773475,0.25771479){\color[rgb]{0,0,0}\makebox(0,0)[lt]{\lineheight{1.25}\smash{\begin{tabular}[t]{l}$3$\end{tabular}}}}%
    \put(0.22326991,0.24479578){\color[rgb]{0,0,0}\makebox(0,0)[lt]{\lineheight{1.25}\smash{\begin{tabular}[t]{l}$2$\end{tabular}}}}%
    \put(0,0){\includegraphics[width=\unitlength,page=6]{eqsurgeries.pdf}}%
  \end{picture}%
\endgroup%

\caption{The top two diagrams are isotopic to the diagrams in Figures~\ref{f:322attack} and~\ref{f:643attack} respectively. Here the emphasis is that the 1-framded 2-handle in red can be added equivariantly with respect to the involution. Analogously, the third diagram is isotopic to Figure~\ref{f:exceptional_family1}, where the two black squares indicate that there is a chain of unlinks.}
\label{f.eqsurgeries}
\endfigure


\section{Bounds on cables from correction terms}\label{cablesHF}


In this section, we use (Heegaard Floer) correction terms to prove Theorem~\ref{bound}.

Correction terms are $\Q$--valued invariants of \spinc rational homology spheres, introduced by Ozsv\'ath and Szab\'o in~\cite{OzsvathSzabo-absolutely}.
We do not recall the definition here, and we refer to the original paper, and to~\cite[Sections 2 and 3]{AcetoGolla} for an exposition more tailored to the scope of this paper. 

The set of \spinc structures on $S^3_n(K)$ has labelling by $\Z/n\Z$, given in~\cite{OzsvathSzabo-integralsurgeries}; when $S^3_n(K)$ bounds a rational homology ball, $n$ is a square, and $m=\sqrt n$ of the correction terms of $S^3_n(K)$ vanish; more precisely, those that vanish are those labelled with $\frac12m(m-1) - km$ as $k$ varies among the integers~\cite[Lemma~5.3]{AcetoGolla}. In particular, we observed that if $S^3_{m^2}(K)$ bounds a rational homology ball, then for $k \le \frac m2$
\begin{equation}\label{e:AG5.3}
V_{m(m-2k-1)/2}(K) = \frac{k(k+1)}2.
\end{equation}

Here $\{V_i(K)\}_{i\ge 0}$ is a sequence of invariants of $K$ that can be used to describe the correction terms of $S^3_n(K)$~\cite{Rasmussen-Goda, NiWu, HomWu}.
Recall that we denote with $\unknot$ the unknot.

\begin{te}[\cite{Rasmussen-Goda, NiWu}]\label{t:RasmussenNiWu}
The sequence $\{V_i(K)\}_{i\ge 0}$ takes values in the non-negative integers and is eventually $0$. Moreover, $V_{i}(K)-1 \le V_{i+1}(K) \le V_i(K)$ for every $i\ge 0$.

For every rational number $p/q$ and for an appropriate indexing of \spinc structures on $S^3_{p/q}(\unknot)$ and $S^3_{p/q}(K)$, we have
\begin{equation}\label{e:NiWu}
d(S^3_{p/q}(K),i) = -2\max\{V_{\lfloor i/q\rfloor}(K),V_{\lceil (p-i)/q\rceil}(K)\} + d(S^3_{p/q}(\unknot),i).
\end{equation}
\end{te}

The invariant $\nu^+(K)$ is defined as the minimal index $i$ such that $V_i(K)=0$~\cite{HomWu}, and in particular it is always non-negative. It gives a lower bound for the slice genus, $g_4$, and an upper bound for the $\tau$-invariant~\cite{OzsvathSzabo-tau}.

\begin{prop}[\cite{HomWu}]\label{p:squeezed}
The invariant $\nu^+(K)$ satisfies $\tau(K) \le \nu^+(K) \le g_4(K)$.
\end{prop}

Wu studied the behavior of the sequence $\{V_i(K)\}$ under cabling~\cite{Wu}.
In the following proposition, we let $\phi_1(i)$ be the remainder of the division of $i-\frac{(p-1)(q-1)}2$ by $q$. We also abandon our convention that torus knots $T_{p,q}$ have $p<q$, and we shall do so throughout this section.

\begin{prop}[\cite{Wu}]
For every $0\le i \le \frac{pq}2$ the following holds:
\begin{equation}\label{e:Wu-cables}
V_i(K_{p,q}) = V_i(T_{p,q}) + \max\left\{
V_{\left\lfloor {\frac{\phi_1(i)}{p}} \right\rfloor}(K),
V_{\left\lfloor {\frac{p+q-1-\phi_1(i)}{p}} \right\rfloor}(K)
\right\}.
\end{equation}
\end{prop}

As mentioned in the introduction, correction terms of surgeries along torus knots (and hence, in light of Wu's formula~\eqref{e:Wu-cables}, of surgeries along cables) are related to semigroups. More precisely, Borodzik and Livingston proved the following relationship~\cite[Theorems~5.4 and~5.6]{BorodzikLivingston}:
\begin{equation}\label{e:BL}
V_i(K) = R_{p,q}(g(T_{p,q}) + i) - i,
\end{equation}
where $R_{p,q}$ is the semigroup-counting function associated to the semigroup $\Gamma_{p,q} = \{hp+kq \mid h,k \in \Z_{\ge0}\}$, defined as:
\[
R_{p,q}(n) = \#(\Gamma_{p,q} \cap [0,n)).
\]
We find it convenient to translate~\eqref{e:BL} into the following relation. Let $\Gamma_{p,q}(n)$ be the $n^{\rm th}$ element of $\Gamma_{p,q}$, so that we have $\Gamma_{p,q}(0) = 0$ and $\Gamma_{p,q}(1) = \min\{p,q\}$; by convention, we let $\Gamma_{p,q}(-1) = -\infty$. Then, for each $0 \le i$ and $k \ge 0$:
\begin{equation}\label{e:BL2}
V_i(T_{p,q}) = a \Longleftrightarrow g(T_{p,q})-\Gamma_{p,q}(a) \le i < g(T_{p,q})-\Gamma_{p,q}(a-1).
\end{equation}

Recall also from~\cite[Theorem~5.2]{AcetoGolla} that if $S^3_{m^2}(K)$ bounds a rational homology ball, then
\begin{equation}\label{e:AG5.2}
\frac{1 + \sqrt{1+8\nu^+(K)}}2 \le m < \frac{3+\sqrt{9+8\nu^+(K)}}2,
\end{equation}
or, equivalently,
\begin{equation}\label{e:AG5.2v2}
\frac{m(m-3)}2 < \nu^+(K) \le \frac{m(m-1)}2.
\end{equation}

The theme of this section is to combine Equations~\eqref{e:Wu-cables},~\eqref{e:BL2}, and~\eqref{e:AG5.3} when $q/p$ is either very large or very small, so as to produce a contradiction with~\eqref{e:AG5.2v2}.

\subsection{An example: cables of the trefoil}


As an example, we start by considering cables of a knot $K$ such that $\tau(K) = g_*(K) = 1$, e.g. the trefoil $K = T_{2,3}$.
While this is a toy case, it displays some of the features of the general case; also, in this case we get better upper and lower bounds for $\acc(K)$.

Note that, by Proposition~\ref{p:squeezed}, the assumption implies that $\nu^+(K) = 1$.
Moreover, for these knots we have $\tau(K_{p,q}) =  p + \frac{(p-1)(q-1)}2$; this follows readily from~\cite[Theorem~1]{Hom-epsilon}, once we observe that $\epsilon(K) = 1$ (see~\cite[Proposition~3.6]{Hom-epsilon}).

\begin{prop}
Let $K$ be a knot such that $\tau(K) = g_*(K) = 1$, and $p,q$ be positive and coprime, with $p>1$. If $K_{p,q}$ has a positive integral surgery that bounds a rational homology ball, then $\frac1{7} < \frac qp < 10$.
\end{prop}

\begin{proof}
We know that $\tau(K_{p,q}) \le \nu^+(K_{p,q}) \le g_*(K_{p,q})$ and that $g_*(K_{p,q}) \le pg_*(K)+\frac{(p-1)(q-1)}2$. Thus we have:
\begin{equation}\label{e:nu-cables-23}
\nu := \nu^+(K_{p,q}) = \tau(K_{p,q}) = g_*(K_{p,q}) = p + \frac{(p-1)(q-1)}2.
\end{equation}

We want to apply Wu's formula~\eqref{e:Wu-cables}.
Since $\nu^+(K) = 1$, $V_i(K) = 0$ if and only if $i>0$.
Since $0\le\phi_1(i)\le q-1$ for each $i$, $\frac{p+q-1-\phi_1(i)}{p}\ge1$, and~\eqref{e:Wu-cables} reads:
\begin{equation}\label{e:V-cables-23}
V_i(K_{p,q}) = V_i(T_{p,q}) +
V_{\left\lfloor {\frac{\phi_1(i)}{p}} \right\rfloor}(K).
\end{equation}

Suppose, by contradiction, that $p \ge 7q$.
In particular, $p>q>\phi_1(i)$ for each $i$, from which $\frac{\phi_1(i)}p < 1$; it follows that the second summand on the right-hand side of equation~\eqref{e:V-cables-23} is always $1$.
That is, $V_i(K_{p,q}) = V_i(T_{p,q}) + 1$ for each $i\le \frac{pq}2$; moreover, since the sequence of $V_i$'s is non-increasing and
\[
\textstyle \nu^+(T_{p,q}) = \frac{(p-1)(q-1)}2 < \frac{pq}2 < \frac{pq+p-q+1}2 = \nu.
\]
We also obtain that $V_i(K) = 1$ for each $\frac{pq}2 \le i<\nu$, and that $V_i(K_{p,q})=0$ for $i\ge\nu$.
From the previous formula, and since $p \ge 7q \ge 7$, we have $\nu \ge 7$; in particular, from~\eqref{e:AG5.2} we get that $m \ge 5$.

We want to apply~\eqref{e:AG5.3} for $k=0,2$, which we can do since we just showed that $m\ge 5$, in conjunction with~\eqref{e:BL2} and~\eqref{e:Wu-cables}. Therefore, we have the following inequalities:
\begin{align*}
\nu \le &{\textstyle \frac{m(m-1)}2},\\
& {\textstyle \frac{m(m-5)}2} < {\textstyle \frac{(p-1)(q-1)}2} - q.
\end{align*}
Subtracting the second inequality from the first yields $2m > \nu - {\textstyle \frac{(p-1)(q-1)}2} + q$, from which $2m > p+q > 8q$.
We now compare this with~\eqref{e:AG5.2v2}:
\[
\textstyle \frac{p+q}2\cdot\big(\frac{p+q}2-3\big) < \frac{m(m-3)}2 < \nu = p+\frac{(p-1)(q-1)}2,
\]
which yields
\[
(p-q)^2 < 10p + 2q + 4,
\]
which in turn never holds if $p>7q$.

Suppose now that $q > 10p$.
Observe that, by definition, $\phi_1(i) \ge p$ whenever $\frac{(p-1)(q-1)}2-q+p < \frac{(p-1)(q-1)}2-14p \le i < \frac{(p-1)(q-1)}2$. In particular, in this interval, equation~\eqref{e:V-cables-23} reads $V_i(K_{p,q}) = V_i(T_{p,q})$.
Moreover, $V_i(K_{p,q}) = 1$ for $\frac{(p-1)(q-1)}2 \le i < \nu$ as well.

We can assume that $m \ge 3$, since $\nu \ge p + \frac12(q-1) > 6$ (recall that $p>1$, otherwise we would be including trivial cables). From~\eqref{e:BL2} and~\eqref{e:AG5.3} for $k=1,2$ we then obtain
\begin{align*}
{\textstyle \frac{(p-1)(q-1)}2} - p & \le {\textstyle \frac{m(m-3)}2} < \nu,\\
{\textstyle \frac{(p-1)(q-1)}2} - 3p & \le  {\textstyle \frac{m(m-5)}2} < {\textstyle \frac{(p-1)(q-1)}2} - 2p,
\end{align*}
and subtracting the second inequality from the first we get $m < 3p$.
We now plug this into~\eqref{e:AG5.2v2} to obtain:
\[
\textstyle p + \frac{(p-1)(q-1)}2  = \nu < \frac{m(m-1)}2 < \frac{9p^2-3p}2,
\]
which is impossible if $q>10p$.
\end{proof}

\subsection{The general case}

In this section we want to prove a quantitative version of Theorem~\ref{bound}. More precisely, we want to show that if $\nu^+(K) > 0$, then $\acc(K)$ is bounded \emph{away from $0$} and from above, by two constants that depend only on $\nu^+(K)$.

When $\nu^+(K) = 0$, we cannot exclude $0$ as an accumulation point for $\acc(K)$. In fact, whenever $K$ is slice (which, in particular, implies $\nu^+(K) = 0$), then $S^3_1(K_{p,1})$ bounds a rational homology ball for each $p$. Viewed differently, $\acc(K)$ is a concordance invariant, so for $K$ slice $\acc(K) = \acc(\unknot)$, which has $0$ as an accumulation point.
Moreover, since $\nu^+(K) = 0$ implies that $d(S^3_n(K_{p,q}), i) = d(S^3_n(T_{p,q}),i)$ for each $i$ and each positive $n$; then the same proof as in~\cite[Theorem~1.3]{AcetoGolla} shows that $q/p < 9$. (In fact, the proof can be refined to show that $q/p < 8$.)
For these reason, in what follows we only treat the case $\nu^+(K) > 0$.

\begin{re}
If the answer to Question~\ref{q:dinvariants-converse} is positive (that is, if correction terms detect whether a positive integral surgery along a torus knot bounds a rational homology ball), then for every knot with $\nu^+(K) = 0$ we have $\acc(K) \subset \acc(\unknot)$. In particular, $q/p \in \acc(K)$ if and only if either $q=1$, or $(p,q;n)\in \Gc \cup \Rc \cup \Lc$ or $(q,p;n)\in \Gc \cup \Rc \cup \Lc$ or for some $n$.
\end{re}

We prove the lower and upper bounds for $\acc(K)$ in Theorems~\ref{t:p<Nq} and~\ref{t:q<Np} below.

\begin{te}\label{t:p<Nq}
Let $K$ be a knot in $S^3$ with $\nu^+(K) > 0$, and $p,q$ positive, coprime integers such that $K_{p,q}$ has an integral surgery that bounds a rational homology ball.
Then either $q=1$ and
\[
\frac qp \ge \frac1{2+\sqrt{1+8V_0(K)}} \ge \frac1{2+\sqrt{1+8\nu^+(K)}},
\]
or $q > 1$ and
\[
\frac{q}{p} > \frac1{9V_0(K)} \ge \frac1{9\nu^+(K)}.
\]
\end{te}

The strategy of the proof will follow closely the one we adopted for the trefoil in the previous subsection. The main difference with that proof is that now we have no a priori control on $\nu^+(K_{p,q})$.

\begin{proof}
First, we note that the two right-most inequalities in the statement follows immediately from the inequality $\nu^+(K) \ge V_0(K)$, which in turn follows from Theorem~\ref{t:RasmussenNiWu}.

Suppose that $p>q$; then~\eqref{e:Wu-cables} shows that $V_i(K_{p,q}) = V_i(T_{p,q}) + V_0(K)$ for each $i \le \frac{pq}2$. Since we assumed $\nu^+(K) > 0$, then also $V_0(K) > 0$, and $\nu^+(K_{p,q}) \ge \lfloor \frac{pq}2 \rfloor + V_0(K)$.
We now need to distinguish two cases in the proof: $q = 1$ and $q > 1$.

Let then $q=1$, so that $V_i(K_{p,q}) = V_0(K)$ for each $i \le \frac{pq}2$, and suppose that $S^3_{m^2}(K_{p,1})$ bounds a rational homology ball. Applying~\eqref{e:AG5.3} for $k=\lfloor \frac m2 \rfloor$ we obtain
\[
V_0(K) \ge \frac{(m-1)(m-2)}2,
\]
that translates into $m \le \frac{3+\sqrt{1+8V_0(K)}}2 =: \ell$. Note that the denominator of the fraction in the statement is $2\ell -1$.

We observed above that $\nu^+(K_{p,q}) \ge \frac{p-1}2 + V_0(K)$; we now combine this with~\eqref{e:AG5.2v2} and the previous inequality to get:
\[
V_0(K) + {\textstyle\frac{p-1}2} - m + 1 \le \nu^+(K_{p,q}) - m + 1 \le {\textstyle\frac{m(m-1)}2}-(m-1) =  {\textstyle\frac{(m-1)(m-2)}2} \le V_0(K),
\]
from which we obtain $\frac{p-1}2 \le m-1 \le \ell -1$, i.e. $p \le 2\ell-1$, as required.

Suppose now that $q > 1$ and $p > 9V_0(K)q$, and that $S^3_{m^2}(K_{p,q})$ bounds a rational homology ball.
Call $x$ the smallest integer such that $\frac{x(x+1)}2 > V_0(K)+1$; note that this is equivalent to asking that $\frac{\sqrt{9+8V_0(K)}-1}2 < x \le \frac{\sqrt{9+8V_0(K)}+1}2$.

We want to make sure that $x+1 \le \frac{m}2$: from~\eqref{e:AG5.2}:
\begin{align*}
m &\ge \left\lceil \frac{1 + \sqrt{1+8\nu^+(K_{p,q})}}2\right\rceil \ge \left\lceil\frac{1 + \sqrt{3+8V_0(K) + 4pq}}2\right\rceil\\
& > \left\lceil\frac{1 + \sqrt{3+152V_0(K)}}2\right\rceil \ge \left\lfloor\sqrt{9+8V_0(K)}+3\right\rfloor \ge 2(x+1).
\end{align*}
Call $\delta_0 := \frac{x(x+1)}2 - V_0(K)-1$, $\delta_1 := \frac{(x+1)(x+2)}2 - V_0(K)$, and note that $\delta_1 = \delta_0 + x + 2 < 2x+2$.
We also want to make sure that $q\cdot (2x+2) < p$: this is clear, since $p/q > 9V_0$, and $x \le \frac{\sqrt{1+8V_0(K)}+1}2$.

Since $x+2 < \frac{m}2$, we can apply Equation~\eqref{e:AG5.3} for $k=x$ and $k=x+1$, thus obtaining:
\[
V_{m(m-2x-1)/2}(K_{p,q}) = {\textstyle \frac{x(x+1)}2} \quad {\rm and} \quad V_{m(m-2x-1)/2}(K_{p,q}) = {\textstyle \frac{(x+1)(x+2)}2}.
\]
From Wu's formula~\eqref{e:Wu-cables} and Borodzik and Livingston's formula~\eqref{e:BL2}, we then obtain:
\begin{align*}
V_{\frac{m(m-2x-1)}2}(T_{p,q}) = \delta_0+1 &\Longrightarrow {\textstyle\frac{m(m-2x-1)}2 < \frac{(p-1)(q-1)}2 - \Gamma_{p,q}(\delta_0) = \frac{(p-1)(q-1)}2 - \delta_0q}, \\
V_{\frac{m(m-2x-3)}2}(T_{p,q}) = \delta_1 &\Longrightarrow {\textstyle \frac{(p-1)(q-1)}2 - \delta_1 q = \frac{(p-1)(q-1)}2 - \Gamma_{p,q}(\delta_1) \le \frac{m(m-2x-3)}2}.
\end{align*}
Subtracting the second inequality from the first yields $m < (\delta_1 - \delta_0)q = (x+2)q$.

We now use~\eqref{e:AG5.2} to get a contradiction:
\begin{align*}
V_0(K)(1+ 9q^2) - 1 &\le {\textstyle  V_0(K) + \frac{pq-1}2 \le \frac{m(m-1)}2 < \frac{m^2}2} <\\
&< {\textstyle\frac{q^2(x+2)^2}2 < \frac{q^2\cdot (2x)^2}2 < 2q^2x^2,}
\end{align*}
where we used that $x > 1$ (since $V_0(K) \ge 1$) and for each such $x$, $x+2 < 2x$.
Since this is clearly a contradiction for $V_0(K) \ge 1$, the proof is concluded.
\end{proof}

\begin{te}\label{t:q<Np}
Let $K$ be a knot with $\nu^+(K) > 0$, and $p,q$ positive, coprime integers such that $K_{p,q}$ has an integral surgery that bounds a rational homology ball.
Then
\[
\frac{q}{p} < 20\nu^+(K).
\]
\end{te}

We note that the result is far from optimal, and can be refined by distinguishing more cases; for instance, the inequality gets better as $\nu^+(K)$ increases, or as $p$ gets larger. We content ourselves with showing that $\acc(K)$ is bounded, rather than finding the sharpest possible result.

\begin{proof}
Suppose, by contradiction, that $q > 20\nu^+(K)$. We claim that there exists $x$ such that
\begin{equation}\label{e:xinbetween}
\textstyle{ p\nu^+(K) + p < p\cdot \frac{x(x+1)}2 < p\cdot \frac{(x+1)(x+2)}2 < \frac{(p-1)(q-1)}2 - p\nu^+(K) - p.}
\end{equation}
Indeed, let $x$ be the smallest integer such that $\frac{x(x+1)}2 > \nu^+(K) + 1$, so that $x$ satisfies $x \le  \frac{\sqrt{9+8\nu^+(K)}+1}2 < \nu^+(K)+1$; then
\[
\frac{(x+1)(x+2)}2 = \frac{x(x+1)}2 + x + 1 \le \nu^+(K) + 2x + 3 < 3\nu^+(K) + 5.
\]
The inequality $p\cdot\frac{(x+1)(x+2)}2 < \frac{(p-1)(q-1)}2 - p\nu^+(K) - p$ is clearly satisfied if $q > 20\nu^+(K)p$.

The inequalities in~\eqref{e:xinbetween} ensure that, when we apply~\eqref{e:AG5.3} and~\eqref{e:BL2}, we are actually looking at $V_i(K_{p,q})$ in the range where $V_i(K_{p,q}) = V_i(T_{p,q})$.
So we have, from~\eqref{e:AG5.3}, similarly as in the proof of Theorem~\ref{t:p<Nq},
\begin{align*}
V_{\frac{m(m-2x-1)}2}(T_{p,q}) = {\textstyle \frac{x(x+1)}2} &\Longrightarrow {\textstyle\frac{m(m-2x-1)}2 < \frac{(p-1)(q-1)}2 - (\frac{x(x+1)}2-1)p}, \\
V_{\frac{m(m-2x-3)}2}(T_{p,q}) = {\textstyle \frac{(x+1)(x+2)}2} &\Longrightarrow {\textstyle \frac{(p-1)(q-1)}2 - {\textstyle \frac{(x+1)(x+2)}2} p \le \frac{m(m-2x-3)}2}.
\end{align*}
Subtracting the second inequality from the first, we obtain that $m < (x+2)p$.
We now apply~\eqref{e:AG5.2v2}:
\[
{\textstyle 9\nu+(K)p^2 < p\nu^+(K) + \frac{(p-1)(q-1)}2 = \nu^+(K_{p,q}) \le \frac{m(m-1)}2 < \frac{(x+2)^2p^2}2 < \frac{(\nu^+(K)+3)^2 p^2}2,}
\]
where we have used Wu's formula~\eqref{e:Wu-cables} to get $\nu^+(K_{p,q}) = p\nu^+(K) + \frac{(p-1)(q-1)}2$ (this is also the statement of~\cite[Theorem~1.1]{Wu}).
Since the last inequality never holds when $\nu^+(K) \ge 1$, we have reached a contradiction.
\end{proof}
\begin{section}{The proof of Proposition~\ref{p:mainA}}\label{latticeemb}

This long section is mainly devoted to the study of the embeddability of the positive lattices defined by the graphs in Proposition~\ref{seifertplumbing} 
into the standard positive lattice of the same rank. Since we are only interested in this property because of Corollary~\ref{Donaldsonobstructionforus}, there are some simplifying assumptions that we can make from the very beginning.
\begin{enumerate}
\item By Proposition~\ref{pruning1} we may assume that $7\geq k\geq 1$. (Recall that $q = kp+r$ with $0<r<p-1$.)
\item Since it is completely understood which integer surgeries on $T_{p,q}$ torus knots bound rational homology balls when $q\equiv\pm1\bmod p$ \cite[Theorem~1.4]{AcetoGolla}, we can restrict our analysis to the case $q\not\equiv \pm1\bmod p$. Via Proposition~\ref{seifertplumbing} it is easy to show that this condition is equivalent to requiring that none of the strings $(a_1,\dots,a_n)$ and $(b_1,\dots,b_m)$ in the graphs $\Gamma_{1}$ and $\Gamma_{2}$ from Proposition~\ref{seifertplumbing} is a $2$-chain. This in turn is equivalent to imposing $n,m>1$.
\end{enumerate}

These two simplifications explain the first assumptions in the statement of Proposition~\ref{p:mainA}. The third assumption, which concerns only the graphs in family $\Gamma_{1}$, will be justified in the next subsection. The proof of this proposition is long and convoluted. We have tried to do our best to guide the reader providing context between the technical results in this section. The idea behind the proof is quite simple and it can be very roughly described as follows. Let $\Gamma$ be a weighted graph with vertices $v_{1},\dots,v_{n}$ and suppose its associated lattice $(\Z\Gamma,Q_{\Gamma})$ embeds into $(\Z^{n}=\langle e_{1},\dots,e_{n}\rangle,\mathrm{Id})$. By an abuse of notation we will identify the $v_{i}$'s with their images and write $v_{i}=\sum_{j=1}^{n}\alpha_{j}e_{j}$ where $\alpha_{j}\in\Z$. To any $\Gamma$ one can associate the integer 
\[
I(\Gamma)=\sum_{i=1}^{n}(v_{i}Q_{\Gamma}v_{i}^{T}-3).
\] 
Now, if $\Gamma$ admits an embedding as described, then this number $I(\Gamma)$ can also be computed \emph{from the embedding}. We will show that if the graphs $\Gamma_{1}$ and $\Gamma_{2}$ with the restrictions of Proposition~\ref{p:mainA} embed, then the quantities $I(\Gamma_{1})$ and $I(\Gamma_{2})$ are bounded from below. To obtain these bounds it will be of particular relevance to understand the basis vectors $e_{i}$ with the property that $\langle v_{j},e_{i}\rangle\neq 0$ for exactly one or exactly two $v_{j}$'s. This is because it is only these types of basis vectors that contribute \emph{negatively} to the quantity $I(\Gamma)$. We will say that such basis vectors \emph{hit} one or two $v_{j}$'s. The first part of the proof shows that there is no basis vector hitting precisely one $v_{j}$. The second part is a very careful analysis of the basis vectors hitting precisely two $v_{j}$'s. We will show that the negative contribution to $I(\Gamma)$ of this type of basis vectors is always `compensated' by some other basis vector hitting many $v_{j}$'s. The end result is that in most cases we can show that, if the lattices embed, then $I(\Gamma_{1})$ and $I(\Gamma_{2})$ are necessarily greater than $-2$. However, if we compute this quantity directly from the above formula we arrive immediately at a contradiction.

\begin{subsection}{Pruning some cases}
This section is different in flavor to all the rest. Proposition~\ref{pruning1} uses Heegaard Floer techniques to show that we need only to deal with graphs $\Gamma_{1}$ and $\Gamma_{2}$ satisfying $7\geq k\geq 1$. Moreover, it also excludes from the lattice embedding analysis the graphs in family $\Gamma_{1}$ with $k=7$ and $N\in\{1,2,3\}$.
Indeed, in Proposition~\ref{pruning2} we show that the associated surgeries on torus knots do not bound rational homology balls. This completes the justification of the assumptions in Proposition~\ref{p:mainA}.

In both Propositions~\ref{pruning1} and~\ref{pruning2} we make use of the following inequality from from~\cite{AcetoGolla}: if $S^3_{m^2}(K)$ bounds a rational homology ball, then:
\begin{equation}\label{e:AG5.1}
m < \frac{3+\sqrt{9+8\nu^+(K)}}2.
\end{equation}
When $K$ is the torus knot $T_{p,q}$, $\nu^+(K) = g(K) = \frac{(p-1)(q-1)}2$.

\begin{prop}\label{pruning1}
If $S^3_{m^2}(T_{p,q})$ bounds a rational homology ball, where $q=kp+r$ and $2\leq r \leq p-2$, then $k\leq7$.
\end{prop}

\begin{proof}
If $m^2<m(T_{p,q})$, where $m(T_{p,q})$ is either $pq-\frac{p}{q^*}$ or $pq-\frac{q}{p^*}$, then $S^3_{m^2}(T_{p,q})$ does not bound a rational homology ball~\cite[Theorem 2]{OwensStrle}.
Since we assumed $2\leq r \leq p-2$, in either of the above cases one can show that $pq-\frac{p}{2} \leq m(T_{p,q})$.
Thanks to the inequality~\eqref{e:AG5.1}, if
\[
\frac{1}{4}\left(3+\sqrt{9+8\nu^+(T_{p,q}})\right)^2 \leq pq-\frac{p}{2},
\]
then no integral surgery on $T_{p,q}$ bounds a rational homology ball.
Expanding and substituting $2\nu = (p-1)(q-1)$ yields 
\[
3\sqrt{9+8\nu} \leq p+2q-11.
\]
After squaring both sides (note the right-hand side is necessarily positive) and further simplifying we obtain
\[
32pq+8q \leq 4q^2+p^2+14p+4.
\]
We now substitute $q=kp+r$ and gather like powers of $p$ to get the inequality
\[
0 \leq (4k^2-32k+1)p^2 + (8kr-8k-32r+14)p +4(r-1)^2.
\]
When $k\geq8$ and $r\geq2$, each of the above coefficients is positive, and so the inequality holds for all $p>0$. For example, when $k=8$ we get
\[
0\leq p^2 +(32r-50)p +4(r-1)^2.
\]
It follows that if $S^3_{m^2}(T_{p,q})$ bounds a rational homology ball, then $k\leq7$.
\end{proof}

\begin{prop}\label{pruning2}
The Seifert spaces associated to graphs of type $\Gamma_{1}$ with $k=7$ and $N\in\{1,2,3\}$ do not bound rational homology balls. 
\end{prop}

\begin{proof}
The statement is equivalent to the following assertion: if $7p+1 < q < 7p+p-1$ and $pq-4 \le m^2 \le pq-2$, then $S^3_{m^2}(T_{p,q})$ does not bound a rational homology ball.

In light of the inequality~\eqref{e:AG5.1}, in order to prove that $S^3_{m^2}(T_{p,q})$ does not bound a rational homology ball if $pq-4 \le m^2 \le pq-2$, it is enough to prove that
\[
\left(\frac{3+\sqrt{9+8\nu}}2\right)^2 \le pq-4.
\]
By expanding and squaring, we reduce it to the following, equivalent inequality:
\[
p^2 - 7pq + q^2 - 10p - 10q + 61 \ge 0.
\]
Call $r = q - 7p$, so that $1<r<p-1$. We now gather terms, and re-write:
\[
p^2 - 7pq + q^2 - 10p - 10q + 61 = (q-7p-10)q + p^2 - 10p + 61 = (r-10)q + (p-5)^2 + 36,
\]
and the latter quantity is clearly positive if $r \ge 10$.

Now suppose that $2 \le r < 10$ and write:
\[
(q-7p-10)q + p^2 - 10p + 61 \ge p^2 - 8q - 10p + 61 \ge p^2 - 66p - 11,
\]
and the latter is non-negative for $p \ge 67 > 33 + \sqrt{1098}$. For $p \le 66$, and $7p+2 \le q \le 7p+9$, we verified with a computer search that one among $pq-4$, $pq-3$, $pq-2$ is a square if and only if $(p,q;m)$ is one of the following:
\begin{align*}
&(7,52;19), (8,61;22), (11,82;30), (12,91;33), (13,100;36),\\
&(19,137;51), (28,201;75), (37,265;99).
\end{align*}
However, for each of these one can individually check that either the lattice embedding obstruction or the correction term obstruction (these correction terms can be computed by combining work of \cite{NiWu} and \cite{BorodzikLivingston}; see, for example, \cite[Section 5.2]{AcetoGolla}) shows that these examples do not bound rational homology balls.
\end{proof}

\end{subsection}

\begin{subsection}{Preliminaries on linear subsets and complementary strings}
The study of the embeddability of the lattices associated to the graphs in families $\Gamma_{1}$ and $\Gamma_{2}$ will heavily rely on previous results about lattice embeddings. Lisca did a complete analysis on the embeddability of \emph{linear subsets}, that is lattices coming from trees with vertices of valency at most two. For the reader's convenience, in this section we recall some basic terminology and facts from \cite{Lisca-ribbon} and \cite{Lisca-sums} and add some new technical results regarding linear subsets that will be useful later on. We warn the reader that even though Lisca's analysis concerns \emph{negative definite lattices} and we are working with positive definite ones, all his results translate verbatim to our context. (Several other sources we will cite in this paper will have the same feature.) At the end of this section, we will collect together the relevant results on \emph{complementary strings}.
\begin{de}
Let $S=\{v_1,\dots,v_n\}\subset(\Z^n,\mathrm{Id})$ be such that 
$$
\begin{cases}
v_{i}\cdot v_{i}\geq 1,\\
v_{i}\cdot v_{j}\in\{0,1\}\mathrm{\ if \ } i\neq j.
\end{cases}
$$
\begin{enumerate}
\item Define the \emph{intersection graph of S} as the graph having a vertex for each element of $S$ and an edge for every pair $(v_{i},v_{j})$ such that $v_{i}\cdot v_{j}=1$. This graph will be denoted $\Gamma_{S}$ and its vertices $v_{i}$ will be given the integral weights $v_{i}\cdot v_{i}$. $S$ will be called a \emph{linear subset} whenever $\Gamma_{S}$ is a linear graph.

\item Given $v\in\Z^{n}$ and some basis vector $e_{i}$ we say that $e_{i}$ \emph{hits} $v$ (or that $v$ hits $e_{i}$) if $v\cdot e_{i}\neq 0$.

\item Two vectors $v,w$ are \emph{linked} if there exists a basis vector that hits both of them.

\item A subset $S\subset \Z^{n}$ is \emph{irreducible} if for every pair of vectors $v,w\in S$ there exits a sequence of vectors in $S$
$$
v_{0}=v,v_{1},\dots,v_{n}=w
$$
such that $v_{i}$ and $v_{i+1}$ are linked for $i=0,\dots,n-1$. A subset which is not irreducible is said to be \emph{reducible}.

\item A linear irreducible subset such that $v_{i}\cdot v_{i}\geq 2$ for all $i$ is called a \emph{good subset}.

\item A good subset whose graph is connected is a \emph{standard subset}.

\item We indicate with $c(S)$ the number of connected components of $\Gamma_{S}$. (Not to be confused with the number of irreducible components for which we introduce no symbol.)

\item There is a particular type of connected component in a good subset $S$ called a \emph{bad component}.
We refer the reader to \cite[Definition~4.1]{Lisca-ribbon} for the details. The number of bad components of $S$ will be denoted by $b(S)$.
\end{enumerate}
\end{de}

\begin{lemma}\label{standardsubset}
Let $S=\{v_1,\dots,v_n\}\subset\Z^n$ be a standard subset. Write $[a_1,\dots,a_n]^-=\frac{p}{q}$ where $a_i=v_i\cdot v_i$. Suppose that there exists $v\in\Z^n$ such that $v\cdot v_1\neq 0$ and $v\cdot v_i=0$ for each $i>1$.
Then,
 $$
 \frac{(v\cdot v_1)^2}{v\cdot v}=\frac{p}{q}.
 $$
\end{lemma}
\begin{proof}
 The pairing associated with the subset $T=\{v\}\cup S$ is described by the following matrix
 $$M_T=\left(
 \begin{array}{lllll}
  x&y&&&\\
  y&a_1&1&&\\
  &1&\ddots&\ddots&\\
  &&\ddots&&1\\
  &&&1&a_n
 \end{array}
\right)
 $$
 where $x=v\cdot v$ and $y=v\cdot v_1$. Since $T\subset \Z^n$ has $n+1$ elements we must have $\det M_T=0$. Expanding the determinant from the first column we obtain
 $$
 0=\det M_T=x\det M_S-y^2\det M_{S\setminus\{v_1\}}.
 $$
 The conclusion follows since $\det M_S=p$ and $\det M_{S\setminus\{v_1\}}=q$.
\end{proof}

\begin{lemma}\label{badcomponents}
Let $S\subset\Z^m$ be a subset such that $\Gamma_S$ is linear and connected. Suppose that there exists $v\in\Z^m$ linked once to a final vector of $S$ and orthogonal to all other elements of $S$.
 Then, $S$ is not a bad component of any linear subset in $\Z^m$.
\end{lemma}
\begin{proof}
 Suppose $S$ is a bad component of some linear subset $S'\subset\Z^m$. By definition, $S$ is obtained from a subset $\{v_1,v_2,v_3\}$ with $v_1\cdot v_1=v_3\cdot v_3=2$ and $v_2\cdot v_2>2$ via $2$-final expansions.
 We may identify the vector $v_2$ with the corresponding element in $S$ since it is not affected by the sequence of expansions. Consider the subset $\tilde{S}:=S\setminus\{v_2\}$. Note that $\tilde{S}$ is obtained from 
 a subset of the form $\{e_1-e_2,e_1+e_2\}$ via $2$-final expansions. We may identify $\tilde{S}$ with its projection in a standard lattice $\Z^n$ where $n$ is the number of elements of $\tilde{S}$. Let $\tilde{v}$ be the projection 
 of $v$ in $\Z^n$. Note that $\tilde{v}\neq 0$ and that it is linked once to a final vector of $\tilde{S}$ and is orthogonal to any other element of $\tilde{S}$.
 
Now we proceed along the lines of Lemma \ref{standardsubset}. We may write  $\tilde{S}=S_1\cup S_2$ where each $S_i$ corresponds to a connected component of $\Gamma_{\tilde{S}}$. Suppose $\tilde{v}$ is linked to an element of $S_1$.
 The pairing associated with $T:=\{\tilde{v}\}\cup\tilde{S}$ is described by a matrix $M_T$ and we have
 $$
 0=\det M_T=\det M_{\{\tilde{v}\}\cup S_1}\cdot\det M_{S_2}\Rightarrow \det M_{\{\tilde{v}\}\cup S_1}=0.
 $$
 We obtain
 $$
 \frac{1}{\tilde{v}\cdot\tilde{v}}=\frac{p}{q}
 $$
 where $p$ and $q$ are obtained from $S_1$ as in Lemma \ref{standardsubset}. This gives us a contradiction since $p>q>0$.
\end{proof}

\begin{lemma}\label{badcomponents2}
Let $S$ be a good subset. If $c(S)=1$ and $I(S)<0$ then $I(S)\in\{-1,-2,-3\}$ and $S$ is not a bad component.
If $c(S)=2$ and $I(S)\leq -2$ then, one of the following holds
\begin{itemize}
 \item $S$ has no bad components and $I(S)=-2$;
 \item $S$ has one bad component and $I(S)\in\{-2,-3\}$;
 \item $S$ has two bad components and $I(S)=-4$.
\end{itemize}
\end{lemma}
\begin{proof}
This is essentially implicit in \cite{Lisca-ribbon} and \cite{Lisca-sums}. Suppose $c(S)=1$ and $I(S)<0$.
Then the conclusion follows from Lemma 5.1 and Theorem 6.4 in \cite{Lisca-ribbon}.

Now assume that $c(S)=2$ and that $I(S)\leq -2$. We distinguish two cases.\\
\emph{First case :} $I(S)+b(S)<0$. In this case we are in the situation described at the beginning of the proof of the main theorem in \cite[Page 2160]{Lisca-sums} (``First case: $S$ irreducible"). The conclusion then follows going through that proof.\\
\emph{Second case :} $I(S)+b(S)=0$. We claim that this possibility does not occur. Since
$b(S)\leq c(S)=2$, we must have $b(S)=2$ and $I(S)=-2$. Write $S=S_1\cup S_2$ where each $S_i$ is a bad component. By the definition of a bad component we can reduce $S$ via 2-final contractions to a subset
$\tilde{S}=\tilde{S_1}\cup\tilde{S_2}\subset\mathbb{Z}^6$ of the form
$$
\{e_1+e_2,v,e_1-e_2\}\cup\{e_3+e_4,w,e_3-e_4\}
$$
where 
\begin{itemize}
\item $I(\tilde{S})=I(S)$;
\item $v\cdot (e_1+e_2)=v\cdot (e_1-e_2)=w\cdot (e_3+e_4)=w\cdot (e_3-e_4)=1$;
\item each element of $\tilde{S_1}$ is orthogonal to each element of $\tilde{S_2}$;
\item $v\cdot v\geq 3$ and $w\cdot w\geq 3$.
\end{itemize}
In particular, since $I(\tilde{S})=-2$ we obtain
$$
v\cdot v+w\cdot w=8.
$$
Up to exchanging the role of $v$ and $w$ this leaves us with two possibilities,. Either
$(v\cdot v,w\cdot w)=(3,5)$ or $(v\cdot v,w\cdot w)=(4,4)$. It is easy to see that in either case such a subset does not exist. 
\end{proof}
\begin{lemma}\label{I=-2,b=1}
Let $S$ be a good subset such that $c(S)=2$, $b(S)=1$ and $I(S)<-1$. Then, one of the following conditions hold
\begin{enumerate}
\item $I(S)=-3$ and $S$ is obtained via $2$-final expansions from a subset of the form
$$
\{e_1+e_2,e_1+e_3+\dots+e_N,e_1-e_2\}\cup\{e_3-e_4,\dots,(-1)^{N-1}e_{N-1}-e_N\}
$$
where $N\geq 4$ and the $2$-final expansions are performed on the first component.
\item $I(S)=-2$ and $S$ is obtained via $2$-final expansions from a subset of the form
$$
\{e_1+e_2,e_1+e_3+\dots+e_N,e_1-e_2\}\cup\{e_3-e_4+e_{N+1},-e_4+e_5,\dots,
(-1)^{N-1}e_{N-1}-e_N,-e_3-e_N+e_{N+1}\}
$$
where $N\geq 4$ and the $2$-final expansions can be performed on both components. 
\end{enumerate}  
\end{lemma}
\begin{proof}
Let $S$ be a subset satisfying the assumptions of the lemma. Let us write $S=S_1\cup S_2$ and assume the $S_1$ is a bad component. Note that we are under the situation described  in the proof of the main theorem in \cite{Lisca-sums}. More specifically we can apply the argument in the first case (``$S$ irreducible") and the second subcase (``$b(S)=1$"). We conclude that the string associated to $S$ is obtained via
$2$-final expansions from a string that is either of the form 
$$
(2,n+1,2)\cup (2^{[n-1]})
$$
or 
$$
(2,n+1,2)\cup (3,2^{[n-2]},3).
$$
Note that in the first case all $2$-final contractions are performed on $S_1$ while in the second case 
$2$-final contractions are in general required on both $S_1$ and $S_2$. Call $\tilde{S}=\tilde{S}_1\cup\tilde{S}_2$ the subset 
obtained after all possible $2$-final contractions. In order to prove the lemma it is enough to show that
$\tilde{S}$ can be described as in the conclusion. In both cases this is a straightforward verification.
First write down all the vectors whose weight is $2$ and at that point all other elements are uniquely determined.  
\end{proof}

\begin{lemma}\label{extravector}
Let $S=\{v_1,\dots,v_n\}\subset\Z^{n}$ be a subset with $\Gamma_S$ linear such that one of the following conditions is satisfied:
\begin{enumerate}
\item $S$ is a standard subset;
\item $S$ is a good subset and $(c(S),I(S))=(2,-2)$.
\end{enumerate}
 Suppose that there exists $w\in\Z^n$ and an index $1\leq i\leq n$ such that
 $w\cdot v_i=1$ and $w\cdot v_j=0$ for each $j\neq i$. Then, $v_i$ is internal and $S$ is standard.
\end{lemma}
\begin{proof}
First assume $S$ is standard. Then, the conclusion follows from Proposition 8.1 in \cite{Paolo}.
Now assume $S$ is a good subset and $c(S)=2$. We need to show that there is no $w$ as above. Let
$A$ be the matrix whose columns are the elements of $S$. The conditions on the extra vector $w$
can be expressed as
$$
{^tA}w=e_i.
$$
 Multiplying both sides by $A$ we obtain
 $$
A{^tA}w=Ae_i=v_i. 
 $$
 Note that the matrix $A{^tA}$ is conjugated to ${^tA}A$ and this last matrix represents the pairing 
 associated with the subset $S$. In particular, these matrices have the same eigenvalues.
 It is easy to see that if $\lambda$ is an eigenvalue for ${^tA}A$ then $\lambda>1$. Therefore we have
 $$
 w\cdot w<(A{^tA}w)\cdot(A{^tA}w)=v_i\cdot v_i.
 $$ 
 Now consider the subset $\tilde{S}=(S\setminus\{v_i\})\cup\{w\}$. We want to reach a contradiction by examining the subset $\tilde{S}$. According to Lemma \ref{badcomponents2} we have $b(S)\in\{0,1\}$. We examine these two possibilities separately. 
 
 Suppose $b(S)=0$. It follows from \cite{Lisca-sums} that $S$ is obtained from
 a subset of the form $\{e_1+e_2,e_1-e_2\}$ via 2-final expansions (see the proof of the main theorem, first subcase of the first case). Using this fact it is easy to see that
 $\tilde{S}$ is a good subset with no bad components. Also, since $w\cdot w<v_i\cdot v_i$ we have 
 $I(\tilde{S})<I(S)=-2$. Finally, note that $c(\tilde{S})\in\{3,4\}$ depending on whether $v_i$
 is final or internal. In any case we obtain a contradiction with Proposition 4.10 in \cite{Lisca-sums}.
 
 Now assume that $b(S)=1$. Then, $S$ satisfies the assumption of Lemma \ref{I=-2,b=1}. In particular 
 the string associated to $S$ is obtained via
$2$-final expansions from a string of the form 
$$
(2,n+1,2)\cup (3,2^{[n-2]},3).
$$
Let us denote by $u$ the element in $S$ corresponding to the vertex of weight $n+1$ in the above string.
Suppose that $u\neq v_i$. Following the description of Lemma \ref{I=-2,b=1} we can verify that $\tilde{S}$ is irreducible and is therefore a good subset with $b(\tilde{S})\leq 1$ and $c(\tilde{S})\in\{3,4\}$. This contradicts Proposition 4.10 in \cite{Lisca-sums}. 
We are left with the possibility $u=v_i$. If $\tilde{S}$ is irreducible then we may argue as before 
and find a contradiction with Proposition 4.10 in \cite{Lisca-sums}. Assume that $\tilde{S}$ is reducible. Note that in $S\setminus\{u\}$ all the basis coordinates of $\mathbb{Z}^n$ are used. It follows that $\tilde{S}$ has two irreducible components. We may write $\tilde{S}=S_1\cup S_2\cup S_3\cup\{w\}$, where $S_1$ and $S_2$ correspond to
the two connected components of the graph associated with $\tilde{S}$ originally linked to $u$. Clearly $S_1$ and $S_2$ are contained in the same irreducible component. It follows that the decomposition of 
$\tilde{S}$ in irreducible components is given by 
$\tilde{S}=U_1\cup U_2$ where $U_1=S_1\cup S_2$ and 
$U_2=S_3\cup\{w\}$. Associated with this decomposition we have a splitting 
$\mathbb{Z}^n=\mathbb{Z}^{n_1}\oplus\mathbb{Z}^{n_2}$
so that each $U_i$ is embedded in $\mathbb{Z}^{n_i}$.
For $i\in\{1,2\}$ let 
$\pi_i:\mathbb{Z}^n\rightarrow\mathbb{Z}^{n_i}$ be the projection corresponding to the above splitting.
Note that $\pi_2(u)\cdot \pi_2(u)\geq 2$ and $\pi_2(u)\cdot w=u\cdot w=1$.
Now consider the subset 
$$
T:=S_3\cup\{w\}\cup\{\pi_2(u)\}.
$$
Since $|T|>n_2$, the pairing associated with this subset must have vanishing determinant. From this we obtain 
$$
(w\cdot w)(\pi_2(u)\cdot \pi_2(u))-1=0
$$
which is impossible.
\end{proof}

Recall that the strings $(a_1,\dots,a_n)$ and $(b_1,\dots,b_m)$ in both families of graphs $\Gamma_1$ and $\Gamma_2$ are \emph{complementary}, that is, they  satisfy  
\begin{equation}\label{e:1}
([a_n,\dots,a_1]^-)^{-1}+([b_m,\dots,b_1]^-)^{-1}=1.
\end{equation}
We proceed to list some of the basic properties of strings of integers satisfying \eqref{e:1} and their embeddings into standard definite lattices. 

\begin{facts}{Complementary string properties.}\label{f}
\begin{enumerate}

\item Given an arbitrary string of integers $(a_1,\dots,a_n)$ with $a_{i}\geq 2$ for all $i$, there is a unique complementary string $(b_1,\dots,b_m)$, where $b_{i}\geq 2$ for all $i$. The two strings are related to each other by Rimenschneider point rule and if $n,m>1$ and say $a_{1}=2$, then $b_{1}>2$ \cite{Riemensh}.
%
%
\item To any weighted graph $\Gamma$, with vertices $v_{i}$ of weight $v_{i}\cdot v_{i}\in\Z$, one can associate the following quantity
$$
I(\Gamma):=\sum_{v_{i}\in\Gamma}(v_{i}\cdot v_{i}-3).
$$ 
It is well known \cite[Lemma~2.6]{Lisca-ribbon} that for a graph $\Lambda$ consisting of two linear components with weights given by complementary strings we have: 
\begin{equation}\label{e:Icom}
I(\Lambda)=\sum_{i=1}^{n}a_{i}+\sum_{i=1}^{m}b_{i}-3(m+n)=-2.
\end{equation}
The two families of graphs in Proposition~\ref{seifertplumbing} satisfy:
\begin{equation}\label{e:I}
I(\Gamma_{1})=k-N-5\quad\mathrm{and}\quad I(\Gamma_{2})=-N-k-1.
\end{equation}
\item\label{f:unique} Let $(\Z\Lambda,Q_{\Lambda})$ be the lattice associated to two complementary strings $(a_1,\dots,a_n)$ and $(b_1,\dots,b_m)$ and consider an embedding $\varphi$ of $(\Z\Lambda,Q_{\Lambda})$ into $(\Z^{r},\mathrm{Id})$. Then,
	\begin{enumerate}
		\item $r\geq m+n;$
		\item if a basis vector $e_{i}$ hits both the vertices with weight $a_{1}$ and $b_{1}$, then the image of $\varphi$ is contained in $\Z^{{n+m}}$ 				and the embedding $\varphi$ is unique \cite[Lemma~3.1]{LiscaLec} and  \cite[Lemma~5.2]{BBL}. For the complementary strings $(2)$ and $(2)$ this unique embedding in $\Z^{2}$ with basis $\{e_{1},e_{2}\}$ is, up to sign, given by $(e_{1}+e_{2})$ and $(e_{1}-e_{2})$.
	\end{enumerate}
\item\label{f:+1} Given two strings of integers $(a_1,\dots,a_n)$ and $(b_1,\dots,b_m)$ we consider the following two operations
\[
\begin{array}{ccc}
\begin{cases}
  (a_1,\dots,a_n)\mapsto (a_1,\dots,a_n,2) \\  (b_1,\dots,b_m)\mapsto (b_1,\dots,b_m+1)    
  \end{cases}
  &\mathrm{and}&
  \begin{cases}
  (a_1,\dots,a_n)\mapsto (a_1,\dots,a_n+1)
 \\ (b_1,\dots,b_m)\mapsto (b_1,\dots,b_m,2) 
  \end{cases}.  
 \end{array}
\]
If we start with the strings $(2)$ and $(2)$, any combination of these operations yields two complementary strings (it is enough to compare this construction with the Riemenschneider point rule). Moreover, these operations are compatible with the unique embedding from the preceding point in the following sense. If we start with $(e_{1}+e_{2})$ and $(e_{1}-e_{2})$, the strings resulting from these operations  admit, up to sign, a unique embedding into $\Z^{3}$
\[
\begin{array}{lcl}
  (e_{1}+e_{2})\mapsto (e_{1}+e_{2},e_{2}+e_{3}) &\quad\quad & (e_{1}+e_{2})\mapsto (e_{1}+e_{2}+e_{3})    \\
  (e_{1}-e_{2})\mapsto (e_{1}-e_{2}+e_{3}) &\quad\quad &(e_{1}-e_{2})\mapsto (e_{1}-e_{2},-e_{2}+e_{3})   
 \end{array}
\]
An easy induction shows that, for generic complementary strings, the unique embedding $\varphi$ from the preceding point is the one compatible with these operations. 

We can extend the arguments above to show that given two complementary strings $(a_{1},\cdots,a_{n})$ and $(b_{1},\cdots,b_{m})$, if the two component linear graph
\[
  \begin{tikzpicture}[xscale=1.5,yscale=-0.5]
    \node (A0_0) at (0, 0) {$a_n$};
    \node (A0_2) at (2, 0) {$a_1$};
    \node (A0_4) at (4, 0) {$b_1$};
    \node (A0_6) at (6, 0) {$b_m+1$};
    \node (A1_0) at (0, 1) {$\bullet$};
    \node (A1_1) at (1, 1) {$\dots$};
    \node (A1_2) at (2, 1) {$\bullet$};
    \node at (6, 1.5) {$u$};
    \node (A1_4) at (4, 1) {$\bullet$};
    \node (A1_5) at (5, 1) {$\dots$};
    \node (A1_6) at (6, 1) {$\bullet$};
    \path (A1_4) edge [-] node [auto] {$\scriptstyle{}$} (A1_5);
    \path (A1_0) edge [-] node [auto] {$\scriptstyle{}$} (A1_1);
    \path (A1_1) edge [-] node [auto] {$\scriptstyle{}$} (A1_2);
    \path (A1_5) edge [-] node [auto] {$\scriptstyle{}$} (A1_6);
  \end{tikzpicture}
  \]
admits an embedding into a standard lattice such that there exists a basis vector hitting both $a_{1}$ and $b_{1}$, then this embedding is necessarily the unique embedding from the preceding point associated to the graph 
\[
  \begin{tikzpicture}[xscale=1.5,yscale=-0.5]
    \node (A0_0) at (0, 0) {$a_n$};
    \node (A0_2) at (2, 0) {$a_1$};
    \node (A0_4) at (4, 0) {$b_1$};
    \node (A0_6) at (6, 0) {$b_m$};
    \node at (6, 1.5) {$u'$};
    \node (A1_0) at (0, 1) {$\bullet$};
    \node (A1_1) at (1, 1) {$\dots$};
    \node (A1_2) at (2, 1) {$\bullet$};
    \node (A1_4) at (4, 1) {$\bullet$};
    \node (A1_5) at (5, 1) {$\dots$};
    \node (A1_6) at (6, 1) {$\bullet$};
    \path (A1_4) edge [-] node [auto] {$\scriptstyle{}$} (A1_5);
    \path (A1_0) edge [-] node [auto] {$\scriptstyle{}$} (A1_1);
    \path (A1_1) edge [-] node [auto] {$\scriptstyle{}$} (A1_2);
    \path (A1_5) edge [-] node [auto] {$\scriptstyle{}$} (A1_6);
  \end{tikzpicture}
  \]
where the only difference is that $\varphi(u)=\varphi(u')+e$ where $e$ is a basis vector which does not hit any other vertex in the graph. The main lines of the argument go as follows: without loss of generality assume $a_{1}=2$ and $b_{1}\geq 2$. Since there is a basis vector that hits both of them, the embedding of the vertex with weight $b_{1}$ necessarily contains all the basis vectors hitting the 2-chain starting with $a_{1}$. The pattern repeats: each time that there is a 2-chain on one of the complementary legs, the vectors hitting it hit also one vertex of weight greater than 2 in the other complementary leg. The basis vectors hitting the vertex $u$ will account for it being connected to another vertex in its complementary string and for being orthogonal to the vertices in the other string. Since the weight of $u$ is $b_{m}+1$, the only way to embed it is for it to be hit by a new basis vector $e$, different from all the rest hitting the complementary strings. 

\end{enumerate}
\end{facts}

\end{subsection}

\begin{subsection}{Operations on graphs and Determinants}
The main strategy in the proofs in this section is to directly manipulate the graphs and embeddings seeking for contradictions. We highlight here some of the most used procedures.

\begin{facts}{Operations on graphs}\label{op}
\begin{enumerate}
\item Given a graph with an embedding $\varphi:(\Z\Gamma,Q_{\Gamma})\hookrightarrow(\Z^{r},\mathrm{Id})$ we might \emph{delete some basis vector} obtaining a new graph $\Gamma'$ with an embedding into $(\Z^{r-1},\mathrm{Id})$. If we are deleting the vector $e_{j}$, this amounts to composing $\varphi$ with the projection $\pi:(\Z^{r},\mathrm{Id})\rightarrow (\Z^{r}\setminus\{e_{j}\},\mathrm{Id})$. For example, deleting $e_{2}$ we obtain:
\[
  \begin{tikzpicture}[xscale=1.5,yscale=-0.5]
  	\node at (-0.5,1) {$\Gamma:$};
    \node (A0_0) at (0, 0) {$\scriptstyle e_{1}+e_{2}$};
    \node at (1, 0) {$\scriptstyle e_{1}+e_{3}$};
    \node (A0_2) at (2, 0) {$\scriptstyle e_{1}-e_{2}$};
    \node (A0_4) at (5, 0) {$\scriptstyle e_{1}$};
    \node at (6, 0) {$\scriptstyle e_{1}+e_{3}$};
    \node (A0_6) at (7, 0) {$\scriptstyle e_{1}$};
    \node (A1_0) at (0, 1) {$\bullet$};
    \node (A1_1) at (1, 1) {$\bullet$};
    \node (A1_2) at (2, 1) {$\bullet$};
    \node at (4.5,1) {$\Gamma':$};
    \node (A1_4) at (5, 1) {$\bullet$};
    \node (A1_5) at (6, 1) {$\bullet$};
    \node (A1_6) at (7, 1) {$\bullet$};
    \path (A1_4) edge [-] node [auto] {$\scriptstyle{}$} (A1_5);
    \path (A1_0) edge [-] node [auto] {$\scriptstyle{}$} (A1_1);
    \path (A1_1) edge [-] node [auto] {$\scriptstyle{}$} (A1_2);
    \path (A1_5) edge [-] node [auto] {$\scriptstyle{}$} (A1_6);
    \draw (A1_4) to[bend left=60] (A1_6);
    \end{tikzpicture}
  \]
\item In a graph $\Gamma$ we might \emph{delete some vertex}. If we delete say the vertex $w$, we mean to consider the graph $\Gamma'$ obtained from $\Gamma$ by deleting the vertex $w$ and all its adjacent edges. If $\Gamma$ has an embedding $\varphi$, then we consider $\Gamma'$ equipped with the restriction of $\varphi$.
\item\label{bd} Given a graph $\Gamma$ with an embedding in a standard lattice of arbitrary rank, and a vertex $v$ of valency 2 and weight 1, the blow down operation is compatible with the embedding: if the vertex with weight 1 had embedding $e_{j}$, then the graph obtained by blowing down the 1 is the graph with embedding $\Gamma'$ obtained by deleting $e_{j}$ and $v$ from $\Gamma$. There might be some issue with the sign of the edge in $\Gamma'$, but as long as we work with trees, this has no relevance \cite[Proposition~2.1,R0]{Neumann}.
\end{enumerate}
\end{facts}

Many of the arguments in the upcoming proofs have to do with weighted graphs having vanishing determinant. We remind the reader that: 

\begin{facts}{Remarks on determinants}\label{rem}
\begin{enumerate}
\item Plumbing along a star-shaped three legged weighted graph yields a 4-manifold with boundary a Seifert space $Y$. If the Seifert invariants are $Y=Y(e_{0};r_{1},r_{2},r_{3})$ with $e_{0}\in\Z$ and $r_{i}\in(0,1)\cap\Q$, then the associated three legged graph has central weight $e_{0}$ and the weights on the legs are the coefficients of the negative continued fractions associated to $\frac{1}{r_{i}}$. The determinant of the graph, which coincides with the order of $H_{1}(Y;\Z)$, will vanish if and only if (see for example \cite[Lemma~4.2]{NeumannRaymond})
\begin{equation}\label{e:det=0}
e_{0}=r_{1}+r_{2}+r_{3}
\end{equation}
\item\label{ddN} The determinant associated to any connected weighted graph whose weights are at least the valency of the vertices and having at least one vertex with weight strictly bigger than its valency is non-vanishing \cite{Shivakumar_KimHoChew}.
\item\label{dd} A consequence of the preceding point is that a linear graph whose weights are all at least 2 has non-vanishing determinant. If a linear graph has exactly one vertex of weight 1 and vanishing determinant, then this vertex needs to be internal.
\item\label{small} A graph with $n$ vertices which admits an embedding into a standard lattice of rank smaller than $n$ has vanishing determinant. 
\end{enumerate}
\end{facts}

The following technical lemmas apply to a much wider range of graphs than the ones studied in this article and they will be extensively used in many of the subsequent proofs.

\begin{lemma}\label{l:det}
Let $\Psi$ be a $3$--legged star shaped positive definite weighted graph with central weight $2$ and all other weights at least $2$. If $\Psi$ admits an embedding into a standard lattice of any rank, then $\Psi$ has non-vanishing determinant. 
\end{lemma}
\begin{proof}
This lemma is an easy consequence of \cite[Lemma~3.3]{LiscaLec}. The determinant of $\Psi$ will vanish if and only if the sum of the inverses of the three continued fractions associated to $\Psi$ equals 2 (see Equation\eqref{e:det=0}). These inverses, which belong to the interval $(0,1)$, are denoted by $r_{1},r_{2},r_{3}$ in \cite{LiscaLec}. If the determinant of $\Psi$ vanishes, the sum of any two of the $r_{i}$'s has to be greater than $1$, so \cite[Lemma~3.3]{LiscaLec} applies and there is no embedding into any standard lattice (even if in \cite{LiscaLec} \emph{negative} lattices instead of positive ones were considered, every result applies to our positive setting). The statement follows.
\end{proof}

\begin{re}
The statement of the last lemma can be strengthened to the case of a star shaped graph with $r$ legs, central weight $r-1$ and all other weights at least $2$. The proof is the same.
\end{re}

\end{subsection}

\begin{subsection}{Embedding $\Gamma_{1}$ and $\Gamma_{2}$: generalities and definition of $\Gamma$}
In this very first part of the proof Proposition~\ref{p:mainA}, we establish some necessary properties of the embeddings of the graphs of type $\Gamma_{1}$ and $\Gamma_{2}$.

\begin{lemma}\label{l:firststep}
Call $v$ the only vector in $\Gamma_{i}$ adjacent to the central vertex and of weight strictly greater than 2 and let $N$ be the length of the 2-leg.
In any embedding of $(\Z\Gamma_{i},Q_{\Gamma_{i}})$ into $(\Z^{\mathrm{rk}(\Gamma_{i})},\mathrm{Id})$, all the vectors hitting the 2-leg hit $v$. Moreover, $|v|>N+1$. 
\end{lemma}
\begin{proof}
Let $\varphi$ be an embedding  of $(\Z\Gamma_{i},Q_{\Gamma_{i}})$ in the standard lattice of the same rank. Up to automorphisms of the standard lattice, we might assume that the embedding of the only trivalent vertex is of the form $e_{1}+e_{2}$. There are now two different cases to study: either there is an $e_{i}$, $i=1,2$, which appears with a non-zero coefficient in the embedding of both the vertices with weights $a_{1}$ and $b_{1}$, or there is not. If the former possibility holds then:
\begin{itemize}
\item If we consider the graph $\Gamma_{1}$, then by \cite[Lemma~3.1]{LiscaLec} the restriction of the embedding $\varphi$ to the subgraph
 \[
  {\Gamma'}:=
  \begin{tikzpicture}[xscale=1.5,yscale=-0.5]
    \node (A0_0) at (0, 0) {$a_n$};
    \node (A0_2) at (2, 0) {$a_1$};
    \node (A0_3) at (3, 0) {$2$};
    \node (A0_4) at (4, 0) {$b_1$};
    \node (A0_6) at (6, 0) {$b_m$};
    \node (A1_0) at (0, 1) {$\bullet$};
    \node (A1_1) at (1, 1) {$\dots$};
    \node (A1_2) at (2, 1) {$\bullet$};
    \node (A1_3) at (3, 1) {$\bullet$};
    \node (A1_4) at (4, 1) {$\bullet$};
    \node (A1_5) at (5, 1) {$\dots$};
    \node (A1_6) at (6, 1) {$\bullet$};
    \path (A1_4) edge [-] node [auto] {$\scriptstyle{}$} (A1_5);
    \path (A1_0) edge [-] node [auto] {$\scriptstyle{}$} (A1_1);
    \path (A1_1) edge [-] node [auto] {$\scriptstyle{}$} (A1_2);
    \path (A1_5) edge [-] node [auto] {$\scriptstyle{}$} (A1_6);
    \path (A1_2) edge [-] node [auto] {$\scriptstyle{}$} (A1_3);
    \path (A1_3) edge [-] node [auto] {$\scriptstyle{}$} (A1_4);
  \end{tikzpicture}
  \]
yields an embedding of $(\Z\Gamma',Q_{\Gamma'})$ into $(\Z^{n+m+1},\mathrm{Id})$. $\Gamma'$ is a standard subset and therefore the existence of $\varphi$ contradicts Lemma~\ref{extravector}, where the role of the vertex $w$ is played by the projection of the vertex in $\Gamma$ attached to the right of the $m^{\mathrm{th}}$ vector on the top right leg of the graph.
\item If we consider the graph $\Gamma_{2}$, then we know by Property~\ref{f}(\ref{f:+1}) that the restriction of the embedding $\varphi$ to the subgraph
 \[
  {\Gamma''}:=
  \begin{tikzpicture}[xscale=1.5,yscale=-0.5]
    \node (A0_0) at (0, 0) {$a_n$};
    \node (A0_2) at (2, 0) {$a_1$};
    \node (A0_4) at (4, 0) {$b_1$};
    \node (A0_6) at (6, 0) {$b_m+1$};
    \node at (6, 1.5) {$u$};
    \node (A1_0) at (0, 1) {$\bullet$};
    \node (A1_1) at (1, 1) {$\dots$};
    \node (A1_2) at (2, 1) {$\bullet$};
    \node (A1_4) at (4, 1) {$\bullet$};
    \node (A1_5) at (5, 1) {$\dots$};
    \node (A1_6) at (6, 1) {$\bullet$};
    \path (A1_4) edge [-] node [auto] {$\scriptstyle{}$} (A1_5);
    \path (A1_0) edge [-] node [auto] {$\scriptstyle{}$} (A1_1);
    \path (A1_1) edge [-] node [auto] {$\scriptstyle{}$} (A1_2);
    \path (A1_5) edge [-] node [auto] {$\scriptstyle{}$} (A1_6);
  \end{tikzpicture}
  \]
coincides, except for the vertex $u$, with the unique embedding of complementary strings discussed in Properties~\ref{f}(\ref{f:unique}) and \ref{f}(\ref{f:+1}). Notice that this restriction yields an embedding $\varphi''$ of $(\Z\Gamma'',Q_{\Gamma''})$ into $(\Z^{n+m+1},\mathrm{Id})$, a lattice of rank one bigger than the number of vertices in $\Gamma''$. To go from $\Gamma''$ to $\Gamma_{2}$ we need to add a series of vertices with weight 2: the central vertex, the final 2-chain attached to $u$ in the qc-leg and the 2-leg. Since 2-chains have a unique embedding into a standard lattice, we conclude that there is a unique embedding of $\Gamma_{2}$ compatible with the restriction $\varphi''$. However, in this fashion we have obtained an embedding of $(\Z\Gamma_{2},Q_{\Gamma_{2}})$ into $(\Z^{\mathrm{rk}(\Gamma_{2})+1},\mathrm{Id})$, and we conclude, by the uniqueness part of the argument, that there is no embedding into $(\Z^{\mathrm{rk}(\Gamma_{2})},\mathrm{Id})$.
\end{itemize}

It only remains to consider the case in which neither $e_{1}$ nor $e_{2}$ hit both the vectors with weights $a_{1}$ and $b_{1}$. It follows that one between $e_{1}$ and $e_{2}$ hits the first weight 2 vertex in the bottom-right $2$-chain and either the vertex with weight $a_{1}$ or weight $b_{1}$. Without loss of generality, let us assume that it is $e_{2}$ and let us call $v$ the described vertex. Since, up to automorphisms, the embedding of the 2-leg is of the form $(e_{2}+e_{3},e_{3}+e_{4},\dots,e_{N+1}+e_{N+2})$, the vertex $v$ with weight $a_{1}$ or $b_{1}$ is orthogonal to this 2-chain and $e_{1}$ does not hit $v$, we have $v=e_{2}-e_{3}+e_{4}\dots(-1)^{N}e_{N+2}+v'$ where $v'\cdot e_{i}=0$ for all $i=1,\dots,N+2$.

If $v'=0$, then the subgraph formed by the string $(|v|)$, the trivalent vertex and the 2-leg is a standard subset of $\Z^{N+2}$. As before, by Lemma~\ref{extravector}, we know that no vector can be attached to $v$ in such a way that the embedding extends. From here we deduce that $n$ or $m$ equals 1, which contradicts our assumptions on the graphs $\Gamma_{1}$ and $\Gamma_{2}$. We then necessarily have that $v'\neq 0$ and it follows that the vector $v$ of the statement has weight $N+1+|v'|>N+1$.
\end{proof}

\begin{re}
In the last lemma we have shown that every graph of type $\Gamma_{2}$ admits an embedding into $(\Z^{\mathrm{rk}(\Gamma_{2})+1},\mathrm{Id})$. In fact this embedding always exists, without any restrictions on the value of $k$ and also for the case when $m$ or $n$ equals 1. A simple example is given by:
\[
  \begin{tikzpicture}[xscale=1.6,yscale=-0.5]
    \node (A0_4) at (4, 0) {$\scriptstyle e_{1}-e_{6}+e_{7}$};
    \node at (5, 0) {$\scriptstyle e_{7}-e_{8}+e_{9}$};
    \node (A0_6) at (6, 0) {$\scriptstyle e_{9}-e_{10}+e_{11}$};
    \node (A0_7) at (7, 0) {$\scriptstyle e_{11}+e_{12}$};
    \node at (8, 0) {$\scriptstyle e_{12}+e_{13}$};
    \node (A1_4) at (4, 1) {$\bullet$};
    \node (A1_5) at (5, 1) {$\bullet$};
    \node (A1_6) at (6, 1) {$\bullet$};
    \node (A1_8) at (8, 1) {$\bullet$};
    \node (A2_0) at (0, 2) {$\scriptstyle e_{8}+e_{9}+e_{10}$};
    \node at (1, 2) {$\scriptstyle e_{6}+e_{7}+e_{8}$};
    \node (A1_7) at (7, 1) {$\bullet$};
    \node (A1_6) at (6, 1) {$\bullet$};
    \node (A2_2) at (2, 2) {$\scriptstyle e_{1}+e_{6}$};
    \node (A2_3) at (3, 2) {$\scriptstyle e_{1}+e_{2}$};
    \node (A3_0) at (0, 3) {$\bullet$};
    \node (A3_1) at (1, 3) {$\bullet$};
    \node (A3_2) at (2, 3) {$\bullet$};
    \node (A3_3) at (3, 3) {$\bullet$};
    \node (A4_4) at (4, 4) {$\scriptstyle e_{2}+e_{3}$};
    \node at (5, 4) {$\scriptstyle e_{3}+e_{4}$};
    \node (A4_6) at (6, 4) {$\scriptstyle e_{4}+e_{5}$};
    \node (A5_4) at (4, 5) {$\bullet$};
    \node (A5_5) at (5, 5) {$\bullet$};
    \node (A5_6) at (6, 5) {$\bullet$};
    \path (A1_4) edge [-] node [auto] {${}$} (A1_5);
    \path (A1_6) edge [-] node [auto] {$\scriptstyle{}$} (A1_7);
    \path (A1_7) edge [-] node [auto] {$\scriptstyle{}$} (A1_8);
    \path (A3_0) edge [-] node [auto] {$\scriptstyle{}$} (A3_1);
    \path (A1_5) edge [-] node [auto] {$\scriptstyle{}$} (A1_6);
    \path (A3_3) edge [-] node [auto] {$\scriptstyle{}$} (A1_4);
    \path (A5_4) edge [-] node [auto] {$\scriptstyle{}$} (A5_5);
    \path (A3_2) edge [-] node [auto] {$\scriptstyle{}$} (A3_3);
    \path (A3_3) edge [-] node [auto] {$\scriptstyle{}$} (A5_4);
    \path (A3_1) edge [-] node [auto] {$\scriptstyle{}$} (A3_2);
    \path (A5_5) edge [-] node [auto] {$\scriptstyle{}$} (A5_6);
  \end{tikzpicture}
\]
However, we will show in this section that no graph in the family $\Gamma_{2}$ admits an embedding into $(\Z^{\mathrm{rk}(\Gamma_{2})},\mathrm{Id})$.
\end{re}

Since many of the arguments in what follows can be applied to the two graphs we are studying in this section, we will use from now on the symbol $\Gamma$ to denote either $\Gamma_{1}$ or $\Gamma_{2}$ with the assumptions of the statement of Proposition~\ref{p:mainA}. We warn the reader that, as the section develops, the same symbol $\Gamma$ will be used to denote also some other graphs we will encounter. 

We summarize what we have learned in the preceding lemma about the embedding of $\Gamma$ in the diagram in Figure~\ref{f:convenzioni1}. We will fix these conventions in the remainder of this section: the central vertex of $\Gamma$ has embedding $e_{1}+e_{2}$, the 2-leg in the bottom of the diagram has embedding $(e_{2}+e_{3},\dots,e_{N+1}+e_{N+2})$, the vertex not in this chain adjacent to the central one with weight 2 embeds as $u=e_{1}+e_{r}$, and the remaining vertex linked to the trivalent one embeds as $v=\sum_{i=2}^{N+2}(-1)^{i}e_{i}+v'$ with $v'\neq 0$.

\begin{figure}
\[
  \begin{tikzpicture}[xscale=1.5,yscale=-0.5]
	\node (A3_0) at (2.25,0) {$\scriptstyle v=\sum_{i=2}^{N+2}(-1)^{i}e_{i}+v'$};
	\node (A3_0) at (2.25,1.5) {$\scriptstyle v$};
	\node (A1_3) at (2.25,1) {$\bullet$};
	\node (A3) at (3.75,0) {$\scriptstyle e_{1}+e_{r}$};
	\node (A1) at (3.75,1) {$\bullet$};
	\node (A3_0) at (3.75,1.5) {$\scriptstyle u$};
    \node (A2_3) at (2.5, 3) {$\scriptstyle e_{1}+e_{2}$};
    \node (A3_3) at (3, 3) {$\bullet$};
    \node (A4_4) at (4, 4) {$\scriptstyle e_{2}+e_{3}$};
    \node (A4_6) at (6, 4) {$\scriptstyle e_{N+1}+e_{N+2}$};
    \node (A5_4) at (4, 5) {$\bullet$};
    \node (A5_5) at (5, 5) {$\dots$};
    \node at (4.75, 1) {$\dots$};
    \node at (1.25, 1) {$\dots$};
    \node (A5_6) at (6, 5) {$\bullet$};
    \path (A3_3) edge [-] node [auto] {$\scriptstyle{}$} (A1_3);
    \path (A3_3) edge [-] node [auto] {$\scriptstyle{}$} (A1);
    \path (A5_4) edge [-] node [auto] {$\scriptstyle{}$} (A5_5);
    \path (A3_3) edge [-] node [auto] {$\scriptstyle{}$} (A5_4);
    \path (A5_5) edge [-] node [auto] {$\scriptstyle{}$} (A5_6);
    \path (A1_3) edge [-] node [auto] {$\scriptstyle{}$} (1.5,1);
    \path (A1) edge [-] node [auto] {$\scriptstyle{}$} (4.5,1);
   \end{tikzpicture}
  \]
  \caption{Conventions of the embedding of $\Gamma$ into the standard lattice after Lemma~\ref{l:firststep}.}\label{f:convenzioni1}
\end{figure}
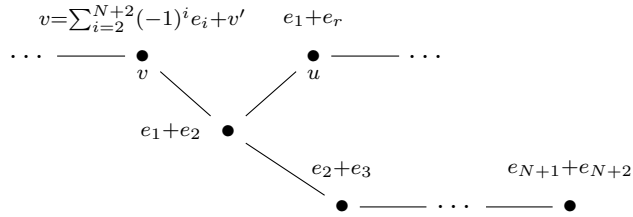
\end{subsection}

\begin{subsection}{Strategy and first bounds on $I(\Gamma)$}
The key idea in the proof of Proposition~\ref{p:mainA} is to compute the value of $I(\Gamma)$, which is given in \eqref{e:I}, via the embedding. Since the graph $\Gamma$ embeds, each vertex can be expressed as $v_{i}=\sum_{j}\alpha_{j}^{i}e_{j}$, where the $e_{j}$ form a basis of the standard lattice and $\alpha_{j}^{i}\in\Z$. We have then
$$
I(\Gamma)=\sum_{i}\left(v_{i}\cdot v_{i}-3\right)=\sum_{i}\left(\left(\sum_{j}(\alpha_{j}^{i})^{2}\right)-3\right).
$$
We will say that a vector $e_{j_{0}}$ \emph{appears n times in the embedding} if $|\{i\,|\,\alpha_{j_{0}}^{i}\neq 0\}|=n$, or equivalently, if it hits precisely $n$ different vertices in the graph.
Since we are assuming $\Gamma$ embeds in the standard lattice of the \emph{same rank}, in the above formula $i$ and $j$ range over identical values.
Hence we can permute the order of the summation to obtain
$$
I(\Gamma)=\sum_{j}\left(\left(\sum_{i}(\alpha_{j}^{i})^{2}\right)-3\right).
$$
From this it makes sense to define the \emph{contribution to $I$} of the basis vector $e_{j}$ to be the quantity $\left(\sum_{i}(\alpha_{j}^{i})^{2}\right)-3.$

Notice that any vector $e_{j}$ which appears exactly 3 times in the embedding with coefficient $\pm 1$ will have no contribution to $I$; and for a vector to contribute \emph{negatively} it needs to appear at most twice with coefficient $\pm 1$. 

In the next lemma we establish a lower bound on the contribution to $I(\Gamma)$ of the vectors in the embedding hitting the 2-leg.
\begin{lemma}\label{l:2rep}
Any embedding of $\Gamma$ into $(\Z^{|\Gamma|},\mathrm{Id})$ must satisfy that  the all the $e_{i}$'s hitting the $2$-leg, hit also a vector in $\Gamma$ different from $v$ in Figure~\ref{f:convenzioni1}. It follows that the vectors  $\{e_{2},\dots,e_{N+2}\}$ contribute at least $N$ to $I(\Gamma)$.
\end{lemma}
\begin{proof}
The existence of the embedding puts us in the assumptions of Lemma~\ref{l:firststep} and thus we follow the conventions in Figure~\ref{f:convenzioni1}. We want to show that the vectors $e_{i}$ with $i=1,2,...,N+2$ hit some vector in $\Gamma$ different from the ones depicted in Figure~\ref{f:convenzioni1}, which readily implies the statement, since then the $N$ vectors $e_{2},\cdots,e_{N+1}$ appear all at least 4 times in the embedding. If this were not the case, one could consider the graph with embedding obtained by deleting all the  $e_{i}$ with $i=1,2,...,N+2$. The resulting graph $\Psi$ with $n$ vertices is linear with two connected components, and admits an embedding in a standard lattice of rank $n-1$. However, this is not possible since its determinant is non-vanishing. Indeed, by our assumptions on $\Gamma$, we know that after deleting the $e_{i}$ with $i=1,2,...,N+2$, all the weights in the remaining linear graph are at least two, except for the weights associated to the final vertex $u$ and perhaps the final vertex $v$. It follows by Remark~\ref{rem}(\ref{dd}) that $\Psi$ has non-vanishing determinant. So at least one of the $e_{i}$ with $i=1,2,...,N+2$ must appear elsewhere in the embedding, but then, by othogonality with the 2-leg, they must all appear.
\end{proof}
\end{subsection}

\begin{subsection}{Vectors appearing only once in the embedding and determinant of $\Psi$}
There are only two possibilities for a vector $e_{j}$ to contribute negatively to $I$: either it appears exactly once or exactly twice in the embedding with coefficient $\pm 1$. In this subsection we show that the former possibility does not occur in the embedding of the graphs $\Gamma$.

\begin{lemma}\label{l:p1}
Suppose that $\Gamma$ with the assumptions of Proposition~\ref{p:mainA} embeds into $(\Z^{|\Gamma|},\mathrm{Id})$ and that there is a basis vector which appears only once in the embedding. Then, 
there is only one such vector, $\Gamma$ is of the form $\Gamma_{1}$, $I(\Gamma_{1})\geq -1$ and the parameters $N,k$ defining $\Gamma_{1}$ satisfy $$
(N,k)\in\{(1,5),(1,6),(2,6)\}.
$$
\end{lemma}
\begin{proof}
Let $e$ be a vector that appears only once in the embedding of $\Gamma$. Notice that $e$ is then necessarily different from all the basis vectors depicted in Figure~\ref{f:convenzioni1}, except possibly $e_{r}$, and thus appears in one vector $w$ of the left or top leg of $\Gamma$. If we delete $e$ from the embedding, we obtain again an embedding of a tree $\Gamma'$ with $n$ vertices in a standard lattice of rank $n-1$. This yields a contradiction if the determinant of $\Gamma'$ is non-vanishing. We proceed to analyze this determinant. If after deleting $e$ from the embedding the vertex $w$ has still weight greater than 1, then Lemma~\ref{l:det} applies and $\Gamma'$ has non-vanishing determinant. Then we must have that the embedding of the original $w$ is of the form $\alpha e\pm e_{j}$ for some $\alpha\in\Z$. Moreover, $e_{j}$ necessarily hits any vectors adjacent to $w$ with coefficient $\pm 1$ and it does not hit any other vertices in the graph. 

We assume now that after deleting the vector $e$ from the embedding of $w$ the graph $\Gamma'$ has a vertex of weight $1$. The blow down operation at the level of the graph is compatible with the embedding (see Operation~\ref{op}(\ref{bd})): blowing down the vertex with weight $1$ yields a graph $\tilde\Gamma'$ with $n-1$ vertices and an embedding in the standard lattice of rank $n-2$, which implies that this graph $\tilde\Gamma'$ has again vanishing determinant. At the level of the embedding we have deleted the basis vector $e_{j}$ and the vertex $w$. Now, several things can happen: 
\begin{enumerate}
\item if after the blow down we obtain a graph with no vertex of weight $1$ and central vertex of weight $2$, then Lemma~\ref{l:det} yields a contradiction.
\item if there is a vertex of weight 1, then we claim that $\tilde\Gamma'$ can have at most one vertex of weight 1. Indeed, if $w$ were final in $\Gamma$, this is evident. If $w$ were internal, then precisely two adjacent vertices in $\tilde\Gamma'$ weigh one less than they did in $\Gamma'$. If they have both weight 1 it means that in $\Gamma'$ we had the following configuration:
\[
  \begin{tikzpicture}[xscale=1.5,yscale=-0.4]
    \node (A0_2) at (2, 0) {$\scriptstyle e_{j}-e_{k}$};
    \node (A0_3) at (3, 0) {$\scriptstyle e_{j}$};
    \node at (2, 1.75) {$\scriptstyle w_{-}$};
    \node at (3, 1.75) {$\scriptstyle w$};
    \node at (4, 1.75) {$\scriptstyle w_{+}$};
    \node (A0_4) at (4, 0) {$\scriptstyle e_{j}+e_{k}$};
    \node (A1_1) at (1, 1) {$\dots$};
    \node (A1_2) at (2, 1) {$\bullet$};
    \node (A1_3) at (3, 1) {$\bullet$};
    \node (A1_4) at (4, 1) {$\bullet$};
    \node (A1_5) at (5, 1) {$\dots$};
    \path (A1_4) edge [-] node [auto] {$\scriptstyle{}$} (A1_5);
    \path (A1_1) edge [-] node [auto] {$\scriptstyle{}$} (A1_2);
    \path (A1_2) edge [-] node [auto] {$\scriptstyle{}$} (A1_3);
    \path (A1_3) edge [-] node [auto] {$\scriptstyle{}$} (A1_4);
  \end{tikzpicture}
  \]
%
However, this is not possible: $w_{-}$ and $w_{+}$ cannot both be final vertices, we can assume that $w_{-}$ has an adjacent vertex $t$ and since $e_j$ does not hit any other vertices in the graph, $t$ must be of the from $e_k + t'$ (with $t' \cdot e_j =0$), which contradicts $t\cdot w_{+}=0$.

So, if a vertex in $\tilde\Gamma'$ has weight $1$, there is only one, and we can iterate the blow down operation at the level of the graph and the embedding. Eventually we will be either in the the preceding case, obtaining a contradiction, or we will have one of the following situations:
	\begin{enumerate}
	\item the vertex $w$ was in the same leg as $v$ and after a series of blow 		downs we arrive to 
\[
  \begin{tikzpicture}[xscale=1.5,yscale=-0.5]
    \node (A3_0) at (2.25,0) {$\scriptstyle e_{s}+\sum_{i=2}^{N+2}(-1)^{i}e_{i}$};
	\node (A1_3) at (2.25,1) {$\bullet$};
	\node at (2.25,1.5) {$\scriptstyle v$};
	\node (B) at (1,0) {$\scriptstyle e_{s}$};
	\node (A) at (1,1) {$\bullet$};
	\node (C) at (0,0) {$\scriptstyle e_{s}+\dots$};
	\node at (0,1.5) {$\scriptstyle u$};
	\node (D) at (0,1) {$\bullet$};
	\node (A3) at (3.75,0) {$\scriptstyle e_{1}+e_{r}$};
	\node (A1) at (3.75,1) {$\bullet$};
    \node (A2_3) at (2.5, 3) {$\scriptstyle e_{1}+e_{2}$};
    \node (A3_3) at (3, 3) {$\bullet$};
    \node (A4_4) at (4, 4) {$\scriptstyle e_{2}+e_{3}$};
    \node (A4_6) at (6, 4) {$\scriptstyle e_{N+1}+e_{N+2}$};
    \node (A5_4) at (4, 5) {$\bullet$};
    \node (A5_5) at (5, 5) {$\dots$};
    \node (A5_6) at (6, 5) {$\bullet$};
    \node at (4.75, 1) {$\dots$};
    \node at (-1, 1) {$\dots$};
    \path (A3_3) edge [-] node [auto] {$\scriptstyle{}$} (A1_3);
    \path (A3_3) edge [-] node [auto] {$\scriptstyle{}$} (A1);
    \path (A5_4) edge [-] node [auto] {$\scriptstyle{}$} (A5_5);
    \path (A3_3) edge [-] node [auto] {$\scriptstyle{}$} (A5_4);
    \path (A5_5) edge [-] node [auto] {$\scriptstyle{}$} (A5_6);
    \path (A1_3) edge [-] node [auto] {$\scriptstyle{}$} (A);
    \path (A) edge [-] node [auto] {$\scriptstyle{}$} (D);
    \path (A1) edge [-] node [auto] {$\scriptstyle{}$} (4.5,1);
    \path (D) edge [-] node [auto] {$\scriptstyle{}$} (-0.75,1);
   \end{tikzpicture}
  \]
where, if the vertex $u$ exists, then it has weight at least 2. In this configuration, the $v'$ part of $v$ in Figure~\ref{f:convenzioni1} has been reduced to $e_{s}$ during the blow down iterations. Now, if $u$ does not exist and the vertex with embedding $e_{s}$ is final, then after blowing down this vertex we obtain a contradiction with the determinant via Lemma~\ref{l:det}. On the other hand, if $u$ exists, it is evident that the described situation is impossible since $u$ must be orthogonal to $v$, but this would require that $u$ is not orthogonal to the 2-leg.
	%
	\item the vertex $w$ was not in the same leg as $v$ and after a series 	of 		blow downs we obtain 
\[
  \begin{tikzpicture}[xscale=1.5,yscale=-0.5]
    \node (A3_0) at (2.25,0) {$\scriptstyle \sum_{i=2}^{N+2}(-1)^{i}e_{i}+v'$};
	\node (A1_3) at (2.25,1) {$\bullet$};
	\node at (2.25,1.5) {$\scriptstyle v$};
	\node (A) at (1,1) {$\bullet$};
	\node (D) at (0,1) {$\dots$};
	\node (A3) at (3.75,0) {$\scriptstyle e_{1}$};
	\node (A1) at (3.75,1) {$\bullet$};
    \node (A2_3) at (2.5, 3) {$\scriptstyle e_{1}+e_{2}$};
    \node (A3_3) at (3, 3) {$\bullet$};
    \node (A4_4) at (4, 4) {$\scriptstyle e_{2}+e_{3}$};
    \node (A4_6) at (6, 4) {$\scriptstyle e_{N+1}+e_{N+2}$};
    \node (A5_4) at (4, 5) {$\bullet$};
    \node (A5_5) at (5, 5) {$\dots$};
    \node (A5_6) at (6, 5) {$\bullet$};
    \node (F) at (4.75, 1) {$\bullet$};
    \node (J) at (5.75, 1) {$\dots$};
    \node at (4.75, 0) {$\scriptstyle e_{1}+\dots$};
    \node at (3.75, 1.5) {$\scriptstyle u$};
    \path (A3_3) edge [-] node [auto] {$\scriptstyle{}$} (A1_3);
    \path (A3_3) edge [-] node [auto] {$\scriptstyle{}$} (A1);
    \path (A5_4) edge [-] node [auto] {$\scriptstyle{}$} (A5_5);
    \path (A3_3) edge [-] node [auto] {$\scriptstyle{}$} (A5_4);
    \path (A5_5) edge [-] node [auto] {$\scriptstyle{}$} (A5_6);
    \path (A1_3) edge [-] node [auto] {$\scriptstyle{}$} (A);
    \path (A) edge [-] node [auto] {$\scriptstyle{}$} (D);
    \path (A1) edge [-] node [auto] {$\scriptstyle{}$} (F);
    \path (F) edge [-] node [auto] {$\scriptstyle{}$} (J);
   \end{tikzpicture}
  \]
where $u$ is the only vertex of weight 1. If $u$ is final, then we can blow down $u$, then the central vertex, and then the rest of the 2-leg. This results in a linear graph with $v$ as a final vertex with weight $v'$. Since $|v'|\geq 1$ and all other weights are greater than or equal to 2, it follows by Remark~\ref{rem}(\ref{dd}) that the determinant is non-vanishing which yields a contradiction.

Now assume that $u$ is not final. Then we can blow down the vertex $u$ obtaining a tree with central vertex of weight 1 and the following embedding into a standard lattice of rank one less than the total number of vertices:
	\[
  \begin{tikzpicture}[xscale=1.8,yscale=-0.5]
    \node (A3_0) at (2.25,0) {$\scriptstyle \sum_{i=2}^{N+2}(-1)^{i}e_{i}+v'$};
	\node (A1_3) at (2.25,1) {$\bullet$};
	\node at (2.25,1.5) {$\scriptstyle v$};
	\node (A) at (1,1) {$\bullet$};
	\node (D) at (0,1) {$\dots$};
	\node (A3) at (3.75,0) {$\scriptstyle \pm \sum_{i=2}^{N+2}(-1)^{i}e_{i}+u'$};
	\node (A1) at (3.75,1) {$\bullet$};
    \node (A2_3) at (3, 3.5) {$\scriptstyle e_{2}$};
    \node (A3_3) at (3, 3) {$\bullet$};
    \node (A4_4) at (4, 4) {$\scriptstyle e_{2}+e_{3}$};
    \node (A4_6) at (6, 4) {$\scriptstyle e_{N+1}+e_{N+2}$};
    \node (A5_4) at (4, 5) {$\bullet$};
    \node (A5_5) at (5, 5) {$\dots$};
    \node (A5_6) at (6, 5) {$\bullet$};
    \node (F) at (4.75, 1) {$\dots$};
    \node at (3.75, 1.5) {$\scriptstyle u$};
    \path (A3_3) edge [-] node [auto] {$\scriptstyle{}$} (A1_3);
    \path (A3_3) edge [-] node [auto] {$\scriptstyle{}$} (A1);
    \path (A5_4) edge [-] node [auto] {$\scriptstyle{}$} (A5_5);
    \path (A3_3) edge [-] node [auto] {$\scriptstyle{}$} (A5_4);
    \path (A5_5) edge [-] node [auto] {$\scriptstyle{}$} (A5_6);
    \path (A1_3) edge [-] node [auto] {$\scriptstyle{}$} (A);
    \path (A) edge [-] node [auto] {$\scriptstyle{}$} (D);
    \path (A1) edge [-] node [auto] {$\scriptstyle{}$} (F);
   \end{tikzpicture}
  \]
We are now going to consider the graph obtained by deleting the trivalent vertex in the above graph and all its adjacent edges. We claim that this linear graph with 3 connected components $L$, which has an embedding into the standard lattice of the same rank, has $I(L)=I(\Gamma)-\alpha^{2}+2$, where $\Gamma$ is the graph with which we started to obtain $L$. Indeed, deleting the basis vector $e$ with coefficient $\alpha$ from the embedding of $\Gamma$ yields $\Gamma'$ with $I(\Gamma')=I(\Gamma)-\alpha^{2}$, the blow downs leave $I$ unchanged, and the last operation, deleting the trivalent vertex with weight 1, adds two to $I$, therefore $I(L)=I(\Gamma)-\alpha^{2}+2$ as claimed.

Now, $L$ is a good subset and we claim that $L$ has no bad components. To see this note that the vector $e_2$ is linked once to a final vector of each connected component of the graph associated with $L$. It then follows from Lemma~\ref{badcomponents} that $L$ has no bad components. Therefore, we can apply \cite[Proposition~4.10]{Lisca-sums} to conclude that $I(L)\geq 0$, which implies  that $I(\Gamma)\geq \alpha^{2}-2\geq -1$. By \eqref{e:I} the graphs $\Gamma_{2}$ are excluded and we conclude that  $\Gamma=\Gamma_{1}$. Moreover, since $I(\Gamma_{1})=k-N-5$ we need $(N,k)\in\{(1,5),(1,6),(2,6),(1,7),(2,7),(3,7)\}$ and it follows also that $|\alpha|=1$. Finally, by Lemma~\ref{pruning2} we know that the graphs $\Gamma_{1}$ with 	$(N,k)\in\{(1,7),(2,7),(3,7)\}$ do not bound rational homology balls and they are excluded from our embeddability analysis.
\end{enumerate}
\end{enumerate} 

We finish the proof of this lemma by showing that there cannot be a second basis vector $f$ such that $f$ also appears only once in the embedding of $\Gamma$. If there were such a vector, then, by the above arguments, it needs to hit a vertex belonging to the same qc-leg as $u$. Observe that it cannot hit $w$, the vertex $e$ hits, since this vertex is not isolated and then it would have at least weight 3 and by deleting either $e$ or $f$ we obtain a contradiction with the determinant of $\Gamma$ via Lemma~\ref{l:det}. Without loss of generality suppose $f$ hits a vertex $w_{2}$ further from the trivalent vertex than $w$ in the qc-leg. By the above arguments there is a basis vector $e_{k}$ hitting exactly $w_{2}$ and its adjacent vertices. After deleting $f$ we obtain a graph $\Gamma'$ with an embedding and a vertex of weight 1 that we can blow down. The key point is that the blow down operation is compatible with the embedding and as explained above we need to be able to keep on blowing down vertices until we arrive to a configuration like 2.(b). This is only possible if to the sides of $w_{2}$ we have a configuration of complementary strings from the trivalent vertex up to $w_{2}$ and from $w_{2}$ to some vertex between $w_{2}$ and the final vertex in its qc-leg. Moreover, the embedding of these strings needs to be as in Property~\ref{f}(\ref{f:unique}). The existence of the vector $e$ makes this impossible.
\end{proof}

\begin{re}
The reader might find it interesting to know that there are indeed graphs with embeddings like the ones described in Case 2.(b) in the last lemma. We will show in what follows though that any such graph will not correspond to a surgery on a torus knot. A random example is:
\[
  \begin{tikzpicture}[xscale=2.,yscale=-0.5]
    \node (A3_0) at (2.25,0) {$\scriptstyle e_{1}-e_{2}+e_{3}+3f_{1}+f_{2}$};
	\node (A1_3) at (2.25,1) {$\bullet$};
	\node at (1,0) {$\scriptstyle f_{2}+3f_{3}$};
	\node (A) at (1,1) {$\bullet$};
	\node (A3) at (3.75,0) {$\scriptstyle e_{1}-e_{2}+e_{3}-f_{1}$};
	\node (A1) at (3.75,1) {$\bullet$};
    \node (A2_3) at (2.75, 3) {$\scriptstyle e_{1}$};
    \node (A3_3) at (3, 3) {$\bullet$};
    \node (A4_4) at (4, 4) {$\scriptstyle e_{1}+e_{2}$};
    \node (A4_6) at (5, 4) {$\scriptstyle e_{2}+e_{3}$};
    \node (A5_4) at (4, 5) {$\bullet$};
    \node (A5_5) at (5, 5) {$\bullet$};
    \node (F) at (4.75, 1) {$\bullet$};
    \node at (4.75, 0) {$\scriptstyle -f_{1}+3f_{2}-f_{3}$};
    \path (A3_3) edge [-] node [auto] {$\scriptstyle{}$} (A1_3);
    \path (A3_3) edge [-] node [auto] {$\scriptstyle{}$} (A1);
    \path (A5_4) edge [-] node [auto] {$\scriptstyle{}$} (A5_5);
    \path (A3_3) edge [-] node [auto] {$\scriptstyle{}$} (A5_4);
    \path (A1_3) edge [-] node [auto] {$\scriptstyle{}$} (A);
    \path (A1) edge [-] node [auto] {$\scriptstyle{}$} (F);
   \end{tikzpicture}
  \]
\end{re}

We now proceed to eliminate the possibilities left open in Lemma~\ref{l:p1}.
\begin{lemma}\label{l:no1}
Suppose that $\Gamma$ under the assumptions of Proposition~\ref{p:mainA} embeds into $(\Z^{|\Gamma|},\mathrm{Id})$, then no basis vector appears only once in the embedding.  
\end{lemma}
\begin{proof}
In light of Lemma~\ref{l:p1} we only need to consider the case $\Gamma=\Gamma_{1}$ and $(N,k)\in\{(1,5),(1,6),(2,6)\}$. Even though the argument is essentially the same, we will distinguish the two cases $N=1$ and $N=2$.

\emph{Case $N=2$}. From Lemma~\ref{l:p1} we know that after deleting the basis vector appearing only once and performing a series of blow downs we have the following scenario (where we have fixed the signs of the basis vectors $e_{i}$ after the blow downs):
\[
  \begin{tikzpicture}[xscale=1.8,yscale=-0.5]
    \node (A3_0) at (2.25,0) {$\scriptstyle e_{2}-e_{3}+e_{4}+v'$};
	\node (A1_3) at (2.25,1) {$\bullet$};
	\node at (2.25,1.5) {$\scriptstyle v$};
	\node (A) at (1,1) {$\bullet$};
	\node (D) at (0,1) {$\dots$};
	\node (A3) at (3.75,0) {$\scriptstyle e_{2}-e_{3}+e_{4}+u'$};
	\node (A1) at (3.75,1) {$\bullet$};
    \node (A2_3) at (3, 3.5) {$\scriptstyle e_{2}$};
    \node (A3_3) at (3, 3) {$\bullet$};
    \node (A4_4) at (4, 4) {$\scriptstyle e_{2}+e_{3}$};
    \node (A4_6) at (5, 4) {$\scriptstyle e_3+e_4$};
    \node (A5_4) at (4, 5) {$\bullet$};
    \node (A5_5) at (5, 5) {$\bullet$};
    \node (F) at (4.75, 1) {$\dots$};
    \node at (3.75, 1.5) {$\scriptstyle u$};
    \path (A3_3) edge [-] node [auto] {$\scriptstyle{}$} (A1_3);
    \path (A3_3) edge [-] node [auto] {$\scriptstyle{}$} (A1);
    \path (A5_4) edge [-] node [auto] {$\scriptstyle{}$} (A5_5);
    \path (A3_3) edge [-] node [auto] {$\scriptstyle{}$} (A5_4);
    \path (A1_3) edge [-] node [auto] {$\scriptstyle{}$} (A);
    \path (A) edge [-] node [auto] {$\scriptstyle{}$} (D);
    \path (A1) edge [-] node [auto] {$\scriptstyle{}$} (F);
   \end{tikzpicture}
  \]
The vectors $e_{2},e_{3}$ and $e_{4}$ do not reappear elsewhere in the embedding and the determinant of this graph $\Psi$ vanishes, since the described embedding has image in a standard lattice of rank one less than the number of vertices in the graph. Moreover, in this case we have that the quantity $I(\Psi)=I(\Gamma_{1})-1=6-2-5-1=-2$. 

We now proceed to delete the basis vectors $e_{2},e_{3},e_{4}$, which implies the loss of three vertices in the graph, yielding the graph $\Psi'$, which still admits an embedding into a standard lattice of rank one less than its number of vertices and satisfies $I(\Psi')= -4$.
\[
  \begin{tikzpicture}[xscale=1.5,yscale=1.5]
    \node at (2, 1.25) {$\scriptstyle v'$};
    \node at (3, 1.25) {$\scriptstyle u'$};
    \node at (2.5, 1.25) {$\scriptstyle -$};
    \node at (2.5, 1.05) {$\scriptstyle -$};
    \node at (2.5, 0.75) {$\scriptstyle -$};
    \node (B) at (0, 1) {$\dots$};
    \node (A1_1) at (1, 1) {$\bullet$};
    \node (A1_2) at (2, 1) {$\bullet$};
    \node (A1_3) at (3, 1) {$\bullet$};
    \node (A1_4) at (4, 1) {$\bullet$};
    \node (A1_5) at (5, 1) {$\dots$};
    \path (A1_4) edge [-] node [auto] {$\scriptstyle{}$} (A1_5);
    \path (B) edge [-] node [auto] {$\scriptstyle{}$} (A1_1);
    \path (A1_1) edge [-] node [auto] {$\scriptstyle{}$} (A1_2);
    \path (A1_2) edge [-] node [auto] {$\scriptstyle{}$} (A1_3);
    \draw (A1_2) to[bend right] (A1_3);
    \draw (A1_2) to[bend left] (A1_3);
    \path (A1_3) edge [-] node [auto] {$\scriptstyle{}$} (A1_4);
  \end{tikzpicture}
  \]
Since $\Psi'$ has vanishing determinant, the weight of at least one between $u'$ and $v'$ needs to be strictly less than $4$. Without loss of generality let us assume that $|u'|\leq 3$ and $|u'|\leq |v'|$. Since $u\cdot v=0$ in $\Psi$ and they are not both final vectors, we have that if $|u'|=1,2,3$ then $|v'|\geq 10,5,4$ respectively. Deleting from $\Psi'$ the vertex $v'$ we obtain a linear graph $L$. 

If $|u'|\neq 1$, then $I(L)\leq -5$. If $L$ is irreducible, then it is good with at most two connected components, which contradicts Lemma~\ref{badcomponents2}. If it were reducible, at least one of its components would have $I=-3$, which implies that in the embedding of this component a basis vector appears only once \cite[Proposition~6.1]{Lisca-ribbon}. This contradicts Lemma~\ref{l:p1}.

If $|u'|=1$, then $|v'|\geq 10$ and $I(L)\leq -11$.  If $u'$ is isolated in $L$, then  we eliminate it, obtaining in this way a standard subset $L'$ (note that in these circumstances $v'$ is not final) with $I(L')\leq -9$, which contradicts Lemma~\ref{badcomponents2}. If $u'$ is not isolated, then it is linked either to a vertex of weight greater than 2, or to the first vertex of a 2-chain $c$ of length $r\geq 1$. In the former case, we proceed to blow down the vertex $u'$, obtaining a subset $L'$ with at most two connected components (depending on whether or not $v$ was final) and $I(L')\leq -10$. Notice that since the vertex linked to $u'$ in $L$ was assumed of weight greater than 2, then all vertices in $L'$ have weight at least 2. If $L'$ is irreducible, then it is good and this contradicts Lemma~\ref{badcomponents2}. If $L'$ consisted of two irreducible components, each would be good and at least one would necessarily have $I< -3$ which again contradicts Lemma~\ref{badcomponents2}.
In the case in which $u'$ is linked to a 2-chain $c$, all the basis vectors hitting this 2-chain also hit $v'$ in $\Psi'$, and they do so with coefficient $\pm3$. It follows that in $L$, once we have deleted $v'$ from $\Psi'$, the quantity $I$ is actually bounded by $I(L)=I(\Psi')+3-v'\cdot v'\leq -4+3-9(r+1)=-10-9r.$ We proceed to blow down the vertex $u'$ and subsequently the 2-chain $c$ until we obtain a subset $L'$ such that all its vertices have weight at least 2. $L'$ has either one or two connected components (depending on whether or not the 2-chain $c$ was final) and we have $I(L')\leq -10-9r+r+2<-6$. This again contradicts Lemma~\ref{badcomponents2} both if $L'$ is irreducible, and thus good, or if it has two irreducible components and at least one needs then to have $I<-3$.

\emph{Case N=1}. With a similar strategy to the one used in the preceding case we proceed to examine the graphs of the form
\[
  \begin{tikzpicture}[xscale=1.8,yscale=-0.5]
    \node (A3_0) at (2.25,0) {$\scriptstyle e_{2}-e_{3}+v'$};
	\node (A1_3) at (2.25,1) {$\bullet$};
	\node at (2.25,1.5) {$\scriptstyle v$};
	\node (A) at (1,1) {$\bullet$};
	\node (D) at (0,1) {$\dots$};
	\node (A3) at (3.75,0) {$\scriptstyle e_{2}-e_{3}+u'$};
	\node (A1) at (3.75,1) {$\bullet$};
    \node (A2_3) at (3, 3.5) {$\scriptstyle e_{2}$};
    \node (A3_3) at (3, 3) {$\bullet$};
    \node (A4_4) at (4, 4) {$\scriptstyle e_{2}+e_{3}$};
    \node (A5_4) at (4, 5) {$\bullet$};
    \node (F) at (4.75, 1) {$\dots$};
    \node at (3.75, 1.5) {$\scriptstyle u$};
    \path (A3_3) edge [-] node [auto] {$\scriptstyle{}$} (A1_3);
    \path (A3_3) edge [-] node [auto] {$\scriptstyle{}$} (A1);
    \path (A3_3) edge [-] node [auto] {$\scriptstyle{}$} (A5_4);
    \path (A1_3) edge [-] node [auto] {$\scriptstyle{}$} (A);
    \path (A) edge [-] node [auto] {$\scriptstyle{}$} (D);
    \path (A1) edge [-] node [auto] {$\scriptstyle{}$} (F);
   \end{tikzpicture}
  \]
We will refer to them again as $\Psi$ and remark that depending on $k=5$ or $k=6$ we have $I(\Psi)=-2$ or $-1$ respectively. Just as before, this time we delete $e_{2}$ and $e_{3}$ obtaining $\Psi'$:
\[
  \begin{tikzpicture}[xscale=1.5,yscale=1.5]
    \node at (2, 1.25) {$\scriptstyle v'$};
    \node at (3, 1.25) {$\scriptstyle u'$};
    \node at (2.5, 1.25) {$\scriptstyle -$};
    \node at (2.5, 0.75) {$\scriptstyle -$};
    \node (B) at (0, 1) {$\dots$};
    \node (A1_1) at (1, 1) {$\bullet$};
    \node (A1_2) at (2, 1) {$\bullet$};
    \node (A1_3) at (3, 1) {$\bullet$};
    \node (A1_4) at (4, 1) {$\bullet$};
    \node (A1_5) at (5, 1) {$\dots$};
    \path (A1_4) edge [-] node [auto] {$\scriptstyle{}$} (A1_5);
    \path (B) edge [-] node [auto] {$\scriptstyle{}$} (A1_1);
    \path (A1_1) edge [-] node [auto] {$\scriptstyle{}$} (A1_2);
    \draw (A1_2) to[bend right] (A1_3);
    \draw (A1_2) to[bend left] (A1_3);
    \path (A1_3) edge [-] node [auto] {$\scriptstyle{}$} (A1_4);
  \end{tikzpicture}
  \]
a graph with $I(\Psi')= -3$ or $-2$ and vanishing determinant, since it embeds in a lattice of rank one less than its number of vertices. Moreover, we know from Lemma~\ref{l:p1} and the construction of $\Psi'$ that $v'$ is not a final vector, while $u'$ might be final. Notice also that the linear graphs, which we will denote by $L$, obtained from $\Psi'$ by deleting either $u'$ or $v'$ have at most two connected components. By Lemma~\ref{badcomponents} the connected component of $L$ that was linked only once with the deleted vertex of $\Psi'$ cannot be bad.  Since $\Psi'$ has vanishing determinant, the smaller between $|u'|$ and $|v'|$ needs to be either 1 or 2. We proceed to analyze separately the following cases:

\begin{enumerate}
\item $|u'|=1$. In this case there is a basis vector $e$ such that $u'=e$. We study the following two configurations. 
	\begin{itemize}
		\item[--] If $u'$ were final in $\Psi'$, after eliminating $e$ we obtain a linear graph $L$ which still embeds in a lattice of rank smaller than its number of vertices and therefore has vanishing determinant. However, this is not possible: $L$ has only one connected component and all weights greater or equal than 2 except possibly the now final vertex that corresponded to $v'$ in $\Psi'$. By Fact~\ref{rem}.\ref{ddN} we conclude that the determinant of $L$ is non-vanishing.
		\item[--] If $u'$ were not final in $\Psi'$, then 		it is linked to its right to a vertex $w=e+\dots$. Notice that the basis vector $e$ hits precisely $v',u'$ and $w$ in $\Psi'$ and moreover $v'=-2e+\dots$. The linear graph $L=\Psi'\setminus{u'}$ is good, with two linear components and $I(L)=-1$ or 0.
\[
  \begin{tikzpicture}[xscale=1.5,yscale=1.5]
    \node at (2, 1.25) {$\scriptstyle v'$};
    \node at (2, 0.75) {$\scriptstyle -2e+\dots$};
    \node at (4, 1.25) {$\scriptstyle w$};
	\node at (4, 0.75) {$\scriptstyle e+\dots$};
	\node at (-1, 1) {$L:$};
    \node (B) at (0, 1) {$\dots$};
    \node (A1_1) at (1, 1) {$\bullet$};
    \node (A1_2) at (2, 1) {$\bullet$};
    \node (A1_4) at (4, 1) {$\bullet$};
    \node (A1_5) at (5, 1) {$\dots$};
    \path (A1_4) edge [-] node [auto] {$\scriptstyle{}$} (A1_5);
    \path (B) edge [-] node [auto] {$\scriptstyle{}$} (A1_1);
    \path (A1_1) edge [-] node [auto] {$\scriptstyle{}$} (A1_2);
  \end{tikzpicture}
  \]
We proceed to change the $-2$ coefficient of the basis vector $e$ hitting $v'$ into a $-1$. This yields a standard subset $\tilde L$ where the vertex $\tilde v$ corresponding to $v'$ in $L$ satisfies $|\tilde v|+3=|v'|$		
\[
  \begin{tikzpicture}[xscale=1.5,yscale=1.5]
    \node at (2, 1.25) {$\scriptstyle \tilde v$};
    \node at (2, 0.75) {$\scriptstyle -e+\dots$};
    \node at (3, 1.25) {$\scriptstyle w$};
	\node at (3, 0.75) {$\scriptstyle e+\dots$};
	\node at (-1, 1) {$\tilde L:$};
    \node (B) at (0, 1) {$\dots$};
    \node (A1_1) at (1, 1) {$\bullet$};
    \node (A1_2) at (2, 1) {$\bullet$};
    \node (A1_4) at (3, 1) {$\bullet$};
    \node (A1_5) at (4, 1) {$\dots$};
    \path (A1_4) edge [-] node [auto] {$\scriptstyle{}$} (A1_5);
    \path (B) edge [-] node [auto] {$\scriptstyle{}$} (A1_1);
    \path (A1_1) edge [-] node [auto] {$\scriptstyle{}$} (A1_2);
    \path (A1_2) edge [-] node [auto] {$\scriptstyle{}$} (A1_4);
  \end{tikzpicture}
  \]
and therefore $I(\tilde L)\leq -3$. It follows by Lemma~\ref{badcomponents2} that $I(\tilde L)=-3$. However, this contradicts \cite[Proposition~6.1]{Lisca-ribbon} since the basis vector $e$ which appears only twice in the embedding hits internal vertices in $\tilde L$.	
	\end{itemize}
\item $|u'|=2$. In this case we assume that $u'=f_{1}+f_{2}$. If $u'$ were final, then, since $\Psi'$ has vanishing determinant, we need $|v'|\leq 2$ which implies $v'=-f_{1}-f_{2}$. This is incompatible with $v'$ not being final and having and adjacent vertex that is orthogonal to $u'$. We can assume then that $u'$ is not final in $\Psi'$ and that $|v'|\geq 3$.
	\begin{itemize}
		\item[--] If $|v'|\geq 5$ then the determinant of $\Psi$ cannot be zero. Indeed, under these assumptions the sum $r_{1}+r_{2}+r_{3}$ in Equation~\eqref{e:det=0} is strictly smaller than $\frac{1}{2}+\frac{1}{3}+\frac{1}{6}=1$.
		\item[--] If $|v'|=4$ we are dealing with the following graph $\Psi'$:
		\[
  \begin{tikzpicture}[xscale=2,yscale=1.5]
    \node at (2, 1.25) {$\scriptstyle v'$};
    \node at (2, 0.65) {$\scriptstyle -g_{1}-g_{2}-g_{3}-g_{4}$};
    \node at (3, 1.25) {$\scriptstyle u'$};
    \node at (3, 0.65) {$\scriptstyle g_{1}+g_{2}$};
    \node (B) at (0, 1) {$\dots$};
    \node (A1_1) at (1, 1) {$\bullet$};
    \node (A1_2) at (2, 1) {$\bullet$};
    \node (A1_3) at (3, 1) {$\bullet$};
    \node (A1_4) at (4, 1) {$\bullet$};
    \node (A1_5) at (5, 1) {$\dots$};
    \path (A1_4) edge [-] node [auto] {$\scriptstyle{}$} (A1_5);
    \path (B) edge [-] node [auto] {$\scriptstyle{}$} (A1_1);
    \path (A1_1) edge [-] node [auto] {$\scriptstyle{}$} (A1_2);
    \draw (A1_2) to[bend right] (A1_3);
    \draw (A1_2) to[bend left] (A1_3);
    \path (A1_3) edge [-] node [auto] {$\scriptstyle{}$} (A1_4);
  \end{tikzpicture}
  \]
The set $L=\Psi'\setminus\{v'\}$ is a linear 2--connected component set. We argue its irreducibility: notice that the vertex $z$ needs to be hit by at least one between $g_{3}$ and $g_{4}$ to guarantee orthogonality with $v'$. If $L$ were irreducible, then only one of these basis vectors would hit $z$, say it is $g_{3}$, and the linking between $w$ and $v'$ would be via the basis vector $g_{4}$. Moreover, $w$ would be the only vertex in its connected component hit by $g_{4}$ (because of reducibility and the needed orthogonality with $v'$). This is not possible since a standard set  with a basis vector appearing only once has the property that such a basis vector hits an internal vertex (see \cite[Lemma~3.2]{Lisca-ribbon}).

It follows that $L$ is a good set with at most one bad component (recall that by Lemma~\ref{badcomponents} the component linked to $v'$ is not bad). Moreover $I(L)=-4$ or $-3$ and by Lemma~\ref{badcomponents2} the only possible value is $I(L)=-3$, i.e.\ $k=6$, and the component linked to $v'$ is bad. Following line by line the proof of \cite[Lemma~5.4, $S$ irreducible, Second subcase] {Lisca-sums} we learn that $L$ is necessarily obtained from the following graph by final 2-expansions of the bad component:
		\[
  \begin{tikzpicture}[xscale=1.0,yscale=1.5]
    \node at (3, 1.25) {$\scriptstyle 2$};
    \node at (4, 1.25) {$\scriptstyle n+1$};
    \node at (5, 1.25) {$\scriptstyle 2$};
    \node (B) at (0, 1) {$\dots$};
    \node (A1_1) at (1, 1) {$\bullet$};
    \node (A1_2) at (2, 1) {$\bullet$};
    \node (A1_3) at (3, 1) {$\bullet$};
    \node (A1_4) at (4, 1) {$\bullet$};
    \node (A1_5) at (5, 1) {$\bullet$};
    \node (A1_11) at (0, 1) {$\dots$};
    \node (A1_12) at (-1, 1) {$\bullet$};
    \node (A1_13) at (-2, 1) {$\bullet$};
    \node at (-1, 1.25) {$\scriptstyle2$};
    \node at (-2, 1.25) {$\scriptstyle2$};
    \node at (1, 1.25) {$\scriptstyle2$};
    \node at (2, 1.25) {$\scriptstyle2$};
    \path (A1_4) edge [-] node [auto] {$\scriptstyle{}$} (A1_5);
    \path (B) edge [-] node [auto] {$\scriptstyle{}$} (A1_1);
    \path (A1_1) edge [-] node [auto] {$\scriptstyle{}$} (A1_2);
    \path (A1_3) edge [-] node [auto] {$\scriptstyle{}$} (A1_4);
    \path (A1_13) edge [-] node [auto] {$\scriptstyle{}$} (A1_12);
    \path (A1_12) edge [-] node [auto] {$\scriptstyle{}$} (A1_11);
  \end{tikzpicture}
  \]
where the left hand side 2-chain has length $n-1$. Since the leg with $v'$ in $\Psi$ is the same as the original leg in the graph in family $\Gamma_{1}$ from which we obtained $\Psi$, we conclude that the vertex with weight $k+1=7$ is linked to the leg starting with $u$. Moreover, since we have assumed that $|v'|=4$ it follows that $|v|=6$ and the corresponding leg in $\Gamma_{1}$ has associated string $(6,2,\dots,2)$ for a number $t$ of 2s. Then, by definition of quasi-complementary legs, the leg in $\Gamma_{1}$ starting with $u$ has associated string $(2,2,2,2,t+2,7)$. However, this leg should have the basis vector appearing only once in the embedding hitting the fourth 2; once this vector is deleted from the embedding, it is not possible to get by a series of blow downs from this string to one obtained by final 2-expansions from the above described bad component. There are not enough 2s at the beginning of the string.
		\item[--] If $|v'|=3$ we proceed like in the last case studying the graph $\Psi'$
		\[
  \begin{tikzpicture}[xscale=2,yscale=1.5]
    \node at (2, 1.25) {$\scriptstyle v'$};
    \node at (1, 1.25) {$\scriptstyle w$};
    \node at (4, 1.25) {$\scriptstyle z$};
    \node at (2, 0.65) {$\scriptstyle -g_{1}-g_{2}-g_{3}$};
    \node at (3, 1.25) {$\scriptstyle u'$};
    \node at (3, 0.65) {$\scriptstyle g_{1}+g_{2}$};
    \node (B) at (0, 1) {$\dots$};
    \node (A1_1) at (1, 1) {$\bullet$};
    \node (A1_2) at (2, 1) {$\bullet$};
    \node (A1_3) at (3, 1) {$\bullet$};
    \node (A1_4) at (4, 1) {$\bullet$};
    \node (A1_5) at (5, 1) {$\dots$};
    \path (A1_4) edge [-] node [auto] {$\scriptstyle{}$} (A1_5);
    \path (B) edge [-] node [auto] {$\scriptstyle{}$} (A1_1);
    \path (A1_1) edge [-] node [auto] {$\scriptstyle{}$} (A1_2);
    \draw (A1_2) to[bend right] (A1_3);
    \draw (A1_2) to[bend left] (A1_3);
    \path (A1_3) edge [-] node [auto] {$\scriptstyle{}$} (A1_4);
  \end{tikzpicture}
  \]
which this time satisfies $I(\Psi'\setminus\{v'\})=-3$ or $-2$. If $I=-3$, the same arguments as before yield this time a graph in $\Gamma_{1}$ with associated $v$-leg  string $(5,2,\dots,2)$ and $u$-leg string $(2,2,2,t+2,6)$ which again cannot be a $\Gamma_{1}$ graph with the desired properties. 

If $I=-2$ according to Lemma~\ref{badcomponents2} there are two cases to consider, either $L=\Psi'\setminus\{v'\}$ has zero or one bad component. By Lemma~\ref{I=-2,b=1}, in the latter case $L$ has associated strings obtained as final 2-expansions from $(2,t,2)$ and $(3,2,\dots,2,3)$ where the number of 2s in the second string is $t-3$ and $t\geq3$. Let us call $L=L_{w}\sqcup B_{u'}$ the two linear connected components of $L$. Since the embedding of $L$ is completely determined (Lemma~\ref{I=-2,b=1}) we know that the basis vectors hitting $w$, the vertex in $L_{w}$ linked to $v'$, do not hit the vertex $u'$. This implies that $g_{3}$ is the only basis vector hitting $v'$ and $L_{w}$ which is incompatible with the embedding of $L_{w}$.

Finally, if $I(L)=-2$ and $L$ has no bad components, then following Lisca's analysis in the proof of \cite[First case: $S$ irreducible, $b(S)=0$]{Lisca-sums} we conclude that the two strings of $L$ are complementary and that their embedding satisfies the following properties \cite[Lemma~4.7]{Lisca-sums}:
	\begin{enumerate}
		\item There are precisely two basis vectors appearing twice in the embedding. All other basis vectors appear exactly three times in the embedding.
		\item The basis vectors appearing only twice hit only final vertices. Each of these basis vectors hits precisely one vertex of weight 2.
		\item Every final vertex is hit by a basis vector appearing only twice in the embedding. 
	\end{enumerate}
$*$ If $g_{1}$ hits $w$, then $g_{2}$ does too since $w\cdot u'=0$ and we have $w=-g_{1}+g_{2}+g_{3}+\dots$. Since $u'$ is final in $L$, at least one between $g_{1}$ and $g_{2}$ appears only twice in the embedding of $L$, say it is $g_{1}$, which then hits precisely $w$ and $u'$. This forces $g_{2}$ to appear precisely 3 times (it cannot appear only twice since $u'$ was not final in $\Psi'$). We conclude that in $L$, the set of vertices hit by $g_{2}$ is $\{w,u',z\}$. The basis vector $g_{3}$ needs to hit at least one vertex different form $w$, and since this vertex will be orthogonal to $v'$ and we have pinned down all the appearances of the basis vectors $g_{1}$ and $g_{2}$ in the embedding of $L$, $g_{3}$ can only hit $z$. Now, since the basis vectors appearing only twice hit exclusively final vertices, it follows that both $w$ and $z$ are final in $L$ and the graph $\Psi'$ with its embedding is necessarily 
		\[
  \begin{tikzpicture}[xscale=2.5,yscale=1.5]\label{prueba}
    \node at (2, 1.25) {$\scriptstyle v'$};
    \node at (1, 1.25) {$\scriptstyle w$};
    \node at (0, 1) {$C:$};
    \node at (1, 0.65) {$\scriptstyle g_{3}+g_{2}-g_{1}$};
    \node at (4, 1.25) {$\scriptstyle z$};
    \node at (4, 0.65) {$\scriptstyle -g_{2}+g_{3}$};
    \node at (2, 0.65) {$\scriptstyle -g_{1}-g_{2}-g_{3}$};
    \node at (3, 1.25) {$\scriptstyle u'$};
    \node at (3, 0.65) {$\scriptstyle g_{1}+g_{2}$};
    \node (A1_1) at (1, 1) {$\bullet$};
    \node (A1_2) at (2, 1) {$\bullet$};
    \node (A1_3) at (3, 1) {$\bullet$};
    \node (A1_4) at (4, 1) {$\bullet$};
    \path (A1_1) edge [-] node [auto] {$\scriptstyle{}$} (A1_2);
    \draw (A1_2) to[bend right] (A1_3);
    \draw (A1_2) to[bend left] (A1_3);
    \path (A1_3) edge [-] node [auto] {$\scriptstyle{}$} (A1_4);
  \end{tikzpicture}
  \]
The $v$ leg in the graph $\Gamma_{1}$ from which we obtained $\Psi$ has associated string $(5,3)$ which implies that the vertex of weight $k+1=7$ is the last one on the $u$ leg which has associated string $(2,2,2,3,2,7)$. That is, we are dealing with the graph below, where the 2-leg is drawn to the left due to space constraints.
\[
  \begin{tikzpicture}[xscale=1.25,yscale=-0.5]
    \node (A3_0) at (3.75,5.75) {$\scriptstyle v$};
	\node (A1_3) at (3.75,5) {$\bullet$};
	\node at (3.75,4.5) {$\scriptstyle 5$};
	\node (A1_31) at (4.75,5) {$\bullet$};
	\node at (4.75,4.5) {$\scriptstyle3$};
	\node (A3) at (3.75,0) {$\scriptstyle 2$};
	\node (A1) at (3.75,1) {$\bullet$};
    \node (A2_3) at (3.5, 3) {$\scriptstyle e_{1}+e_{2}$};
    \node (A3_3) at (3, 3) {$\bullet$};
    \node (A4_4) at (2, 2) {$\scriptstyle e_{2}+e_{3}$};
    \node (A5_4) at (2, 3) {$\bullet$};
    \node (F) at (4.75, 1) {$\bullet$};
    \node (J) at (5.75, 1) {$\bullet$};
    \node (K) at (6.75, 1) {$\bullet$};
    \node (L) at (7.75, 1) {$\bullet$};
    \node (M) at (8.75, 1) {$\bullet$};
    \node at (4.75, 0) {$\scriptstyle 2$};
    \node at (5.75, 0) {$\scriptstyle 2$};
    \node at (6.75, 0) {$\scriptstyle 3$};
    \node at (7.75, 0) {$\scriptstyle 2$};
    \node at (8.75, 0) {$\scriptstyle 7=k+1$};
    \node at (3.75, 1.5) {$\scriptstyle u$};
    \path (A3_3) edge [-] node [auto] {$\scriptstyle{}$} (A1_3);
    \path (A3_3) edge [-] node [auto] {$\scriptstyle{}$} (A1);
    \path (A3_3) edge [-] node [auto] {$\scriptstyle{}$} (A5_4);
    \path (A1_3) edge [-] node [auto] {$\scriptstyle{}$} (A1_31);
    \path (A1) edge [-] node [auto] {$\scriptstyle{}$} (F);
    \path (F) edge [-] node [auto] {$\scriptstyle{}$} (J);
    \path (J) edge [-] node [auto] {$\scriptstyle{}$} (K);
    \path (K) edge [-] node [auto] {$\scriptstyle{}$} (L);
    \path (L) edge [-] node [auto] {$\scriptstyle{}$} (M);
   \end{tikzpicture}
  \]
It is very easy to show that such a graph does not admit an embedding into $\mathbb Z^{10}$ with a basis vector appearing only once hitting one of the weight 2 vertices on the $u$ leg.

$*$ If $g_{1}\cdot w=0$, then $g_{2}\cdot w=0$ and $g_{3}\cdot w=1$. Since $u'$ is not final in $\Psi'$, we have that precisely one between $g_{1}$ and $g_{2}$ hits $z$. Say it is $g_{2}$. Since $z\cdot v'=0$ it follows that $z=-g_{2}+g_{3}+\dots$. The basis vector $g_{1}$ hits at least one vertex $a$ in $L$ different from $u'$. Since $a\neq z$, then $a\cdot u'=0$ and thus $g_{2}$ also hits $a$. This implies that $g_{2}$ hits precisely the vertices $u',z$ and $a$. Since $u'$ is final in $L$, we deduce that $g_{1}$ appears only twice in the embedding hitting $u'$ and $a$, which is then necessarily a final vertex in $L$ different from $w$ and $u'$. If $g_{3}$ were to appear three times in the embedding of $L$, it would have to hit a vertex orthogonal to $v'$ different from $z$ and $w$. However, this is not possible, since this new vertex needs to be orthogonal to $v'$, it cannot be $a$ and we have pinned down all the appearances of $g_{1}$ and $g_{2}$ in the embedding of $L$. It follows that $g_{3}$ hits only $w$ and $z$ in $L$, which are therefore final vertices. Thus so far we have that $\Psi'$ is of the form
		\[
  \begin{tikzpicture}[xscale=2,yscale=1.5]
    \node at (2, 1.25) {$\scriptstyle v'$};
    \node at (1, 1.25) {$\scriptstyle w$};
    \node at (-1, 1.25) {$\scriptstyle a$};
    \node at (-1, 0.65) {$\scriptstyle g_{1}-g_{2}+\dots$};
    \node at (1, 0.65) {$\scriptstyle -g_{3}+\dots$};
    \node at (4, 1.25) {$\scriptstyle z$};
    \node at (4, 0.65) {$\scriptstyle g_{2}-g_{3}+\dots$};
    \node at (2, 0.65) {$\scriptstyle -g_{1}-g_{2}-g_{3}$};
    \node at (3, 1.25) {$\scriptstyle u'$};
    \node at (3, 0.65) {$\scriptstyle g_{1}+g_{2}$};
    \node (A1) at (-1, 1) {$\bullet$};
    \node (A1_0) at (0, 1) {$\dots$};
    \node (A1_1) at (1, 1) {$\bullet$};
    \node (A1_2) at (2, 1) {$\bullet$};
    \node (A1_3) at (3, 1) {$\bullet$};
    \node (A1_4) at (4, 1) {$\bullet$};
    \path (A1_1) edge [-] node [auto] {$\scriptstyle{}$} (A1_2);
    \path (A1) edge [-] node [auto] {$\scriptstyle{}$} (A1_0);
    \path (A1_0) edge [-] node [auto] {$\scriptstyle{}$} (A1_1);
    \draw (A1_2) to[bend right] (A1_3);
    \draw (A1_2) to[bend left] (A1_3);
    \path (A1_3) edge [-] node [auto] {$\scriptstyle{}$} (A1_4);
  \end{tikzpicture}
  \]
Since $g_{1}$ hits precisely $u'$ and $a$ we conclude that $|a|>2$. Analogously, since $g_{3}$ hits $w$ and $z$, and $|z|>2$ (since $z\cdot w=0$), we have that $|w|=2$. The string complementary to $(2,n)$, $n\geq 3$, is $(3,2,\dots,2)$ whit $n-2$ entries of value 2. We conclude that the only candidates are the graphs $\Psi'$ in the family
		\[
  \begin{tikzpicture}[xscale=1.75,yscale=1.5]
    \node at (2, 1.25) {$\scriptstyle v'$};
    \node at (1, 1.25) {$\scriptstyle w$};
    \node at (-2, 1.25) {$\scriptstyle a$};
    \node at (-3, 1) {$B_{n}:$};
    \node at (-0.5, 1.25) {$\scriptstyle 2$};
    \node at (-2, 0.65) {$\scriptstyle g_{1}-g_{2}+f_{n-2}$};
    \node at (-0.5, 0.65) {$\scriptstyle f_{n-3}+f_{n-2}$};
    \node at (1, 0.65) {$\scriptstyle -g_{3}+f_{1}$};
    \node at (4, 1.25) {$\scriptstyle z$};
    \node at (4, 0.65) {$\scriptstyle g_{2}-g_{3}\pm\sum f_{i}$};
    \node at (2, 0.65) {$\scriptstyle -g_{1}-g_{2}-g_{3}$};
    \node at (3, 1.25) {$\scriptstyle u'$};
    \node at (3, 0.65) {$\scriptstyle g_{1}+g_{2}$};
    \node (A1) at (-2, 1) {$\bullet$};
    \node (A11) at (-0.5, 1) {$\bullet$};
    \node (A1_0) at (0.25, 1) {$\dots$};
    \node (A1_1) at (1, 1) {$\bullet$};
    \node (A1_2) at (2, 1) {$\bullet$};
    \node (A1_3) at (3, 1) {$\bullet$};
    \node (A1_4) at (4, 1) {$\bullet$};
    \path (A1_1) edge [-] node [auto] {$\scriptstyle{}$} (A1_2);
    \path (A1) edge [-] node [auto] {$\scriptstyle{}$} (A11);
    \path (A11) edge [-] node [auto] {$\scriptstyle{}$} (A1_0);
    \path (A1_0) edge [-] node [auto] {$\scriptstyle{}$} (A1_1);
    \draw (A1_2) to[bend right] (A1_3);
    \draw (A1_2) to[bend left] (A1_3);
    \path (A1_3) edge [-] node [auto] {$\scriptstyle{}$} (A1_4);
  \end{tikzpicture}
  \]
with the appropriate sign adjustments in the embedding of the 2-chain. From here we deduce, just as in the preceding case, that the corresponding $\Gamma_{1}$ graphs are of the form
\[
  \begin{tikzpicture}[xscale=1.25,yscale=-0.5]
    \node (A3_0) at (3.75,5.75) {$\scriptstyle v$};
	\node (A1_3) at (3.75,5) {$\bullet$};
	\node at (3.75,4.5) {$\scriptstyle 5$};
	\node (A1_31) at (4.75,5) {$\bullet$};
	\node (N) at (5.75,5) {$\dots$};
	\node (O) at (6.75,5) {$\bullet$};
	\node at (6.75,4.5) {$\scriptstyle 2$};
	\node (P) at (7.75,5) {$\bullet$};
	\node at (7.75,4.5) {$\scriptstyle 3$};
	\node at (4.75,4.5) {$\scriptstyle2$};
	\node (A1) at (3.75,1) {$\bullet$};
    \node (A2_3) at (3.5, 3) {$\scriptstyle e_{1}+e_{2}$};
    \node (A3_3) at (3, 3) {$\bullet$};
    \node (A4_4) at (2, 2) {$\scriptstyle e_{2}+e_{3}$};
    \node (A5_4) at (2, 3) {$\bullet$};
    \node (F) at (4.75, 1) {$\bullet$};
    \node (J) at (5.75, 1) {$\bullet$};
    \node (K) at (6.75, 1) {$\bullet$};
    \node (L) at (7.75, 1) {$\bullet$};
    \node (M) at (8.75, 1) {$\bullet$};
    \node (A3) at (3.75,0) {$\scriptstyle 2$};
    \node at (4.75, 0) {$\scriptstyle 2$};
    \node at (5.75, 0) {$\scriptstyle 2$};
    \node at (6.75, 0) {$\scriptstyle n+1$};
    \node at (7.75, 0) {$\scriptstyle 2$};
    \node at (8.75, 0) {$\scriptstyle 7=k+1$};
    \node at (3.75, 1.5) {$\scriptstyle u$};
    \path (A3_3) edge [-] node [auto] {$\scriptstyle{}$} (A1_3);
    \path (A3_3) edge [-] node [auto] {$\scriptstyle{}$} (A1);
    \path (A3_3) edge [-] node [auto] {$\scriptstyle{}$} (A5_4);
    \path (A1_3) edge [-] node [auto] {$\scriptstyle{}$} (A1_31);
    \path (A1) edge [-] node [auto] {$\scriptstyle{}$} (F);
    \path (F) edge [-] node [auto] {$\scriptstyle{}$} (J);
    \path (J) edge [-] node [auto] {$\scriptstyle{}$} (K);
    \path (K) edge [-] node [auto] {$\scriptstyle{}$} (L);
    \path (L) edge [-] node [auto] {$\scriptstyle{}$} (M);
    \path (A1_31) edge [-] node [auto] {$\scriptstyle{}$} (N);
    \path (N) edge [-] node [auto] {$\scriptstyle{}$} (O);
    \path (O) edge [-] node [auto] {$\scriptstyle{}$} (P);  \draw[decorate,decoration={brace,amplitude=7pt,mirror},xshift=0.4pt,yshift=-0.3pt](4.75,5.5) -- (6.75,5.5) node[black,midway,yshift=-0.5cm] {\footnotesize $n-2$};
   \end{tikzpicture}
  \]
with $n\geq 3$. Again it is immediate to check that none of these admits an embedding into a standard lattice of the same rank with a basis vector hitting only once a weight 2 vertex in the $u$ leg.

\end{itemize}

\item $|v'|=1$ is, by symmetry of the problem, analogous to $|u'|=1$ and $u'$ not final. So, we only need to discuss the case with $u'$ final; that is, graphs $\Psi'$ of the form
		\[
  \begin{tikzpicture}[xscale=2,yscale=1.5]
    \node at (2, 1.25) {$\scriptstyle v'$};
    \node at (2, 0.65) {$\scriptstyle e$};
    \node at (1, 0.65) {$\scriptstyle e+\dots$};
    \node at (1, 1.25) {$\scriptstyle w$}; 
    \node at (3, 1.25) {$\scriptstyle u'$};
    \node at (3, 0.65) {$\scriptstyle -2e+\dots$};
    \node (B) at (0, 1) {$\dots$};
    \node (A1_1) at (1, 1) {$\bullet$};
    \node (A1_2) at (2, 1) {$\bullet$};
    \node (A1_3) at (3, 1) {$\bullet$};    \path (B) edge [-] node [auto] {$\scriptstyle{}$} (A1_1);
    \path (A1_1) edge [-] node [auto] {$\scriptstyle{}$} (A1_2);
    \draw (A1_2) to[bend right] (A1_3);
    \draw (A1_2) to[bend left] (A1_3);
  \end{tikzpicture}
  \]
Since $w\cdot u'=0$, we have $|u'|\geq 5$. If $|u'|= 5$ then $u'=-2e+f$ and $w=e+2f+\dots$. The standard set $L$ obtained by deleting from $\Psi'$ both the vertex $u'$ and the vector $e$ has $I(L)=-3$ or $-4$. By Lemma~\ref{badcomponents2} the only actual possibility is $I(L)=-3$, however this is incompatible with having the vector $f$ hitting with coefficient 2 the vertex $w$ \cite[Proposition~3.3 and Corollary~3.5]{Lisca-ribbon}. If $|u'|>5$ and $|w|>2$ then the standard set obtained as explained before has $I\leq -4$ which contradicts Lemma~\ref{badcomponents2}. Finally, if $w$ is the first vertex in a 2-chain of length $\ell$, then by orthogonality $|u'|\geq 4(\ell+1)$ and in $\Psi'$ there is some other vertex linked to the 2-chain (if not $I(\Psi')\geq 4\ell+4-3-2-\ell=3\ell-1>-2$). We can then proceed to delete $u'$ and $e$ from $\Psi'$ and to blow down the 2-chain until we obtain a standard subset $L$. However, this contradicts Lemma~\ref{badcomponents2} since in this case we have $I(L)\leq -2-4\ell-1+\ell+1<-3$.

\item $|v'|=2$ is the last case left to discuss. If $|u'|\geq 5$ we obtain as before a contradiction to the fact that $\Psi$ has vanishing determinant. The cases with  $|u'|=1,2$ have already been analyzed. The cases $|u'|=3,4$ and $u'$ is not final, can be ruled out with similar arguments to the ones used for $|u'|=2$ and  $|v'|=3,4$. 
	\begin{itemize}
		\item If $|v'|=2$ and $|u'|=4$, then the set $L=\Psi'\setminus{u'}$ is good with one bad component ending in $v'$, $k=6$ and $I(L)=-3$. From Lemma~\ref{I=-2,b=1} we know that $\Psi'$ has the form
		\[
  \begin{tikzpicture}[xscale=1.5,yscale=1.5]
    \node at (2, 1.25) {$\scriptstyle v'$};
    \node at (4, 1.25) {$\scriptstyle w$};
    \node at (2, 0.65) {$\scriptstyle 2$};
    \node at (3, 1.25) {$\scriptstyle u'$};
    \node at (3, 0.65) {$\scriptstyle 4$};
    \node at (4, 0.65) {$\scriptstyle 2$};
    \node at (6, 0.65) {$\scriptstyle 2$};
    \node (B) at (0, 1) {$\dots$};
    \node (A1_1) at (1, 1) {$\bullet$};
    \node (A1_2) at (2, 1) {$\bullet$};
    \node (A1_3) at (3, 1) {$\bullet$};
    \node (A1_4) at (4, 1) {$\bullet$};
    \node (A1_5) at (5, 1) {$\dots$};
    \node (A1_6) at (6, 1) {$\bullet$};
    \path (A1_4) edge [-] node [auto] {$\scriptstyle{}$} (A1_5);
    \path (A1_5) edge [-] node [auto] {$\scriptstyle{}$} (A1_6);
    \path (B) edge [-] node [auto] {$\scriptstyle{}$} (A1_1);
    \path (A1_1) edge [-] node [auto] {$\scriptstyle{}$} (A1_2);
    \draw (A1_2) to[bend right] (A1_3);
    \draw (A1_2) to[bend left] (A1_3);
\draw[decorate,decoration={brace,amplitude=7pt,mirror},xshift=0.4pt,yshift=-0.3pt](3.85,0.55) -- (6.15,0.55) node[black,midway,yshift=-0.5cm] {\footnotesize $t-1$};    \draw[decorate,decoration={brace,amplitude=7pt,mirror},xshift=0.4pt,yshift=-0.3pt](-0.75,0.55) -- (2.15,0.55) node[black,midway,yshift=-0.5cm] {\footnotesize expanded from $(2,t+1,2)$};
    \path (A1_3) edge [-] node [auto] {$\scriptstyle{}$} (A1_4);
  \end{tikzpicture}
  \]
and it is easy to check that for no value of $t$ there is an embedding of this graph into the standard lattice of rank one less than the number of vertices in $\Psi'$ (beware that there is such an embedding if $|u'|=3$ and $t=2$).
	\item If $|v'|=2$ and $|u'|=3$ and $b(\Psi'\setminus u')=1$, the argument is verbatim the one explained before with the roles of $v'$ and $u'$ switched. If there are no bad components in $L$, then the analysis performed with the roles of $v'$ and $u'$ exchanged goes through, except the final step, in which the values of the $v$ leg in $\Gamma_{1}$ come from the string that was before associated to $u'$. That is, from the graph $C$ we obtain 
	
\[
  \begin{tikzpicture}[xscale=1.5,yscale=-0.5]
    \node (A3_0) at (3.75,5.75) {$\scriptstyle v$};
	\node (A1_3) at (3.75,5) {$\bullet$};
	\node at (3.75,4.5) {$\scriptstyle 4$};
	\node (A1_31) at (4.75,5) {$\bullet$};
	\node at (4.75,4.5) {$\scriptstyle2$};
	\node (A3) at (3.75,0) {$\scriptstyle 2$};
	\node (A1) at (3.75,1) {$\bullet$};
    \node (A2_3) at (3.5, 3) {$\scriptstyle e_{1}+e_{2}$};
    \node (A3_3) at (3, 3) {$\bullet$};
    \node (A4_4) at (2, 2) {$\scriptstyle e_{2}+e_{3}$};
    \node (A5_4) at (2, 3) {$\bullet$};
    \node (F) at (4.75, 1) {$\bullet$};
    \node (J) at (5.75, 1) {$\bullet$};
    \node (K) at (6.75, 1) {$\bullet$};
    \node at (4.75, 0) {$\scriptstyle 2$};
    \node at (5.75, 0) {$\scriptstyle 3$};
    \node at (6.75, 0) {$\scriptstyle 7=k+1$};
    \node at (3.75, 1.5) {$\scriptstyle u$};
    \path (A3_3) edge [-] node [auto] {$\scriptstyle{}$} (A1_3);
    \path (A3_3) edge [-] node [auto] {$\scriptstyle{}$} (A1);
    \path (A3_3) edge [-] node [auto] {$\scriptstyle{}$} (A5_4);
    \path (A1_3) edge [-] node [auto] {$\scriptstyle{}$} (A1_31);
    \path (A1) edge [-] node [auto] {$\scriptstyle{}$} (F);
    \path (F) edge [-] node [auto] {$\scriptstyle{}$} (J);
    \path (J) edge [-] node [auto] {$\scriptstyle{}$} (K);
   \end{tikzpicture}
  \]
and respectively the family $B_{n}$ yields in this case
\[
  \begin{tikzpicture}[xscale=1.25,yscale=-0.5]
    \node (A3_0) at (3.75,5.75) {$\scriptstyle v$};
	\node (A1_3) at (3.75,5) {$\bullet$};
	\node at (3.75,4.5) {$\scriptstyle 4$};
	\node (A1_31) at (4.75,5) {$\bullet$};
	\node at (4.75,4.5) {$\scriptstyle n+1$};
	\node (A1) at (3.75,1) {$\bullet$};
    \node (A2_3) at (3.5, 3) {$\scriptstyle e_{1}+e_{2}$};
    \node (A3_3) at (3, 3) {$\bullet$};
    \node (A4_4) at (2, 2) {$\scriptstyle e_{2}+e_{3}$};
    \node (A5_4) at (2, 3) {$\bullet$};
    \node (F) at (4.75, 1) {$\bullet$};
    \node (J) at (5.75, 1) {$\bullet$};
    \node (K) at (6.75, 1) {$\bullet$};
    \node (L) at (7.75, 1) {$\dots$};
    \node (M) at (8.75, 1) {$\bullet$};
    \node (Q) at (9.75, 1) {$\bullet$};
    \node (A3) at (3.75,0) {$\scriptstyle 2$};
    \node at (4.75, 0) {$\scriptstyle 2$};
    \node at (5.75, 0) {$\scriptstyle 3$};
    \node at (6.75, 0) {$\scriptstyle 2$};
    \node at (8.75, 0) {$\scriptstyle 2$};
    \node at (9.75, 0) {$\scriptstyle 7=k+1$};
    \node at (3.75, 1.5) {$\scriptstyle u$};
    \path (A3_3) edge [-] node [auto] {$\scriptstyle{}$} (A1_3);
    \path (A3_3) edge [-] node [auto] {$\scriptstyle{}$} (A1);
    \path (A3_3) edge [-] node [auto] {$\scriptstyle{}$} (A5_4);
    \path (A1_3) edge [-] node [auto] {$\scriptstyle{}$} (A1_31);
    \path (A1) edge [-] node [auto] {$\scriptstyle{}$} (F);
    \path (F) edge [-] node [auto] {$\scriptstyle{}$} (J);
    \path (J) edge [-] node [auto] {$\scriptstyle{}$} (K);
    \path (K) edge [-] node [auto] {$\scriptstyle{}$} (L);
    \path (L) edge [-] node [auto] {$\scriptstyle{}$} (M);
    \path (M) edge [-] node [auto] {$\scriptstyle{}$} (Q);  \draw[decorate,decoration={brace,amplitude=7pt,mirror},xshift=0.4pt,yshift=-0.3pt](6.7,1.25) -- (8.8,1.25) node[black,midway,yshift=-0.5cm] {\footnotesize $n-1$};
   \end{tikzpicture}
  \]
(Notice that the vertex with weight $n+1$ cannot coincide with the vertex with weight $k+1$. This violates the condition on the length of the strings associated with the $u$ and $v$ legs when regarded without the vertex of weight $k+1$.)
These $\Gamma_{1}$ type graphs do not admit an embedding in the standard lattice of the same rank with a basis vector hitting only once a weight 2 vertex on the $u$ leg, so there are no examples to consider.
\end{itemize}

Finally, if $|u'|=3$ or $4$ and $u'$ is final, then we are looking at graphs $\Psi'$ of the following type
		\begin{equation}\label{forlater}
  \begin{tikzpicture}[xscale=2,yscale=1.5,baseline=(current  bounding  box.center)]
    \node at (2, 1.25) {$\scriptstyle v'$};
    \node at (2, 0.65) {$\scriptstyle f_{1}+f_{2}$};
    \node at (3, 1.25) {$\scriptstyle u'$};
    \node at (3, 0.65) {$\scriptstyle 3\mathrm{\ or\ } 4$};
    \node (B) at (0, 1) {$\dots$};
    \node (A1_1) at (1, 1) {$\bullet$};
    \node (A1_2) at (2, 1) {$\bullet$};
    \node (A1_3) at (3, 1) {$\bullet$};    \path (B) edge [-] node [auto] {$\scriptstyle{}$} (A1_1);
    \path (A1_1) edge [-] node [auto] {$\scriptstyle{}$} (A1_2);
    \draw (A1_2) to[bend right] (A1_3);
    \draw (A1_2) to[bend left] (A1_3);
  \end{tikzpicture}
  \end{equation}
Since they have vanishing determinant, we know from the proof of Lemma~\ref{standardsubset} that 
$$
\frac{(v'\cdot u')^{2}}{u'\cdot u'}=[2,\dots]^{-}.
$$
For $|u'|=4$ this yields no solutions. So we are left with the case $|u'|=3$ and $[2,\dots]^{-}=[2,2,2]^{-}$ which corresponds to a $v$-leg in $\Gamma_{1}$ with associated string $(4,2,2)$. This in turn implies that $k=5$, since $I(\Psi')=-3$, and that the quasi-complementary $u$-leg in $\Gamma_{1}$ has associated string $(2,2,4,6=k+1)$. Collecting all these findings we learn that if there were any examples in family $\Gamma_{1}$ under these assumptions they necessarily are of the following form, where the 2-leg is drawn to the left due to space constraints.
\[
  \begin{tikzpicture}[xscale=1.5,yscale=-0.5]
    \node (A3_0) at (3.75,5.75) {$\scriptstyle v$};
	\node (A1_3) at (3.75,5) {$\bullet$};
	\node at (3.75,4.5) {$\scriptstyle 4$};
	\node (A1_31) at (4.75,5) {$\bullet$};
	\node at (4.75,4.5) {$\scriptstyle2$};
	\node (A) at (5.75,5) {$\bullet$};
	\node at (5.75,4.5) {$\scriptstyle 2$};
	\node (A3) at (3.75,0) {$\scriptstyle 2$};
	\node (A1) at (3.75,1) {$\bullet$};
    \node (A2_3) at (3.5, 3) {$\scriptstyle e_{1}+e_{2}$};
    \node (A3_3) at (3, 3) {$\bullet$};
    \node (A4_4) at (2, 2) {$\scriptstyle e_{2}+e_{3}$};
    \node (A5_4) at (2, 3) {$\bullet$};
    \node (F) at (4.75, 1) {$\bullet$};
    \node (J) at (5.75, 1) {$\bullet$};
    \node (K) at (6.75, 1) {$\bullet$};
    \node at (4.75, 0) {$\scriptstyle 2$};
    \node at (5.75, 0) {$\scriptstyle 4$};
    \node at (6.75, 0) {$\scriptstyle 6=k+1$};
    \node at (3.75, 1.5) {$\scriptstyle u$};
    \path (A3_3) edge [-] node [auto] {$\scriptstyle{}$} (A1_3);
    \path (A3_3) edge [-] node [auto] {$\scriptstyle{}$} (A1);
    \path (A3_3) edge [-] node [auto] {$\scriptstyle{}$} (A5_4);
    \path (A1_3) edge [-] node [auto] {$\scriptstyle{}$} (A1_31);
    \path (A1_31) edge [-] node [auto] {$\scriptstyle{}$} (A);
    \path (A1) edge [-] node [auto] {$\scriptstyle{}$} (F);
    \path (F) edge [-] node [auto] {$\scriptstyle{}$} (J);
    \path (J) edge [-] node [auto] {$\scriptstyle{}$} (K);
   \end{tikzpicture}
  \]
It is not difficult to check that this graph admits no embedding into $\Z^{9}$, let alone one with a basis vector appearing only once.
\end{enumerate}
\end{proof}

Once that we have determined that in the embeddings of the graphs $\Gamma$ no basis vector appears only once, and before we continue with our analysis of the vectors appearing only twice in the embedding, we pause to prove the following lemma, which is a technical step that shows up in many of the forthcoming proofs. It is a non-vanishing determinant statement for a variation of the graphs $\Gamma$. The essential features we need from $\Gamma$ are the $2$-leg, the central vertex of weight 2 and the vertices $u$ and $v$ with their embeddings as in illustrated  Figure~\ref{f:convenzioni1}. We also admit the possibility  that the vertex $u$ is final in its leg with weight 1 and embedding $u=e_{1}$. Call any graph with these features $\Psi$.

\begin{lemma}\label{l:3final1}
If $\Psi$ admits an embedding into a standard lattice of any rank and all its vertices have weight at least 2 except possibly one of its final vertices which do not belong to the 2-leg, then $\Psi$ has non-vanishing determinant.
\end{lemma}
\begin{proof}
If all vertices in $\Psi$ have weight at least 2, then we are in a particular case of Lemma~\ref{l:det} and we conclude that the determinant of $\Psi$ is non zero.

If a final vertex $x$ has weight 1 we can proceed to blow it down in a compatible way with the embedding (see Operation~\ref{op}(\ref{bd})). If we obtain a graph with no vertex of weight 1, then the determinant is non-vanishing by Lemma~\ref{l:det}. If we do obtain a graph with weight 1, we iterate the procedure. If the vertex $x$ were in the same qc-leg as the vertex $v$ (following the notation of Figure~\ref{f:convenzioni1}), then this blowing down process necessarily ends at the latest with the vertex $v$ and the determinant is non-vanishing by Lemma~\ref{l:det}. If $x$ had been in the qc-leg with the vertex $u$, then the determinant will be non-vanishing (again, by Lemma~\ref{l:det}) unless possibly if all the vertices in this qc-leg between the trivalent one and $x$ have weight precisely 2. After  blowing down we arrive to the following configuration with embedding:
\[
  \begin{tikzpicture}[xscale=1.7,yscale=-0.5]
    \node (A0_2) at (2, 0) {$\scriptstyle\sum_{i=2}^{N+2}(-1)^{i}e_{i}+v'$};
    \node at (2, 1.5) {$v$};
    \node (A0_3) at (3, 0) {$1$};
    \node (A0_4) at (4, 0) {$\scriptstyle e_{2}+e_{3}$};
    \node (A0_6) at (6, 0) {$\scriptstyle e_{N+1}+e_{N+2}$};
    \node (A1_0) at (0, 1) {$\bullet$};
    \node (A1_1) at (1, 1) {$\dots$};
    \node (A1_2) at (2, 1) {$\bullet$};
    \node (A1_3) at (3, 1) {$\bullet$};
    \node (A1_4) at (4, 1) {$\bullet$};
    \node (A1_5) at (5, 1) {$\dots$};
    \node (A1_6) at (6, 1) {$\bullet$};  
\draw[decorate,decoration={brace,amplitude=7pt,mirror},xshift=0.4pt,yshift=-0.3pt](3.9,1.3) -- (6.1,1.3) node[black,midway,yshift=-0.5cm] {\footnotesize $N$};
    \path (A1_4) edge [-] node [auto] {$\scriptstyle{}$} (A1_5);
    \path (A1_0) edge [-] node [auto] {$\scriptstyle{}$} (A1_1);
    \path (A1_1) edge [-] node [auto] {$\scriptstyle{}$} (A1_2);
    \path (A1_5) edge [-] node [auto] {$\scriptstyle{}$} (A1_6);
    \path (A1_2) edge [-] node [auto] {$\scriptstyle{}$} (A1_3);
    \path (A1_3) edge [-] node [auto] {$\scriptstyle{}$} (A1_4);
  \end{tikzpicture}
  \]
However, since $|v'| \geq 1$ by Lemma~\ref{l:firststep}, the determinant of the above configuration is non-vanishing. Indeed, the determinant will vanish if and only if the sum of the inverses of the continued fractions associated to the 2-leg, $\frac{N}{N+1}$, and the one associated to the string starting with $v$ add up to 1. However, since $|v'|\geq 1$, this second fraction is strictly smaller than $\frac{1}{N+1}$. 
\end{proof}
\end{subsection} 
 
\begin{subsection}{Vectors appearing twice in the embedding} 
We move on to study the basis vectors which appear exactly twice in the embedding of $\Gamma$ with coefficients $\pm1$.  
They are, together with the vectors which appear exactly once with coefficient $\pm1$, the only type of vectors in the embedding that contribute negatively to $I(\Gamma)$. The following lemmas are a careful study of different relevant properties that these vectors satisfy in the embedding of $\Gamma$. We start looking at a couple of such vectors and we establish that they either hit 4 different vertices in $\Gamma$ or they hit 3 in a particular configuration. 

\begin{lemma}\label{l:disjoint}
Suppose that $\Gamma$ embeds in a standard lattice of the same rank and that there are two basis vectors $f_{i}$ and $f_{j}$ which appear only twice with coefficient $\pm 1$ in the embedding. Then,
\begin{enumerate}
\item $f_{i}$ and $f_{j}$ do not both hit the same pair of vertices in $\Gamma$.
\item If there is a vertex $t$ in $\Gamma$ hit by both $f_{i}$ and $f_{j}$, then $f_{i}$ or $f_{j}$ hit an internal vertex $x$ (necessarily different from $t$) of weight 2. 
\end{enumerate}
\end{lemma}
\begin{proof}
We will start with the first item in the statement arguing by contradiction. Suppose we had two vertices in $\Gamma$ with embeddings $t=f_{i}+f_{j}+t'$ and  $w=\pm f_{i}\pm f_{j}+w'$.
Since we are assuming $\Gamma$ embeds we follow the conventions of Figure~\ref{f:convenzioni1}. Notice that $f_{i}$ and $f_{j}$ do not hit the 2-leg or the trivalent vertex in $\Gamma$ (an immediate consequence of Lemma~\ref{l:2rep} is that the basis vector $e_{1}$ in Figure~\ref{f:convenzioni1} appears at least 3 times in the embedding). So, if there were vertices $t=f_{i}+f_{j}+t'$ and $w=\pm f_{i}\pm f_{j}+w'$, we can consider one of the following situations:
\begin{enumerate}
\item If $|t'|=0$ then, since $t$ is not isolated, $w$ needs to be adjacent to $t$, but $w\cdot t\in\{0,\pm2\}$. A contradiction. The case $|w'|=0$ is analogous. 
\item If $|t'|>1$ (the case $|w'|>1$ is analogous) we consider the graph with embedding $\Gamma'$ obtained from $\Gamma$ by deleting the vertex $w$ and the basis vectors $f_{i}$ and $f_{j}$. $\Gamma'$ embeds in a lattice of rank smaller than the number of its vertices and thus should have determinant zero. However, if $w$ was a final vertex, by Lemma~\ref{l:det} we know $\Gamma'$ has non-vanishing determinant. Finally, if $w$ were internal, $\Gamma'$ has two connected components: by Remark~\ref{rem}(\ref{dd}) any linear components have non vanishing determinant and again by Lemma~\ref{l:det} if there is a component with a trivalent vertex, then its determinant is also non-vanishing.
\item If $|w'|=|t'|=1$, then we know that $t$ and $w$ are different from the vertex $u$, since $|u|=2$, and also from the vertex $v$, since at least two vectors hitting $v$ appear at least 3 times in the embedding (this follows from Lemma~\ref{l:2rep}).
	\begin{itemize}
	\item If at least one between $w$ and $t$ is in the leg which starts with the vertex $v$, we assume without loss of generality that $w$ is closest to the vertex $v$ in its leg. Then we consider the graph with embedding $\Gamma'$ obtained by deleting the vertex $t$ and the basis vectors $f_{i}$ and $f_{j}$. This graph has vanishing determinant (Remark~\ref{rem}(\ref{small})). However, if $\Gamma'$ has a linear component, by Remark~\ref{rem}(\ref{dd}) this component has non-vanishing determinant. It follows that the component of $\Gamma'$ with a trivalent vertex should have determinant zero, but this component has an embedding and a vertex of weight 1 in the same leg as $v$.	
We are then in a situation equivalent to the one analyzed in Lemma~\ref{l:p1}, subcase 2.(a) and thus either the determinant is non zero or the embedding does not exist, and we obtain again a contradiction. 
	\item If both $w$ and $t$ are on the leg starting with the vertex $u$ we consider the graph $\Gamma'$ obtained by deleting the closest vertex to $u$, say $w$ and deleting once more the basis vectors $f_{i}$ and $f_{j}$. $\Gamma'$ needs to have vanishing determinant, however by Lemma~\ref{l:det} the connected component with the trivalent vertex has determinant different from zero. It follows that $t$ is internal in $\Gamma$ and that in $\Gamma'$ there is a linear component $L'$ with vanishing determinant, an embedding into a standard lattice and a vertex with weight 1. This vertex needs to be internal in $\Gamma'$ (Remark~\ref{rem}(\ref{dd})), and the linear component is of the following type:
	\[
  \begin{tikzpicture}[xscale=1.5,yscale=-0.5]
    \node at (-1, 0.75) {$L':$};
    \node (A0_0) at (0, 0) {$c_n$};
    \node (A0_2) at (2, 0) {$c_1$};
    \node (A0_3) at (3, 0) {$1$};
    \node at (0, 1.5) {$r'$};
    \node at (3, 1.5) {$t'$};
    \node (A0_4) at (4, 0) {$d_1$};
    \node (A0_6) at (6, 0) {$d_m$};
    \node (A1_0) at (0, 1) {$\bullet$};
    \node (A1_1) at (1, 1) {$\dots$};
    \node (A1_2) at (2, 1) {$\bullet$};
    \node (A1_3) at (3, 1) {$\bullet$};
    \node (A1_4) at (4, 1) {$\bullet$};
    \node (A1_5) at (5, 1) {$\dots$};
    \node (A1_6) at (6, 1) {$\bullet$};
    \path (A1_4) edge [-] node [auto] {$\scriptstyle{}$} (A1_5);
    \path (A1_0) edge [-] node [auto] {$\scriptstyle{}$} (A1_1);
    \path (A1_1) edge [-] node [auto] {$\scriptstyle{}$} (A1_2);
    \path (A1_5) edge [-] node [auto] {$\scriptstyle{}$} (A1_6);
    \path (A1_2) edge [-] node [auto] {$\scriptstyle{}$} (A1_3);
    \path (A1_3) edge [-] node [auto] {$\scriptstyle{}$} (A1_4);
  \end{tikzpicture}
  \]
Since the determinant vanishes, the strings $(c_{1},\cdots,c_{n})$ and $(d_{1},\cdots,d_{m})$ are complementary and since they are both adjacent to the vertex $t'$ there is a basis vector that hits both the vertices with weights $c_{1}$ and $d_{1}$. Then, by Property~\ref{f}(\ref{f:unique}) the embedding of $L'$ in a standard lattice of any rank is unique. Now, in $\Gamma$ the vertex $w=f_{i}+f_{j}+f$ is linked to $r'$ and since $r'\neq t'$ (recall $t'$ is internal in $L'$) the only basis vector that hits both $r'$ and $w$ is $f$. By the uniqueness of the embedding of $L'$ we know that $f$ hits some other vertex $s'$ in $L'$ which is then necessarily orthogonal to $w$, but this is impossible. Indeed, since $w\cdot s'=0$, then $f_{i}$ or $f_{j}$ must hit $s'$ and then we necessarily have $s'=t'$. It follows that the three vectors $f_{i}$, $f_{j}$ and $f$ hit both $t$ and $w$ while $t\cdot w=0$. A contradiction.
	\end{itemize}  
\end{enumerate}

We now consider the second item in the statement and suppose there are three vectors in $\Gamma$ whose embeddings are, up to signs, of the form $w=f_{i}+w'$, $t=f_{i}+f_{j}+t'$ and $y=f_{j}+y'$, and independently of the vertices $w$, $t$ and $y$ being adjacent or in the same or different legs, we consider the following manipulation: delete from the embedding the basis vectors $f_{i}$ and $f_{j}$ and the vertex $t$ from $\Gamma$ to obtain $\Gamma'$, a graph with embedding into a standard lattice of rank one less than its number of vertices. Now,  if $|w|,|y|>2$ then the vertices $w$ and $y$ have weight at least 2 in $\Gamma'$ and we arrive to the contradiction that $\Gamma'$ has non-vanishing determinant by Lemma~\ref{l:det} and Remark~\ref{rem}(\ref{dd}). Moreover, 
if we had $|w|=2$ and $|y|>2$ and $w$ were final, the above manipulation yields a contradiction with Lemma~\ref{l:3final1} or Remark~\ref{rem}(~\ref{ddN}). These two contradictions imply the second assertion in the statement. 
\end{proof}

Let $B:=\{f_{1},\dots,f_{n}\}$ be the set of all basis vectors which appear exactly twice in the embedding of $\Gamma$ with coefficients $\pm1$. By Lemma~\ref{l:disjoint} the set of pairs of vertices in $\Gamma$ defined as $V_{2}:=\{(u_{i},v_{i})\,|\, f_{i}\cdot u_{i}, f_{i}\cdot v_{i}\neq 0\}$ consists of $n$ couples of vertices, where the couple $(u_{i},v_{i})$ is formed of the only two vertices in $\Gamma$ hit by the vector $f_{i}$. Notice that the vertices in $V_{2}$ all belong to the qc-legs, since $\Gamma$ embeds and therefore the conventions in Figure~\ref{f:convenzioni1} hold. We will divide $V_{2}$ into two subsets: let $V_{2}^{f}$ be the subset containing couples with a final vertex and $V_{2}^{Nf}$ the subset of couples with only internal vertices. The next crucial lemma shows that the vectors appearing only twice in the embedding hit `mainly' internal vertices of weight at most 3. The key idea in the proof of this lemma is to manipulate $\Gamma$ in such a way that the trivalent vertex, the only one whose valency is greater than its weight, belongs to a separate connected component.

\begin{lemma}\label{l:2times}
If $\Gamma$ embeds, then at most two couples in the set $V_{2}^{Nf}$ and at most two couples in $V_{2}^{f}$ have no internal vertex of weight 2 or 3. 
\end{lemma}
\begin{proof}
Let us assume that there are at least three such couples in $V_{2}^{f}$. Then, at least one of the final vertices of $\Gamma$ is necessarily hit by two different basis vectors which appear only twice in the embedding with coefficient $\pm1$. It follows, by the second statement of Lemma~\ref{l:disjoint}, that one of these three couples has a vertex of weight 2. If it is internal, we are done. If it is final, either one final vertex is hit by three vectors appearing only twice, and we conclude with Lemma~\ref{l:disjoint} that at least one internal vertex hit by one of these vectors has weight 2. Or both final vertices are hit by 2 vectors appearing only twice, and again we conclude with Lemma~\ref{l:disjoint}.

We move onto studying now the set $V_{2}^{Nf}$. Let us assume by contradiction that there is a subset of $V_{2}^{Nf}$, the set $S:=\{(u_{i},v_{i})\,|\, i=1,2,3\}$, such that every vertex in it has weight at least 4. As before, we call $f_{i}$ the vector appearing exactly twice in the embedding with coefficient $\pm1$ hitting $u_{i}$ and $v_{i}$. From the second statement in Lemma~\ref{l:disjoint} we know that $S$ consists of three disjoint couples of vertices.

We will consider the graph with embedding $\Gamma'$ obtained from $\Gamma$ by deleting all $f_{i}$, $i=1,2,3$ from the embedding and deleting two vertices in $S$. The precise choice of the two vertices we need to delete in $S$ will depend on the configuration and it will be one vertex from two of the three pairs $(u_{i},v_{i})$, $i=1,2,3$: From the trivalent vertex move along one of the qc-legs to the first vector in $S$, which we call $x$. Now, consider the other qc-leg and again from the trivalent vertex move to the first vector in $S$ with different subindex from $x$. If this vertex exists call it $y$.
	\begin{enumerate}
	\item If both $x$ and $y$ exist delete these two vertices.
	\item If $y$ does not exist, delete $x$ and the next vertex in $S$ found on its qc-leg and having a different subindex from $x$.   
	\end{enumerate}
The graph $\Gamma'$ admits an embedding into a standard lattice of rank smaller than the number of vertices in $\Gamma'$ and therefore $\Gamma'$ has vanishing determinant (Remark~\ref{rem}(\ref{small})). The connected component of $\Gamma'$ which contains the only trivalent vertex of $\Gamma$ has an embedding, and thus, if it has a trivalent vertex in $\Gamma'$ by Lemma~\ref{l:det} the determinant of this component is non-vanishing; and if it is a linear set, then by Remark~\ref{rem}(\ref{dd}) its determinant is again non vanishing. It follows that at least one of the other connected components of $\Gamma'$ needs to have vanishing determinant. However, this is not possible: by Remark~\ref{rem}(\ref{ddN}) we need a connected component with a vertex with weight smaller than its valency, and by our assumptions on $S$ we do not have such a vertex.
\end{proof}

We have learned in Lemma~\ref{l:2times} that in $V_{2}$ there are at most 4 couples of vertices with no internal vertex of weight 2 or 3. Our aim now is to reduce this quantity to 2 vertices. Let us call $B_{2}$ ($B$ for `bad') the subset of $V_{2}$ consisting of couples in which the internal vertices have weight at least 4, and define consequently $B_{2}^{Nf}$ and $B_{2}^{f}$. In this language Lemma~\ref{l:2times} reads $|B_{2}^{Nf}|,|B_{2}^{f}|\leq 2 $ and implies $|B_{2}|\leq 4$. In the next lemma we show that if $|B_{2}|>2$, the set $B_{2}^{f}$ is formed of a particular kind of couple of vertices. The proof of Lemma~\ref{l:bad} builds on and uses the notation in the proof of Lemma~\ref{l:2times}.

\begin{lemma}\label{l:bad}
Suppose $\Gamma$ embeds and that we have $|B_{2}|>2$, then
\begin{enumerate}
\item if $|B_{2}^{Nf}|=2$, then every couple in $V_{2}^{f}$ has an internal vertex and the final vertex in every couple in $B_{2}^{f}$ has weight 2.
\item if $|B_{2}^{Nf}|=1$, then $|B_{2}^{f}|=2$ and at least one couple in $B_{2}^{f}$ has a final vertex of weight 2, and the other vertex in this couple is interior.
%
\end{enumerate}
\end{lemma}
\begin{proof}
 
We start with the first case assuming there are two couples $(u_{1},v_{1}),(u_{2},v_{2})$ in $B_{2}^{Nf}$ and we proceed to analyze an element $(u_{3},v_{3})\in V_{2}^{f}$. We will call $S$ the set formed by these three couples of vertices. We first show that $u_{3}$ and $v_{3}$ cannot be both final vertices in $\Gamma$. Suppose this were the case and consider the following manipulation of $\Gamma$ yielding the graph with embedding $\Gamma'$:
\begin{enumerate}
\item If $u_{3}$ or $v_{3}$ is the only vertex in its qc-leg belonging to $S$, then delete it from $\Gamma$ and delete in the other leg the first vertex among $u_{1},v_{1},u_{2}$ and $v_{2}$ which appears when moving along the qc-leg, starting from the trivalent vertex. Moreover delete from the embedding $f_{1},f_{2}$ and $f_{3}$.
\item If we are not in the preceding case, delete from $\Gamma$ the basis vectors $f_{1},f_{2}$ and $f_{3}$ and from each qc-leg the vertex in $S$ closest to $u_{3}$ or $v_{3}$ such that the two deleted vertices have different subindices.\end{enumerate}

In both cases, by Remark~\ref{rem}(\ref{small}), the graph $\Gamma'$ should have vanishing determinant. In the first case this is impossible: the only difference with the arguments presented in the proof of Lemma~\ref{l:2times} is that we might have in $\Gamma'$ the final vector $u_{3}$ or $v_{3}$ with weight $1$. However, since this vector is final in $\Gamma'$ and in its connected component there is at least one vertex with weight greater than its valency (recall that any other final vertex in its connected component has this property since all internal vertices in $S$ have weight at least 4 in $\Gamma$), it follows by Remark~\ref{rem}(\ref{dd}) that this component has non-vanishing determinant and then by Lemma~\ref{l:det} that the determinant of $\Gamma'$ is non zero.

In the second case, notice that $\Gamma'$ has two connected components: a linear one $\Gamma'_{0}$, containing $u_{3}$ and $v_{3}$, which are adjacent in $\Gamma'$, and another one containing the trivalent vertex in $\Gamma$. By Lemma~\ref{l:det} or Remark~\ref{rem}(\ref{dd}) this latter connected component has non vanishing determinant, so we need $\Gamma'_{0}$ to have determinant zero. Notice that by construction $\Gamma'_{0}$ contains at most one vertex from the set $S$ besides $u_{3}$ and $v_{3}$. If $u_{3}$ and $v_{3}$ have both weight at least 2 in $\Gamma'_{0}$, then we obtain a contradiction by Remark~\ref{rem}(\ref{dd}) with the fact that $\Gamma'_{0}$ has vanishing determinant. Since in $\Gamma$ they are connected to other vertices, they cannot both have weight equal to 1 in $\Gamma'_{0}$. The only possibility we are left to study is the case in which, say $|u_{3}|=1$ and $|v_{3}|\geq 2$ in $\Gamma'_{0}$. Notice that $\Gamma'_{0}$ is a linear graph with all vertices of weight at least 2 except for one internal vertex of weight 1 and that $\Gamma'_{0}$ has vanishing determinant. It follows that $\Gamma'_{0}$ consists of two complementary strings linked to a vertex of weight $1$, which implies that the embedding is as in Property~\ref{f}(\ref{f:unique}). We then have that $\Gamma'_{0}$ (without the vertex of weight 1) satisfies the assumptions of Lemma~\ref{extravector} part (2) and no vertex can be linked to this connected component without the introduction of new basis vectors. Now, since $\Gamma'_{0}$ contains at most one vertex of $S$ besides $u_{3}$ and $v_{3}$, at least one of its final vertices has the same weight and embedding in $\Gamma'_{0}$ as in $\Gamma$ and it is adjacent to some vertex in $\Gamma$. By Lemma~\ref{extravector} this is not possible and we conclude that in the couple $(u_{3},v_{3})\in V_{2}^{f}$ at least one of the vertices in internal. 

From now on we assume, without loss of generality, that $u_{3}$ is internal in $\Gamma$ and $v_{3}$ final. Notice that $u_{3}\not\in\{u_{1},v_{1},u_{2},v_{2}\}$, since if it belonged to this set, by Lemma~\ref{l:disjoint} we would have that one among $u_{1},v_{1},u_{2}$ and $v_{2}$ has weight 2, contradicting $(u_{1},v_{1}),(u_{2},v_{2})$ in $B_{2}^{Nf}$.
We have the following configuration in $\Gamma$: three basis vectors $f_{1},f_{2}$ and $f_{3}$ each appearing exactly twice in the embedding with coefficient $\pm1$; The vectors $f_{1}$ and $f_{2}$ hitting respectively the couples $(u_{1},v_{1})$ and $(u_{2},v_{2})$ of internal vertices of weight at least 4, and the vector $f_{3}$ hitting an internal vertex $u_{3}$, different from $u_{1},v_{1},u_{2}$ and $v_{2}$, and a final vertex $v_{3}$. We now proceed to manipulate $\Gamma$ in the same fashion we did in the proof of Lemma~\ref{l:2times} to obtain $\Gamma'$ and we distinguish the two following outcomes:
\begin{enumerate}
\item One of the two vertices of $\Gamma$ erased to obtain $\Gamma'$ belongs to the couple $(u_{3},v_{3})$.
\item The vertices $u_{3}$ and $v_{3}$ belong to $\Gamma'$.
\end{enumerate}
If we are in the first case, by the same arguments we have developed several times before based on Lemma~\ref{l:det} and Remarks~\ref{rem}(\ref{ddN}) and~\ref{rem}(\ref{dd}), we arrive to the contradiction that $\Gamma'$ has non vanishing determinant, and thus case 2 is the only possible option. In this case, if the weights of $u_{3}$ and $v_{3}$ in $\Gamma'$ are respectively at least 3 and 2, then the determinant of $\Gamma'$ would be non-vanishing by Remark~\ref{rem}(\ref{ddN}). So if we assume $u_{3}$ has weight at least 4 in $\Gamma$, that is $(u_{3},v_{3})\in B_{2}^{f}$, this implies that $v_3$ has weight 2 in $\Gamma$ and 1 in $\Gamma'$,  finishing the proof of the first case in the lemma.

We now deal with the second possibility in the statement, that is, $B_{2}^{Nf}=\{(u_{1},v_{1})\}$. Since we are assuming $|B_{2}|>2$ and we know from Lemma~\ref{l:2times} that $|B_{2}^{f}|\leq2$ it follows that $B_{2}^{f}=\{(u_{2},v_{2}),(u_{3},v_{3})\}$ and we need to show that one of the couples needs to have a final vertex of weight 2. Let us assume any final vertex in $(u_{2},v_{2})$ had weight greater than 2. One of the following 2 possibilities holds:
\begin{itemize}
\item[a)] $(u_{2},v_{2})$ consists of the two final vertices in the qc-legs. Then, by Lemma~\ref{l:disjoint}, only one between $u_{3}$ and $v_{3}$ is final. Let us assume it is $v_{3}$. It follows that $v_{3}$ is hit both by $f_{2}$ and $f_{3}$ and thus, the second statement in Lemma~\ref{l:disjoint} guarantees $|u_{3}|=2$ and we arrive to the contradiction $(u_{3},v_{3})\not\in B_{2}^{f}$.
\item[b)] $(u_{2},v_{2})$ consists of one internal vertex of weight at least 4 and one final of weight at least 3. We assume $u_{2}$ is internal and $v_{2}$ is final. By the second statement of Lemma~\ref{l:disjoint} only one between $u_{3}$ and $v_{3}$ is final and again we assume it is $v_{3}$. The exact same argument used before allows us to conclude as follows:  manipulate $\Gamma$ as explained in the proof of Lemma~\ref{l:2times} to obtain $\Gamma'$ and we distinguish the two following outcomes:
\begin{enumerate}
\item One of the two vertices of $\Gamma$ erased to obtain $\Gamma'$ belongs to the couple $(u_{3},v_{3})$.
\item The vertices $u_{3}$ and $v_{3}$ belong to $\Gamma'$.
\end{enumerate}
The first case leads to a contradiction through Lemma~\ref{l:det} and Remarks~\ref{rem}(\ref{ddN}) and~\ref{rem}(\ref{dd}).  In the second case, since the weight of $u_{3}$ is assumed to be at least 4, in order to avoid a contradiction with Remark~\ref{rem}(\ref{ddN}) we need $|v_{3}|=2$.
\end{itemize}
\end{proof}

We will use the couples in $B_{2}^{f}$ with a final vertex of weight 2, whose existence is guaranteed by Lemma~\ref{l:bad}, to reduce the bound of 4 on $|B_{2}|$, obtained from Lemma~\ref{l:2times}, to 2. In the process we will be forced to change the graph $\Gamma$ to a new graph $\Gamma_{r}$. Suppose $\Gamma$ embeds and has a final 2-chain $c$ of length $k$ and some internal vertex $y$ such that the embedding looks as follows:
\[
  \begin{tikzpicture}[xscale=1.5,yscale=-0.5]
    \node (A2_0) at (0, 3.5) {$x$};
    \node at (1, 3.5) {$v_{k}$};
    \node at (0, 2) {$\scriptstyle g_{k}+x'$};
    \node at (1, 2) {$\scriptstyle g_{k-1}+g_{k}$};
    \node (A3_0) at (0, 3) {$\bullet$};
    \node (A3_1) at (1, 3) {$\bullet$};
    \node (A3_2) at (2, 3) {$\dots$};
    \node (A3_3) at (3, 3) {$\bullet$};
    \node (C) at (-3, 3) {$\bullet$};
    \node at (-3, 2) {$\scriptstyle \pm f \pm \sum_{i=1}^{k}(-1)^{i}g_{i}+y'$};
    \node at (-3, 3.5) {$y$};
    \node at (3, 2) {$\scriptstyle g_{1}+g_{2}$};
    \node (A) at (-1, 3) {$\dots$};
    \node at (4, 3.5) {$w$};
    \node at (4, 2) {$\scriptstyle f+g_{1}$};
    \node (B) at (4, 3) {$\bullet$};
    \path (A3_0) edge [-] node [auto] {$\scriptstyle{}$} (A3_1);
    \path (C) edge [-] node [auto] {$\scriptstyle{}$} (-3.5,3);
    \path (C) edge [-] node [auto] {$\scriptstyle{}$} (-2.5,3);
    \path (A3_2) edge [-] node [auto] {$\scriptstyle{}$} (A3_3);
    \path (A3_1) edge [-] node [auto] {$\scriptstyle{}$} (A3_2);
    \path (A) edge [-] node [auto] {$\scriptstyle{}$} (A3_0);
    \path (A3_3) edge [-] node [auto] {$\scriptstyle{}$} (B);
  \end{tikzpicture}
\]
where $w$ is a final vertex in a qc-leg, $v_{k}$ might coincide with $w$, $x\neq y$ has weight at least 3 and the vector $f$ appears only twice in the embedding hitting $w$ and $y$. Then the graph $\Gamma_{r}$ is obtained from $\Gamma$ by deleting the basis vectors $f,g_{1},\dots,g_{k-1}$ and all the vertices in the 2-chain, obtaining the configuration:
\[
  \begin{tikzpicture}[xscale=1.5,yscale=-0.5]
    \node at (0, 3.5) {$x$};
    \node (A3_0) at (0, 3) {$\bullet$};
    \node (A) at (-1, 3) {$\dots$};
    \node at (0, 2) {$\scriptstyle g_{k}+x'$};
   \node (C) at (-3, 3) {$\bullet$};
    \node at (-3, 2) {$\scriptstyle g_{k}+y'$};
    \node at (-3, 3.5) {$y$};
    \path (C) edge [-] node [auto] {$\scriptstyle{}$} (-3.5,3);
    \path (C) edge [-] node [auto] {$\scriptstyle{}$} (-2.5,3);
    \path (A) edge [-] node [auto] {$\scriptstyle{}$} (A3_0);
  \end{tikzpicture}
\]
where now the vertex $x$, which still has weight greater than 2, is final in $\Gamma_{r}$ and the vector $g_{k}$ appears only twice in the embedding of $\Gamma_{r}$. Moreover, if the other qc-leg had an analogous configuration, we can contract it as well, obtaining a graph with two final vertices of weight at least 3, each hit by a different basis vector appearing only twice in the embedding with coefficient $\pm1$. We will also call $\Gamma_{r}$ this `doubly reduced' graph. Notice that $I(\Gamma)=I(\Gamma_{r})$ and that $\Gamma_{r}$ admits an embedding into a standard lattice of rank equal its number of vertices. 

\begin{lemma}\label{l:only2}
If $\Gamma$ embeds and $|B_{2}(\Gamma)|>2$, then we can modify $\Gamma$ as explained above to obtain a graph $\Gamma_{r}$ with $|B_{2}(\Gamma_{r})|=2$.
\end{lemma}
\begin{proof}
To start with we show that if $|B_{2}(\Gamma)|>2$, we can find a final 2-chain in $\Gamma$ such as the one described in the paragraph preceding this lemma which allows us to define $\Gamma_{r}$. By Lemma~\ref{l:bad} we have that there is a couple in $B_{2}^{f}$ with a final vertex $w$ of weight 2. Let us call $f$ the vector which appears exactly twice in the embedding with coefficient one and hits $w$ and an internal vertex $y$ of weight at least 4. We consider $w$ as the last vertex of a 2-chain $c$ of length $k\geq1$. Since we are assuming that no qc-leg in $\Gamma$ is a 2 chain, there exists some vertex $x$ with weight greater than 2 adjacent to $c$. Since the embedding of a 2-chain is unique up to automorphisms of the standard lattice, we can assume it is as described above. If the vertex $x$ coincided with $y$, then since $f$ has always coefficient $\pm1$ this would mean $g_{k}\cdot x=0$ and the vector $g_{k}$ would appear only once in the embedding, since $f$ appears only twice. This is a contradiction since $x$ is not the trivalent vertex in $\Gamma$. In the latter case, we are dealing with  
\[
  \begin{tikzpicture}[xscale=2.5,yscale=-0.5]
    \node (A2_0) at (0, 3.5) {$x=y$};
    \node at (1.5, 3.5) {$v_{k}$};
    \node at (0, 2) {$\scriptstyle 2g_{k}-g_{k-1}+\dots\pm g_{1}\mp f+x'$};
    \node at (1.5, 2) {$\scriptstyle g_{k-1}+g_{k}$};
    \node (A3_0) at (0, 3) {$\bullet$};
    \node (A3_1) at (1.5, 3) {$\bullet$};
    \node (A3_2) at (2.5, 3) {$\dots$};
    \node (A3_3) at (3, 3) {$\bullet$};
    \node at (3, 2) {$\scriptstyle g_{1}+g_{2}$};
    \node at (3, 3.5) {$v_{2}$};
    \node (A) at (-1, 3) {$\dots$};
    \node at (4, 3.5) {$w$};
    \node at (4, 2) {$\scriptstyle f+g_{1}$};
    \node (B) at (4, 3) {$\bullet$};
    \path (A3_0) edge [-] node [auto] {$\scriptstyle{}$} (A3_1);
    \path (A3_2) edge [-] node [auto] {$\scriptstyle{}$} (A3_3);
    \path (A3_1) edge [-] node [auto] {$\scriptstyle{}$} (A3_2);
    \path (A) edge [-] node [auto] {$\scriptstyle{}$} (A3_0);
    \path (A3_3) edge [-] node [auto] {$\scriptstyle{}$} (B);
  \end{tikzpicture}
\]
and deleting from the embedding the vectors $f,g_{1},\dots,g_{k}$ we obtain a graph that embeds in a lattice of rank one smaller than its number of vertices contradicting Lemma~\ref{l:3final1}. 
It follows that $x\neq y$, that $g_{k}\cdot x=1$ and $g_{i}\cdot x=0$ for every $i\neq k$ and that $y=\scriptstyle f+\sum_{i=1}^{k}(-1)^{i}g_{i}+y'$. That is, we are precisely in the situation described above that allows us to define $\Gamma_{r}$ which has a final vertex of weight at least 3 hit by a basis vector appearing only twice with coefficient $\pm1$. If both qc-legs end in 2-chains in the situation described, we call $\Gamma_{r}$ the graph obtained by contracting both of them by the described method, and arriving to a graph with no final vertex of weight 2 in the qc-legs.

Now, since the pairs of vertices in $B_{2}^{Nf}(\Gamma)$ are still the same in $\Gamma_{r}$ we have that $B_{2}^{Nf}(\Gamma)=B_{2}^{Nf}(\Gamma_{r})$. Moreover, since in Lemmas~\ref{l:2times} and~\ref{l:bad} we have never used the specific weights on the qc-legs of $\Gamma$, their conclusions hold also for the graph $\Gamma_{r}$ and therefore, by Lemma~\ref{l:bad} we conclude that if $|B_{2}^{Nf}(\Gamma)|=2$ then $B_{2}^{f}(\Gamma_{r})=\emptyset$  and if $|B_{2}^{Nf}(\Gamma)|=1$ then $|B_{2}^{f}(\Gamma_{r})|\leq 1$. In both cases we have $|B_{2}(\Gamma_{r})|=2$ and the lemma is proved.
\end{proof}

In the last lemma we have shown that either the graph with embedding $\Gamma$ satisfies $|B_{2}(\Gamma)|\leq 2$ or we can modify it to a graph $\Gamma_{r}$ which satisfies this assumption. From now on we will use the symbol $\Gamma$ to denote: any graph $\Psi$ in the  families $\Gamma_{1}$ or $\Gamma_{2}$ in the case in which $|B_{2}(\Psi)|\leq 2$; or the corresponding modification $\Psi_{r}$ in case we had $|B_{2}(\Psi)|\in\{3,4\}$. We will call these modified graphs the \emph{reduced version} of the graphs in $\Gamma_{1}$ or $\Gamma_{2}$. Thus, we have proved the following.

\begin{co}\label{c:controledset}
If $\Gamma$ embeds, then at most two vectors appearing only twice in the embedding with coefficient $\pm1$ hit no internal vertex of weight 2 or 3.
\end{co}
\end{subsection}

\begin{subsection}{Internal vertices of weight 2 and weight 3}
We have established so far that we need to focus on internal vertices of weight 2 and 3 to pursue in our study of the basis vectors which appear only twice in the embedding with coefficient $\pm1$. In the next lemma we show that if we deal with these low weight vertices, then we do not need to worry about the second possibility in Lemma~\ref{l:disjoint}.

\begin{lemma}\label{l:3no2}
If the graph $\Gamma$ embeds, then no vertex of weight 2 or 3 is hit by two different basis vectors appearing exactly twice in the embedding with coefficients $\pm1$. 
\end{lemma}
\begin{proof}
By Lemma~\ref{l:disjoint} we know that if a vertex $t=f_{i}+f_{j}+t'$ 
is hit by $f_{i}$ and $f_{j}$ which appear only twice in the embedding with coefficient $\pm1$, then there are two other vertices with embeddings $w=f+f_{i}$ and $y=f_{j}+y'$ (up to signs), where $|y'| \geq 1$. If $t'=0$ this is impossible, since this would force the local configuration:
\[
  \begin{tikzpicture}[xscale=1.5,yscale=-0.4]
    \node (A0_2) at (2, 0) {$\scriptstyle f_{i}+f$};
    \node (A0_3) at (3, 0) {$\scriptstyle f_i + f_j$};
    \node at (2, 1.75) {$\scriptstyle w$};
    \node at (3, 1.75) {$\scriptstyle t$};
    \node at (4, 1.75) {$\scriptstyle y$};
    \node (A0_4) at (4, 0) {$\scriptstyle f_{j}+y'$};
    \node (A1_1) at (1, 1) {$\dots$};
    \node (A1_2) at (2, 1) {$\bullet$};
    \node (A1_3) at (3, 1) {$\bullet$};
    \node (A1_4) at (4, 1) {$\bullet$};
    \node (A1_5) at (5, 1) {$\dots$};
    \path (A1_4) edge [-] node [auto] {$\scriptstyle{}$} (A1_5);
    \path (A1_1) edge [-] node [auto] {$\scriptstyle{}$} (A1_2);
    \path (A1_2) edge [-] node [auto] {$\scriptstyle{}$} (A1_3);
    \path (A1_3) edge [-] node [auto] {$\scriptstyle{}$} (A1_4);
  \end{tikzpicture}
  \]
Deleting the vectors $f_i$ and $f_j$ produces a 2-component graph with embedding into the standard lattice of rank one less than the number of vertices, and so the determinant should be vanishing. However, the determinant must be non vanishing by Remark~\ref{rem}(\ref{ddN}) and Lemma~\ref{l:3final1}.

We now show that assuming $|t|=3$ again leads to a contradiction. We consider the two following scenarios:
\begin{enumerate}
\item The vertices $w=f_{i}+f$ and $t=f_{i}+f_{j}+g$ are adjacent in $\Gamma$. In this case consider the two  component graph $\Gamma'$ with embedding and vanishing determinant obtained from $\Gamma$ by deleting the vector $f_{i}$. In $\Gamma'$ the vertex $w$ has weight 1, but since it is final in its connected component, the determinant of $\Gamma'$ is non vanishing by Remark~\ref{rem}(\ref{ddN}) or Lemma~\ref{l:3final1}. This contradiction shows that the configuration is not possible.
\item The vertices $w$ and $t$ are not adjacent in $\Gamma$. We assume that $t=f_{i}+f_{j}+f$ and $w=\pm f \mp f_{i}$.\begin{itemize}
\item If $t$ were an internal vertex, then neither of its two adjacent vertices can be hit by $f_{i}$, since we are assuming $w\cdot t=0$. Moreover, both of them cannot be hit by $f_{j}$, since this basis vector appears only twice; nor by $f$, since at least one of these adjacent vertices is orthogonal to $w$ and $f_{i}$ appears only twice in the embedding. It follows that one vertex adjacent to $w$ is hit by $f$ and the other one is $y$ forcing the local configuration:
\[
  \begin{tikzpicture}[xscale=2,yscale=-0.5]
    \node (A2_0) at (0, 4) {$s$};
    \node at (0, 2) {$\scriptstyle f+\dots$};
    \node at (1, 2) {$\scriptstyle f_{i}+f_{j}+f$};
    \node (A2_2) at (1, 4) {$t$};
    \node (A2_2) at (2, 4) {$y$};
    \node at (2, 2) {$\scriptstyle f_{j}+y'$};
    \node (A3_0) at (0, 3) {$\bullet$};
    \node (A3_1) at (1, 3) {$\bullet$};
    \node (A3_2) at (2, 3) {$\bullet$};
    \node (A3_3) at (3, 3) {$\dots$};
    \node (A) at (-1, 3) {$\bullet$};
    \node at (-1, 4) {$w$};
    \node at (-1, 2) {$\scriptstyle -f_{i}+f$};
    \path (A3_0) edge [-] node [auto] {$\scriptstyle{}$} (A3_1);
    \path (A3_2) edge [-] node [auto] {$\scriptstyle{}$} (A3_3);
    \path (A3_1) edge [-] node [auto] {$\scriptstyle{}$} (A3_2);
    \path (A) edge [-] node [auto] {$\scriptstyle{}$} (A3_0);
  \end{tikzpicture}
\]
In this situation $t\cdot y=1$ and we consider the two connected component graph $\Gamma'$ obtained from $\Gamma$ by deleting the basis vector $f_{j}$. This graph with embedding needs to have vanishing determinant by Remark~\ref{rem}(\ref{small}), but this is impossible by Lemma~\ref{l:det} or by Lemma~\ref{l:3final1} if $|y'|=1$. We therefore conclude that this configuration is not possible.
\item If the vertex $t$ is final there is only one possible local configuration (up to sign):
\[
  \begin{tikzpicture}[xscale=2,yscale=-0.5]
    \node (A2_0) at (0, 4) {$s$};
    \node at (0, 2) {$\scriptstyle f+\dots$};
    \node at (1, 2) {$\scriptstyle f-f_{i}$};
    \node (A2_2) at (1, 4) {$w$};
    \node (A2_2) at (2, 4) {$y$};
    \node at (2, 2) {$\scriptstyle f-f_{j}+y'$};
    \node (A3_0) at (0, 3) {$\bullet$};
    \node (A3_1) at (1, 3) {$\bullet$};
    \node (A3_2) at (2, 3) {$\bullet$};
    \node (A3_3) at (3, 3) {$\dots$};
    \node (A) at (-1, 3) {$\bullet$};
    \node at (-1, 4) {$t$};
    \node at (-1, 2) {$\scriptstyle f_{i}+f_{j}+f$};
    \path (A3_0) edge [-] node [auto] {$\scriptstyle{}$} (A3_1);
    \path (A3_2) edge [-] node [auto] {$\scriptstyle{}$} (A3_3);
    \path (A3_1) edge [-] node [auto] {$\scriptstyle{}$} (A3_2);
    \path (A) edge [-] node [auto] {$\scriptstyle{}$} (A3_0);
  \end{tikzpicture}
\]
Indeed, the vertex $w$ is not isolated and therefore any adjacent vertex to it needs to be hit by $f$. Call $s$ such a vertex. Since we are assuming that $f_{i}$ and $f_{j}$ appear only twice in the embedding of $\Gamma$, the vertex $s$ is forced to be adjacent to $t$, which we are assuming is final. Now, the vertex $y$ is necessarily different from $s$, since $f\cdot s\neq 0$ and $f_{j}\cdot s\neq 0$, and therefore if $y=s$ we could not have $s\cdot t=1$. It follows that $y$ is orthogonal to $t$, which means in particular that $f$ hits $y$ and thus $y\cdot w=1$ and we have the claimed configuration.

Now, if there exists a basis vector $f'$ such that $y'=f'$ and $s=f-f'$ then the linear set formed by the vertices $t,s,w$ and $y$ is a standard subset of $\Z^{4}$ and by Lemma~\ref{extravector} we know that no final vertex can be attached to this configuration. It follows then that either $s$ or $y$ is hit yet by another basis vector, that is, we have the local configuration: 
\[
  \begin{tikzpicture}[xscale=2,yscale=-0.5]
    \node (A2_0) at (0, 4) {$s$};
    \node at (0, 2) {$\scriptstyle f-f'+s'$};
    \node at (1, 2) {$\scriptstyle f-f_{i}$};
    \node (A2_2) at (1, 4) {$w$};
    \node (A2_2) at (2, 4) {$y$};
    \node at (2, 2) {$\scriptstyle f-f_{j}+f'+y'$};
    \node (A3_0) at (0, 3) {$\bullet$};
    \node (A3_1) at (1, 3) {$\bullet$};
    \node (A3_2) at (2, 3) {$\bullet$};
    \node (A3_3) at (3, 3) {$\dots$};
    \node (A) at (-1, 3) {$\bullet$};
    \node at (-1, 4) {$t$};
    \node at (-1, 2) {$\scriptstyle f_{i}+f_{j}+f$};
    \path (A3_0) edge [-] node [auto] {$\scriptstyle{}$} (A3_1);
    \path (A3_2) edge [-] node [auto] {$\scriptstyle{}$} (A3_3);
    \path (A3_1) edge [-] node [auto] {$\scriptstyle{}$} (A3_2);
    \path (A) edge [-] node [auto] {$\scriptstyle{}$} (A3_0);
  \end{tikzpicture}
\]
where at least one between $s'$ and $y'$ is non zero. We proceed now with the following manipulation of the graph $\Gamma$: we delete the three basis vectors $f, f_{i}$ and $f_{j}$ (notice that $f$ appears exactly 4 times in the embedding, precisely in the local configuration depicted), obtaining a graph $\Gamma'$ with embedding and two vertices less than the original $\Gamma$. It follows that the determinant of $\Gamma'$ is zero (Remark~\ref{rem}(\ref{small})). However, we will show that this is not possible. The graph $\Gamma'$ has the following local configuration:
\[
  \begin{tikzpicture}[xscale=2,yscale=-0.5]
    \node (A2_0) at (0, 4) {$y$};
    \node at (-0.5, 2.75) {$\scriptstyle -$};
    \node at (0, 2) {$\scriptstyle f'+y'$};
    \node (A3_0) at (0, 3) {$\bullet$};
    \node (A3_3) at (1, 3) {$\dots$};
    \node (A) at (-1, 3) {$\bullet$};
    \node at (-1, 4) {$s$};
    \node at (-1, 2) {$\scriptstyle -f'+s'$};
    \path (A) edge [-] node [auto] {$\scriptstyle{}$} (A3_0);
    \path (A3_0) edge [-] node [auto] {$\scriptstyle{}$} (A3_3);
  \end{tikzpicture}
\]
and beyond these two vertices and the negative edge (which does not change anything in the argumentation, see Operation~\ref{op}(\ref{bd})) $\Gamma'$ is a tree with vanishing determinant, one trivalent vertex and all weights at least 2. If $s'$ and $y'$ were both non zero, then we have a contradiction with Lemma~\ref{l:det}. If $s'$ were zero, then $y'\neq 0$ and we have a contradiction with Lemma~\ref{l:3final1}. Finally, we consider the case $y'=0$. Since $y$ cannot be final, it is linked to a vertex different from $s$ and also hit by $f'$. We will argue now that in fact $y$ is necessarily connected to a 2-chain linked to the trivalent vertex. Indeed, if this were not the case, the local configuration would be:
\[
  \begin{tikzpicture}[xscale=1.7,yscale=-0.5]
	\node (A3) at (3.75,0) {$\scriptstyle h_{1}+h_{2}+g_{k}+\dots$};
	\node (A0) at (2.75,1) {$\dots$};
	\node (A1) at (3.75,1) {$\bullet$};
    \node (F) at (4.75, 1) {$\bullet$};
    \node (J) at (5.75, 1) {$\dots$};
    \node (K) at (6.75, 1) {$\bullet$};
    \node (L) at (7.75, 1) {$\bullet$};
    \node (M) at (8.75, 1) {$\bullet$};
    \node at (6.75, 0) {$\scriptstyle g_{1}+f'$};
    \node at (7.75, 1.5) {$\scriptstyle y$};
    \node at (7.75, 0) {$\scriptstyle f'$};
    \node at (8.25, 0.75) {$\scriptstyle -$};
    \node at (8.75, 1.5) {$\scriptstyle s$};
    \node at (8.75, 0) {$\scriptstyle -f'+s'$};
    \node at (4.75, 0) {$\scriptstyle g_{k}+g_{k-1}$};
    %
    \path (A0) edge [-] node [auto] {$\scriptstyle{}$} (A1);
    \path (A1) edge [-] node [auto] {$\scriptstyle{}$} (F);
    \path (F) edge [-] node [auto] {$\scriptstyle{}$} (J);
    \path (J) edge [-] node [auto] {$\scriptstyle{}$} (K);
    \path (K) edge [-] node [auto] {$\scriptstyle{}$} (L);
    \path (L) edge [-] node [auto] {$\scriptstyle{}$} (M);
   \end{tikzpicture}
  \]
However, for this embedding to exist we need $|s'|\geq k+1$ and, in these circumstances, after blowing down $y$ and the successive vertices with weight 1, we obtain a trivalent graph with embedding which should have vanishing determinant, but this contradicts Lemma~\ref{l:det} or Lemma~\ref{l:3final1}. So, we are only left with one possibility for $\Gamma'$:  $s$ is final and between the vertex $y$ and the trivalent vertex there is a 2-chain. This implies that $s$ and $y$ are not in the same qc-leg as the vertex $v$ (following the notation of Figure~\ref{f:convenzioni1}). In fact, the graph $\Gamma'$ is as follows:
\[
  \begin{tikzpicture}[xscale=1.2,yscale=-0.5]
    \node (A3_0) at (2.25,0) {$\scriptstyle \sum_{i=2}^{N+2}(-1)^{i}e_{i}+v'$};
	\node (A1_3) at (2.25,1) {$\bullet$};
	\node at (2.25,1.5) {$\scriptstyle v$};
	\node (A) at (1,1) {$\bullet$};
	\node (D) at (0,1) {$\dots$};
	\node (A3) at (3.75,0) {$\scriptstyle e_{1}+g_{k}$};
	\node (A1) at (3.75,1) {$\bullet$};
    \node (A2_3) at (2.5, 3) {$\scriptstyle e_{1}+e_{2}$};
    \node (A3_3) at (3, 3) {$\bullet$};
    \node (A4_4) at (4, 4) {$\scriptstyle e_{2}+e_{3}$};
    \node (A4_6) at (6, 4) {$\scriptstyle e_{N+1}+e_{N+2}$};
    \node (A5_4) at (4, 5) {$\bullet$};
    \node (A5_5) at (5, 5) {$\dots$};
    \node (A5_6) at (6, 5) {$\bullet$};
    \node (F) at (4.75, 1) {$\bullet$};
    \node (J) at (5.75, 1) {$\dots$};
    \node (K) at (6.75, 1) {$\bullet$};
    \node (L) at (7.75, 1) {$\bullet$};
    \node (M) at (8.75, 1) {$\bullet$};
    \node at (6.75, 0) {$\scriptstyle g_{1}+f'$};
    \node at (7.75, 1.5) {$\scriptstyle y$};
    \node at (7.75, 0) {$\scriptstyle f'$};
    \node at (8.25, 0.75) {$\scriptstyle -$};
    \node at (8.75, 1.5) {$\scriptstyle s$};
    \node at (8.75, 0) {$\scriptstyle -f'+s'$};
    \node at (4.75, 0) {$\scriptstyle g_{k}+g_{k-1}$};
    \node at (3.75, 1.5) {$\scriptstyle u$};
    \path (A3_3) edge [-] node [auto] {$\scriptstyle{}$} (A1_3);
    \path (A3_3) edge [-] node [auto] {$\scriptstyle{}$} (A1);
    \path (A5_4) edge [-] node [auto] {$\scriptstyle{}$} (A5_5);
    \path (A3_3) edge [-] node [auto] {$\scriptstyle{}$} (A5_4);
    \path (A5_5) edge [-] node [auto] {$\scriptstyle{}$} (A5_6);
    \path (A1_3) edge [-] node [auto] {$\scriptstyle{}$} (A);
    \path (A) edge [-] node [auto] {$\scriptstyle{}$} (D);
    \path (A1) edge [-] node [auto] {$\scriptstyle{}$} (F);
    \path (F) edge [-] node [auto] {$\scriptstyle{}$} (J);
    \path (J) edge [-] node [auto] {$\scriptstyle{}$} (K);
    \path (K) edge [-] node [auto] {$\scriptstyle{}$} (L);
    \path (L) edge [-] node [auto] {$\scriptstyle{}$} (M);
   \end{tikzpicture}
  \]
where $s'=\sum_{i=1}^{k}(-1)^{i+1}g_{i}+\sum_{i=1}^{N+2}(-1)^{k+i-1}e_{i}+s''$, and since $y$ in $\Gamma$ had weight 3, then $y\neq u$ and $k\geq 0$ (if $k=0$, then $u\cdot y =1$ and $u=e_{1}+f'$). It is clear that in the embedding of $\Gamma'$ the basis vector $f'$ cannot hit any vertices different from $y$ and its two adjacent ones. This in turn implies that the vectors $g_{i}$ and $e_{i}$ hit only the vertices depicted. We now proceed to delete from $\Gamma'$ the $N+k+3$ basis vectors $e_{1},\dots,e_{N+2},g_{1},\dots,g_{k},f'$. This reduces the number of vertices precisely by $N+k+3$ and leaves us with the following graph $\Gamma''$, which admits an embedding and has vanishing determinant (just as $\Gamma'$, it embeds into a standard lattice of rank smaller than its number of vertices):
 \[
  \begin{tikzpicture}[xscale=1.5,yscale=-0.5]
    \node (A0_2) at (2, 0) {$\scriptstyle v'$};
    \node at (2, 1.5) {$\scriptstyle v$};
    \node (A0_4) at (4, 0) {$\scriptstyle s''$};
    \node  at (4, 1.5) {$\scriptstyle s$};
    \node (A1_1) at (1, 1) {$\bullet$};
    \node (A1_2) at (2, 1) {$\bullet$};
    \node (A1_3) at (4, 1) {$\bullet$};
    \node (A1_0) at (0, 1) {$\dots$};
    \path (A1_0) edge [-] node [auto] {$\scriptstyle{}$} (A1_1);
    \path (A1_1) edge [-] node [auto] {$\scriptstyle{}$} (A1_2);
  \end{tikzpicture}
  \]
However, the linear graph $\Gamma''$ cannot have vanishing determinant by Remark~\ref{rem}(\ref{dd}), and this contradiction proves the lemma.
\end{itemize}
\end{enumerate}
\end{proof}
\end{subsection}

\begin{subsection}{Contributions to $I(\Gamma)$}
As we have mentioned before, the strategy of the proof of Proposition~\ref{p:mainA} is to compute the quantity $I(\Gamma) $ from the embedding. In the next lemma we show that each internal vertex of weight 2 or 3 hit by a vector contributing negatively to $I(\Gamma)$ is also hit by a vector with a positive contribution. A key result we will be using repeatedly is Lemma~\ref{l:3final1}, which we want to stress, applies also to the graphs $\Gamma_{r}$.

\begin{lemma}\label{l:3and2comp}
Suppose that $\Gamma$ embeds and that there is a basis vector $e$ which appears only twice in the embedding with coefficient $\pm1$. Moreover, assume $e$ hits an internal vertex of weight 3, $w=e+e_{i}+e_{j}$, or an internal vertex of weight 2, $w=e+e_{i}$. Then, the contribution to $I(\Gamma)$ of the set $\{e_{i},e_{j}\}$ is at least 1.
\end{lemma}
\begin{proof}
We start with the case $w=e+e_{i}+e_{j}$. Notice that by Lemma~\ref{l:3no2} and Lemma~\ref{l:no1} we already know that the contribution of $\{e_{i},e_{j}\}$ to $I(\Gamma)$ is at least 0.

Call $w_{-}$ and $w_{+}$ the two vertices adjacent to $w$. If $e_{i}$ hits $w_{-}$ with a coefficient greater than $1$ in absolute value, then it contributes to $I(\Gamma)$ with at least $2$. 
If $e_{i}$ hits $w_{-}$ with coefficient $-1$ then, either $|e_{j}\cdot w_{-}|=2$, which implies that the contribution of $e_{j}$ to $I(\Gamma)$ is at least 2, or $e_{j}\cdot w_{-}=e\cdot w_{-}=1$. 
We will show that the latter case leads to a contradiction: since the vector $e$ hits precisely $w$ and $w_{-}$ in the embedding of $\Gamma$, we can consider the graph $\Gamma'$ obtained from $\Gamma$ by deleting the vector $e$. This graph $\Gamma'$ should have vanishing determinant, but it consists of two connected components to which Lemma~\ref{l:det} and Remark~\ref{rem}(\ref{dd}) apply and we deduce $\Gamma'$ has non-zero determinant.

Assume now that $e_{i}$ hits $w$, $w_{-}$ and $w_{+}$ always with coefficient $1$ and it hits no other vertex in $\Gamma$. If $e_{j}$ also hits $w_{-}$ or $w_{+}$, we fall back into previously analyzed scenarios.
If $e_{j}\cdot w_{-}=e_{j}\cdot w_{+}=0$, then $e_{j}$ appears exactly two times in the embedding of $\Gamma$ with coefficient $\pm1$: it hits $w$ and the other vertex in the graph which is hit by $e$, which we will call $v$. This contradicts Lemma~\ref{l:3no2} and we conclude that $e_{i}$ needs to hit at least another vertex  $u\in\Gamma\setminus\{w_{-},w,w_{+}\}$ and thus it contributes to $I(\Gamma)$ at least with $1$.

Summing up, we have shown so far that we can find a basis vector with a positive contribution to $I(\Gamma)$ except in the case yet to be studied where $w_{-}=e_{i}+...$ and $w_{+}=e_{j}+...$, that is $e_{i}$ and $e_{j}$ each hit with coefficient $+1$ only one of the adjacent vertices to $w$. We proceed to analyze further this configuration. The vector $e$ hits a vertex $u\in\Gamma$ such that $u\cdot w=0$ (if $u\cdot w\neq 0$ we are in a previously analyzed configuration), and without loss of generality we can assume that $u=\pm e \mp \alpha e_{i}+\dots$. If $|\alpha|\geq 2$, then $e_{i}$ contributes at least with 3 to $I(\Gamma)$ and we are done. We assume then that $u=\pm e \mp e_{i}+\dots$. Now, by Lemma~\ref{l:3no2} the vector $e_{j}$ needs to hit some other vertex in $\Gamma$ besides $w$ and $w_{+}$
: If it hits $w_{-}$ we fall back in one of the previously analyzed cases and it cannot hit $u$ since $u=\pm e \mp e_{i}+\dots$, $w=e+e_{i}+e_{j}$ and $u\cdot w=0$. The only possibility left is then the existence of a different vertex $v\in\Gamma\setminus\{w,w_{-},w_{+},u\}$ such that $v\cdot e_{j}\neq 0$. From this configuration we deduce that $e_{i}\cdot v\neq 0$, since $v\cdot w=0$ and we conclude that $e_{i}$ appears at least four times in the embedding, thus contributing to $I(\Gamma)$ at least with $1$. 

Consider now the case in which the internal vector $w$ has weight 2 with embedding $w=e+e_{i}$ and $e$ appears only twice in the embedding of $\Gamma$ with coefficients $\pm 1$. It is immediate to check that then the basis vector $e_{i}$ contributes at least with 1 to $I(\Gamma)$: since $w$ is internal, $e_{i}$ hits its two adjacent vertices (if this were not the case we arrive to a contradiction by deleting $e$ from the embedding) and since $e$ appears at least on another vertex $u$ with $u\cdot w=0$ it follows that $e_{i}\cdot u\neq 0$, so $e_{i}$ appears at least 4 times in $\Gamma$, contributing thus at least with 1 to $I(\Gamma)$.
\end{proof}

We would like to be able to deduce from Lemma~\ref{l:3and2comp} that the total contribution of the vectors hitting internal vertices of weight 2 or 3 is at least zero. However, this is not yet immediate, since we cannot exclude the possibility of having, for example, two vertices of weight 3, say $w=e+e_{i}+e_{j}$ and $w'=e'+e_{i}-e_{j}$, where $e$ and $e'$ contribute negatively to $I(\Gamma)$. In this case Lemma~\ref{l:3and2comp} only guarantees that the set $\{e,e',e_{i},e_{j}\}$ contributes with at least $-1$ to $I(\Gamma)$. The next two lemmas show that we need not worry about more than two vertices hit by a vector appearing only twice and sharing some other basis vector.

\begin{lemma}\label{l:3ad}
Suppose that $\Gamma$ embeds and that there are three different basis vectors $e,e'$, and $e''$ which appear each exactly twice in the embedding with coefficient $\pm1$. Moreover, assume that these vectors hit three different vertices $w, w'$, and $w''$ of weight 3. Then, there is no basis vector $f$ which hits all three vertices. 
\end{lemma}
\begin{proof}
Suppose there were a basis vector $f$ which hits the three vertices with embeddings $w=e+f+e_{i}$, $w'=e'+f+e_{j}$ and $w''=e''+f+e_{k}$ up to signs. 
\begin{enumerate}
\item If every two of these three vertices were orthogonal, then $\{e_{i},e_{j},e_{k}\}\cap\{e,e',e''\}\neq\emptyset$ and we arrive to a contradiction because of Lemma~\ref{l:3no2}.
%
\item If we have $w\cdot w'=1$ and $w\cdot w''=w'\cdot w''=0$, then we need $w=e+f+e_{i}$ and $w'=e'+f+e_{j}$ with $e_{i}\neq e_{j}$. The embedding of the vertex $w''$ is then impossible, again because of Lemma~\ref{l:3no2}.
\item If the three vertices verify $w\cdot w'=1$ and $w'\cdot w''=1$, then the local configuration needs to be the following:
\[
  \begin{tikzpicture}[xscale=2,yscale=-0.5]
    \node (A2_0) at (0, 3.5) {$w'$};
    \node at (1, 3.5) {$w''$};
    \node at (0, 2) {$\scriptstyle e'+f+e_{j}$};
    \node at (1, 2) {$\scriptstyle e''+f-e_{i}$};
    \node (A3_0) at (0, 3) {$\bullet$};
    \node (A3_3) at (1, 3) {$\bullet$};
    \node (A) at (-1, 3) {$\bullet$};
    \node at (-1, 3.5) {$w$};
    \node at (-1, 2) {$\scriptstyle e+f+e_{i}$};
    \path (A) edge [-] node [auto] {$\scriptstyle{}$} (A3_0);
    \path (A3_0) edge [-] node [auto] {$\scriptstyle{}$} (A3_3);
 
  \end{tikzpicture}
\]
where $e_{i}$ and $e_{j}$ are different from each other. This configuration leads to a contradiction: there is some other vertex $u$ in the graph hit by the basis vector $e$ and
\begin{itemize}
\item If  $u\cdot w=u\cdot w''=0$, then the embedding of $u$ is not possible because of Lemma~\ref{l:disjoint}.
\item If $u\cdot w=1$ and $e\cdot u=1$, then we can delete from $\Gamma$ the basis vector $e$, obtaining a two connected component graph $\Gamma'$ whose determinant should vanish. However, $\Gamma'$ has at most one vertex of weight 1 which is final and thus by Lemma~\ref{l:3final1} or Remark~\ref{rem}(\ref{ddN}) its determinant is non-vanishing.
\item If $u\cdot w=1$ and $e\cdot u=-1$, and by the symmetry of the roles played by $w,w''$ and $e,e''$, the following configuration is forced: 
\[
  \begin{tikzpicture}[xscale=2.2,yscale=-0.5]
    \node (A2_0) at (0, 4) {$w$};
    \node at (0, 2) {$\scriptstyle e+f+e_{i}$};
    \node at (1, 2) {$\scriptstyle e'+f+e_{j}$};
    \node at (1, 4) {$w'$};
    \node at (2, 4) {$w''$};
    \node at (3, 4) {$u''$};
    \node at (2, 2) {$\scriptstyle e''+f-e_{i}$};
    \node at (3, 2) {$\scriptstyle -e''+f-e_{i}-e_{j}+\dots$};
    \node (A3_0) at (0, 3) {$\bullet$};
    \node (A3_1) at (1, 3) {$\bullet$};
    \node (A3_2) at (2, 3) {$\bullet$};
    \node (A3_3) at (3, 3) {$\bullet$};
    \node (A) at (-1, 3) {$\bullet$};
    \node at (-1, 4) {$u$};
    \node at (-1, 2) {$\scriptstyle -e+f+e_{i}-e_{j}+\dots$};
    \path (A3_0) edge [-] node [auto] {$\scriptstyle{}$} (A3_1);
    \path (A3_2) edge [-] node [auto] {$\scriptstyle{}$} (A3_3);
    \path (A3_1) edge [-] node [auto] {$\scriptstyle{}$} (A3_2);
    \path (A) edge [-] node [auto] {$\scriptstyle{}$} (A3_0);
  \end{tikzpicture}
\]
Now, deleting the basis vector $e'$ and the vertex $w'$ from the embedding yields a new graph $\Gamma'$ with two connected components and one of them is linear. Without loss of generality we assume that the linear component is the one containing the vertex $w''$. We now proceed to delete the vector $e''$ from $\Gamma'$. This yields a graph with vanishing determinant, but the connected component with the trivalent vertex has non vanishing determinant by Lemma~\ref{l:det} and the other, even if it has a double edge between the vertices $w''$ and $u''$, we know its determinant is non vanishing by Remark~\ref{rem}(\ref{ddN}). Here we use the fact that the other vertex $u'$ containing $e'$ has weight at least 2 after deleting $e'$. To see this, observe that $u' \cdot w'=0$ and $u'$ cannot contain $f$ (since $u'$ is also orthogonal to $w, w''$), so $u'=\pm e' \mp e_j + \dots$. Since $u' \neq u$ or $u''$, and $u'$ must be orthogonal to at least one of them, $u'$ must contain at least one more basis vector.
\item If $u\cdot w''=1$, then by symmetry of the roles played by $w$ and $w''$, the only vertex $u''\neq w''$ in the graph hit by $e''$ must satisfy $u''\cdot w=1$. The vertex $u'\neq w'$ hit by $e'$ satisfies $u'\cdot f=0$, since $u'\cdot w=u'\cdot w''=0$. The same argument implies that $f$ can only hit the vertices $\{u'',w,w',w'',u\}$ in $\Gamma$. We then have that the embedding of $u'$ is of the form $u'=e'-e_{j}+\dots$. Now, by Lemma~\ref{l:disjoint}, the basis vector $e_{j}$ must hit a vertex different from $w'$ and $u'$, and since any vertex different from $u'$ hit by $e_{j}$ must be hit by $f$, then $e_{j}$ hits at least one between $u$ and $u''$. Without loss of generality we may assume it hits $u$. 

Since $u\cdot w''=1$, $u\cdot w=0$ and $f\cdot u\neq 0$, we have $e_{i}\cdot u=0$. Again by Lemma~\ref{l:3no2} we know that the vector $e_{i}$ hits some vertex different from $w$ and $w''$, and since this vertex cannot be orthogonal to $w$ and $w''$ it needs to be $u''$. Moreover, we have $f\cdot u''=0$, since $e_{i}\cdot u''\neq 0$ and $u''\cdot w''=0$. We are left with the following local configuration: 
\[
  \begin{tikzpicture}[xscale=2,yscale=-0.5]
    \node (A2_0) at (0, 4) {$w$};
    \node at (0, 2) {$\scriptstyle e+f+e_{i}$};
    \node at (1, 2) {$\scriptstyle e'+f+e_{j}$};
    \node at (1, 4) {$w'$};
    \node at (2, 4) {$w''$};
    \node at (3, 4) {$u$};
    \node at (4.5, 4) {$u'$};
    \node at (2, 2) {$\scriptstyle e''+f-e_{i}$};
    \node at (3, 2) {$\scriptstyle -e+f-e_{j}+\dots$};
    \node at (4.5, 2) {$\scriptstyle e'-e_{j}+\dots$};
    \node (A3_0) at (0, 3) {$\bullet$};
    \node (A3_1) at (1, 3) {$\bullet$};
    \node (A3_2) at (2, 3) {$\bullet$};
    \node (A3_3) at (3, 3) {$\bullet$};
    \node (A3_5) at (4.5, 3) {$\bullet$};
    \node (A) at (-1, 3) {$\bullet$};
    \node at (-1, 4) {$u''$};
    \node at (-1, 2) {$\scriptstyle e''+e_{i}+\dots$};
    \path (A3_5) edge [-] node [auto] {$\scriptstyle{}$} (4.2,3);
    \path (A3_5) edge [-] node [auto] {$\scriptstyle{}$} (4.8,3);
    \path (A3_0) edge [-] node [auto] {$\scriptstyle{}$} (A3_1);
    \path (A3_2) edge [-] node [auto] {$\scriptstyle{}$} (A3_3);
    \path (A3_3) edge [-] node [auto] {$\scriptstyle{}$} (3.3,3);
    \path (A3_1) edge [-] node [auto] {$\scriptstyle{}$} (A3_2);
    \path (A) edge [-] node [auto] {$\scriptstyle{}$} (A3_0);
    \path (A) edge [-] node [auto] {$\scriptstyle{}$} (-1.3,3);
  \end{tikzpicture}
\]
and in the graph $\Gamma$ the basis vectors $e,e'',f$ and $e_{i}$ only appear in the depicted vertices. If we delete these 4 basis vectors we are left with a graph $\Gamma'$, which is still a tree and has an embedding into a standard lattice. In $\Gamma'$ we have the local configuration:
\[
  \begin{tikzpicture}[xscale=2,yscale=-0.5]
    \node at (1, 2) {$\scriptstyle e'+e_{j}$};
    \node at (1, 4) {$w'$};
    \node at (3, 4) {$u$};
    \node at (4.5, 4) {$u'$};
    \node at (2, 2.75) {$-$};
    \node at (3, 2) {$\scriptstyle -e_{j}+\dots$};
    \node at (4.5, 2) {$\scriptstyle e'-e_{j}+\dots$};
    \node (A3_1) at (1, 3) {$\bullet$};
    \node (A3_3) at (3, 3) {$\bullet$};
    \node (A3_5) at (4.5, 3) {$\bullet$};
    \node (A) at (-1, 3) {$\bullet$};
    \node at (-1, 4) {$u''$};
    \node at (-1, 2) {$\scriptstyle \dots$};
    \path (A3_5) edge [-] node [auto] {$\scriptstyle{}$} (4.2,3);
    \path (A3_5) edge [-] node [auto] {$\scriptstyle{}$} (4.8,3);
    \path (A3_1) edge [-] node [auto] {$\scriptstyle{}$} (A3_3);
    \path (A3_3) edge [-] node [auto] {$\scriptstyle{}$} (3.3,3);
    \path (A) edge [-] node [auto] {$\scriptstyle{}$} (-1.3,3);
  \end{tikzpicture}
\]
and the rest of $\Gamma'$ and its embedding coincides with that of $\Gamma$. Notice that if in $\Gamma$ we had $|u''|=2$, then $u''$ was necessary final and in this case $\Gamma'$ embeds into a lattice of rank one smaller than its number of vertices. On the other hand, if in $\Gamma$ we had $|u''|>2$ then $\Gamma'$ would embed into a lattice of rank two less than its number of vertices. In both cases the graph $\Gamma'$ must have vanishing determinant. We will finish the proof of the lemma by showing this cannot hold.

If in $\Gamma$ the vertex $u''$ had weight 2, then $\Gamma'$ has only one connected component. If it had weight greater than 2, then its connected component has at most one vertex of weight 1 which is final and thus by Lemma~\ref{l:3final1} or Remark~\ref{rem}(\ref{ddN}) its determinant is non vanishing. We focus then on the connected component of $\Gamma'$ containing the vertex $w'$, which we call $\Gamma'_{0}$. If in $\Gamma$ the vertices $u$ and $u'$ were orthogonal, or if $|u|>3$, then in $\Gamma'_{0}$ there is no vertex of weight 1 and therefore, if it is a linear graph, it has non-vanishing determinant by Remark~\ref{rem}(\ref{dd}); and if it has a trivalent vertex, since it embeds in a standard lattice, its determinant is non-vanishing by Lemma~\ref{l:det}. Finally, if in $\Gamma$ we had  $u\cdot u'=1$ and $|u|=3$, then in $\Gamma'_{0}$ we find:
\[
  \begin{tikzpicture}[xscale=2,yscale=-0.5]
    \node at (1, 2) {$\scriptstyle e'+e_{j}$};
    \node at (1, 4) {$w'$};
    \node at (3, 4) {$u$};
    \node at (4.5, 4) {$u'$};
    \node at (2, 2.75) {$-$};
    \node at (3, 2) {$\scriptstyle -e_{j}$};
    \node at (4.5, 2) {$\scriptstyle e'-e_{j}+\dots$};
    \node (A3_1) at (1, 3) {$\bullet$};
    \node (A3_3) at (3, 3) {$\bullet$};
    \node (A3_5) at (4.5, 3) {$\bullet$};
    \path (A3_5) edge [-] node [auto] {$\scriptstyle{}$} (4.8,3);
    \path (A3_1) edge [-] node [auto] {$\scriptstyle{}$} (A3_3);
    \path (A3_3) edge [-] node [auto] {$\scriptstyle{}$} (A3_5);
  \end{tikzpicture}
\]
If $|u'|\geq 3$ in $\Gamma'_{0}$, then after blowing down the vectors $e_{j}$ and $e'$ we obtain a graph with embedding and at most one final vertex with weight one, and thus, again by Lemma~\ref{l:3final1} or Remark~\ref{rem}(\ref{ddN}) its determinant, which coincides with the determinant of $\Gamma'_{0}$  is not vanishing. There is only one case left to discuss: $|u'|=2$ in $\Gamma'_{0}$, which implies $u'$ final in $\Gamma$. In this case, the original local configuration in $\Gamma$ had to be:
\[
  \begin{tikzpicture}[xscale=2,yscale=-0.5]
    \node (A2_0) at (0, 4) {$w$};
    \node at (0, 2) {$\scriptstyle e+f+e_{i}$};
    \node at (1, 2) {$\scriptstyle e'+f+e_{j}$};
    \node at (1, 4) {$w'$};
    \node at (2, 4) {$w''$};
    \node at (3, 4) {$u$};
    \node at (4, 4) {$u'$};
    \node at (2, 2) {$\scriptstyle e''+f-e_{i}$};
    \node at (3, 2) {$\scriptstyle -e+f-e_{j}$};
    \node at (4, 2) {$\scriptstyle e'-e_{j}$};
    \node (A3_0) at (0, 3) {$\bullet$};
    \node (A3_1) at (1, 3) {$\bullet$};
    \node (A3_2) at (2, 3) {$\bullet$};
    \node (A3_3) at (3, 3) {$\bullet$};
    \node (A3_5) at (4, 3) {$\bullet$};
    \node (A) at (-1, 3) {$\bullet$};
    \node at (-1, 4) {$u''$};
    \node at (-1, 2) {$\scriptstyle e''+e_{i}+\dots$};
    \path (A3_0) edge [-] node [auto] {$\scriptstyle{}$} (A3_1);
    \path (A3_2) edge [-] node [auto] {$\scriptstyle{}$} (A3_3);
    \path (A3_3) edge [-] node [auto] {$\scriptstyle{}$} (A3_5);
    \path (A3_1) edge [-] node [auto] {$\scriptstyle{}$} (A3_2);
    \path (A) edge [-] node [auto] {$\scriptstyle{}$} (A3_0);
    \path (A) edge [-] node [auto] {$\scriptstyle{}$} (-1.3,3);
  \end{tikzpicture}
\]
with $u''$ not final and thus $|u''|>2$. The same arguments presented before show that in this case, if we delete from $\Gamma$ the vectors $e,e',f$ and $e_{j}$ we arrive to a trivalent graph with an embedding into a lattice of rank one less than its number of vertices, but which has non-vanishing determinant. This contradiction finishes the proof.
\end{itemize}
\end{enumerate}
\end{proof}

\begin{lemma}\label{l:3adplus}
Suppose that $\Gamma$ embeds and that there are three different basis vectors $e,e'$, and $e''$ which appear each exactly twice in the embedding with coefficient $\pm1$. Moreover, assume that these vectors hit three different vertices $w, w'$, and $w''$ of weight at most three. Then, there is no basis vector $f$ which hits all three vertices. 
\end{lemma}
\begin{proof}
Since the case $w, w'$, and $w''$ all of weight 3 was dealt with in Lemma~\ref{l:3ad}, we might suppose $|w|=2$.
\begin{enumerate}
\item If the three vertices had weight 2 and they were all hit by a common basis vector $f$, the embeddings would be, up to signs, of the form $w=e+f$, $w'=e'+f$ and $w''=e''+f$. However, at least two of these vertices need to be orthogonal but since $e,e'$ and $e''$ are assumed to be different, this is impossible.
\item If $|w|=|w'|=2$ and $|w''|=3$ we are dealing with an embedding, up to signs, of the form $w=e+f$, $w'=e'+f$ and $w''=e''+f+g$. By Lemma~\ref{l:3no2} we know that $g\neq e,e'$ and just like in the preceding case we conclude that this configuration is impossible since $w''$ is orthogonal to at least one between $w$ and $w'$.
\item If $|w|=2$ and $|w''|=|w''|=3$ the embedding, up to signs, is of the form $w=e+f$, $w'=e'+f+g$ and $w''=e''+f+h$. To avoid the orthogonality issues of the two preceding cases we need to have $w$ linked to both $w'$ and $w''$ and $h=-g$. We are dealing with the following configuration: 
\[
  \begin{tikzpicture}[xscale=2,yscale=-0.5]
    \node (A2_0) at (0, 3.5) {$w$};
    \node at (1, 3.5) {$w''$};
    \node at (0, 2) {$\scriptstyle e+f$};
    \node at (1, 2) {$\scriptstyle e''+f-g$};
    \node (A3_0) at (0, 3) {$\bullet$};
    \node (A3_3) at (1, 3) {$\bullet$};
    \node (A) at (-1, 3) {$\bullet$};
    \node at (-1, 3.5) {$w'$};
    \node at (-1, 2) {$\scriptstyle e'+f+g$};
    \path (A) edge [-] node [auto] {$\scriptstyle{}$} (A3_0);
    \path (A3_0) edge [-] node [auto] {$\scriptstyle{}$} (A3_3);
 
  \end{tikzpicture}
\]
Since $e$ appears a second time in the embedding on a vertex that we call $u$ and $u\cdot w=0$, we have $u\cdot f\neq 0$. We claim that $u$ cannot be orthogonal to both $w'$ and $w''$. Indeed, if this were the case we would need one of the following to happen:
	\begin{itemize}
	\item $e'\cdot u\neq 0$, which implies $g\cdot u\in\{0,\pm2\}$. Since we are assuming $u\cdot w''=0$ we deduce that $e''\cdot u\neq 0$. It follows that $e'$ and $e''$ hit $u$, which contradicts the second statement in Lemma~\ref{l:disjoint}.
	\item $e''\cdot u\neq 0$, which by symmetry can be dealt with as the preceding case, yielding the same contradiction.
	\item $g\cdot u\neq0$ and $e''\cdot u=e'\cdot u=0$. However, since $g$ hits $w'$ and $w''$ with opposite signs, we cannot obtain $u\cdot w'=u\cdot w''=0$.
	\end{itemize}
Therefore, as claimed, $u$ is not orthogonal to both $w'$ and $w''$ and by symmetry we might assume $u\cdot w''\neq 0$. We study the following two possibilities:
	\begin{itemize}
	\item If $u=-f+e+\dots$, then either $u\cdot g=-2$ or $u\cdot e''=1$ and $u\cdot g=-1$. Since $w'\cdot u=0$ neither of these two cases can happen (recall that the coefficient of $e'$ is $\pm1$ and that $e''$ and $e'$ cannot hit the same vertex by Lemma~\ref{l:disjoint}).
	\item $u=f-e+\dots$, then either $e'\cdot u=0$ or $e'\cdot u\neq0$. The former case implies, since $w'\cdot u=0$ that $u\cdot g=-1$, from where it follows $u\cdot e''=-1$. Now, there is exactly one vertex in the graph different from $w'$, call it $u'$, such that $u'\cdot e'\neq 0$. Notice that then necessarily $f\cdot u'=g\cdot u'=0$. We are then dealing with the following situation:
\[
  \begin{tikzpicture}[xscale=2.3,yscale=-0.5]
    \node (A2_0) at (0, 4) {$w'$};
    \node at (0, 2) {$\scriptstyle e'+f+g$};
    \node at (1, 2) {$\scriptstyle e+f$};
    \node (A2_2) at (1, 4) {$w$};
    \node (A2_2) at (2, 4) {$w''$};
    \node at (3, 4) {$u$};
    \node at (2, 2) {$\scriptstyle e''+f-g$};
    \node at (3, 2) {$\scriptstyle f-e-g-e''+\dots$};
    \node (A3_0) at (0, 3) {$\bullet$};
    \node (A3_1) at (1, 3) {$\bullet$};
    \node (A3_2) at (2, 3) {$\bullet$};
    \node (A3_3) at (3, 3) {$\bullet$};
    \node (A) at (-1, 3) {$\bullet$};
    \node at (-1, 4) {$u'$};
    \node at (-1, 2) {$\scriptstyle e'+\dots$};
    \path (A3_0) edge [-] node [auto] {$\scriptstyle{}$} (A3_1);
    \path (A3_2) edge [-] node [auto] {$\scriptstyle{}$} (A3_3);
    \path (A3_1) edge [-] node [auto] {$\scriptstyle{}$} (A3_2);
    \path (A) edge [-] node [auto] {$\scriptstyle{}$} (A3_0);
  \end{tikzpicture}
\]
Deleting from the embedding the basis vector $e'$ we get a two component graph with an embedding, whose determinant should vanish, but this contradicts Lemma~\ref{l:3final1} or Remark~\ref{rem}(\ref{dd}).

There is thus only one possibility left to analyze, when $e'\cdot u\neq0$. In this case the graph we are considering is as follows:
\[
  \begin{tikzpicture}[xscale=2.3,yscale=-0.5]
    \node (A2_0) at (0, 4) {$w'$};
    \node at (0, 2) {$\scriptstyle e'+f+g$};
    \node at (1, 2) {$\scriptstyle e+f$};
    \node (A2_2) at (1, 4) {$w$};
    \node (A2_2) at (2, 4) {$w''$};
    \node at (3, 4) {$u$};
    \node at (2, 2) {$\scriptstyle e''+f-g$};
    \node at (3, 2) {$\scriptstyle f-e-e'+\dots$};
    \node (A3_0) at (0, 3) {$\bullet$};
    \node (A3_1) at (1, 3) {$\bullet$};
    \node (A3_2) at (2, 3) {$\bullet$};
    \node (A3_3) at (3, 3) {$\bullet$};
    \node (A) at (-1, 3) {$\bullet$};
    \node at (-1, 4) {$u''$};
    \node at (-1, 2) {$\scriptstyle e''+g+\dots$};
    \path (A3_0) edge [-] node [auto] {$\scriptstyle{}$} (A3_1);
    \path (A3_2) edge [-] node [auto] {$\scriptstyle{}$} (A3_3);
    \path (A3_1) edge [-] node [auto] {$\scriptstyle{}$} (A3_2);
    \path (A) edge [-] node [auto] {$\scriptstyle{}$} (A3_0);
  \end{tikzpicture}
\]
since in this scenario $e''$ cannot hit $u$ (Lemma~\ref{l:disjoint}) and the vertex different from $w''$ it hits cannot be orthogonal to $w'$. Now, either $|u''|>2$ or $u''$ is final, since $g$ cannot hit any vertex in the graph outside $\{u'',w',w''\}$; similarly either $|u|>3$ or $u$ is final. Moreover, notice that the basis vectors in $\{e,e',e'',f,g\}$ can only hit the depicted vertices. If we delete from the graph the basis vectors $\{e,e',e'',f,g\}$ we obtain a graph with an embedding into a lattice of rank smaller than its number of vertices, and it therefore should have vanishing determinant. However, this contradicts  Lemma~\ref{l:3final1} or Remark~\ref{rem}(\ref{dd}).
	\end{itemize}

\end{enumerate}

\end{proof}

Now that we have shown in the last two lemmas that at most two internal vertices of weight at most 3 hit by vectors contributing negatively to $I(\Gamma)$ share other vectors, we analyze thoroughly this possibility in the next two lemmas.

\begin{lemma}\label{l:cases}
Suppose that $\Gamma$ embeds and that there are two distinct basis vectors $e$ and $e'$ which appear each exactly twice in the embedding with coefficient $\pm1$. Moreover, assume that these vectors hit two internal vertices of weight 3 with embeddings $w=e+e_{i}+e_{j}$ and $w'=e'\pm e_{i}+e'_{j}$. Then, at least one of the following holds.
\begin{itemize}
\item[-] $e_{i}$ contributes to $I(\Gamma)$ with at least 2.
\item[-] the set $\{e_{j},e'_{j}\}$, if $e'_{j}\neq e_{j}$, or the set $\{e_{i},e_{j}\}$, if $e'_{j}=e_{j}$, contributes to $I(\Gamma)$ with at least 2.
\end{itemize}
\end{lemma}
\begin{proof}
If any one of $e_{i},e_{j}$ and $e'_{j}$ appears somewhere in the embedding of $\Gamma$ with a coefficient different from $\pm1$ then we are done, since it would contribute to $I(\Gamma)$ at least with 2. So from now on we assume that these basis vectors appear always with coefficient $\pm1$.

If $w\cdot w'=0$, then we necessarily have $e'_{j}=\mp e_{j}$. Let us call $u$ the only vertex in $\Gamma$ besides $w$ hit by the vector $e$. If we assume $u\cdot w=u\cdot w'=0$ we immediately arrive to a contradiction: $u\cdot w=0$ requires $u$ to be hit by precisely one between $e_{i}$ and $e_{j}$, but $u\cdot w'=0$ requires $u$ to be hit by both these vectors because of Lemma~\ref{l:disjoint}. We analyze the following two cases separately:
\begin{enumerate}
\item If $u\cdot w=1$ then, if $e\cdot u=1$, by deleting the vector $e$ from the embedding we arrive to the usual contradiction with Lemma~\ref{l:det} and Remark~\ref{rem}(\ref{dd}) (note that in this scenario $u$ cannot hit $e_i$ or $e_j$ since if $u=e \pm e_i \mp e_j + \dots$, then $|u \cdot w'|=2$, impossible). So we assume $e\cdot u=-1$, which forces the embedding of $u$ to be of the form $u=-e+e_{i}+e_{j}+\dots$. Now, $e'$ also reappears in the embedding on a vertex which we call $u'$ and by symmetry of this problem $u'\cdot w=u'\cdot w'=0$ leads to a contradiction. 

If $u'\cdot w'=1$ then $u'\cdot w=0$ since $u\neq u'$ 
by Lemma~\ref{l:disjoint} and by the same arguments we have used to study the embedding of $u$ we conclude that the only possible configuration is:
\[
  \begin{tikzpicture}[xscale=2,yscale=-0.5]
    \node (A2_0) at (0, 4) {$w$};
    \node at (0, 2) {$\scriptstyle e+e_{i}+e_{j}$};
    \node (A2_2) at (1, 4) {$w_{+}$};
    \node  at (3, 4) {$w'$};
    \node  at (4, 4) {$u'$};
    \node at (2, 4) {$w'_{-}$};
    \node (A2_3) at (3, 2) {$\scriptstyle e'\pm e_{i} \mp e_{j}$};
    \node at (4, 2) {$\scriptstyle -e'\pm e_{i} \mp e_{j}+\dots$};
    \node (A3_0) at (0, 3) {$\bullet$};
    \node (A3_1) at (1, 3) {$\bullet$};
    \node at (1.5, 3) {$\cdots$};
    \node (A3_2) at (2, 3) {$\bullet$};
    \node (A3_3) at (3, 3) {$\bullet$};
    \node (A) at (-1, 3) {$\bullet$};
    \node (B) at (4, 3) {$\bullet$};
    \node at (-1, 4) {$u$};
    \node at (-1, 2) {$\scriptstyle -e+e_{i}+e_{j}+\dots$};
    \path (A3_0) edge [-] node [auto] {$\scriptstyle{}$} (A3_1);
    \path (A3_2) edge [-] node [auto] {$\scriptstyle{}$} (A3_3);
    \path (A) edge [-] node [auto] {$\scriptstyle{}$} (A3_0);
    \path (A3_3) edge [-] node [auto] {$\scriptstyle{}$} (B);
  \end{tikzpicture}
\]
Now, since the vertices $w$ and $w'$ are internal, and even in the case $w_{+}=w'_{-}$, at least one between $e_{i}$ and $e_{j}$ reappears in the embedding besides in the depicted vertices $u,w,w',u'$ and in this case the result is proved.

If $u'\cdot w=1$, then it is again immediate to check that $u'\cdot w'=0$  also holds and we are dealing with the local configuration:
\[
  \begin{tikzpicture}[xscale=2,yscale=-0.5]
    \node (A2_0) at (0, 4) {$w$};
    \node at (0, 2) {$\scriptstyle e+e_{i}+e_{j}$};
    \node at (1, 2) {$\scriptstyle \pm e'+\cdots$};
    \node (A2_2) at (1, 4) {$u'$};
    \node  at (3, 4) {$w'$};
    \node at (4, 4) {$w'_{+}$};
    \node at (2, 4) {$w'_{-}$};
    \node at (3, 2) {$\scriptstyle e'\pm e_{i} \mp e_{j}$};
    \node (A3_0) at (0, 3) {$\bullet$};
    \node (A3_1) at (1, 3) {$\bullet$};
    \node (A3_2) at (3, 3) {$\bullet$};
    \node (A3_3) at (4, 3) {$\bullet$};
     \node (B) at (2, 3) {$\bullet$};
    \node (A) at (-1, 3) {$\bullet$};
    \node at (-1, 4) {$u$};
    \node at (-1, 2) {$\scriptstyle -e+e_{i}+e_{j}+\dots$};
    \path (A3_0) edge [-] node [auto] {$\scriptstyle{}$} (A3_1);
    \path (A3_2) edge [-] node [auto] {$\scriptstyle{}$} (A3_3);
    \path (B) edge [-] node [auto] {$\scriptstyle{}$} (A3_2);
    \path (A) edge [-] node [auto] {$\scriptstyle{}$} (A3_0);

  \end{tikzpicture}
\]
where the vertex $u'$ embeds either as $u'=-e'+e_{i}+\dots$ or $u'=e'+e_{j}+\dots$. Since both vertices $w'_{-}$ and $w'_{+}$ need to be hit at least by one between $e_{i}$ and $e_{j}$ in this case the result is proved.
\item If $u\cdot w=0$ then $u\cdot w'=1$. If $u'\cdot w'=1$, then as before we have $u'\cdot w=0$ and we are in a case completely analogous to the last one studied. We therefore assume $u'\cdot w'=0$, which implies $u'\cdot w=1$ and the following configuration is then the only possibility left:
\[
  \begin{tikzpicture}[xscale=2,yscale=-0.5]
    \node (A2_0) at (0, 4) {$w$};
    \node at (0, 2) {$\scriptstyle e+e_{i}+e_{j}$};
    \node (A2_2) at (1, 4) {$w_{+}$};
    \node  at (3, 4) {$w'$};
    \node  at (4, 4) {$u$};
    \node at (2, 4) {$w'_{-}$};
    \node (A2_3) at (3, 2) {$\scriptstyle e'\pm e_{i} \mp e_{j}$};
    \node at (4, 2) {$\scriptstyle \pm e+\dots$};
    \node (A3_0) at (0, 3) {$\bullet$};
    \node (A3_1) at (1, 3) {$\bullet$};
    \node at (1.5, 3) {$\cdots$};
    \node (A3_2) at (2, 3) {$\bullet$};
    \node (A3_3) at (3, 3) {$\bullet$};
    \node (A) at (-1, 3) {$\bullet$};
    \node (B) at (4, 3) {$\bullet$};
    \node at (-1, 4) {$u'$};
    \node at (-1, 2) {$\scriptstyle \pm e'+\dots$};
    \path (A3_0) edge [-] node [auto] {$\scriptstyle{}$} (A3_1);
    \path (A3_2) edge [-] node [auto] {$\scriptstyle{}$} (A3_3);
    \path (A) edge [-] node [auto] {$\scriptstyle{}$} (A3_0);
    \path (A3_3) edge [-] node [auto] {$\scriptstyle{}$} (B);
  \end{tikzpicture}
\]
Now, if $w_{+}\neq w'_{-}$, 
then we are done since the total contribution to $I(\Gamma)$ of $\{e_{i},e_{j}\}$ would be at least 2. Moreover, if $x:=w_{+}=w'_{-}$, then exactly one of $e_i$ and $e_j$ hit $x$ depending on the signs of $e_i$ and $e_j$ in $w'$. Without loss of generality let us assume $e_i$ hits $x$. If in addition $e_{i}$ hits both $u$ and $u'$ we are done too, since $e_{i}\cdot x=1$.  We are thus left with the following configuration to further analyze:
\[
  \begin{tikzpicture}[xscale=2,yscale=-0.5]
    \node (A2_0) at (0, 4) {$w$};
    \node at (0, 2) {$\scriptstyle e+e_{i}+e_{j}$};
    \node at (1, 2) {$\scriptstyle e_{i}+\dots$};
    \node (A2_2) at (2, 4) {$w'$};
    \node at (1, 4) {$x$};
    \node at (3, 4) {$u$};
    \node at (2, 2) {$\scriptstyle e'+e_{i}-e_{j}$};
    \node (A2_3) at (3, 2) {$\scriptstyle e-e_{j}+\dots$};
    \node (A3_0) at (0, 3) {$\bullet$};
    \node (A3_1) at (1, 3) {$\bullet$};
    \node (A3_2) at (2, 3) {$\bullet$};
    \node (A3_3) at (3, 3) {$\bullet$};
    \node (A) at (-1, 3) {$\bullet$};
    \node (B) at (4, 3) {$\bullet$};
    \node at (4, 4) {$u_{+}$};
    \node at (-1, 4) {$u'$};
    \node at (-1, 2) {$\scriptstyle \pm e'+\cdots$};
    \path (A3_0) edge [-] node [auto] {$\scriptstyle{}$} (A3_1);
    \path (A3_2) edge [-] node [auto] {$\scriptstyle{}$} (A3_3);
    \path (A3_1) edge [-] node [auto] {$\scriptstyle{}$} (A3_2);
    \path (A) edge [-] node [auto] {$\scriptstyle{}$} (A3_0);
    \path (A3_3) edge [-] node [auto] {$\scriptstyle{}$} (B);
  \end{tikzpicture}
\]
where $u'=-e'+e_{i}+ \dots$ or $u'=e'+e_{j}+\dots$ and the vertex $u_{+}$ may or may not exist. We will be done if we are able to show that either $e_{i}$ or $e_{j}$ appears at least once more in the embedding of $\Gamma$. Suppose that this were not the case and that 
\begin{itemize}
\item $u'=-e'+e_{i}+ \dots$. In this case we have that $|u'|\geq 3$, since $u'\cdot x=0$. Moreover, either $u$ is final or $|u|\geq 3$. If we are in the latter case, after deleting from $\Gamma$ the basis vectors $e,e',e_{i}$ and $e_{j}$ and the vertices $x,w$ and $w'$ we obtain a graph $\Gamma'$ with two connected components and an embedding into a lattice of rank smaller than its number of vertices. It follows that det$(\Gamma')=0$. However, this contradicts Lemma~\ref{l:3final1} and Remark~\ref{rem}(\ref{dd}). 

If we are in the case where $u$ is a final vertex, the local configuration we are looking at is the following:
\[
  \begin{tikzpicture}[xscale=2,yscale=-0.5]
    \node (A2_0) at (0, 4) {$w$};
    \node at (0, 2) {$\scriptstyle e+e_{i}+e_{j}$};
    \node at (1, 2) {$\scriptstyle e_{i}+f+\tilde x$};
    \node (A2_2) at (2, 4) {$w'$};
    \node at (1, 4) {$x$};
    \node at (3, 4) {$u$};
    \node at (2, 2) {$\scriptstyle e'+e_{i}-e_{j}$};
    \node (A2_3) at (3, 2) {$\scriptstyle e-e_{j}$};
    \node (A3_0) at (0, 3) {$\bullet$};
    \node (A3_1) at (1, 3) {$\bullet$};
    \node (A3_2) at (2, 3) {$\bullet$};
    \node (A3_3) at (3, 3) {$\bullet$};
    \node (A) at (-1, 3) {$\bullet$};
    \node (B) at (-2, 3) {$\dots$};
    \node at (-1, 4) {$u'$};
    \node at (-1, 2) {$\scriptstyle -e'+e_{i}-f+\tilde u$};
    \path (A) edge [-] node [auto] {$\scriptstyle{}$} (B);
    \path (A3_0) edge [-] node [auto] {$\scriptstyle{}$} (A3_1);
    \path (A3_2) edge [-] node [auto] {$\scriptstyle{}$} (A3_3);
    \path (A3_1) edge [-] node [auto] {$\scriptstyle{}$} (A3_2);
    \path (A) edge [-] node [auto] {$\scriptstyle{}$} (A3_0);
  \end{tikzpicture}
\]
where $|\tilde u|,|\tilde x|\geq 0$. We delete again from the embedding the vectors $e,e',e_{i}$ and $e_{j}$ and the vertices $w,w'$ and $u$. This yields a connected graph with vanishing determinant and local configuration
\[
  \begin{tikzpicture}[xscale=2,yscale=-0.5]
    \node (A2_0) at (0, 4) {$x$};
    \node at (-0.5, 2.75) {$\scriptstyle -$};
    \node at (0, 2) {$\scriptstyle f+\tilde x$};
    \node (A3_0) at (0, 3) {$\bullet$};
    \node (A3_3) at (-2, 3) {$\dots$};
    \node (A) at (-1, 3) {$\bullet$};
    \node at (-1, 4) {$u'$};
    \node at (-1, 2) {$\scriptstyle -f+\tilde u$};
    \path (A) edge [-] node [auto] {$\scriptstyle{}$} (A3_0);
    \path (A) edge [-] node [auto] {$\scriptstyle{}$} (A3_3);
  \end{tikzpicture}
\]
which was shown in the proof of Lemma~\ref{l:3no2} second scenario, second bullet, to lead to a contradiction.

\item $u'=e'+e_{j}+\dots$. In this case, if we delete the basis vector $e_{i}$ which appears exactly 3 times in the embedding, we obtain a graph $\Gamma'$ with vanishing determinant, a connected component consisting of a single vertex, the vertex $x$, and another connected component in which every vertex has weight at least 2. There is one negative edge in $\Gamma'$, but this does not affect the arguments (Operation~\ref{op}(\ref{bd})) and $\Gamma'$ should have non-vanishing determinant by Lemma~\ref{l:det}, a contradiction.
\end{itemize}
\end{enumerate}

Now that we have studied all possible configurations with $w\cdot w'=0$ and have established the lemma in this case, we  move on to study the case $w\cdot w'=1$. We are dealing this time with the local configuration:
\[
  \begin{tikzpicture}[xscale=2,yscale=-0.5]
    \node (A2_0) at (0, 4) {$w$};
    \node at (0, 2) {$\scriptstyle e+e_{i}+e_{j}$};
    \node at (1, 2) {$\scriptstyle e'+e_{i}+e'_{j}$};
    \node (A2_2) at (2, 4) {$w_{+}$};
    \node at (1, 4) {$w'$};
    \node (A3_0) at (0, 3) {$\bullet$};
    \node (A3_1) at (1, 3) {$\bullet$};
    \node (A3_2) at (2, 3) {$\bullet$};
    \node (A) at (-1, 3) {$\bullet$};
    \node at (-1, 4) {$w_{-}$};
    \path (A3_0) edge [-] node [auto] {$\scriptstyle{}$} (A3_1);
    \path (A3_1) edge [-] node [auto] {$\scriptstyle{}$} (A3_2);
    \path (A) edge [-] node [auto] {$\scriptstyle{}$} (A3_0);
 \end{tikzpicture}
\]
where $e_{j}\neq\pm e'_{j}$. We remind the reader that by Lemma~\ref{l:3no2} the vectors $e_{i},e_{j}$ and $e_{j}'$ appear each at least 3 times in the embedding of $\Gamma$. We claim that $e'\cdot w_{+}\neq 1$ (and by symmetry $e\cdot w_{-}\neq 1$). Indeed: if $e'\cdot w_{+}=1$ by deleting $e'$ from the embedding we arrive to a contradiction with Lemma~\ref{l:3final1} and Remark~\ref{rem}(\ref{dd}). The possible embeddings of $w_{+}$ are then: 
	\begin{itemize}
		\item[$-$] $w_{+}=-e'+e_{i}+e_{j}'-e_{j}+\dots$,
		\item[$-$] $w_{+}=e_{i}-e_{j}+\dots$,
		\item[$-$] $w_{+}=e_{j}'+\dots$,
		\item[$-$] $w_{+}=e_{i}-e+\dots$, or
		\item[$-$] $w_{+}=e_{j}'\pm e_{j}\mp e+\dots$.		
	\end{itemize}
Notice that these possibilities are mutually exclusive: with this notation we mean that the only vertices from the set $A:=\{e,e',e_{i},e_{j},e_{j}'\}$ hitting $w_{+}$ are the ones evidenced.  We start by considering the two last cases, characterized by $e\cdot w_{+}\neq 0$. Notice that the five vectors in $A$ need to hit at least five vertices in $\Gamma$ (if this were not the case, by deleting at most 4 vertices we would have an embedding of a graph with non vanishing determinant into a smaller rank lattice). We show first that if the vectors in $A$ hit at least 6 vertices, then the result holds. In this case there is a vertex $v$, orthogonal to $w$ and $w'$, hit by a vector in $A$ and such that $e'\cdot v=0$. This immediately implies that $v$ is hit by $e_{i},e_{j}, e'_{j}$. There are at least two more vertices hit by the vectors in $A$: $w_{-}$ and some other vertex $v'$. Precisely one of these two is hit by $e'$. If it is $w_{-}$ we are done: by the same argument as before, $v'$ is hit by $e_{i},e_{j}, e'_{j}$ and looking at the vertices $v,v',w,w'$ and $w_{+}$ we have either $e_{i}$ appearing 5 times and $e_{j},e_{j}'$ three times each or $e_{i}$ appearing 4 times and $e_{j},e_{j}'$ four times each. If $w_{-}\cdot e'=0$, we are dealing with:
\[
  \begin{tikzpicture}[xscale=2.7,yscale=-0.75]
    \node (A2_0) at (0, 4) {$w$};
    \node at (0, 2) {$\scriptstyle e+e_{i}+e_{j}$};
    \node at (1, 2) {$\scriptstyle e'+e_{i}+e'_{j}$};
    \node (A2_2) at (2, 4) {$w_{+}$};
    \node at (1, 4) {$w'$};
    \node at (3, 4) {$v$};
    \node at (2, 2) {$\scriptstyle\pm e\begin{cases}\scriptstyle+e_{i}+\dots\\ \scriptstyle\mp e_{j}+e_{j}'+\dots\end{cases}$};
    \node at (3, 2) {$\scriptstyle e_{i}-e_{j}-e_{j}'+\dots$};
    \node at (4, 2) {$\scriptstyle e'\begin{cases}\scriptstyle-e_{i}+e_{j}+\dots\\ \scriptstyle -e_{j}'+\dots\end{cases}$};
    \node (A3_0) at (0, 3) {$\bullet$};
    \node (A3_1) at (1, 3) {$\bullet$};
    \node (A3_2) at (2, 3) {$\bullet$};
    \node (A3_3) at (3, 3) {$\bullet$};
    \node (A) at (-1, 3) {$\bullet$};
    \node (B) at (4, 3) {$\bullet$};
    \node at (4, 4) {$v'$};
    \node at (-1, 4) {$w_{-}$};
    \path (A3_0) edge [-] node [auto] {$\scriptstyle{}$} (A3_1);
    \path (A3_1) edge [-] node [auto] {$\scriptstyle{}$} (A3_2);
    \path (A) edge [-] node [auto] {$\scriptstyle{}$} (A3_0);
  \end{tikzpicture}
\]
where the braces mean that one of the two options must occur. Since $w_{-}$ must be hit either by both $e_{i}$ and $e_{j}'$ or by $e_{j}$, it is immediate to check that all possible embeddings depicted above yield $e_{i}$ or $\{e_{j},e_{j}'\}$ contributing to $I(\Gamma)$ with at least 2 except for the following configuration:
\[
  \begin{tikzpicture}[xscale=2,yscale=-0.5]
    \node (A2_0) at (0, 4) {$w$};
    \node at (0, 2) {$\scriptstyle e+e_{i}+e_{j}$};
    \node at (1, 2) {$\scriptstyle e'+e_{i}+e'_{j}$};
    \node (A2_2) at (2, 4) {$w_{+}$};
    \node at (1, 4) {$w'$};
    \node at (3, 4) {$v$};
    \node at (2, 2) {$\scriptstyle -e+e_{i}+\dots$};
    \node at (3, 2) {$\scriptstyle e_{i}-e_{j}-e_{j}'+\dots$};
    \node at (4, 2) {$\scriptstyle e'-e_{j}'+\dots$};
    \node (A3_0) at (0, 3) {$\bullet$};
    \node (A3_1) at (1, 3) {$\bullet$};
    \node (A3_2) at (2, 3) {$\bullet$};
    \node (A3_3) at (3, 3) {$\bullet$};
    \node (A) at (-1, 3) {$\bullet$};
    \node (B) at (4, 3) {$\bullet$};
    \node at (4, 4) {$v'$};
    \node at (-1, 4) {$w_{-}$};
    \node at (-1, 2) {$\scriptstyle e_{j}+\dots$};
    \path (A3_0) edge [-] node [auto] {$\scriptstyle{}$} (A3_1);
    \path (A3_1) edge [-] node [auto] {$\scriptstyle{}$} (A3_2);
    \path (A) edge [-] node [auto] {$\scriptstyle{}$} (A3_0);
  \end{tikzpicture}
\]
where no other vertex in $\Gamma$ is hit by a vector in $A$. This scenario, however, leads to a contradiction. Indeed, if $w_{+}$ has weight greater than 2, we can delete from the embedding the basis vectors $e,e',e_{i}$ and $e_{j}'$ and the vertices $w',v$ and $v'$ yielding:
\[
  \begin{tikzpicture}[xscale=2,yscale=-0.5]
    \node (A2_0) at (0, 4) {$w$};
    \node at (0, 2) {$\scriptstyle e_{j}$};
    \node (A2_2) at (2, 4) {$w_{+}$};
    \node at (2, 2) {$\scriptstyle +\dots$};
    \node (A3_0) at (0, 3) {$\bullet$};
    \node (A3_2) at (2, 3) {$\bullet$};
    \node (A) at (-1, 3) {$\bullet$};
    \node at (4, 4) {$\ $};
    \node at (-1, 4) {$w_{-}$};
    \node at (-1, 2) {$\scriptstyle e_{j}+\dots$};
    \path (A) edge [-] node [auto] {$\scriptstyle{}$} (A3_0);
  \end{tikzpicture}
\]
where $w$ and $w_{+}$ are now final vertices. The determinant of the graph thus obtained needs to be zero, but this contradicts Lemma~\ref{l:3final1} and Remark~\ref{rem}(\ref{dd}). If $|w_{+}|=2$, then $w_{+}\cdot v=1$, which implies $v\cdot w_{-}=0$, which in turn yields $|v|\geq 4$. This time we delete again from the embedding of $\Gamma$ the same basis vectors $e,e',e_{i}$ and $e_{j}'$, but the vertices $w',w_{+}$ and $v'$ yielding:
\[
  \begin{tikzpicture}[xscale=2,yscale=-0.5]
    \node (A2_0) at (0, 4) {$w$};
    \node at (0, 2) {$\scriptstyle e_{j}$};
    \node (A2_2) at (3, 4) {$v$};
    \node at (3, 2) {$\scriptstyle -e_{j}+\dots$};
    \node (A3_0) at (0, 3) {$\bullet$};
    \node (A3_2) at (3, 3) {$\bullet$};
    \node (A) at (-1, 3) {$\bullet$};
    \node at (4, 4) {$\ $};
    \node at (-1, 4) {$w_{-}$};
    \node at (-1, 2) {$\scriptstyle e_{j}+\dots$};
    \path (A) edge [-] node [auto] {$\scriptstyle{}$} (A3_0);
  \end{tikzpicture}
\]
where $w$ is final and $|v|\geq 2$. Just as before, the determinant of this graph needs to be zero, but this contradicts Lemma~\ref{l:3final1} and Remark~\ref{rem}(\ref{dd}).

We now need to analyze the case $e\cdot w_{+}\neq 0$ and precisely 5 vertices in $\Gamma$ are hit by the five vectors in $A$. The five vertices are $w_{-},w,w',w_{+}$ and another vertex, call it $v$. The basis vector $e'$ hits precisely one between $w_{-}$ and $v$. We start considering the case $e'\cdot w_{-}\neq 0$, which implies, since $v\cdot w,v\cdot w'=0$, that $e_{i},e_{j}$ and $e_{j}'$ all hit $v$. The possibilities for the embedding in this situation are:
\[
  \begin{tikzpicture}[xscale=2.7,yscale=-0.75]
    \node (A2_0) at (0, 4) {$w$};
    \node at (0, 2) {$\scriptstyle e+e_{i}+e_{j}$};
    \node at (1, 2) {$\scriptstyle e'+e_{i}+e'_{j}$};
    \node (A2_2) at (2, 4) {$w_{+}$};
    \node at (1, 4) {$w'$};
    \node at (3, 4) {$v$};
    \node at (2, 2) {$\scriptstyle \pm e\begin{cases}\scriptstyle+e_{i}+\dots\\ \scriptstyle +e_{j}'\mp e_{j}+\dots\end{cases}$};
    \node at (3, 2) {$\scriptstyle e_{i}-e_{j}-e_{j}'+\dots$};
    \node at (-1, 2) {$\scriptstyle \pm e'\begin{cases}\scriptstyle+e_{i}+\dots\\ \scriptstyle +e_{j}\mp e_{j}'+\dots\end{cases}$};
    \node (A3_0) at (0, 3) {$\bullet$};
    \node (A3_1) at (1, 3) {$\bullet$};
    \node (A3_2) at (2, 3) {$\bullet$};
    \node (A3_3) at (3, 3) {$\bullet$};
    \node (A) at (-1, 3) {$\bullet$};
    \node at (-1, 4) {$w_{-}$};
    \path (A3_0) edge [-] node [auto] {$\scriptstyle{}$} (A3_1);
    \path (A3_1) edge [-] node [auto] {$\scriptstyle{}$} (A3_2);
    \path (A) edge [-] node [auto] {$\scriptstyle{}$} (A3_0);
  \end{tikzpicture}
\]
If the embedding of $w_{-}$ and $w_{+}$ is ``of the same type'', then we are done since either $e_i$ or $\{e_{j},e_{j}'\}$ contributes with 2 to $I(\Gamma)$. By the symmetry of the problem we need only to analyze one of the two cases in which the embedding of these vertices is not of the same type. We pick this case:
\[
  \begin{tikzpicture}[xscale=2.5,yscale=-0.5]
    \node (A2_0) at (0, 4) {$w$};
    \node at (0, 2) {$\scriptstyle e+e_{i}+e_{j}$};
    \node at (1, 2) {$\scriptstyle e'+e_{i}+e'_{j}$};
    \node (A2_2) at (2, 4) {$w_{+}$};
    \node at (1, 4) {$w'$};
    \node at (3, 4) {$v$};
    \node at (2, 2) {$\scriptstyle \pm e +e_{j}'\mp e_{j}+\dots$};
    \node at (3, 2) {$\scriptstyle e_{i}-e_{j}-e_{j}'+\dots$};
    \node at (-1, 2) {$\scriptstyle -e'+e_{i}+\dots$};
    \node (A3_0) at (0, 3) {$\bullet$};
    \node (A3_1) at (1, 3) {$\bullet$};
    \node (A3_2) at (2, 3) {$\bullet$};
    \node (A3_3) at (3, 3) {$\bullet$};
    \node (A) at (-1, 3) {$\bullet$};
    \node at (-1, 4) {$w_{-}$};
    \path (A3_0) edge [-] node [auto] {$\scriptstyle{}$} (A3_1);
    \path (A3_1) edge [-] node [auto] {$\scriptstyle{}$} (A3_2);
    \path (A) edge [-] node [auto] {$\scriptstyle{}$} (A3_0);
  \end{tikzpicture}
\]
and show that it leads to a contradiction. If $|w_{-}|>2$, then we proceed to delete from $\Gamma$ the basis vectors $e,e_{i},e_{j}$ and $e_{j}'$ and the vertices $w,w_{+}$ and $v$ obtaining:
\[
  \begin{tikzpicture}[xscale=2.2,yscale=-0.5]
    \node (A2_0) at (1, 4) {$w'$};
    \node at (1, 2) {$\scriptstyle e'$};
    \node (A3_0) at (1, 3) {$\bullet$};
    \node (A3_2) at (3, 3) {$\ $};
    \node (A) at (-1, 3) {$\bullet$};
    \node at (4, 4) {$\ $};
    \node at (-1, 4) {$w_{-}$};
    \node at (-1, 2) {$\scriptstyle -e'+\dots$};
    \path (A) edge [-] node [auto] {$\scriptstyle{}$} (A3_0);
  \end{tikzpicture}
\]
where $w'$ is a final vertex and the only one in the graph with weight 1. The determinant of this graph should vanish, but this contradicts Lemma~\ref{l:3final1} and Remark~\ref{rem}(\ref{dd}). If $|w_{-}|=2$, then we are dealing with
\[
  \begin{tikzpicture}[xscale=2.5,yscale=-0.5]
    \node (A2_0) at (0, 4) {$w$};
    \node at (0, 2) {$\scriptstyle e+e_{i}+e_{j}$};
    \node at (1, 2) {$\scriptstyle e'+e_{i}+e'_{j}$};
    \node (A2_2) at (2, 4) {$w_{+}$};
    \node at (1, 4) {$w'$};
    \node at (-2, 4) {$v$};
    \node at (2, 2) {$\scriptstyle \pm e +e_{j}'\mp e_{j}+\dots$};
    \node at (-2, 2) {$\scriptstyle e_{i}-e_{j}-e_{j}'+\dots$};
    \node at (-1, 2) {$\scriptstyle -e'+e_{i}$};
    \node (A3_0) at (0, 3) {$\bullet$};
    \node (A3_1) at (1, 3) {$\bullet$};
    \node (A3_2) at (2, 3) {$\bullet$};
    \node (A3_3) at (-2, 3) {$\bullet$};
    \node (A) at (-1, 3) {$\bullet$};
    \node at (-1, 4) {$w_{-}$};
    \path (A3_0) edge [-] node [auto] {$\scriptstyle{}$} (A3_1);
    \path (A3_1) edge [-] node [auto] {$\scriptstyle{}$} (A3_2);
    \path (A) edge [-] node [auto] {$\scriptstyle{}$} (A3_0);
    \path (A3_3) edge [-] node [auto] {$\scriptstyle{}$} (A);
  \end{tikzpicture}
\]
Since $w_{+}$ and $v$ cannot both be final vertices in $\Gamma$, at least one of them has weight greater than 3. Eliminating from $\Gamma$ the five vectors in $A$ will then yield a graph with 3 or 4 vertices less (depending on the weights of $w_{+}$ and $v$) and final vertices with weight at least 1. This graph should have vanishing determinant, but this contradicts again Lemma~\ref{l:3final1} and Remark~\ref{rem}(\ref{dd}).

The last case we need to consider now is $e'\cdot w_{-}=0$, which implies $e'\cdot v\neq 0$. We are assuming again that the vectors in $A$ only hit the set of five vertices $B=\{w_{-},w,w'w_{+},v\}$. Notice that this time we cannot have $e_{i}$ hitting $w_{+}$: if this were the case, neither $e_{j}$ nor $e_{j}'$ could hit $w_{+}$; however, both $e_{j}$ and $e_{j}'$ need to hit at least 3 vertices in $B$, while they cannot both hit $w_{-}$ and they cannot both hit $v$. We have thus $w_{+}=\pm e\mp e_j+e_{j}'+\dots$ and the possible embeddings are: 
\[
  \begin{tikzpicture}[xscale=3,yscale=-0.75]
    \node (A2_0) at (0, 4) {$w$};
    \node at (0, 2) {$\scriptstyle e+e_{i}+e_{j}$};
    \node at (1, 2) {$\scriptstyle e'+e_{i}+e'_{j}$};
    \node (A2_2) at (2, 4) {$w_{+}$};
    \node at (1, 4) {$w'$};
    \node at (3, 4) {$v$};
    \node at (2, 2) {$\scriptstyle \pm e\mp e_{j}+e_{j}'+\dots$};
    \node at (3, 2) {$\scriptstyle e'\begin{cases}\scriptstyle -e_{i}+e_{j}+\dots\\ \scriptstyle -e_{j}'+\dots\end{cases}$};
    \node at (-1, 2) {$\scriptstyle \begin{cases}\scriptstyle e_{i}-e_{j}'+\dots\\ \scriptstyle e_{j}+\dots\end{cases}$};
    \node (A3_0) at (0, 3) {$\bullet$};
    \node (A3_1) at (1, 3) {$\bullet$};
    \node (A3_2) at (2, 3) {$\bullet$};
    \node (A3_3) at (3, 3) {$\bullet$};
    \node (A) at (-1, 3) {$\bullet$};
    \node at (-1, 4) {$w_{-}$};
    \path (A3_0) edge [-] node [auto] {$\scriptstyle{}$} (A3_1);
    \path (A3_1) edge [-] node [auto] {$\scriptstyle{}$} (A3_2);
    \path (A) edge [-] node [auto] {$\scriptstyle{}$} (A3_0);
  \end{tikzpicture}
\]
The only combination of these that guarantees that each of the vectors $e_{i},e_{j}$ and $e_{j}'$ appear at least three times in the embedding is
\[
  \begin{tikzpicture}[xscale=3,yscale=-0.5]
    \node (A2_0) at (0, 4) {$w$};
    \node at (0, 2) {$\scriptstyle e+e_{i}+e_{j}$};
    \node at (1, 2) {$\scriptstyle e'+e_{i}+e'_{j}$};
    \node (A2_2) at (2, 4) {$w_{+}$};
    \node at (1, 4) {$w'$};
    \node at (3, 4) {$v$};
    \node at (2, 2) {$\scriptstyle \pm e\mp e_{j}+e_{j}'+\dots$};
    \node at (3, 2) {$\scriptstyle e'-e_{i}+e_{j}+\dots$};
    \node at (-1, 2) {$\scriptstyle  e_{i}-e_{j}'+\dots$};
    \node (A3_0) at (0, 3) {$\bullet$};
    \node (A3_1) at (1, 3) {$\bullet$};
    \node (A3_2) at (2, 3) {$\bullet$};
    \node (A3_3) at (3, 3) {$\bullet$};
    \node (A) at (-1, 3) {$\bullet$};
    \node at (-1, 4) {$w_{-}$};
    \path (A3_0) edge [-] node [auto] {$\scriptstyle{}$} (A3_1);
    \path (A3_1) edge [-] node [auto] {$\scriptstyle{}$} (A3_2);
    \path (A) edge [-] node [auto] {$\scriptstyle{}$} (A3_0);
  \end{tikzpicture}
\]
However, this configuration is not possible in $\Gamma$: deleting from the embedding the basis vectors $e',e_{i}$ and $e_{j}$ and the vertices $w'$ and $v$ we obtain
\[
  \begin{tikzpicture}[xscale=3,yscale=-0.5]
    \node (A2_0) at (0, 4) {$w$};
    \node at (0, 2) {$\scriptstyle e$};
    \node (A2_2) at (2, 4) {$w_{+}$};
    \node at (2, 2) {$\scriptstyle \pm e+e_{j}'+\dots$};
    \node at (-1, 2) {$\scriptstyle  -e_{j}'+\dots$};
    \node (A3_0) at (0, 3) {$\bullet$};
    \node (A3_2) at (2, 3) {$\bullet$};
    \node (A3_3) at (3, 3) {$\ $};
    \node (A) at (-1, 3) {$\bullet$};
    \node at (-1, 4) {$w_{-}$};
    \path (A3_0) edge [-] node [auto] {$\scriptstyle{}$} (A3_2);
  \end{tikzpicture}
\]
where $w_{-}$ and $w_{+}$ are not linked since they were not before. The only vertices that have (or might have) weight 1 are final vertices, which implies that this graph should have a vanishing determinant, but this contradicts again Lemma~\ref{l:3final1} and Remark~\ref{rem}(\ref{dd}).

Now that we are done analyzing the case when $w_{+}\cdot e\neq 0$, and by symmetry of the problem also the case when $w_{-}\cdot e'\neq 0$, we still need to consider the possibilities:
	\begin{itemize}
		\item[$-$] $w_{+}=-e'+e_{i}+e_{j}'-e_{j}+\dots$,
		\item[$-$] $w_{+}=e_{i}-e_{j}+\dots$ and
		\item[$-$] $w_{+}=e_{j}'+\dots$.	
	\end{itemize}
There are three analogous possibilities for $w_{-}$.
We proceed by studying all possible combinations, which again by the symmetry of the roles played by $w_{+}$ and $w_{-}$ amounts to 6 cases.
\begin{enumerate}

\item If $w_{+}=-e'+e_{i}+e_{j}'-e_{j}+\dots$ and $w_{-}=-e+e_{i}+e_{j}-e_{j}'+\dots$, we are dealing with a graph $\Gamma$ which will have other vertices hit by the vectors in $A$, but that locally somewhere looks like:
\[
  \begin{tikzpicture}[xscale=2.5,yscale=-0.5]
    \node (A2_0) at (0, 4) {$w$};
    \node at (0, 2) {$\scriptstyle e+e_{i}+e_{j}$};
    \node at (1, 2) {$\scriptstyle e'+e_{i}+e'_{j}$};
    \node (A2_2) at (2, 4) {$w_{+}$};
    \node at (1, 4) {$w'$};
    \node (A3_0) at (0, 3) {$\bullet$};
    \node (A3_1) at (1, 3) {$\bullet$};
    \node (A3_2) at (2, 3) {$\bullet$};
    \node (A) at (-1, 3) {$\bullet$};
    \node at (-1, 4) {$w_{-}$};
    \node at (-1,2) {$\scriptstyle -e+e_{i}+e_{j}-e_{j}'+\dots$};
    \node at (2,2) {$\scriptstyle -e'+e_{i}+e_{j}'-e_{j}+\dots$};
    \path (A3_0) edge [-] node [auto] {$\scriptstyle{}$} (A3_1);
    \path (A3_1) edge [-] node [auto] {$\scriptstyle{}$} (A3_2);
    \path (A) edge [-] node [auto] {$\scriptstyle{}$} (A3_0);
 \end{tikzpicture}
\]
There must be at least another vertex in $\Gamma$ hit by some vector in $A$ (deleting the 4 evidenced vertices and the five vectors in $A$ yields to a contradiction). By our assumptions on $e,e'$ this vertex, call it $v$, needs to be hit by all $e_{i},e_{j}$ and $e_{j}'$. The claim follows.
\item Consider the case $w_{+}=-e'+e_{i}+e_{j}'-e_{j}+\dots$ and $w_{-}=e_{i}-e_{j}'+\dots$. If the vector $e_{i}$ reappears in the embedding then we are done since $e_{i}$ appears 5 times contributing 2 to $I(\Gamma)$. We assume that this is not the case and we will arrive at a contradiction. The vector $e$ reappears on a vertex $u$ and under the current assumptions  $u\neq w_{\pm}$. It follows that the embedding of $u$ is of the form $u=\pm e \mp e_{j}+\dots$. If $|u|=2$ then $w_{+}\cdot u=1$ and $u$ is final (we are assuming $e_{i}$ does not reappear). Moreover, if $|w_{-}|=2$, then $w_{-}$ is necessarily final, and thus only one between $u$ and $w_{-}$ can have modulus 2.  

We consider the following manipulation: delete from $\Gamma$ the vertices $w$ and $w'$ and the basis vectors $e$ and $e'$. This yields a disconnected graph without cycles and an embedding into a standard lattice of rank its number of vertices. We further proceed to delete the basis vector $e_{i}$ obtaining a graph $\Gamma'$ with one connected component and vanishing determinant from which we know:
\[
  \begin{tikzpicture}[xscale=2,yscale=-0.5]
    \node (A2_0) at (0, 4) {$w_{-}$};
    \node at (0, 2) {$\scriptstyle -e_{j}'+\dots$};
    \node at (1, 2) {$\scriptstyle e'_{j}-e_{j}+\dots$};
    \node at (1, 4) {$w_{+}$};
    \node at (0.5, 2.75) {$-$};
    \node (A3_0) at (0, 3) {$\bullet$};
    \node (A3_1) at (1, 3) {$\bullet$};
    \node (A3_2) at (2, 3) {$\dots$};
    \node (A) at (-1, 3) {$\dots$};
    \path (A3_0) edge [-] node [auto] {$\scriptstyle{}$} (A3_1);
    \path (A3_1) edge [-] node [auto] {$\scriptstyle{}$} (A3_2);
    \path (A) edge [-] node [auto] {$\scriptstyle{}$} (A3_0);
 \end{tikzpicture}
\]
If $|u|,|w_{-}|\geq 2$ in $\Gamma'$, we arrive to a contradiction with Lemma~\ref{l:det}. If one of them has modulus 1, since it is necessarily final, we arrive to a contradiction with Lemma~\ref{l:3final1}.
\item If $w_{+}=-e'+e_{i}+e_{j}'-e_{j}+\dots$ and $w_{-}=e_{j}+\dots$ then we have two possibilities for the vertex $u$ which is the only vertex in $\Gamma$ different from $w$ hit by $e$:
	\begin{itemize}
		\item $u\cdot e_{i}\neq 0$, which means $e_{i}$ appears already 4 times in the embedding of $\Gamma$ and we claim that if it does not reappear at least once more, we have a contradiction. Indeed, assume $e_{i}$ does not reappear and proceed to delete the vertices $u$ and $w_{+}$, and the vectors $e$ and  $e'$. We obtain a graph with possibly several connected components and an embedding into a lattice of the same rank as the number of vertices, and in which the vector $e_{i}$ appears only twice. The local configuration we are interested in is:
\[
  \begin{tikzpicture}[xscale=2,yscale=-0.5]
    \node (A2_0) at (0, 4) {$w_{-}$};
    \node at (0, 2) {$\scriptstyle e_{j}+\dots$};
    \node at (1, 2) {$\scriptstyle e_{i}+e_{j}$};
    \node at (2, 2) {$\scriptstyle e_{i}+e'_{j}$};
    \node at (1, 4) {$w$};
    \node at (2, 4) {$w'$};
    \node (A3_0) at (0, 3) {$\bullet$};
    \node (A3_1) at (1, 3) {$\bullet$};
    \node (A3_2) at (2, 3) {$\bullet$};
    \path (A3_0) edge [-] node [auto] {$\scriptstyle{}$} (A3_1);
    \path (A3_1) edge [-] node [auto] {$\scriptstyle{}$} (A3_2);
 \end{tikzpicture}
\]
Now, if we delete $e_{i}$ from this graph we obtain $\Gamma'$, a graph with embedding and vanishing determinant. But all weights in $\Gamma'$ are at least 2, except for the vertices $w$ and $w'$ and since these vertices are each final in their connected component, by Lemma~\ref{l:3final1} and Remark~\ref{rem}(\ref{ddN}) we arrive to the contradiction that $\Gamma'$ has non-vanishing determinant.
\item $u\cdot e_{i}=0$. In this case $u=\pm e \mp e_{j}+\dots$, and we know there needs to be another vertex $v$ in $\Gamma$ hit by the basis vector $e_{j}'$. Notice that by our current assumptions $v\not\in\{w_{-},u,w_{+},w,w'\}$. Since $v\cdot w'=0$ and $v\cdot e'=0$ we have $e_{i}\cdot v\neq 0$. Then since $v\cdot w=0$ and $v\cdot e=0$ we have $e_{j}\cdot v\neq 0$. This means that the basis vector $e_{j}$ appears at least 5 times in the embedding, contributing with at least 2 to $I(\Gamma)$ and in this case we are done.
	\end{itemize}
\item If $w_{+}=e_{i}-e_{j}+\dots$ and $w_{-}=e_{i}-e_{j}'+\dots$ then there are two different vertices in $\Gamma$, $u$ and $u'$, hit respectively by $e$ and $e'$. If $e_{i}\cdot u\neq 0$ or $e_{i}\cdot u'\neq 0$, then we are done since $e_{i}$ appears at least 5 times. We assume now that $e_{i}$ appears only in $\{w_{-},w,w',w_{+}\}$ and we show that this configuration leads to a contradiction. To start with notice that $w_{-}\cdot w_{+}=0$ implies that these two vertices have weight at least 3. Moreover, if $|u|=2$ or $|u'|=2$ we have then that $w_{+}\cdot u=1$ or $w_{-}\cdot u'=1$ and the vertices $u$ or $u'$ are necessarily final. It follows that at most one between $u$ and $u'$ has modulus 2 in $\Gamma$. Proceed to delete from $\Gamma$ the vertices $w$ and $w'$ and the vectors $e$ and $e'$. If we further delete the basis vector $e_{i}$, which now appears only two times, we obtain a graph $\Gamma'$ with an embedding and vanishing determinant. However, all its vertices have weight at least 2 except for possibly a final vertex of weight 1. It follows by Lemma~\ref{l:det}, Remark~\ref{rem}(\ref{dd}) and Lemma~\ref{l:3final1} that this is not possible.
\item If $w_{+}=e_{i}-e_{j}+\dots$ and $w_{-}=e_{j}+\dots$ then there are two different vertices $u$ and $u'$ hit respectively by $e$ and $e'$. 
	\begin{itemize}
	\item If both are also hit by $e_{i}$, we are done. 
	%
	%
	\item If one is hit by $e_{i}$ and the other by $e_{j}$, then each appear 4 times in the embedding and so together contribute at least 2 to $I(\Gamma)$.
	\item If one is hit by $e_{i}$ and the other by $e_{j}'$, then we still need $e_{j}'$ to hit another vertex, since it appears at least 3 times in the embedding of $\Gamma$. By orthogonality with $w'$ this forces $e_{i}$ to appear at least 5 times.
	\item If one is hit by $e_{j}$ and the other by $e_{j}'$, the fact that we need at least a third appearance of $e_{j}'$ finishes the proof since as above it implies that $e_{i}$ must appear once more. So $e_i$ and $e_j$ both appear 4 times and together contribute at least $2$ to $I(\Gamma)$.
	\end{itemize}
\item If $w_{+}=e_{j}'+\dots$ and $w_{-}=e_{j}+\dots$ then the vectors $e_{i},e_{j}$ and $e_{j}'$ need to reappear. 
	\begin{itemize}
	\item If the three of them hit the same vertex $v$ then $v\neq u,u'$ which are then either hit by $e_{j}$ and $e_{j}'$ respectively, and we are done since each of these appears at least 4 times; or both are hit by $e_{i}$, which then appears 5 times and we are done; or precisely one of them  is hit by $e_{i}$, say it is $u$, and it follows that $u=e-e_{i}+e_{j}'+\dots$ and $u'=e'-e_{j}'+\dots$ and we are done since $e_{j}'$ appears at least 5 times (had we chosen $u'$ the same argument finishes with $e_{j}$ appearing at least 5 times).
	\item If the three of them hit only two other vertices, they have to be $u$ and $u'$ and we are dealing with the configuration (the embeddings of $u$ and $u'$ up to sign):
\[
  \begin{tikzpicture}[xscale=2,yscale=-0.5]
    \node (A2_0) at (0, 4) {$w$};
    \node at (0, 2) {$\scriptstyle e+e_{i}+e_{j}$};
    \node at (1, 2) {$\scriptstyle e'+e_{i}+e'_{j}$};
    \node (A2_2) at (2, 4) {$w_{+}$};
    \node at (1, 4) {$w'$};
    \node at (3, 4) {$u$};
    \node at (2, 2) {$\scriptstyle e_{j}'+\dots$};
    \node at (3, 2) {$\scriptstyle e-e_{i}+e_{j}'+\dots$};
    \node at (4, 2) {$\scriptstyle e'-e_{i}+e_{j}+\dots$};
    \node (A3_0) at (0, 3) {$\bullet$};
    \node (A3_1) at (1, 3) {$\bullet$};
    \node (A3_2) at (2, 3) {$\bullet$};
    \node (A3_3) at (3, 3) {$\bullet$};
    \node (A) at (-1, 3) {$\bullet$};
    \node (B) at (4, 3) {$\bullet$};
    \node at (4, 4) {$u'$};
    \node at (-1, 4) {$w_{-}$};
    \node at (-1, 2) {$\scriptstyle e_{j}+\dots$};
    \path (A3_0) edge [-] node [auto] {$\scriptstyle{}$} (A3_1);
    \path (A3_1) edge [-] node [auto] {$\scriptstyle{}$} (A3_2);
    \path (A) edge [-] node [auto] {$\scriptstyle{}$} (A3_0);
  \end{tikzpicture}
\]
If $e_{i}$ reappears in the embedding, we are done. We show that if this is not the case, we arrive to a contradiction. The argument is the same one we have repeatedly used in this proof: delete $u$, $u'$ and the vectors $e$ and $e'$ from the embedding obtaining a graph in which $e_{i}$ appears only twice on two adjacent vertices. Now, delete $e_{i}$, obtaining a graph with embedding $\Gamma'$ which should have vanishing determinant, but its only vertices with weight smaller than 2 are final, and thus $\Gamma'$ has non vanishing determinant by Lemma~\ref{l:3final1} and Remark~\ref{rem}(\ref{dd}).
	\item If there are three vertices outside the set $\{w_{-},w,w',w_{+}\}$ hit by the vectors $\{e_{i},e_{j},e_{j}'\}$, then one of these vertices is different from $u$ and $u'$ and orthogonal to $w$ and $w'$ and it is thus hit by the 3 vectors $e_{i},e_{j}$ and $e_{j}'$.
Then, if neither $u$ nor $u'$ are hit by $e_{i}$ we have that the vectors $e_{j}$ and $e_{j}'$ appear each 4 times in the embedding, and we are done. If both $u$ and $u'$ are hit by $e_{i}$ we are also done since $e_{i}$ appears 5 times in the embedding. Finally, if $e_{i}$ hits say only $u$ between $u$ and $u'$, then $e_{j}'\cdot u\neq 0$ and $e_{j}'\cdot u'\neq 0$ and we are done since $e_{j}'$ appears 5 times. 
	\end{itemize} 
\end{enumerate}

We have completed the study of all possible cases and the lemma is proved. 
\end{proof}

In Lemmas~\ref{l:3ad} and~\ref{l:cases} we have studied the appearance in the embedding of basis vectors which hit more than one internal vertex of weight 3 hit by a vector appearing exactly twice in the embedding with coefficient $\pm 1$. There is one more possibility that we have to contemplate before concluding that the negative contribution to $I(\Gamma)$ is compensated by other vectors: in the next lemma we study the contribution to $I(\Gamma)$ of the basis vectors hitting two vertices of weight 2 or a vertex of weight 2 and a vertex of weight 3, which are hit by a basis vector contributing negatively to $I(\Gamma)$.

\begin{lemma}\label{l:case2}
Suppose that $\Gamma$ embeds and that there are two distinct basis vectors $e$ and $e'$ which appear each exactly twice in the embedding with coefficient $\pm1$. 
\begin{itemize}
\item[-] If $v=e+e_{i}$ and $v'=e'+e_{i}$ are two different internal vertices of weight 2, then $e_{i}$ contributes to $I(\Gamma)$ with at least 2.
\item[-] If $v=e+e_{i}$ is an internal vertex of weight 2 and $v'=e'+e_{i}+e_{j}$ is an internal vertex of weight 3, then $v\cdot v'=1$ and the set $\{e_{i},e_{j}\}$ contributes to $I(\Gamma)$ with at least 2.
\end{itemize}
\end{lemma}
\begin{proof}
Let us suppose that there are two vertices $v=e+e_{i}$ and $v'=e'+e_{i}$. Then they are linked, and if $e_{i}$ appears anywhere in the embedding with a coefficient greater than 1, we are done. So let us suppose that the coefficient of $e_{i}$ is always $\pm 1$. If $e$ or $e'$ appear in two linked vertices with the same sign, then by deleting it from the embedding we obtain a contradiction with  Lemma~\ref{l:3final1} and Remark~\ref{rem}(\ref{dd}). It follows that the only possible local configuration is the following
\[
  \begin{tikzpicture}[xscale=2,yscale=-0.4]
    \node (A2_0) at (0, 4) {$v$};
    \node at (0, 2) {$\scriptstyle e+e_{i}$};
    \node at (1, 2) {$\scriptstyle e_{i}+e'$};
    \node (A2_2) at (1, 4) {$v'$};
    \node (A2_2) at (2, 4) {$v'_{+}$};
    \node at (2, 2) {$\scriptstyle e_{i}-e+f+\dots$};
    \node (A3_0) at (0, 3) {$\bullet$};
    \node (A3_1) at (1, 3) {$\bullet$};
    \node (A3_2) at (2, 3) {$\bullet$};
    \node (A3_3) at (3, 3) {$\dots$};
    \node (A) at (-1, 3) {$\bullet$};
    \node (B) at (-2, 3) {$\dots$};
    \node at (-1, 4) {$v_{-}$};
    \node at (-1, 2) {$\scriptstyle e_{i}-e'-f+\dots$};
    \path (A3_0) edge [-] node [auto] {$\scriptstyle{}$} (A3_1);
    \path (A3_2) edge [-] node [auto] {$\scriptstyle{}$} (A3_3);
    \path (A3_1) edge [-] node [auto] {$\scriptstyle{}$} (A3_2);
    \path (A) edge [-] node [auto] {$\scriptstyle{}$} (A3_0);
    \path (A) edge [-] node [auto] {$\scriptstyle{}$} (B);
  \end{tikzpicture}
\]
Observe that the subgraph formed by the 4 depicted vertices and the vectors $e,e',e_{i},f$ is a standard subset of $\Z^{4}$ and since by Lemma~\ref{extravector} no vertex can be linked to this set we conclude that at least one between $v_{-}$ and $v_{+}'$ has weight at least 4. If $e_{i}$ appears somewhere else in the embedding, then we are done. If it does not, then we consider the following manipulation: delete from the embedding the basis vectors $e,e',e_{i}$ and the vertices $v,v'$. The resulting graph with embedding, $\Gamma'$, has vanishing determinant, since it admits an embedding into a smaller rank lattice. We conclude, by Lemma~\ref{l:3final1}, that precisely one between $v_{-}$ and $v'_{+}$ has weight 3 in $\Gamma$ and this vertex is not final in $\Gamma'$. Without loss of generality we suppose $v_{-}=e_{i}-e'-f$ and thus the local configuration in $\Gamma'$ is 
\[
  \begin{tikzpicture}[xscale=2,yscale=-0.4]
    \node (A2_0) at (0, 4) {$v_{-}$};
    \node at (0, 2) {$\scriptstyle -f$};
    \node at (1, 2) {$\scriptstyle f+\dots$};
    \node (A2_2) at (1, 4) {$v'_{+}$};
    \node (A3_0) at (0, 3) {$\bullet$};
    \node  at (0.5, 2.7) {$\scriptstyle -$};
    \node (A3_1) at (1, 3) {$\bullet$};
    \node (A3_2) at (2, 3) {$\dots$};
    \node (A) at (-1, 3) {$\bullet$};
    \node (B) at (-2, 3) {$\dots$};
    \path (A3_0) edge [-] node [auto] {$\scriptstyle{}$} (A3_1);
    \path (A3_1) edge [-] node [auto] {$\scriptstyle{}$} (A3_2);
    \path (A) edge [-] node [auto] {$\scriptstyle{}$} (A3_0);
    \path (A) edge [-] node [auto] {$\scriptstyle{}$} (B);
  \end{tikzpicture}
\]
This sort of configuration has already been analyzed in the proof of Lemma~\ref{l:3no2}, second case, second item `If the vertex $t$ is final...', when the option `the case $y'=0$' is considered. Indeed, in that proof the vertex $s$ plays the same role as the vertex $v'_{+}$ here and the same argument goes through, yielding a contradiction. Notice that in the arguments presented in the proof of Lemma~\ref{l:3no2} the fact that $s$ is final does not play any part. 

We move on now to analyze the second claim in the statement of the lemma, with two internal vertices $v=e+e_{i}$ and $v'=e'+e_{i}+e_{j}$. By Lemma~\ref{l:3no2}, we know that $e_{j}\neq\pm e$ and therefore $v\cdot v'=1$ and we necessarily have a local configuration like the following one:
\[
  \begin{tikzpicture}[xscale=2,yscale=-0.4]
    \node (A2_0) at (0, 4) {$v$};
    \node at (0, 2) {$\scriptstyle e+e_{i}$};
    \node at (1, 2) {$\scriptstyle e_{i}+e'+e_{j}$};
    \node (A2_2) at (1, 4) {$v'$};
    \node (A2_2) at (2, 4) {$v'_{+}$};
    \node (A3_0) at (0, 3) {$\bullet$};
    \node (A3_1) at (1, 3) {$\bullet$};
    \node (A3_2) at (2, 3) {$\bullet$};
    \node (A3_3) at (3, 3) {$\dots$};
    \node (A) at (-1, 3) {$\bullet$};
    \node (B) at (-2, 3) {$\dots$};
    \node at (-1, 4) {$v_{-}$};
    \node at (-1, 2) {$\scriptstyle e_{i}+\dots$};
    \path (A3_0) edge [-] node [auto] {$\scriptstyle{}$} (A3_1);
    \path (A3_2) edge [-] node [auto] {$\scriptstyle{}$} (A3_3);
    \path (A3_1) edge [-] node [auto] {$\scriptstyle{}$} (A3_2);
    \path (A) edge [-] node [auto] {$\scriptstyle{}$} (A3_0);
    \path (A) edge [-] node [auto] {$\scriptstyle{}$} (B);
  \end{tikzpicture}
\]
Since $e$ appears somewhere else in the embedding, $e_{i}$ contributes at least with 1 to $I(\Gamma)$. We claim that so does $e_{j}$ in this situation. From Lemma~\ref{l:no1} we know that $e_{j}$ appears at least twice in the embedding, and from Lemma~\ref{l:3no2} we learn that $e_{j}$ appears at least three times. Suppose by means of contradiction that $e_{j}$ appears exactly 3 times with coefficient $\pm 1$. If $e_{j}\cdot v'_{+}=0$, then $e_{j}$ hits two vertices orthogonal to $v'$ and their embeddings must be of the form $\pm e_{j}\mp e'+\dots$ and $\pm e_{j}\mp e_{i}\pm e+\dots$. This forces the contradiction $e_{j}\cdot v'_{+}\neq 0$. If $e_{j}\cdot v'_{+}\neq 0$, then there is another vertex $u$ hit by $e_{j}$. There are two possibilities:
\begin{itemize}
\item $u\cdot e_{i}\neq 0$, which implies $u=e_{j}-e_{i}+e+\dots$, up to sign. We know that there is some other vertex, say $u'$, necessarily different from $v'_{+}$, hit by $e'$. By Lemma~\ref{l:3no2} we have $e\cdot u'=0$ and therefore $e_{j}\cdot u'\neq 0$, contradicting our assumptions.
\item $u\cdot e_{i}=0$ which implies $u=e_{j}-e'+\dots$. The unique vertex in $\Gamma$ different from $v$ and hit by $e$, call it $u'$, is different from $v'_{+}$ since $u'\cdot e, u'\cdot e_{i}\neq 0$ and we have assumed $v'_{+}\cdot e_{j}\neq 0$. However, the described configuration forces $u'\cdot e_{j}\neq 0$. Again a contradiction.
\end{itemize}

\end{proof}

All the above analysis yields the following conclusion.

\begin{co}\label{c:compensate}
Suppose that $\Gamma$ embeds and consider the set $W$ of internal vertices of weight 2 or 3 hit by a basis vector which appears only twice in the embedding with coefficient $\pm1$. Let $E$ be the set of all basis vectors hitting the vertices in the set $W$. Then, the total contribution to $I(\Gamma)$ of the set $E$ is non-negative. Moreover, if $N>1$ the set $E$ is disjoint from the basis vectors hitting the 2-leg.
\end{co}
\begin{proof}
The contribution  to $I(\Gamma)$ of any basis vector which appears only twice in the embedding with coefficient $\pm1$ is $-1$. Let us briefly call these vectors in $E$ the `negative vectors'. Lemma~\ref{l:3no2} guarantees that every vertex in $W$ is hit by at most one such vector. In Lemma~\ref{l:3and2comp} we have shown that each of the vertices in $W$ is also hit by some basis vectors which contribute to $I(\Gamma)$ at least 1. Let us call these vectors in $E$ the `non-negative vectors'. From Lemmas~~\ref{l:3no2} and~\ref{l:3and2comp} alone we cannot conclude that the contribution to $I(\Gamma)$ of the negative vectors in $E$ is compensated by the non-negative ones: we might be counting the same non-negative vectors to compensate for different negative ones. This is where the analysis done in Lemmas~\ref{l:3ad}, \ref{l:3adplus}, \ref{l:cases} and~\ref{l:case2} becomes essential. The first two lemmas guarantee that each non-negative vector in $E$ hits at most two different vertices in $W$. We will match precisely which non-negative vectors compensate the negative ones:
\begin{enumerate}
\item Let $v\in W$ be hit by the negative vector $e_{v}$ and assume that the non-negative vectors hitting $v$ do not hit any other vertex in $W$. It follows from Lemma~\ref{l:3and2comp} that the negative contribution of $e_{v}$ is compensated. 
\end{enumerate}
Now that we have matched the negative and non-negative vectors in the first case, we consider the other vectors $E$.
\begin{enumerate}
\item[(2)] If two vertices $v,v'\in W$ are hit by negative vectors $e_{v}$ and $e_{v'}$ respectively, by one or two common non-negative vectors and \emph{none of their non-negative vectors hit other vertices in $W$}, then Lemmas~\ref{l:cases} and~\ref{l:case2} apply and the negative contribution is compensated. 

If we had $e_{v}=e_{v'}$, then, since we only need to compensate once for the negative contribution of $e_{v}$, we use Lemma~\ref{l:3and2comp}
to conclude that the non-negative vectors hitting $v$ and $v'$ compensate for $e_{v}$.
\end{enumerate}
Again, we can now ignore from $E$ the negative and non-negative vectors in this second case, since matched up together the negative contribution is compensated.
\begin{enumerate}
\item[(3)] There is one final situation to analyze: there is at least one vertex $v$ in $W$ whose non-negative vectors hit at least two other vertices in $W$. Notice that, by Lemma~\ref{l:3adplus}, the vertex $v$ cannot have weight 2, and thus has weight 3. For clarity, let us write $v=e+f+g$ with $e$ negative and $f,g$ non-negative. If there is any vertex in $W\setminus\{v\}$ hit by $e$, we can ``ignore it'' since we need to compensate the negative contribution of $e$ and this is independent of $e$ hitting more than one vertex in $W$. That is, if we had three vertices in $W$, for example $v_{1}=e+f+g$, $v_{2}=e-f+h$ and $v_{3}=e'+g+k$, since $v_{1}$ and $v_{2}$ are both hit by $e$ we can ignore say $v_{2}$ and consider $v_{1}$ and $v_{3}$ as an example of case (2) above (if $f,g,k$ do not hit any further vertices in $W$); or we could also ignore $v_{1}$ and consider $v_{2}$ and $v_{3}$ as two instances of case (1) (again if $f,g,k$ do not hit any further vertices in $W$).

It follows from the above that it suffices to look into the case of $v\in W$ whose non-negative vectors hit at least two other vertices in $W$ and all of these vertices in $W$ are hit by different negative vectors.
In this case, using Lemmas~\ref{l:3ad}, \ref{l:3adplus}, we are dealing
 
with the following local type of configuration:
\[
  \begin{tikzpicture}[xscale=2,yscale=-0.5]
    \node at (1, 2) {$\scriptstyle e_{2}+g_{1}+g_{2}$};
    \node at (2, 2) {$\scriptstyle e_{3}+g_{2}+g_{3}$};
    \node at (3, 2) {$\scriptstyle e_{4}+g_{3}+h$};
    \node (A3_0) at (0, 3) {$\bullet$};
    \node (A3_1) at (1, 3) {$\bullet$};
    \node (A3_2) at (2, 3) {$\bullet$};
    \node (A3_3) at (3, 3) {$\bullet$};
    \node (A) at (-1, 3) {$\bullet$};
    \node (B) at (4, 3) {$\bullet$};
    \node at (0, 2) {$\scriptstyle e_{1}+f+g_{1}$};
    \path (A3_0) edge [-] node [auto] {$\scriptstyle{}$} (A3_1);
    \path (A3_1) edge [-] node [auto] {$\scriptstyle{}$} (A3_2);
    \path (A3_2) edge [-] node [auto] {$\scriptstyle{}$} (A3_3);
    \path (A) edge [-] node [auto] {$\scriptstyle{}$} (A3_0);
    \path (A3_3) edge [-] node [auto] {$\scriptstyle{}$} (B);
  \end{tikzpicture}
\]
where the vectors $e_{i}\in E$ are negative, the vectors $f,g_{i},k\in E$ are non-negative and the 4 vertices with shown embeddings are in $W$ and the final ones are  either not in $W$ (recall that all vertices in $W$ are internal in $\Gamma$) or in $W$ and hit by negative vectors which we are ignoring as explained before. This local configuration is only meant to be an example: this `chain' in $W$, which needs to be at least of length 3 to not fall in case $(2)$, could be much longer and its two final vertices could have weight 2 instead of 3. We need to establish which non-negative vectors will compensate in this case for each of the $e_{i}$'s.
We start from one end of the chain and follow it in a chosen direction. If in a pair of consecutive vertices the non-negative vector hitting both vertices contributes to $I(\Gamma)$ with at least 2 (see Lemmas~\ref{l:cases} and~\ref{l:case2}), we compensate the two corresponding negative $e_{i}$'s with it. After finishing going through the chain in this manner, there might be some vertices with negative $e_{i}$'s that have not been compensated yet. If they are now between already compensated vertices, then the non-negative vectors hitting them do compensate by Lemma~\ref{l:3and2comp} and have not been considered before, so we use them. We are finally left with perhaps some non-compensated $e_{i}$'s in a subchain. In this case, that is, when the repeated non-negative vectors are not compensating for the negative contribution we use the example above to illustrate what happens. By Lemmas~\ref{l:cases} and~\ref{l:case2}, the pairs $\{f,g_{2}\},\{g_{1},g_{3}\}$ and $\{g_{2},h\}$ each contribute to $I(\Gamma)$ with at least two. Since we are assuming that none among $g_{1},g_{2},g_{3}$ appears more than 4 times in the embedding, we conclude that the set $\{f,g_{1},g_{2},g_{3},h\}$ contributes at least with 5 to $I(\Gamma)$ and it compensates for the negative contribution of $\{e_{1},\dots,e_{4}\}$. 
There is only one configuration that could seem problematic, and it is the following:
\[
  \begin{tikzpicture}[xscale=2,yscale=-0.5]
    \node at (1, 2) {$\scriptstyle e_{2}+g_{1}+g_{2}$};
    \node at (2, 2) {$\scriptstyle e_{3}+g_{2}+g_{3}$};
    \node at (3, 2) {$\scriptstyle e_{4}+g_{3}$};
    \node (A3_0) at (0, 3) {$\bullet$};
    \node (A3_1) at (1, 3) {$\bullet$};
    \node (A3_2) at (2, 3) {$\bullet$};
    \node (A3_3) at (3, 3) {$\bullet$};
    \node (A) at (-1, 3) {$\bullet$};
    \node (B) at (4, 3) {$\bullet$};
    \node at (0, 2) {$\scriptstyle e_{1}+g_{1}$};
    \node at (-1, 3.5) {$\scriptstyle x$};
    \path (A3_0) edge [-] node [auto] {$\scriptstyle{}$} (A3_1);
    \path (A3_1) edge [-] node [auto] {$\scriptstyle{}$} (A3_2);
    \path (A3_2) edge [-] node [auto] {$\scriptstyle{}$} (A3_3);
    \path (A) edge [-] node [auto] {$\scriptstyle{}$} (A3_0);
    \path (A3_3) edge [-] node [auto] {$\scriptstyle{}$} (B);
  \end{tikzpicture}
\]
since we could run into the situation of wanting to compensate the negative contribution of $\{e_{1},\dots,e_{4}\}$ with the set $\{g_{1},g_{2},g_{3}\}$; however, from our previous lemmas we can only guarantee that these three basis vectors contribute to $I(\Gamma)$ with at least 3. The point is that any such configuration is not possible. Under the current assumptions, each of these $g_{i}$'s appears exactly 4 times in the embedding and each of the $e_{i}$'s exactly twice. To start with let us focus on $g_{1}$. It hits the vertex $x$ and we have that either $g_{2}$ or $e_{2}$ also hit $x$. There is yet another vertex, orthogonal to $e_{1}+g_{1}$ hit by $g_{1}$ (since it appears exactly 4 times in the embedding). Let us call this vertex $v$ with embedding, up to sign, $v=e_{1}-g_{1}+\dots$. Again, $v$ is either hit by $g_{2}$ or $e_{2}$. Now, $g_{2}$, which appears twice in the chain we are discussing, needs to appear exactly 4 times in the embedding. If both $x$ and $v$ are hit by $g_{2}$, then neither is hit by $e_{2}$, which needs to hit some other vertex $w=e_{2}+\dots$. This leads to a contradiction since $w$ is orthogonal to $e_{2}+g_{1}+g_{2}$ but neither $g_{1}$ nor $g_{2}$ hit $w$. It follows that precisely one between $x$ and $v$ is hit by $e_{2}$ and the other by $g_{2}$. Again, this leads to a contradiction since the vector $g_{2}$ necessarily appears a fourth time in the embedding, hitting a vector $w'=g_{2}+\dots$ orthogonal to $e_{2}+g_{1}+g_{2}$, but neither $e_{2}$ nor $g_{1}$ hit $w'$.
\end{enumerate}

From this analysis we conclude that the vectors in the set $E=\{e\, |\, e\cdot w\neq 0, w\in W\}$ which contribute positively to $I(\Gamma)$ compensate the negative contribution of the vectors appearing only twice, and thus the total contribution to $I(\Gamma)$ of the set $E$ is non-negative

Since $\Gamma$ embeds we follow the conventions presented in Figure~\ref{f:convenzioni1}. We have that any vertex in the qc-legs, different from $v$ and hit by a vector in $\{e_{2},\dots,e_{N+2}\}$ is hit by at least all the vectors in $\{e_{2},\dots,e_{N+2}\}$ and $e_{1}$; in the case of the vertex $v$, it is hit by $\{e_{2},\dots,e_{N+2}\}$ and $v'\neq 0$. Assuming $N>1$, it follows that these vectors cannot hit any vertex in $W$ and thus the set $E$ is disjoint from $\{e_{2},\dots,e_{N+2}\}$.
\end{proof}

The more information we have about the contribution to $I(\Gamma)$ of the vectors in the set $E$ defined in the last corollary, the stronger the bound we may obtain. The following corollary will prove useful at the end of this section, when we will search for the family of graphs in the first statement of Proposition~\ref{p:mainA}.

\begin{co}\label{c:Emore}
Let $\Gamma$, $W$ and $E$ be as in the preceding corollary. Suppose there is a basis vector $f\in E$ whose contribution to $I(\Gamma)$ is $n\geq 2$ and define $V_{f}=\{w\in W\,|\, f\cdot w\neq 0\}$. Then, we have:
	\begin{itemize}
	\item If $|V_{f}|=1$, the contribution to $I(\Gamma)$ of the set $E$ is at least $n-1$.
	\item If $|V_{f}|=2$, the contribution to $I(\Gamma)$ of the set $E$ is at least $n-2$.
	\end{itemize}
\end{co}
\begin{proof}
To start with notice that by Lemmas~\ref{l:3ad} and~\ref{l:3adplus} we have $|V_{f}|\leq 2$. If $|V_{f}|=1$ there is only one vertex $w\in W$ hit by $f$, either $w=e+f+e_{j}$ or $w=e+f$, where $e$ has a contribution to $I(\Gamma)$ of $-1$ and $e_{j}$ a contribution of at least 0 (Lemma~\ref{l:3no2}). From the proof of Corollary~\ref{c:compensate}, we know that since $f$ hits only one vertex in $W$, its positive contribution is used to compensate at most one $-1$ coming from a negative vector. Since, in the set $E$ every negative vector is paired with some non-negative ones that compensate their negative contribution and $f$ is paired with one only negative vertex, we have that the contribution to $I(\Gamma)$ of the set $E$ is at least $n-1$ as claimed.

The case $|V_{f}|=2$ can be dealt with in a completely analogous fashion: this time, from the proof of Corollary~\ref{c:compensate} we know that the positive contribution of $f$ to $I(\Gamma)$ might be compensating for at most two negative vectors. It follows immediately that the contribution to $I(\Gamma)$ of the set $E$ is at least $n-2$. 
\end{proof}

\end{subsection}

\begin{subsection}{Final steps in the proof of Proposition~\ref{p:mainA}}
We are now ready to prove the second statement of Proposition~\ref{p:mainA} and to reduce the possibilities of the embedding of the graphs in the family $\Gamma_{1}$ to the case $N=1$. For the reader's convenience we will write the details of this part separately and then we will focus on the embeddings of the remaining graphs in $\Gamma_{1}$.

\begin{proof}[Proof of second statement of Proposition~\ref{p:mainA} and reduction of the cases in first statement] 

We argue by contradiction: Suppose there was an embedding of $\Gamma$ into a standard lattice of the same rank. We are now going to bound $I(\Gamma)$ from below with the information we have gathered from the embedding of $\Gamma$ in the preceding analysis. By Lemma~\ref{l:no1} we know that for a basis vector to contribute negatively to $I(\Gamma)$ it needs to appear precisely twice in the embedding with coefficient $\pm1$. By Corollary~\ref{c:controledset} we know that at most two such vectors in the embedding of $\Gamma$, (or of its reduced version $\Gamma_{r}$ which we still denote by $\Gamma$), do not hit an internal vertex of weight 2 or 3. Moreover, by Corollary~\ref{c:compensate} we know that every time such a vector hits an internal vertex of weight 2 or 3 their negative contribution to $I(\Gamma)$ is compensated by some other vector in the embedding. It follows that $I(\Gamma)\geq -2$. 

Since $I(\Gamma_{2})=-N-k-1\leq -3$ (see Facts~\ref{f} Equation~\eqref{e:I}), we conclude that no graph in the family $\Gamma_{2}$ admits an embedding into the standard lattice of the same rank and the second statement of the proposition is proved.

In the family $\Gamma_{1}$ we have $I(\Gamma_{1})=k-N-5$. We know by Lemma~\ref{l:2rep} that the vectors hitting the bottom 2-leg of length $N$ contribute with at least $N$ to $I(\Gamma_{1})$, and assuming $N>1$, Corollary~\ref{c:compensate} implies they do not intersect the set of vectors already accounted for. So, on the one hand we know from the analysis of the embedding that $I(\Gamma_{1})\geq N-2\geq 0$, and on the other $I(\Gamma_{1})=k-N-5\leq 0$. The two last inequalities, obtained assuming $N>1$ in $\Gamma_{1}$, are true at the same time only if $(N,k)=(2,7)$. However, we are not interested in this case since by Proposition~\ref{pruning2} we know that the associated surgeries on torus knots do not bound rational homology balls.
\end{proof}

To finish the proof of Proposition~\ref{p:mainA} we need to analyze the embeddability of the graphs in family $\Gamma_{1}$ with $N=1$ and $1\leq k\leq 6$. Following the conventions of Figure~\ref{f:convenzioni1}, we know, by Lemma~\ref{l:2rep}, that the basis vectors $e_{2}$ and $e_{3}$ hit at least one more vertex in $\Gamma$, call it $w$, different from $v$ and from the only vertex in the 2-leg, which we will denote $s$. We will show in the next lemma that the vertex $w$ is necessarily adjacent to $v$.

\begin{lemma}\label{l:adjacent}
Assume that $\Gamma_{1}$ with $N=1$ and $1\leq k\leq 6$ embeds into a standard lattice of the same rank. Then, with the above conventions, we have $v\cdot w=1$. 
Furthermore, the couple of vectors $e_{2},e_{3}$ hits only the vertices $s,v$ and $w$, and with coefficients $\pm1$.
\end{lemma}
\begin{proof}
We start the proof by showing that the vectors $e_{2}$ and $e_{3}$ only hit the vertices $s,v$ and $w$ in the embedding of $\Gamma_{1}$ and they do so with coefficients $\pm1$. If there were another vertex, $w'$, hit by $e_{2}$ and $e_{3}$ we would have that $e_{2}$ appears at least 5 times in the embedding, and $e_{3}$ at least 4 times. It follows that the contribution to $I(\Gamma_{1})$ of the set $\{e_{2},e_{3}\}$ would be of at least 3. On the other hand, the contribution to $I(\Gamma_{1})$ of the set $E$ of Corollary~\ref{c:compensate} is non-negative. Since we are dealing with $N=1$, we cannot guarantee that the set $E$ is disjoint from $\{e_{2},e_{3}\}$, so we analyze both scenarios. If these sets are disjoint, then we would have $I(\Gamma_{1})\geq 3+0-2\geq 1$, since by Corollary~\ref{c:controledset} there are at most two vectors contributing negatively to $I(\Gamma_{1})$, each contributing $-1$, outside the set $E$. However, we know from Facts~\ref{f} Equation~\eqref{e:I} that $I(\Gamma_{1})=k-N-5\leq 0$. This contradiction implies that $\{e_{2},e_{3}\}$ and $E$ cannot be disjoint. Assuming the two sets are not disjoint, the only possible vertex in the set $W$ of Corollary~\ref{c:compensate} hit by the vectors $\{e_{2},e_{3}\}$ is $v$, since all other vertices are either final or their weight is greater than 3. This means that there is some basis vector $f$ such that $v=e_{2}-e_{3}+v'=e_{2}-e_{3}+f$ and $f$ appears only twice in the embedding with coefficient $\pm1$. However, this is not possible since at least one between $w$ and $w'$ is orthogonal to $v$, say it is $w$, and this requires $f$ to hit $w$ with coefficient $\pm2$. Finally, if $|e_{2}\cdot w|=\alpha>1$,
then also $|e_{3}\cdot w|=\alpha>1$ and the contribution to $I(\Gamma_{1})$ of the set $\{e_{2},e_{3}\}$ is at least $2\alpha^{2}-1\geq 7$. From here we arrive to the same contradiction as before with the same type of argument.

The vertex $u$ in Figure~\ref{f:convenzioni1} is the first one in a 2-chain of length $\ell\geq 1$. We might assume the embedding of this 2-chain $c$ is of the form $(e_{1}+e_{r},e_{r}+e_{r+1},\dots,e_{r+\ell-2}+e_{r+\ell-1})$. Furthermore, by the assumptions on $\Gamma_{1}$ we know that this 2-chain is not final, and therefore it is linked to a vertex $x$ of weight at least 3. Independently of $x=w$ or $x\neq w$ we have that $w$ is necessarily hit by all the basis vectors in the set $S:=\{e_{1},e_{2},e_{3},e_{r},\dots,e_{r+\ell-2}\}$ with coefficients $\pm1$. 

We proceed now to show that the vertex $w$ cannot be orthogonal to $v$. By means of contradiction let us assume $w\cdot v=0$. Since both $w$ and $v$ are hit by $e_{2}$ and $e_{3}$, and $w\cdot s=v\cdot s=0$, there is at least one more basis vector $f\not\in S$ hitting $w$ and guaranteeing $w\cdot v=0$. If there were only one such vector, we would have $f\cdot v=\pm2$ or $f\cdot w=\pm2$ and this is not possible. Since the argument is completely analogous, we discuss only the former possibility. If $f\cdot v=\pm2$ then we would have the set $\{e_{2},f\}$ contributing at least with 3 to $I$. Just as before, if $E$ and $\{e_{2},f\}$ are disjoint, then $I(\Gamma_{1})\geq 3+0-2\geq 1$, which contradicts $I(\Gamma_{1})\leq 0$. Since we are assuming $f\cdot v=\pm2$, then $e_{2}\not\in E$. To conclude, assume that $f\in E$. If $f$ hits only one vertex in $W$, then the contribution of $f$ to $I(\Gamma_{1})$ is at least 3, and if it hits two different vertices, the only other possibility, then its contribution is of at least $4$ (since $v,w\not\in W$). Then, by Corollary~\ref{c:Emore}, the contribution to $I(\Gamma_{1})$ of the set $E$ is at least 2 in both cases, and therefore $I(\Gamma_{1})\geq 1+2-2\geq 1$ (where the $+1$ comes from the contribution of $e_{2}$, the $+2$ of $E$ and the $-2$ from the only two possible vectors not in $E$ contributing negatively to $I(\Gamma_{1})$). This contradiction implies that there are at least two vectors $f_{1},f_{2}\not\in S$ which hit $w$ and $v$ and guarantee $w\cdot v=0$.

To end the proof of this lemma we will consider two cases separately:
\begin{enumerate}
\item If $w\neq x$, we have that the only vector in the set $S\cup\{e_{r+\ell-1}\}$ hitting $x$ is $e_{r+\ell-1}$ with coefficient $\pm1$. 
Furthermore, since both $e_{2}$ and $e_{3}$ only hit the vertices $s,v$ and $w$, it follows that the vectors in $S\cup\{e_{r+\ell-1}\}$ only hit the vertices $s,v,w,x$, the central vertex and the vertices in the 2-chain $c$. We proceed to delete from the graph the $\ell+3$ vertices $s,w$, the central vertex, and the whole 2-chain $c$ and to delete from the embedding the $\ell+3$ basis vectors in $S\cup\{e_{r+\ell-1}\}$. We obtain a linear graph $\Gamma_{1}'$ (recall $v=e_{2}-e_{3}+v'=e_{2}-e_{3}\pm f_{1}\pm f_{2}+\dots$) and $x=e_{r+\ell-1}+x'$ with $|x'|\geq 2$. We consider separately the case when $w$ is final, and thus $\Gamma_{1}'$ has two connected components, or not final.
\begin{itemize}
	\item Assume $w$ were final. We claim that $\Gamma_{1}'$ is a good set, that is, its two connected components form an irreducible set. Call these two components $C_{v}$ and $C_{x}$, where the subscript indicates which vertex in $\Gamma_{1}$ is a final one in $\Gamma_{1}'$. If $\Gamma_{1}'$ was reducible, then both $C_{v}$ and $C_{x}$ would be standard sets. In this case, if $w$ was linked to $C_{x}$ in $\Gamma_{1}$, we consider the graph with embedding
\[
  \begin{tikzpicture}[xscale=2.5,yscale=-1]
    \node at (3, 1.5) {$\tilde w$};
    \node (A0_3) at (3, 0.5) {$\scriptstyle f_{1}+f_{2}+\dots$};
    \node (A0_4) at (4, 0.5) {$\scriptstyle v'=f_{1}+f_{2}+\dots$};
    \node (A1_3) at (3, 1) {$\bullet$};
    \node (A1_4) at (4, 1) {$\bullet$};
    \node (A1_5) at (5, 1) {$\dots$};
    \node (A1_6) at (6, 1) {$\bullet$};  
\draw[decorate,decoration={brace,amplitude=7pt,mirror},xshift=0.4pt,yshift=-0.3pt](3.9,1.3) -- (6.1,1.3) node[black,midway,yshift=-0.5cm] {\footnotesize $C_{v}$};
    \path (A1_4) edge [-] node [auto] {$\scriptstyle{}$} (A1_5);
    \path (A1_5) edge [-] node [auto] {$\scriptstyle{}$} (A1_6);
	\draw (A1_3) to[bend right] (A1_4);
    \draw (A1_3) to[bend left] (A1_4);
  \end{tikzpicture}
  \]
where $\tilde w$ is the vertex $w$ from $\Gamma_{1}$ where in its embedding we only consider the basis vectors hitting the standard set $C_{v}$. Analogously, if $w$ was linked to $C_{v}$ in $\Gamma_{1}$, we consider the graph with embedding
\[
  \begin{tikzpicture}[xscale=2.5,yscale=-1]
    \node at (3, 1.5) {$\tilde w$};
    \node (A0_4) at (4, 0.5) {$\scriptstyle x'$};
    \node (A1_3) at (3, 1) {$\bullet$};
    \node (A1_4) at (4, 1) {$\bullet$};
    \node (A1_5) at (5, 1) {$\dots$};
    \node (A1_6) at (6, 1) {$\bullet$};  
\draw[decorate,decoration={brace,amplitude=7pt,mirror},xshift=0.4pt,yshift=-0.3pt](3.9,1.3) -- (6.1,1.3) node[black,midway,yshift=-0.5cm] {\footnotesize $C_{x}$};
    \path (A1_4) edge [-] node [auto] {$\scriptstyle{}$} (A1_5);
    \path (A1_5) edge [-] node [auto] {$\scriptstyle{}$} (A1_6);
	\draw (A1_3) to (A1_4);
  \end{tikzpicture}
  \]
This time, $\tilde w$ and $x'$ are linked precisely once since the only change in the embedding of $x'$ with respect to $x\in\Gamma_{1}$ is the basis vector $e_{r+\ell-1}$.
Notice that both of these graphs should have vanishing determinant, since they embed in a lattice of rank smaller than their number of vertices. However, this is not the case. Indeed, the latter case follows readily from Remark~\ref{rem}(\ref{dd}); while the former can be shown to have non-vanishing determinant by a dominant diagonal argument unless $|v'|=2$. If $|v'|=2$, then, following the same arguments used when analyzing the graph~\ref{forlater} in the proof of Lemma~\ref{l:no1}, we know that 
this graph will have vanishing determinant if and only if $|w'|=3$ and $C_{v}$ is a string of three vertices with weight 2. However, this yields an impossible configuration for $\Gamma_{1}$: under these assumptions, we have $|w|\geq 8$, so this vertex cannot be the one of weight $k+1$ and then the qc-legs do not have the correct structure. 

Now that we have shown that if $w$ was final in $\Gamma_{1}$ then $\Gamma_{1}'$ is a good subset, we proceed to compare $I(\Gamma_{1}')$ with $I(\Gamma_{1})$. We have that the former obtains a positive contribution of $2+\ell$, because of the suppression of the $2+\ell$ vertices of weight $2$, and also a negative one of at least $-2-1-3-2-\ell+3$, because of the new weights of the vertices $v,x$ and the deletion of $w$. It follows that $I(\Gamma_{1}')\leq I(\Gamma_{1})-3\leq -3$ and by Lemma~\ref{badcomponents2} this implies that $\Gamma_{1}'$ has either one or two bad components. We argue now that both scenarios yield to a contradiction. If $\Gamma_{1}'$ had two bad components and $I(\Gamma_{1}')=-4$, we infer that $w$ in $\Gamma_{1}$ is hit by precisely one more basis vector with coefficient $\pm1$ than the ones we had already considered (to yield the correct $I$). The structure of the two bad components, described in Lemma~\ref{I=-2,b=1}, forces there to be two basis vectors hitting the `central vertex' of each bad component that appear only twice in the embedding with coefficient $\pm1$ (they cannot hit $w$ because of its weight constraint and the embedding of all the other vertices in $\Gamma_{1}$ is fixed). This contradicts Lemma~\ref{l:3no2}. Finally, if $\Gamma_{1}'$ had precisely one bad component and $I(\Gamma_{1}')=-3$, then, it is a routine check to establish that the structure of the two components in $\Gamma_{1}'$, described explicitly in Lemma~\ref{l:3no2}, is incompatible with the qc-leg structure we assume in the legs of $\Gamma_{1}$.

	\item If $w$ was internal, $\Gamma_{1}'$ would have 3 connected components. There are two cases, either $w$ joins the components $C_{x}$ and $C_{3}$ and the third one is $C_{v}$ (Case I); or $w$ joins the components $C_{v}$ and $C_{3}$ and the other component is $C_{x}$ (Case II). In any of these cases let us denote by $C_{nw}$ (for \emph{no $w$}) the component not linked to $w$ and by $C_{lw}$ (for \emph{left of $w$}) the component linked to $w$ which is not $C_{3}$. We start arguing that $\Gamma_{1}'$ is a good subset by analyzing its different connected components.
	
	The component $C_{nw}$ cannot be irreducible: if it were, it would be a good subset and we could consider, exactly as in the preceding point, the effect of adding $\tilde w$ (that is the vertex $w$ with the vectors hitting the vertices in $C_{nw}$). This configuration needs to have a vanishing determinant, but, as explained in the preceding point, this does not happen.
	
	The component $C_{3}$ cannot be irreducible: again, if it were, it'd be a good subset and adding the appropriate $\tilde w$ we would obtain a configuration with vanishing determinant, however this is not possible since we are dealing with a diagonally dominant intersection matrix.
	
	If $C_{lw}$ is irreducible, then it is a standard subset and, by the above arguments, we necessarily have that $C_{nw}\cup C_{3}$ is a good subset. If we are in case I above, that is $C_{lw}=C_{x}$, then there is precisely one basis vector, call it $g$ hitting $w$ and $C_{x}$ (if there were more than one, by considering $\tilde w\cup C_{x}$ we would have a contradiction with the determinant). Moreover, since $w$ is linked to $C_{3}$ there has to be a basis vector different from $g$ which links $w$ and $C_{3}$. With these considerations in mind, it is easy to see that $I(\Gamma_{1}')\leq -5$. (The computation is the same as in the preceding bullet point but now we can guarantee the existence of two more basis vectors hitting $w$ which account for the bound going from $-3$ to $-5$.). The component $C_{x}$ cannot have $I(C_{x})=-3$, since standard subsets with $I=-3$ have an embedding in which a basis vector appears only once, see \cite[Proposition~6.1]{Lisca-ribbon},  and this contradicts Lemma~\ref{l:no1}. Moreover, it follows directly from Lemma~\ref{badcomponents}, that $C_{3}$ is not a bad component, and therefore, using Lemma~\ref{badcomponents2}, we conclude that in this Case~I we have $I(C_{x})=-2$ and $I(C_{v}\cup C_{3})=-3$ where $C_{v}$ is a bad component. The concrete structure of $C_{v}\cup C_{3}$ is given in Lemma~\ref{I=-2,b=1} and it is immediate to check that these conditions are incompatible with the qc-leg structure. We conclude that in Case~I $C_{lw}$ is not irreducible, which implies that $\Gamma_{1}'$ is a good subset.
	
	We have one more case to discuss, that is $C_{lw}$ irreducible under the assumptions of Case II above. We are now working with $C_{lw}=C_{v}$ and $C_{nv}=C_{x}$. The linear set $\tilde w\cup C_{x}\cup C_{3}$ needs to have vanishing determinant, and this is only possible if $w$ is hit precisely by one vector appearing in the embedding of the good subset $C_{x}\cup C_{3}$ and moreover, $C_{x}$ and $C_{3}$ need to be complementary strings. Once again, this is immediately seen to be in contradiction with the structure of the qc-legs. 
	
	We have shown so fart that the set $\Gamma_{1}'$ is irreducible, and therefore a good subset, with 3 connected components and $I(\Gamma_{1}')\leq -3$ (by the same argument used in the previous bullet point). However, this contradicts \cite[Lemma~4.10]{Lisca-sums} (recall that $C_{3}$ is not a bad component of $\Gamma_{1}'$).
\end{itemize}

\item If $w=x$ then, as shown before, we have that $w$ is hit by all the vertices in $S$ with coefficient $\pm1$. Notice that by Lemma~\ref{l:no1} and since $e_{2}$ and $e_{3}$ hit only the vertices $s,v,w\in\Gamma_{1}$, we have $e_{r+\ell-1}\cdot w\neq0$. This implies, since $w=x$, that $e_{r+\ell-1}\cdot w=\pm2$. As in the preceding case, we proceed to delete from the graph the $\ell+3$ vertices $s,w$, the central vertex, and the whole 2-chain $c$ and to delete from the embedding the $\ell+3$ basis vectors in $S\cup\{e_{r+\ell-1}\}$. The linear graph we obtain $\Gamma'_{1}$, whose vertices have all at least weight 2, and has at most 2 connected components (it is indeed a connected graph if $w$ is final). It  satisfies $I(\Gamma'_{1})\leq I(\Gamma_{1})+2+\ell-2-3-2-(\ell-1)-4+3\leq -5$, so if $\Gamma_{1}'$ is good, which in this case is equivalent to being irreducible, we obtain a contradiction with Lemma~\ref{badcomponents2}. If $\Gamma_{1}'$ were reducible, then it necessarily has two connected components and both of them have $I\leq -2$. The component linked to $w$, call it $L_{w}$, is a good subset and needs to be linked to $w$ via a basis vector, call it $g_{1}$ that has not been accounted for. Moreover, since $g_{1}$ hits at least two different vertices in $L_{w}$ there is at least a basis vector $g_{2}$ hitting $w$ different from all the ones accounted so far and guaranteeing orthogonality. This implies that the estimate we had for $I(\Gamma_{1}')$ can be improved in this case to $I(\Gamma_{1}')\leq -7$, which in turn implies that at least one of its good connected components has $I\leq -4$ which contradicts again Lemma~\ref{badcomponents2}.
\end{enumerate}
\end{proof}

We are now ready to finish the proof of Proposition~\ref{p:mainA} which involves the construction of the fourth family in Theorem~\ref{maintheorem}.

\begin{proof}[End of proof of second statement of Proposition~\ref{p:mainA}]

We have shown so far that, under the assumptions of Proposition~\ref{p:mainA}, if a graph $\Gamma$ embeds, then it belongs to the family $\Gamma_{1}$ with $N=1$ and $1\leq k\leq 6$. In the next figure we summarize what we know so far about these embeddings, from the qc-leg structure and from Lemma~\ref{l:adjacent}. Notice that due to space constraints the 2-leg is depicted to the left and the qc-legs are both to the right of the trivalent vertex.
\[
  \begin{tikzpicture}[xscale=1.5,yscale=-0.5]
    \node (A3_0) at (3.75,6) {$\scriptstyle e_{2}-e_{3}+v'$};
    \node at (6.75,6) {$\scriptstyle \pm(e_{2}-e_{3}-e_{1}+\sum_{i=0}^{\ell-1}(-1)^{i}e_{r+i}+w')$};
	\node (A1_3) at (3.75,5) {$\bullet$};
	\node at (3.75,4.5) {$\scriptstyle v$};
	\node (A) at (6.75,5) {$\bullet$};
	\node at (6.75,4.5) {$\scriptstyle w$};
	\node (D) at (8.75,5) {$\dots$};
	\node (A3) at (3.75,0) {$\scriptstyle e_{1}+e_{r}$};
	\node (A1) at (3.75,1) {$\bullet$};
    \node (A2_3) at (3.5, 3) {$\scriptstyle e_{1}+e_{2}$};
    \node (A3_3) at (3, 3) {$\bullet$};
    \node (A4_4) at (2, 2) {$\scriptstyle e_{2}+e_{3}$};
    \node (A5_4) at (2, 3) {$\bullet$};
    \node (F) at (4.75, 1) {$\bullet$};
    \node (J) at (5.75, 1) {$\dots$};
    \node (K) at (6.75, 1) {$\bullet$};
    \node (L) at (8.25, 1) {$\bullet$};
    \node (M) at (9.25, 1) {$\bullet$};
    \node (N) at (10.25, 1) {$\bullet$};
    \node (O) at (11.25, 1) {$\dots$};
    \node at (6.75, 0) {$\scriptstyle e_{r+\ell-2}+e_{r+\ell-1}$};
    \node at (8.25, 1.5) {$\scriptstyle x$};
    \node at (8.25, 0) {$\scriptstyle e_{r+\ell-1}+g+f_{1}$};
    \node at (9.25, 0) {$\scriptstyle f_{1}+f_{2}$};
    \node at (10.25, 0) {$\scriptstyle f_{2}+f_{3}$};
    \node at (4.75, 0) {$\scriptstyle e_{r}+e_{r+1}$};
    \node at (3.75, 1.5) {$\scriptstyle u$};
    \path (A3_3) edge [-] node [auto] {$\scriptstyle{}$} (A1_3);
    \path (A3_3) edge [-] node [auto] {$\scriptstyle{}$} (A1);
    \path (A3_3) edge [-] node [auto] {$\scriptstyle{}$} (A5_4);
    \path (A1_3) edge [-] node [auto] {$\scriptstyle{}$} (A);
    \path (A) edge [-] node [auto] {$\scriptstyle{}$} (D);
    \path (A1) edge [-] node [auto] {$\scriptstyle{}$} (F);
    \path (F) edge [-] node [auto] {$\scriptstyle{}$} (J);
    \path (J) edge [-] node [auto] {$\scriptstyle{}$} (K);
    \path (K) edge [-] node [auto] {$\scriptstyle{}$} (L);
    \path (L) edge [-] node [auto] {$\scriptstyle{}$} (M);
    \path (M) edge [-] node [auto] {$\scriptstyle{}$} (N);
    \path (N) edge [-] node [auto] {$\scriptstyle{}$} (O);
   \end{tikzpicture}
  \]
The length of the 2-chain $c$ starting with the vertex $u$ is $\ell=|v'|\geq 1$, $|x|=3$ and to the right of $x$ there is a 2-chain of length at least 2.

We start by showing that $|w'|\geq 2$. Notice that $|w'|\geq 1$, since $w\cdot x=0$ and $e_{r+\ell-1}$ hits both $w$ and $x$. If $f_{1}\cdot w\neq 0$ then we have also that $f_{2}$ and $f_{3}$ hit $w$ and we are done. So, let us assume $w\cdot g\neq0$ and $w\cdot f_{1}=0$. By contradiction we assume $w'=\pm g$ and distinguish the following two possibilities:
	\begin{itemize}
		\item $w$ is a final vertex. In this scenario the 2-chain to the right of $x$, call it $c'$, has length precisely $\ell+2$. Since $|v'|=\ell$, we cannot have any of the $f_{i}'s$ in $c'$ hitting $v$. However, $v\cdot w=1$ forces at least one more basis vector, apart from $e_{2},e_{3}$, to hit both $v$ and $w$. The only possible one is $g$, but this implies that either $e_{r+\ell-1}$ or $f_{1}$ hits $v$, and both possibilities yield an immediate contradiction. 
		\item $w$ is an internal vertex. In this case $g$ appears at least 4 times in the embedding, since it hits $x,v,w$ and the vertex to the right of $w$. Moreover, since $e_{1}\cdot v=0$, and since we know by Lemma~\ref{l:adjacent} that $e_{2}$ and $e_{3}$ do not reappear in the embedding, it follows that $f_{1}$ and $f_{2}$ also appear each at least 4 times in the embedding. Indeed, they necessarily hit $v$, the vertex to the right of $w$ and they appear twice in their depicted 2-chain.
It follows that the contribution to $I(\Gamma_{1})$ of the set $\{e_{2},g,f_{1},f_{2}\}$ is at least 4. Moreover, it is immediate to check that none of these basis vectors belong to the set $E$ in Corollary~\ref{c:compensate}. It follows, by the same argument we have developed several times in the preceding proofs, that $I(\Gamma_{1})\geq 4-2+0=2$ which contradicts $I(\Gamma_{1})\leq 0$.
\end{itemize}

Our next goal in this final proof of this section is to establish that $|v'|=1$. Since the vectors in the set $S=\{e_{1},e_{2},e_{3},e_{r},\dots,e_{r+\ell-1}\}$ only hit the depicted vertices of the graph $\Gamma_{1}$, after deleting them from the embedding together with the vertex $w$ we obtain the following linear graph $L$:
\[
  \begin{tikzpicture}[xscale=1.5,yscale=-0.5]
	\node (A1_3) at (5,5) {$\bullet$};
	\node at (5,5.5) {$\scriptstyle v'$};
	\node (D) at (8.75,5) {$\bullet$};
	\node (E) at (9.75,5) {$\dots$};
    \node (L) at (8.25, 1) {$\bullet$};
    \node (M) at (9.25, 1) {$\bullet$};
    \node (N) at (10.25, 1) {$\bullet$};
    \node at (8.75, 5.5) {$\scriptstyle w_{+}$};
    \node (O) at (11.25, 1) {$\dots$};
    \node at (8.25, 1.5) {$\scriptstyle x$};
    \node at (8.25, 0) {$\scriptstyle g+f_{1}$};
    \node at (9.25, 0) {$\scriptstyle f_{1}+f_{2}$};
    \node at (10.25, 0) {$\scriptstyle f_{2}+f_{3}$};
    \path (D) edge [-] node [auto] {$\scriptstyle{}$} (E);
    \path (L) edge [-] node [auto] {$\scriptstyle{}$} (M);
    \path (M) edge [-] node [auto] {$\scriptstyle{}$} (N);
    \path (N) edge [-] node [auto] {$\scriptstyle{}$} (O);
   \end{tikzpicture}
  \]
which admits an embedding into a lattice of rank its number of vertices, since we have deleted from $\Gamma_{1}$ the same number of vertices as basis vectors from the embedding. Moreover, this graph satisfies $I(L)=I(\Gamma_{1})-|w'|-1\leq -3$. The one-vertex component cannot be a bad component, since bad components have at least 3 vertices; the component starting with $w_{+}$ (which exists if $w$ is internal) cannot be bad by a direct application of Lemma~\ref{badcomponents}; finally, the component $L_{x}$ starting with the vertex labelled $x$ is not bad either: we can apply Lemma~\ref{badcomponents} when we consider $L_{x}\cup\tilde w$, where $\tilde w$ is $w$ without the basis vector $e_{r+\ell-1}$. So, since we have a three component set with $I\leq -3$ and no bad components, we conclude,  by \cite[Lemma~4.9]{Lisca-sums} and Lemma~\ref{badcomponents2} that $L$ is not a good set.
\begin{itemize}
\item If $L$ has two connected components, meaning $w$ was final in $\Gamma_{1}$, then, since $L$ is not a good set, there must be a unique basis vector, $e$, hitting $v'$. The coefficient of $e$ is necessarily $\pm 1$ since, if it were in absolute value greater or equal than 2, it would imply that the connected component starting with $x$, $L_{x}$, is a standard subset with $I(L_{x})\leq-4$ which contradicts Lemma~\ref{badcomponents2}. So, in this case $|v'|=\ell=1$, as claimed. 
\item The only possibility left to consider is when $L$ has 3 connected components, call them $v', L_{x}$ and $L_{+}$. Since $L$ is not good, it cannot consist of a unique irreducible set. (The only way this could happen is having $|v'|=1$, so that the condition of having every vertex with weight at least two is violated; however, this would make the component $v'$ irreducible.) It follows that $L$ can be decomposed into 2 or 3 irreducible sets. 

We claim that $L_{+}$ cannot be an irreducible good set of $L$. If it were, then $L_{+}$ has to be standard and in $\Gamma_{1}$ all the basis vectors hitting the vertices in $L_{+}$ would not appear elsewhere in the embedding except on $w$. It follows that deleting from $\Gamma_{1}$ all the basis vectors that do not hit $L_{+}$ we would obtain a linear set $\tilde w\cup L_{x}$ with non-vanishing determinant that admits an embedding into a standard lattice of rank one less (the one spanned by the vectors hitting the standard set $L_{+}$) than its number of vertices ($|L_{+}|+1$), a contradiction. So, assuming by contradiction that $|v'|\geq 2$, the irreducible sets of $L$ are either $L_{+}\cup v'$ and $L_{x}$ or $L_{+}\cup L_{x}$ and $v'$. The latter decomposition is not possible: 
since $L$ embeds in a lattice of rank its number of vertices and $L_{+}\cup L_{x}$ is under these assumptions a good set, there is only one basis vector left in the embedding of $L$ that can hit $v'$. Since we are assuming $|v'|\geq 2$, this means that the basis vector hitting $v'$ has a coefficient, and thus we actually have $|v'|\geq 4$. This configuration contradicts Lemma~\ref{badcomponents2}, since $I(L)=I(v')+I(L_{x}\cup L_{+})\leq -3$, which implies $I(L_{x}\cup L_{+})\leq -4$, but $L_{x}$ and $L_{+}$ are not bad components.

 The final task is excluding $L_{+}\cup v'$ and $L_{x}$ from being the irreducible good components of $L$ while $|v'|\geq 2$. In this case $L_{x}$ is a standard set. If we delete from $\Gamma_{1}$ all the basis vectors that do not hit $L_{x}$, we are left with $L_{x}$ linked once with the vertex called $w$ in $\Gamma_{1}$. Just as in the preceding paragraph, this connected linear set has non vanishing determinant while it embeds in a lattice of rank smaller than its number of vertices, a contradiction.
\end{itemize}

The above arguments show that $|v'|=1$ and now we move on to arguing  that $|w'|=2$ and that $w$ is not a final vertex. Notice that since $|v'|=1$ there is some basis vector $h$ such that $v=e_{2}-e_{3}+h$. Moreover, $w\cdot v=1$ implies that $h\cdot w\neq 0$. There are two possibilities, either  $w=-e_{2}+e_{3}+3h+\dots$ or $w=e_{2}-e_{3}-h+\dots$. The former one can be excluded, since it implies that the vector $h$ contributes at least 7 to $I(\Gamma_{1})$, $h\not\in E$ and therefore in this scenario we arrive to the contradiction $I(\Gamma_{1})\geq 7-2+0=5$. We now proceed to analyze the graph $\Gamma'_{1}$ we obtain by deleting from the embedding all the basis vectors in the set $S$.
\[
  \begin{tikzpicture}[xscale=1.5,yscale=-0.5]
    \node (A3_0) at (4.5,5.75) {$\scriptstyle h$};
    \node at (5.65,4.75) {$\scriptstyle -$};
    \node at (7.75,3) {$\scriptstyle -$};
    \node at (6.75,5.75) {$\scriptstyle -h+\dots$};
	\node (A1_3) at (4.5,5) {$\bullet$};
	\node at (4.5,4) {$\scriptstyle v$};
	\node (A) at (6.75,5) {$\bullet$};
	\node at (6.75,4) {$\scriptstyle w$};
	\node (D) at (8.75,5) {$\dots$};
    \node (L) at (8.25, 1) {$\bullet$};
    \node (M) at (9.25, 1) {$\bullet$};
    \node (N) at (10.25, 1) {$\bullet$};
    \node (O) at (11.25, 1) {$\dots$};
    \node at (8.25, 1.5) {$\scriptstyle x$};
    \node at (8.25, 0) {$\scriptstyle g+f_{1}$};
    \node at (9.25, 0) {$\scriptstyle f_{1}+f_{2}$};
    \node at (10.25, 0) {$\scriptstyle f_{2}+f_{3}$};
    \path (A1_3) edge [-] node [auto] {$\scriptstyle{}$} (A);
    \path (A) edge [-] node [auto] {$\scriptstyle{}$} (D);
    \path (A) edge [-] node [auto] {$\scriptstyle{}$} (L);
    \path (L) edge [-] node [auto] {$\scriptstyle{}$} (M);
    \path (M) edge [-] node [auto] {$\scriptstyle{}$} (N);
    \path (N) edge [-] node [auto] {$\scriptstyle{}$} (O);
   \end{tikzpicture}
  \]
The determinant of this graph needs to vanish, since it admits an embedding into a standard lattice of rank smaller than its number of vertices. For this to happen, we need at least one vertex with weight strictly less than its valency. The only possible such vertex is $w$, if it has valency 3 (so $w$ is not final) and $w'=-g-h$. 

Thus we have that the configuration we are dealing with is necessarily the following:
\[
  \begin{tikzpicture}[xscale=1.5,yscale=-0.5]
    \node (A3_0) at (3.75,5.75) {$\scriptstyle e_{2}-e_{3}+h$};
    \node at (5.75,5.75) {$\scriptstyle e_{2}-e_{3}-e_{1}+e_{r}-h-g$};
	\node (A1_3) at (3.75,5) {$\bullet$};
	\node at (3.75,4.5) {$\scriptstyle v$};
	\node (A) at (5.75,5) {$\bullet$};
	\node at (5.75,4.5) {$\scriptstyle w$};
	\node (D) at (7,5) {$\bullet$};
	\node (E) at (8,5) {$\dots$};
	\node (A3) at (3.75,0) {$\scriptstyle e_{1}+e_{r}$};
	\node (A1) at (3.75,1) {$\bullet$};
    \node (A2_3) at (3.5, 3) {$\scriptstyle e_{1}+e_{2}$};
    \node (A3_3) at (3, 3) {$\bullet$};
    \node (A4_4) at (2, 2) {$\scriptstyle e_{2}+e_{3}$};
    \node (A5_4) at (2, 3) {$\bullet$};
    \node (F) at (4.75, 1) {$\bullet$};
    \node (J) at (5.75, 1) {$\bullet$};
    \node (K) at (6.75, 1) {$\bullet$};
    \node (L) at (7.75, 1) {$\bullet$};
    \node (M) at (8.75, 1) {$\bullet$};
    \node (N) at (9.75, 1) {$\dots$};
    \node at (4.75, 1.5) {$\scriptstyle x$};
    \node at (4.75, 0) {$\scriptstyle e_{r}+g+f_{1}$};
    \node at (5.75, 0) {$\scriptstyle f_{1}+f_{2}$};
    \node at (6.75, 0) {$\scriptstyle f_{2}+f_{3}$};
    \node at (7.75, 0) {$\scriptstyle f_{3}+f_{4}$};
    \node at (3.75, 1.5) {$\scriptstyle u$};
    \path (A3_3) edge [-] node [auto] {$\scriptstyle{}$} (A1_3);
    \path (A3_3) edge [-] node [auto] {$\scriptstyle{}$} (A1);
    \path (A3_3) edge [-] node [auto] {$\scriptstyle{}$} (A5_4);
    \path (A1_3) edge [-] node [auto] {$\scriptstyle{}$} (A);
    \path (A) edge [-] node [auto] {$\scriptstyle{}$} (D);
    \path (D) edge [-] node [auto] {$\scriptstyle{}$} (E);
    \path (A1) edge [-] node [auto] {$\scriptstyle{}$} (F);
    \path (F) edge [-] node [auto] {$\scriptstyle{}$} (J);
    \path (J) edge [-] node [auto] {$\scriptstyle{}$} (K);
    \path (K) edge [-] node [auto] {$\scriptstyle{}$} (L);
    \path (L) edge [-] node [auto] {$\scriptstyle{}$} (M);
    \path (M) edge [-] node [auto] {$\scriptstyle{}$} (N);
   \end{tikzpicture}
  \]
Observe that the bottom leg begins $[3,6,\dots]$.
If, ignoring the final vertex $v_{k}$ of weight $k+1$ in the graph $\Gamma_{1}$, we have that the bottom complementary string is exactly $[3,6]$, then the other one is forced to be $[2,3,2,2,2,2]$. Moreover, since we have shown that the vertex $w$ is not final, the vertex $v_{k}$ needs to be linked to it. Since no vector in $S$ can hit $v_{k}$, it follows that $g\cdot v_{k}=-1$ and the embedding is forced, yielding the graph:
\[
  \begin{tikzpicture}[xscale=1.5,yscale=-0.5]
    \node (A3_0) at (3.75,5.75) {$\scriptstyle e_{2}-e_{3}+h$};
    \node at (5.75,5.75) {$\scriptstyle e_{2}-e_{3}-e_{1}+e_{r}-h-g$};
    \node at (8,5.75) {$\scriptstyle -g+f_{1}-f_{2}+f_{3}-f_{4}+f_{5}$};
    \node at (8,4.5) {$\scriptstyle v_{k}$};
	\node (A1_3) at (3.75,5) {$\bullet$};
	\node at (3.75,4.5) {$\scriptstyle v$};
	\node (A) at (5.75,5) {$\bullet$};
	\node at (5.75,4.5) {$\scriptstyle w$};
	\node (D) at (8,5) {$\bullet$};
	\node (A3) at (3.75,0) {$\scriptstyle e_{1}+e_{r}$};
	\node (A1) at (3.75,1) {$\bullet$};
    \node (A2_3) at (3.5, 3) {$\scriptstyle e_{1}+e_{2}$};
    \node (A3_3) at (3, 3) {$\bullet$};
    \node (A4_4) at (2, 2) {$\scriptstyle e_{2}+e_{3}$};
    \node (A5_4) at (2, 3) {$\bullet$};
    \node (F) at (4.75, 1) {$\bullet$};
    \node (J) at (5.75, 1) {$\bullet$};
    \node (K) at (6.75, 1) {$\bullet$};
    \node (L) at (7.75, 1) {$\bullet$};
    \node (M) at (8.75, 1) {$\bullet$};
    \node at (4.75, 1.5) {$\scriptstyle x$};
    \node at (4.75, 0) {$\scriptstyle e_{r}+g+f_{1}$};
    \node at (5.75, 0) {$\scriptstyle f_{1}+f_{2}$};
    \node at (6.75, 0) {$\scriptstyle f_{2}+f_{3}$};
    \node at (7.75, 0) {$\scriptstyle f_{3}+f_{4}$};
    \node at (8.75, 0) {$\scriptstyle f_{4}+f_{5}$};
    \node at (3.75, 1.5) {$\scriptstyle u$};
    \path (A3_3) edge [-] node [auto] {$\scriptstyle{}$} (A1_3);
    \path (A3_3) edge [-] node [auto] {$\scriptstyle{}$} (A1);
    \path (A3_3) edge [-] node [auto] {$\scriptstyle{}$} (A5_4);
    \path (A1_3) edge [-] node [auto] {$\scriptstyle{}$} (A);
    \path (A) edge [-] node [auto] {$\scriptstyle{}$} (D);
    \path (A1) edge [-] node [auto] {$\scriptstyle{}$} (F);
    \path (F) edge [-] node [auto] {$\scriptstyle{}$} (J);
    \path (J) edge [-] node [auto] {$\scriptstyle{}$} (K);
    \path (K) edge [-] node [auto] {$\scriptstyle{}$} (L);
    \path (L) edge [-] node [auto] {$\scriptstyle{}$} (M);
   \end{tikzpicture}
  \]
which is the first one in the infinite family described in the statement of Proposition~\ref{p:mainA}. 

To end the proof we analyze how the above graph yields the family in the statement, by induction on the number $n$ of vertices in the complementary string in the bottom leg. We have treated the initial case of $n=2$. If $n=3$, then, since the vertex to the right of $w$, call it $w_{+}$, is linked to $w$ via $g$, the embedding forces $w_{+}=-g+f_{1}-f_{2}+f_{3}-f_{4}+f_{5}+w_{+}'$ and $|w_{+}|\geq 6$. This in turn forces the other complementary string to be $[2,3,2,2,2,3,2,2,2,2,\dots]$, with forced embedding $(e_{1}+e_{r},e_{r}+g+f_{1},f_{1}+f_{2},\dots,f_{4}+f_{5}+t_{1},t_{1}+t_{2},t_{2}+t_{3},\dots)$.

We claim that $|w_{+}'|=0$ and that $w_{+}$ is not final in $\Gamma_{1}$. The argument is verbatim the one we used before to show that $|w'|=2$ and $w$ not final. This time we consider the graph obtained from $\Gamma_{1}$ by deleting from the embedding all the vectors in the set $\{e_{1},e_{2},e_{3},e_{r},h,g,f_{1},f_{2},f_{3},f_{4}\}$. This amounts to deleting 10 basis vectors and 9 vertices in the diagram, and therefore the resulting linear graph $L$ must have vanishing determinant:
\[
  \begin{tikzpicture}[xscale=1.5,yscale=-0.5]
    \node at (7.75,3) {$\scriptstyle -$};
    \node at (6.75,5.75) {$\scriptstyle f_{5}+w'_{+}$};
	\node (A) at (6.75,5) {$\bullet$};
	\node at (6.75,4) {$\scriptstyle w_{+}$};
	\node (D) at (8.75,5) {$\dots$};
    \node (L) at (8.25, 1) {$\bullet$};
    \node (M) at (9.25, 1) {$\bullet$};
    \node (N) at (10.25, 1) {$\bullet$};
    \node (O) at (11.25, 1) {$\dots$};
    \node at (8.25, 0) {$\scriptstyle f_{5}+t_{1}$};
    \node at (9.25, 0) {$\scriptstyle t_{1}+t_{2}$};
    \node at (10.25, 0) {$\scriptstyle t_{2}+t_{3}$};
    \path (A) edge [-] node [auto] {$\scriptstyle{}$} (D);
    \path (A) edge [-] node [auto] {$\scriptstyle{}$} (L);
    \path (L) edge [-] node [auto] {$\scriptstyle{}$} (M);
    \path (M) edge [-] node [auto] {$\scriptstyle{}$} (N);
    \path (N) edge [-] node [auto] {$\scriptstyle{}$} (O);
   \end{tikzpicture}
  \]
There is only one possibility for the determinant to be zero: $w_{+}$ is not final in $\Gamma_{1}$ and $|w_{+}'|=0$. 

Summing up, we have established that if $n=3$ the complementary strings are precisely $[3,6,6]$ and $[2,3,2,2,2,3,2,2,2,2]$ and the final vertex with weight $k+1$ needs to be linked to the final 6 in the complementary string. The embedding then forces this final vertex to have weight 6 and embedding $f_{5}-t_{1}+t_{2}-t_{3}+t_{4}-t_{5}$. This is the second member of the family in the statement. From here, a clear induction yields the whole family and the proposition is proved.
\end{proof} 
\end{subsection}

\end{section}

\appendix
\section{Lens space and reducible surgeries}\label{s:appendix}

In this appendix we will classify all pairs $(p,q)$ such that $S^3_{pq\pm 1}(T_{p,q})$ or $S^3_{pq}(T_{p,q})$ bounds a rational homology ball.
By work of Moser~\cite{Moser}, $S^3_{pq\pm 1}(T_{p,q})$ is the lens space $L(pq\pm 1, \mp q^2) = L(pq\pm 1, \mp p^2)$, while $S^3_{pq}(T_{p,q})$ is a connected sum of two lens spaces, namely $L(p,-q) \# L(q,-p)$.
Hence we can use Lisca's classification of which lens spaces~\cite{Lisca-ribbon} and sums of lens spaces~\cite{Lisca-sums} bound rational homology balls.

For convenience, we recall the definition of the Fibonacci sequence $\{F_n\}$, and the sequences $\{S_n\}$ and $\{T_n\}$.
\[
\left\{
\begin{array}{l}
F_0 = 0,\\
F_1 = 1,\\
F_{n+1} = F_n + F_{n-1};
\end{array}
\right.
\quad
\left\{
\begin{array}{l}
S_0 = 1,\\
S_1 = 1,\\
S_{n+1} = 6S_n - S_{n-1};
\end{array}
\right.
\quad
\left\{
\begin{array}{l}
T_0 = 0,\\
T_1 = 1,\\
T_{n+1} = 6T_n - T_{n-1}.
\end{array}
\right.
\]
The first few values of $F_n$ are: $0, 1, 1, 2, 3, 5, 8, 13, 21, 34, 55, \dots$
The first few values of $S_n$ are: $1, 1, 5, 29, 169, 985, \dots$, while the first few values of $T_n$ are $0, 1,6,35,204,1189\dots$.

The main results of this appendix are the following.

\begin{te}\label{reduciblebounding}
The $3$--manifold $S^3_{pq}(T_{p,q})$ bounds a rational homology ball if and only if $(p,q;pq) \in \Rc$.
\end{te}

\begin{te}\label{lensspacesbounding}
The $3$--manifold $S^3_{pq\pm 1}(T_{p,q})$ bounds a rational homology ball if and only if $(p,q;pq\pm1) \in \Lc$.
\end{te}

For convenience, we list here the pairs $(p,q)$ for which $(p,q;pq) \in \Rc$:
\begin{align*}
&(r^2,(r+1)^2), (F_n^2,F_{n+1}^2), (s^2, (2s\pm1)^2), (t^2, (2t\pm 2)^2), (S_n^2,4T_n^2), \\
&(4T_n^2,S_{n+1}^2), (F_{2n-1}^2,F_{2n+1}^2), (9^2, 14^2), (11^2, 14^2),
\end{align*}
where $n \ge 2$ (except for the pair $(F_2^2,F_{3}^2) = (1,4)$), $r \ge 2$, $s \equiv 2 \pmod 4$ and $s\ge 6$, and $t \equiv \mp 3 \pmod 8$ and $t\ge 5$.
The pairs $(p,q)$ for which $(p,q;pq+1) \in \Lc$ are:
\[
(2r+1,2r+3), (F_{2n},F_{2n+2}),
\]
with $r, n\ge 1$.
Finally, the pairs $(p,q)$ for which $(p,q;pq-1) \in \Lc$ are:
\[
(F_{2n+1},F_{2n+3}), (S_{n+1},S_{n+2}), (F_{2n+1},F_{2n+5}),
\]
with $n \ge 1$.

The strategy in the reducible case is the following; first we prove that $S^3_{pq}(T_{p,q})$ bounds if and only if each of the two summands bound, i.e. if $L(p,-q)$ and $L(q,-p)$ both bound rational homology balls.
Then, we notice that the plumbing graph for $L(p,-q)$ is obtained from that of $L(q,-p)$ by adding a single vertex with weight $k+1 = \lceil q/p \rceil$.
We will call the two plumbing graphs the \emph{short} and \emph{long} graph (or diagram, or plumbing).

We then go through Lisca's list of (plumbing graphs of) lens spaces that bound rational homology balls, and check case by case when this can happen.
There are eleven different families of plumbing graphs, so eleven candidates for the short diagram, and eleven for the long diagram.
Each combination determines the value of $k$;
some (few) of them can be excluded by using the a priori restriction $1\le k \le 7$ (see Proposition~\ref{pruning1}).
Whenever we find a short and a long graph, we check that their connected sum indeed arises as a surgery along a torus knot.

In the lens space case, we are in a similar situation: $S^3_{pq\pm1}(T_{p,q}) = \pm L(pq\pm1,p^2)$;
however,
\[
\frac{pq\pm 1}{p^2} = \frac{p(kp+r)\pm 1}{p^2} = \left[k+1,\frac{p^2}{p(p-r)\mp 1}\right]^-,
\]
so the associated continued fraction starts with $k+1$, and then contains the continued fraction of a lens space that bounds a rational homology ball.
That is, removing the vertex with weight $k+1$ from the graph, we obtain a lens space that again bounds a rational homology ball; moreover, the \emph{short} graph belongs to a specific subfamily (namely the families $A$ and $E$ in Section~\ref{ss:bounding-lens} below).

\begin{lemma}
If $p$ and $q$ are coprime, then $L(p,q)\# L(q,p)$ bounds a rational homology ball if and only if $L(p,q)$ and $L(q,p)$ both do.
\end{lemma}

\begin{proof}
In Lisca's list of linear relations among lens spaces in the rational homology cobordism groups~\cite{Lisca-ribbon}, by a direct check, there is no relation of the form $L(a,b)\# L(a',b')$ with $\gcd(a,a') = 1$.
\end{proof}

The rest of the section contains some background on Riemenschneider diagrams, continued fractions, and Lisca's classification on lens spaces that bound rational homology balls; finally, we prove the classification of surgeries bounding rational balls in the reducible case, and we then deduce the classification in the lens space case.

\subsection{Riemenschneider diagrams}\label{ss:riemen}

To a linear plumbing with weights $\mathbf{a} = [a_1,\dots,a_m]$ we associate a \emph{Riemenschneider diagram}, or \emph{R-diagram} for short, consisting of $m$ rows of dots, one at each height $m, m-1, \dots, 0$.
For each $j\ge 1$, the $j$th line contains $a_j-1$ points, sitting on the line $y=m-j$, and with $x$--coordinades $\sum_{i < j} (a_i-2),\sum_{i < j} (a_i-2)+1,\dots,(\sum_{i\le j} (a_i-2)$.
(We consider all such diagrams up to translations, for convenience; the concrete realization in the plane is of no importance.)
For example, the diagram associated to $[2,2,3,5]$ is
\[
  \begin{tikzpicture}[xscale=-0.35,yscale=-0.35]
    \node at (-1, 4) {$\bullet$};
    \node at (-1, 3) {$\bullet$};
    \node at (-1, 2) {$\bullet$};
    \node at (-2, 4) {$\bullet$};
    \node at (-2, 5) {$\bullet$};
    \node at (-3, 5) {$\bullet$};
    \node at (-4, 5) {$\bullet$};
    \node at (-5, 5) {$\bullet$};
  \end{tikzpicture}
\]
The string $\mathbf{a}'$ dual to $\mathbf{a}$ is obtained by Riemenschneider's point rule;
translated geometrically, it is the string $\mathbf{a}$ whose R-diagram is obtained from the R-diagram of $\mathbf{a}$ by reflection across the line $x+y = 0$ in the plane.
Continuing with the above example, the string dual to $[2,2,3,5]$ is $[4,3,2,2,2]$ with associated R-diagram
\[
  \begin{tikzpicture}[xscale=0.35,yscale=0.35]
    \node at (0, 0) {$\bullet$};
    \node at (1, 0) {$\bullet$};
    \node at (2, 0) {$\bullet$};
    \node at (2, -1) {$\bullet$};
    \node at (3, -1) {$\bullet$};
    \node at (3, -2) {$\bullet$};
    \node at (3, -3) {$\bullet$};
    \node at (3, -4) {$\bullet$};
  \end{tikzpicture}
\]

The following two remarks are elementary, but extremely useful.
\begin{re}\label{complementary-cloud-count}
In particular, both a string and its dual have $\sum_j (a_j -1)$ points in their R-diagram;
in analogy with Lisca's $I$-function, we call $J(\mathbf{a}) = \sum_j (a_j - 1)$, so that the R-diagram of $\mathbf{a}$ comprises $J(\mathbf{a})$ points.
\end{re}

\begin{re}\label{complementary-twos}
The string $[2]$ is self-dual;
if ${\bf b} \neq [2]$ is dual to $\bf c$, then exactly one among $b_1$ and $c_1$ is equal to 2, and exactly one among the last entry in the string $\bf b$ and  the last entry in the string $\bf c$ is equal to 2.
\end{re}

\subsection{Lens spaces bounding rational homology balls}\label{ss:bounding-lens}

We recast here Lisca's classification in terms of plumbing graphs.
This classification is implicit (and incomplete) in Lisca's paper~\cite{Lisca-ribbon}, and was completed by the fourth author~\cite{Lecuona}.
Here we get rid of some redundancies in the latter list for practical reasons, and we also make explicit the dual graphs to her three groups of families
\footnote{The differences are due to the relations: $B^1_{s,t} = B^1_{t,s}$, $B^1_{s,0} = B^2_{s,0}$, $B^2_{s,0} = B^3_{[s+2],[2^{[s+1]}]}$, and $C^1_{0,t} = C^2_{0,t}$.
}.

The lens space $L(p,q)$ will be associated to the linear plumbing graph with weights ${\bf a} = [a_1,\dots,a_m]$, where ${\bf a}$ is the negative continued fraction expansion of $p/q$.
For a lens space that bounds a rational homology ball, the associated plumbing belongs to one of the families $A, B^1,B^2,B^3,C^1,C^2,C^3,D^1,D^2,D^3,$ or $E$, described below.

Recall that ${\bf b} = [b_1,\dots,b_h]$ and ${\bf c} =[c_1,\dots,c_\ell]$ are complementary if $[{\bf b}]^- = x/y$ and $[{\bf c}]^- = x/(x-y)$.
As customary, $y^*$ will denote the representative of the inverse of $y$ modulo $x$ with $0<y^*<x$.
We recall that, for a string $\mathbf{s}$ and a non-negative integer $t$, we use the shorthand notation $\mathbf{s}^{[t]}$ for the concatenation of $t$ copies of $\mathbf{s}$.

From this point on, ${\bf b}$ and $\bf c$ will always be two complementary strings.

All facts about the families $A$, $B^3$, $D^3$, and $E$ follow from direct inspection of their R-diagrams, and from elementary computations with continued fractions.
The lens space associated to a string $\bf a$ will be the boundary of the linear plumbing with \emph{negative} weights $-a_1,\dots,-a_m$.
This is done to be more consistent with Lisca's and the fourth author's conventions.

\begin{itemize}
\item[(I)] $I = -3$: $A_{\bf b} = [b_h,\dots,b_1,2,c_1,\dots,c_\ell]$; equivalently, ${\bf a}$ is obtained from $[2,2,2]$ by final 2-expansions.
We have $[A_{\bf b}]^- = x^2/(xy^*+1)$, so the corresponding lens space is $L(x^2, xy^* + 1) = L(x^2,x(x-y^*)+1)$.
The R-diagram of $A_{\bf b}$ has a row containing a single point, such that the number of points above it equals the number of points below it;
we will call this point the \emph{center} of (the R-diagram of) $A_{\bf b}$.
Moreover, unless $\bb = [2]$, exactly one of the two adjacent rows contains exactly one point, and exactly one end of the string $A_{\bf b}$ is a 2.

\item[(II)] $I = -2$: there are three families:
\begin{enumerate}
\item $B^1_{s,t} = [2^{[t]},3,s+2,t+2,3,2^{[s]}]$, with $s\ge t > 0$;
\item $B^2_{s,t} = [2^{[t]},s+3,2,t+2,3,2^{[s]}]$, with $s\ge0$, $t > 0$;
\item $B^3_{\bf b} = [b_h,\dots,b_2,b_1+1,2,2,c_1+1,c_2,\dots,c_\ell]$;
equivalently, $B^3_{\bf b}$ is obtained from $[3,2,2,3]$ by final 2-expansions.
We have $[B^3_{\bf b}]^- = 4x^2/(4xy^*-1)$, so the corresponding lens space is $L(4x^2, 4xy^* - 1) = L(4x^2,4x(x-y^*)-1)$.
The R-diagram of $B^3_{\bf b}$ has two consecutive rows containing a single point, such that the number of points above it equals the number of points below it;
we will call these two points the \emph{center} of (the R-diagram of) $B^3_{\bf b}$.
Moreover, unless $\bb = [2]$, exactly one of the two adjacent rows contains exactly two points, and exactly one end of the string $B^3_{\bf b}$ is a 2.
\end{enumerate}

\item[(III)] $I = -1$: there are three families:
\begin{enumerate}
\item $C^1_{s,t} = [t+2,s+2,3,2^{[t]},4,2^{[s]}]$, with $s > 0$, $t \ge 0$;
\item $C^2_{s,t} = [t+2,2,s+3,2^{[t]},4,2^{[s]}]$, with $s,t \ge 0$;
\item $C^3_{s,t} = [t+3,2,s+3,3,2^{[t]},3,2^{[s]}]$, with $s, t \ge 0$.
\end{enumerate}

\item[(IV)] $I = 0$: there are three families:
\begin{enumerate}
\item $D^1_{s,t} = [t+3,3,2^{[s]},3,2^{[t]},3,s+3]$, $s \ge t \ge 0$;
\item $D^2_{s,t} = [t+3,2^{[s]},4,2^{[t]},3,s+2]$, $s,t\ge 0$;
\item $D^3_{\bf b} = [b_h,\dots,b_1,5,c_1,\dots,c_\ell]$;
equivalently, $D^3_{\bf b}$ is obtained from $[2,5,2]$ by final 2-expansions.
We have $[D^3_{\bf b}]^- = 4x^2/(4xy^*+1)$, so the corresponding lens space is $L(4x^2, 4xy^* + 1) = L(4x^2,4x(x-y^*)+1)$.
The R-diagram of $D^3_{\bf b}$ has a row containing exactly four points, such that the number of points above it equals the number of points below it;
we will call these points the \emph{center} of (the R-diagram of) $D^3_{\bf b}$.
Moreover, unless $\bb = [2]$, exactly one of the two adjacent rows contains exactly one point, and exactly one end of the string $D^3_{\bf b}$ is a 2.
\end{enumerate}

\item[(V)] $I = 1$: $E_{\bf b} = [b_h,\dots,b_2,b_1+c_1,c_2,\dots,c_\ell]$;
equivalently, $E_{\bf b}$ is obtained from $[4]$ by final 2-expansions.
We have $[E_{\bf b}]^- = x^2/(xy^*-1)$, so the corresponding lens space is $L(x^2, xy^* - 1) = L(x^2,x(x-y^*)-1)$.
If $\bb \neq [2]$, the R-diagram of $E_{\bf b}$ has a row $r$ with at least four points, and the vertical line through either the second or the second-to-last point of $r$ divides the R-diagram into two parts with the same number of points;
we will call this point the \emph{center} of (the R-diagram of) $E_{\bf b}$.
Moreover, if ${\bf b} \neq [2]$, exactly one end of the string $E_{\bf b}$ is a 2.
\end{itemize}

The dualities are the following: $A_{\bf b}$ is dual to $E_{\bf b}$; $B^1_{s+1,t+1}$ is dual to $D^1_{s,t} $; $B^2_{s,t+1}$ is dual to $D^2_{s,t}$; $B^3_{\bf b}$ is dual to $D^3_{\bf b}$; $C^1_{s,t}$ is dual to $C^3_{s-1,t}$ for $s>0$, $t \ge 0$; $C^2_{s,t}$ is dual to $C^2_{t,s}$ with $s,t\ge 0$.

\subsection{Continued fractions and the recursive sequences}

We recall here a few useful facts about the sequences $\{F_n\}$, $\{S_n\}$, and $\{T_n\}$, and their relation to certain continued fraction expansions, whose proofs are entirely elementary;
these facts will be used later without explicit mention.

First, we recall that the Fibonacci numbers satisfy the well-known identities:
\[
F_n^2 = F_{n+1}F_{n-1} - (\pm1)^n, \quad F_{2n+1}^2 = F_{2n-1}F_{2n+1} - 1.
\]
We also observe that
\[
4T_nT_{n+1} + 1 = S_{n+1}^2, \quad S_nS_{n+1}-1 = 4T_n^2.
\]
Both identities can be proven by using the explicit formulae for $S_n$ and $T_n$ given below;
alternatively, one can observe that $4T_nT_{n+1} + 1$, $S_{n+1}^2$, $S_nS_{n+1}$, and $4T_n^2$ all satisfy the recursion $Q_{n+1} = 34Q_n - Q_{n-1}$, and the first two and last two take the same values for $n=1$ and $n=2$.

Finally, we recall that, for each integer $m$,
\[
\lim_{n\to \infty} \frac{F_{n+m}}{F_{n}} = \phi^m,
\]
and we claim that
\[
\lim_{n\to \infty} \frac{S_{n+1}}{S_n} = \psi = 3+2\sqrt2\quad {\rm and} \quad
\lim_{n\to \infty} \frac{2T_n}{S_n} = \lim_{n\to\infty} \frac{S_{n+1}}{2T_n} = 1+\sqrt2.
\]
The latter identity can be proved by computing an explicit formula for $S_n$ and $T_n$, namely:
\[
S_n = {\textstyle \frac{2-\sqrt2}4}\cdot \psi^n + {\textstyle \frac{2+\sqrt2}4}\cdot \psi^{-n} \quad {\rm and} \quad T_n = {\textstyle \frac{1}{8\sqrt2}}\cdot(\psi^n - \psi^{-n}),
\]
where $\psi = 3+2\sqrt2 = (1+\sqrt2)^2$ is the largest root of the polynomial $\lambda^2 - 6\lambda + 1$.

In what follows, by convention we are letting $F_{-1} = 1$.

\begin{itemize}
\item If $[{\bf b}]^- =  x/y$, then $[2^{[t]},{\bf b}]^- = \frac{x + t(x-y)}{y + t(x-y)}$.
\item $[3^{[t]},2]^- = F_{2t+3}/F_{2t+1}$ for $t \ge 0$.
\item $[3^{[t]}]^- = F_{2t+2}/F_{2t}$ for $t\ge 1$.
\item $[6^{[t]},5]^- = S_{t+2}/S_{t+1}$ for $t \ge 0$.
\item $[6^{[t]}]^- = T_{t+2}/T_{t+1}$ for $t \ge 1$.
\item $[7^{[t]},2]^- = F_{4t+3}/F_{4t-1}$ for $t \ge 0$.
\item $[7^{[t]},5]^- = F_{4t+5}/F_{4t+1}$ for $t \ge 0$.
\end{itemize}

\subsection{The reducible case}

The proof will be on a case-by-case analysis;
for the reader's convenience, we collect all the results in Table~\ref{f:reducible_table} below.

\begin{table}[h!]
\begin{center}
\begin{tabular}{c|c|c|c|c|c|c|c|c}
& $A$ & $B^2$ & $B^3$ & $C^1$ & $C^2$ & $C^3$ & $D^2$ & $E$\\
\hline
$A$ & ${\frac{F_{2s+4}^{2}}{F_{2s+3}^{2}}}$ {\tiny (9)} &  & $\frac{(2s+3)^2}{(2s+2)^2}$ \tiny{(6)} & $\times$ & $\times$ & $\times$ & $\times$ & $\times$ \\
\hline
$C^1$  &  &  & $\frac{(8s+11)^2}{(4s+6)^2}$ {\tiny (7)} &  & & &  & $\times$ \\
\hline
$C^2$  &  &  & $\frac{11^2}{6^2}$ {\tiny (7)} &  &  &  &  & $\times$ \\
\hline
$C^3$  &  &  &  & $\frac{14^2}{9^2}$ {\tiny (13)}  &  &  & $\frac{11^2}{9^2}$ {\tiny (13)} & $\times$ \\
\hline
$D^2$  &  &  & $\frac{(8s+13)^2}{(4s+6)^2}$ {\tiny (7)} &  &  &  & & \\
\hline
$D^3$ &  $\frac{4T_{s+2}^2}{S_{s+2}^2}$ {\tiny (11)} & $\frac{(16s+28)^2}{(8s+13)^2}$ {\tiny (8)} &  &  & $\frac{20^2}{11^2}$ {\tiny (8)} & $\frac{(16s+36)^2}{(8s+19)^2}$ {\tiny (8)} &  &  $\frac{(2s+2)^2}{(2s+1)^2}$ {\tiny (6)} \\
\hline
$E$ & $\frac{F_{2s+3}^2}{F_{2s+1}^2}$ {\tiny (10)} &  & $\frac{S_{s+1}^2}{4T_{s}^2}$ {\tiny (12)} & &  &  &  & $\frac{F_{2s+3}^2}{F_{2s+2}^2}$ {\tiny (9)}
\end{tabular}
\end{center}
\vspace{0.4cm}
\caption{The columns in the table correspond to the possible short graphs, while the rows correspond to the possible long graphs;
note that we omitted empty rows and columns to ease readability.
The empty cells in the table correspond to cases where no connected sum of lens spaces arises;
the cells marked with a $\times$ are those for which the value of $k$, predicted using Lisca's $I$-function, is smaller than $1$.
In each box, the fraction are the values of $q/p$ for which $S^3_{pq}(T_{p,q})$ bounds a rational homology ball.
In the columns $A, B^3, D^3, E$, the parameter $s$ varies among positive integers; in the other columns, it varies in the same range as the parameter $s$ in the corresponding column (e.g. $s>0$ in the column $B^1$, and $s\ge0$ in the column $B^2$ (see definitions in Subsection~\ref{ss:bounding-lens})).
The number in square brackets denote the family it belongs to, according to the notation in the introduction.}\label{f:reducible_table}
\end{table}

We are now ready to prove Theorem~\ref{reduciblebounding}.

\begin{proof}[Proof of Theorem~\ref{reduciblebounding}]
As mentioned above, the value of $k$ is determined by the values of $I$ of the short and the long graphs;
more precisely, if we call $S$ the short graph, and $L$ the long graph, then
\[
I(L) = I(S) + I(k+1) = I(S) + k-2.
\]
In the course of the proof, we will use this to compute $k$ without explicit mention. It follows that whenever $(S,L)$ is one of the types $\{(C,A),(D,A),(D,B),(E,A),(E,B),(E,C)\}$, then $k < 1$, which is not compatible with our setting;
so, there are no families in either of these cases. (These are precisely the cases marked with a $\times$ in Table~\ref{f:reducible_table}.)

By looking at Table~\ref{f:reducible_table} we notice that, when neither of the two graphs involved is of type $A$, $B^3$, $D^3$, or $E$, there are only two sporadic examples;
since the arguments to rule the configurations out in these cases are simpler, and all quite similar, we collect here the type of observation we made, and we leave the (tedious, but straightforward) case-by-case checks to the reader.

Either one or more of these strategies rule out a case, or they force the values of sufficiently many parameters (sometimes all);
this leaves just smaller subfamilies to check (and, a posteriori, rule out), or even just a single case, which are easily ruled out by direct inspection.

\begin{itemize}
\item Counting the number of final 2s; for instance, in the case when $S$ and $L$ are both of type $B^1$, we need to add a vertex of weight 3 to the short graph, but in the family $B^1$ all strings start and end with a 2.

\item Looking at internal sequences of entries larger than 2; for instance, in the case when $S$ is of type $B^1$ and $L$ is of type $C^1$, in $S$ we have four consecutive internal vertices with weight greater than 2, while in $L$ there can be no such substring.

\item Counting the number of non-2 vertices;
for instance, when $S$ and $L$ are both of type $B^{2}$, namely $S = B^2_{s,t}$ and $L = B^2_{s',t'}$ (where $s,s'\geq 0$, $t,t'>0$);
here $k=2$, so we need to remove a vertex of weight 3 in $L$, but both $S$ and $L$ have exactly three vertices of weight at least 3.

\item Looking at left- and right-most non-$2$ vertices;
for example, when $S = B^2_{s,t}$, $L=C^{1}_{s',t'}$, with $t, s'>0$, $s, t'\geq 0$; here $k=3$, so we add a vertex of weight 4;
we immediately see that $t'=2$, since $L$ must start or end with a 4;
removing the 4 leaves a string that starts with $s'+2 \ge 3$, so in turn $s=0$;
but then the extremal non-2 vertices of $S$ are two 3s, while one of the extremal vertices of $L$ (after removing the 4) is 3.

\item Looking at internal vertices of weight $4$ or $5$;
for instance, when $S = C^1_{s,t}$ and $L = B^{1}_{s',t'}$.
Indeed, in $L$ there is no 4 next to a 2-chain and thus it cannot come from $S$ by adding a vertex.
\end{itemize}

We therefore restrict ourselves to the case when at least one among the short and long graph are of type $A$, $B^3$, $D^3$, or $E$, plus the two exceptions.
To fix the notation, the parameters of the short graph will always be $(s,t)$ or $\bf b$, while the parameters of the long graph will be $(s',t')$ or $\bf b'$;
the dual strings to $\bf b$ and $\bf b'$ will be $\bf c$ and $\bf c'$;
the length of the strings $\bf b$, $\bf b'$, $\bf c$, and $\bf c'$ will be $h$, $h'$, $\ell$, and $\ell'$, respectively.

We will start by choosing the type of the short graph, and then run through all possibilities for the long graph $L$;
that is, in terms of the table, we will sweep it column by column, top to bottom and then left to right.

Regardless of the assumption made during the proof, at the end of each case $L$ will be read starting from the entry with weight $k+1$, so that $[L] = [k+1,S]$.
With this convention, the fraction $q/p$ will be $q/p = [L]^- = [k+1,S]^-$.

The amount of details in each proof will be decreasing as we go along, since most arguments are similar.
We will use the general yoga of dualizing strings with the Riemenschneider point rule quite freely, and without mention.

\begin{itemize}
\item $S=A_{\bf b}$.
	\begin{itemize}
		\item $L=A_{\bf b'}$; in this case, since we have $I(S)=I(L)=-3$ the value of $k$ is constrained to be $k=2$. This translates to adding a vertex of weight 3 to $S$ to obtain $L$.
		Since $A_{\bf b}$ is symmetric in $\bf b$ and $\bf c$, we can assume, without loss of generality, that $b_h = 2$; moreover, unless ${\bf b} = [2]$, this implies $c_\ell > 2$.
		Since $A_{\bb'}$ is symmetric, too, we can assume that $b'_{h'} = 2$, and $c'_{\ell'} = k+1 = 3$.
		Since the R-diagram of $L$ has two more points on the right of the center of the R-diagram of $S$ (because of the added vertex of weight $k+1=3$), the center of $L$ is the the vertex with weight $c_{1}$ in $S$. From here we deduce, $c_1 = 2$, ${\bf b}' = [2,{\bf b}]$, $\cc' = [c_2,\dots,c_\ell,k+1]$.
		If the length $\ell$ of $\bf c$ is $\ell = 1$, then we have one example, namely $S = A_{[2]} = [2,2,2]$, and $L = [3,2,2,2]$.
		Suppose now $\ell > 1$; since $c_1 = 2$, $b_1 > 2$;
		since the center of $L$ is $c_1$, we also know that $c'_1 = c_2 > 2$, and since $b_1 > 2$ we deduce that $c_2 = c'_1 = 3$.
		An easy induction shows that the only possible case is $\cc = [2,3^{[s]}]$, $\bb = [3^{[s]},2]$;
		indeed, suppose the contrary, and consider the smallest $i>0$ such that $c_{i+1} \neq 3$;
		that is, $\cc$ starts with $[2,3^{[i-1]},c]$, with $c\neq 3$;
		we proved that $i \ge 2$.
		Then, $\cc'$ starts with $[3^{[i-1]},c]$;
		by duality, $\bb'$ starts with $[2,3^{[i-2]},b']$ with $b' > 2$.
		In turn, this implies that $\bb$ starts with $[3^{[i-2]},2]$;
		since $b_1 > 2$, we also have $i>2$, and by duality $\cc$ starts with $[2,3^{[i-3]},c]$, which contradicts the minimality of $i$.
		
		Summing up, we have found that, changing the direction in which we read the strings, $S = [3^{[s]},2,2,3^{[s]},2]$, and $L = [3,S] = [3^{[s+1]},2,2,3^{[s]},2]$.
		Note that this includes the case $\bf b=[2]$ by setting $s=0$.
		We have $L = A_{\cc'}$, where $\cc' = [3^{[s+1]}]$ for $s\ge 0$;
		since $[3^{[s+1]}]^- = F_{2s+4}/F_{2s+2}$, it holds
		\[
		\frac{q}{p} = \frac{F_{2s+4}^2}{F_{2s+4}F_{2s+2}+1}=\frac{F_{2s+4}^{2}}{F_{2s+3}^{2}}.
		\]			 
		
		\item $L=B^{1}_{s',t'}$; here $k=3$, so we add a vertex of weight $4$.
		Recall that $s'$ and $t'$ are both positive, so $L$ has only final vertices of weight 2, and therefore it cannot be obtained from $S$ by adding a vertex of weight 4.
		 
		\item $L=B^{2}_{s',t'}$; here $k=3$, and we add a vertex of weight $4$.
		Recall that $s'\geq 0$ and $t'>0$.
		Since $t'$ is positive, the leftmost end of the string $L$ is a $2$, so the added vertex of weight 4 needs to be the rightmost vertex of $L$; this implies $s'=0$, in which case the string $L$ ends with a 3, and not with a 4.
		Therefore, there are no examples.

		\item $L=B^{3}_{\bb'}$; also in this case $k=3$, and we add a vertex of weight 4.
		By symmetry, we can assume that we add the vertex attached to the rightmost vertex.
		Adding a vertex of weight 4 adds three points to the R-diagram of $S$;
		by counting the number of points in the R-diagram, it follows that the center of $L$ (which consists of two rows, each with one vertex), starts immediately to the right of the center of $S \subset L$;
		but on both sides of the center of $L$ we have two vertices of weight strictly larger than 2, while the center of $S$ has weight 2, so we obtain a contradiction.
		 
		\item $L=C^{1}_{s',t'}$; here $k=4$, and the new vertex has weight $5$.
		Recall that in this family $s'>0$, $t'\geq 0$.
		The final $5$ added to $S$ needs to be the vertex of weight $t'+2$ in $L$, and thus $t'=3$.
		We have that $S$ is of the form $[s'+2,3,2,2,2,4,2^{[s']}]$;
		by counting points in the R-diagram, the center of $S$ must be the fourth vertex, which is a 2 surrounded by two 2s, but this is impossible since $\bb$ and $\cc$ are complementary, and both have length at least 2.
		 
		\item $L=C^{2}_{s',t'}$; here $k=4$, and the new vertex has weight $5$.
		The parameters satisfy $s',t'\geq 0$.
		Since the right-most non-2 vertex has weight 4, similarly as in the previous case, we need $t'+2=5$;
		this implies $t'=3$, and a point count in the R-diagram shows that the center of $S$ is the central 2 in a chain of three 2s, leading to the same contradiction.

		\item $L=C^{3}_{s',t'}$; we have $k=4$ and the new vertex has weight 5.
		The parameters satisfy $s',t'\geq 0$
		With the same argument as in the previous case, we obtain that $t'=2$ and after deleting the added 5 from $L$ we obtain the string $[2,s'+3,3,2,2,3,2^{[s']}]$.
		This shows that $s'=0$, since $S$ cannot start and end with a 2 (the length of $S$ is at least 6, so $S$ cannot be $A_{[2]}$), and it is immediate to see that $[2,3,3,2,2,3,2]$ is not of type $A$.
		
		\item $L=D^{1}_{s',t'}$; now $k=5$ and the new vertex has weight 6.
		The parameter vary as $s',t'\geq 0$.
		Removing a final vertex from $L$ we obtain a string with no final 2s, so there are no examples.
		 
		\item $L=D^{2}_{s',t'}$; again, $k=5$ and the new vertex has weight $6$.
		also, $s',t'\geq 0$.
		If $t'+3\neq 6$, then the vertex with weight $k+1$ in $L$ is that labelled by $s'+2$, and removing it we get a string without final 2s, which cannot be of type $A$;
		so, $t' = 3$.
		After deleting the vertex with weight $t'+3$, we obtain the string $[2^{[s']},4,2,2,2,3,s'+2]$.
		As above, counting points in the R-diagram, the center of $S$ is the fourth vertex from the right, which is surrounded by two 2s, leading to a contradiction.
		
		\item $L=D^{3}_{\bb'}$; here $k=5$ and we add a vertex of weight $6$.
		Without loss of generality, the new vertex can be added by linking it to $c_\ell$;
		in $L$, we have $J(\cc) + 5$ points to the right of the center of $S$, and $J(\bb)$ points to its left.
		The four points comprising the center of $L$ must start immediately after the center of $S$ (so that we have $J(\bb)+1$ points to their left, and $J(\cc)+1$ to their right), which means that $c_1 = 5$ corresponds to the center of $L$.
		Now, if $\cc = [5]$, we get the strings $L = [6,S] = [6,5,2^{[5]}]$, corresponding to the fraction
		\[
		\frac{q}{p} = \frac{144}{25} = \frac{4T_2^2}{4T_2T_1+1} = \frac{4T_2^2}{S_2^2}.
		\]
		If $\cc$ has length more than 1, $\bb$ starts with exactly four 2s, and then with a number larger than 2.
		This implies, in turn, that, $\bb'$ starts with exactly five 2s, and therefore $\cc'$ starts with a 6.
		Now an easy induction shows that the only strings we obtain are $L = [6^{[s+1]},5,2,(2^{[3]},3)^{[s]},2^{[4]}]$, i.e. ${\bf b}' = [6^{[s+1]}]$ and $[{\bf b'}]^- = T_{s+3}/T_{s+2}$; the corresponding fractions are
		\[
		\frac{q}{p} = \frac{4T_{s+3}^2}{4T_{s+3}T_{s+2}+1} = \frac{4T_{s+3}^2}{S_{s+3}^2}.
		\]
		\item $L=E_{\bb'}$; here $k=6$ and we add a vertex of weight $7$.
		By symmetry, we can assume that we add the new vertex attaching it to $c_\ell$.
		Adding this vertex, we have added six points to the R-diagram of $S$.
		By counting points, we find out that there are two different possibilities: either $c_{1}=2$, or $c_1>2$.
		
		If $c_1 = 2$, then the center of $L$ belongs to the row in the R-diagram for $S$ corresponding $c_{2}$, and it necessarily splits as $b_1' + c_1' = 2+c_{1}'$, i.e. $b_1'=2$.
		We then obtain also that $b'_{2}=c_{1}=2$, $b_{3}'=2$ and by duality we also obtain $c_{1}'>2$.
		Since $b_{1}=2$, we know that $c_{1}>2$ which implies in turn that $c_{1}'=5$, so that $b_1'+c'_1 = 7$.
		From here an easy induction shows that the string $L$ we are looking at is either $[7,2,2,2]$, or  $[7^{[s+1]},2,2,(3,2^{[4]})^{[s]},2]$.
		Since this corresponds to $\bb' = [7^{[s+1]},5]$, we obtain the fractions
		\[
		\frac{q}{p} = \frac{F_{4s+5}^2}{F_{4s+5}F_{4s+1} - 1} = \frac{F_{4s+5}^2}{F_{4s+3}^2}.
		\]
		 
		In the case $c_{1}>2$, notice that because of the symmetry in the R-diagram we immediately conclude that $c_{1}=5$.
		Moreover, it follows that under the identification the center of $S$ is identified with the vertex of weight $b'_{2}$ in $L$.
		From here we infer that $b_{1}'+c_{1}'=5=3+2$.
		The first sequence verifying the desired properties is for $L$ to be $[7,5,2^{[5]}]$ and this is the first one in the family $[7^{[s+1]},5,2,(2^{[4]},3)^{[s]},2^{[5]}]$.
		Since this corresponds to $\bb' = [7^{[s+1]},2]$ we obtain
		\[
		\frac{q}{p} = \frac{F_{4s+3}^2}{F_{4s+3}F_{4s-1} - 1} = \frac{F_{4s+3}^2}{F_{4s+1}^2}.
		\]
	\end{itemize}

\item $S=B^{1}_{s,t}$, where the parameters $s$ and $t$ are both positive.

	\begin{itemize}

		\item $L=A_{\bb'}$; here $k=1$, so we add a vertex of weight 2.
	Regardless of which end we add the vertex to, though, we get the contradiction of the long string starting and ending with a 2.
	
		
		\item $L=B^{3}_{\bb'}$; here $k=2$, and we add a vertex of weight 3.
		If $\bb'=[2]$, then $L = [3,2,2,3]$, and $S=[2,2,3]$ which is not of type $B^1$.
		So we can assume $\bb' \neq [2]$, which implies that, even after removing a vertex of weight 3 from $L$, the string has two consecutive internal $2$s surrounded by two vertices of weight larger than 2; such a string does not appear in $S$, so there are no examples.
		


	
	
	
		\item $L=D_{\bb'}^{3}$; here $k=4$ and we add a vertex of weight 5.
		The center of $L$ is an \emph{internal} 5 linked to (at least) a 2; after removing the vertex with weight 5, this configuration survives; however, in $S$ there is no vertex of weight 5 linked to a 2.
		
		\item $L=E_{\bb'}$; here $k=5$ and we add a vertex of weight 6.
		If $\cc' = [2]$, then $L = [4]$, so this gives no examples;
		if $\cc' = [c'_1]$, with $c'_1 > 2$, then the only way we can obtain a vertex of weight 6 in $L$ is if $c'_1 = 4$, from which $L = [6,2,2,2]$, which, again, gives no examples of type $B^1$.
		By symmetry, we can now assume $h',\ell' > 1$ and that $c'_{\ell'} = 6$;
		by duality, $b'_{h'}=...=b'_{h'-4}=2$ and $b'_{h'-5}>2$.
		So $S$ ends with a string of exactly five 2s; by symmetry of $S$, we can assume $t = 5$, so $L = [2^{[5]},3,s+2,7,3,2^{[s]},6]$; but, by duality, since $s>0$, $b'_{h'-5}$ should be strictly larger than 3, a contradiction.
	\end{itemize}
\item $S=B_{s,t}^{2}$, where the parameters $s,t$ satisfy $s\geq 0$ and $t>0$.

	\begin{itemize}
		\item $L=A_{\bb'}$; here $k=1$, so we add a vertex of weight 2.
		Since $S$ has length 5, $\bb \neq [2]$, so $L$ ends with 2 on one side, and with something larger on the other.
		Since $t>0$, this means that the extra vertex has to be added on the left of $S$, and that $s = 0$;
		therefore, $L$ ends with a string 2s of length of $t+1 \ge 2$, and with a 3 on the other side; but this can only happen if $\cc' = [3]$, which is easily seen not be the case.
		
		\item $L=B^{3}_{\bb'}$; here $k=2$, so we add a vertex of weight 3.
		Unless $\bb' = [2]$, the two internal vertices of weight 2 surrounded by two vertices with larger weight survive even after removing the extra vertex;
		but in $S$ there is no such configuration.
		The case $\bb' = [2]$ is excluded by direct inspection.
		\item $L=D^{3}_{\bb'}$; here $k=4$ and we add a vertex of weight 5.
		Note that $\bb' \neq [2]$, since otherwise $L$ would not end in a 5;
		in particular, the central 5 in $L$ has a 2 only on one side, and it follows that $t+2=5$;
		Adding the 5 to the left of $S$ yields $[5,2,2,2,3+s,2,5,3,2^{[s]}]$, and this is never of the form $D^3_{[\bb']}$ (since, for instance, the dual string to $[3,2^{[s]}]$ is $[2,s+2]$).
		However, if we add the 5 to the right; indeed, we obtain the family $S = [2,2,2,3+s,2,5,3,2^{[s]}]$ and $L = [5,2^{[s]},3,5,2,3+s,2,2,2] = D^3_{\bb'_s}$ with $\bb'_s = [3,2^{[s]},5]$.
		We compute
		\[
		\frac{x}{y^*} = [5,2^{[s]},3] = \big[5,{\textstyle\frac{2s+3}{2s+1}}\big]^- = \frac{8s+14}{2s+3},
		\]
		so that
		\[
		\frac{q}{p} = [L]^- = \frac{4(8s+14)^2}{4(8s+14)(2s+3)+1} = \frac{(16s+28)^2}{(8s+13)^2}
		\]
		
		\item $L=E_{\bb'}$; now $k=5$, so we add a vertex of weight 6.
		If $\ell' = 1$, in order to find a final 6 in $L$, we need $c_1' = 4$, $\bb' = [2,2,2]$, and $L = [6,2,2]$, but then removing the 6-vertex we obtain a string that is not of type $B^2$.
		Now, by symmetry, we can assume $\ell' > 1$ and $c'_{\ell'}=6$;
		this implies $b'_{h'}=\dots=b'_{h'-4}=2$ and $b_{h'-5}>2$.
		We have two possibilities for $S$, by looking at the two chains of final 2s: either $t=5$ and the new vertex is attached on the right-hand side of $S$, or $s=5$ and the new vertex is attached on the left-hand side of $S$.
		In the first case, $S = [2,2,2,2,2,3+s,2,7,3,2^{[s]}]$;
		as argued above, the only possibility for the vertex of weight $b'_{1}+c'_{1}$ in $L$ is $b'_{1}+c'_{1}=7$;
		for there to be a two complementary legs, we need $c'_{1}=2$ and $b_{1}=5$, but this is easily checked not to yield any example.
		In the second case, then $L = [6,2^{[t]},8,2,t+2,3,2^{[5]}]$;
		since the first non-2 element on the right is a 3, the first sequence of 2s to the right must have length 0, i.e. $t=0$ which contradicts the assumption that $t$ be positive.
	\end{itemize}
\item $S=B^{3}_\bb$ ending with two complementary legs.

	\begin{itemize}
		\item $L=A_{\bb'}$; $k=1$, and the new vertex has weight 2.
		Since $S$ has length at least 4, $L$ has length at least 5, and $\bb' \neq [2]$;
		it follows that exactly one of the two ends of $L$ is a 2.
		We will examine the case where $\cc = [c_1]$ first;
		in this case, $S = [2^{[c_1-2]},3,2,2,c_1+1]$, and the extra vertex can only be added on the left-hand side;
		so the only possibility is $L = [2^{[c_1-1]},3,2,2,c_1+1] = A_{[3,2^{[c_1-1]}]}$;
		since $[2^{[s]},3]^- = \frac{2s+3}{2s+1}$, we then have a family
		\[
		\frac{q}{p} =  \frac{(2s+3)^2}{(2s+3)(2s+1)+1} = \frac{(2s+3)^2}{(2s+2)^2}.
		\]
		Having ruled out the case $\cc = [c_1]$ (and, by symmetry, the case $\bb = [b_1]$, too), we can assume $h, \ell > 1$;
		again, by symmetry assume $c_\ell > 2$, so that $b_h = \dots = b_{h-c_\ell+3} = 2$ and $b_{h-c_\ell+2} > 2$;
		note that $S$ has at least two vertices of weight larger than 2 (namely, $b_1 + 1$ and $c_1 + 1$), so also $h', \ell' > 1$;
		we can only add the new vertex on the side of $b_h$, but by doing so $\bb'$ ends with a string of 2s of length $c_1-1 = c_1'-1$, which contradicts the fact that $\cc'$ is dual to $\bb'$, and both have length larger than 1.

		\item $L=B_{s',t'}^{1}$; $k=2$, so we add a vertex of weight 3;
		but $L$ starts and ends with 2s, so there are no cases in this family.
		
		\item $L=B^{2}_{s',t'}$; $k=2$, so we add a vertex of weight 3;
		in $L$ there are no two consecutive 2's linked to vertices of weight greater than 2, so there are no cases here, either.
		
		\item $L=B^{3}_{\bb'}$; $k=2$, so the new vertex has weight 3;
		by symmetry, we assume that it is added on the side of $c_\ell$.
		We want to look for the center of $L$:
		since we added two points in the R-diagram of S, the centre of L must move one line down;
		but this is not possible because the next line of the diagram has more than one point in it (since the corresponding vertex had weight $c_1+1 > 2$).

		\item $L=C^{1}_{s',t'}$; $k=3$, so the new vertex has weight 4;
		since there is an internal string of two 2s in $S$, surrounded by non-2 vertices, this configuration must also appear in $L$, which readily implies $t'=2$;
		moreover, since $s'>0$, the extra vertex is the one with weight $t'+2 = 4$;
		thus $L = [4,s'+2,3,2,2,4, 2^{[s']}]$, and $S = [s'+2,3,2,2,4,2^{[s']}] = B^3_{[2,s'+2]}$.
		We have $\frac{x}{y^*} = [s'+2,2]^- = \frac{2s'+3}{2}$, so that
		\[
		[S]^- =  \frac{4(2s'+3)^2}{4(2s'+3)\cdot 2-1} = \frac{(4s'+6)^2}{16s'+23},
		\]
		and
		\[
		\frac{q}{p} = [4,S]^- = \left[4,\frac{(4s'+6)^2}{16s'+23}\right] = \frac{(8s'+11)^2}{(4s'+6)^2}
		\]
		
		\item $L=C^{2}_{s',t'}$; $k=3$ and the new vertex has weight 4;
		looking at the central string of 2s in $S$ as in the preceding case, we obtain $t'=2$;
		then $L$ is of the form $[4,2,s'+3,2,2,4, 2^{[s']}]$.
		Deleting the vertex of weight 4 we conclude that $s'=0$, since in $S$ we cannot have two final weight-2 vertices.
		Thus, $L$ is of the form $[4,2,3,2,2,4]$, and the corresponding short string is $S=[2,3,2,2,4] = B^3_{[2,2]}$. It follows that we have a single solution in this family, corresponding to $\frac{q}{p} = \frac{11^{2}}{6^{2}}$, which would be the case $s'=0$ in the previous family.
		
		\item $L=C^{3}_{s',t'}$; $k=3$ and the new vertex has weight 4;
		as in the two preceding cases, we deduce $t'=2$; but then $L$ starts with a 5 and ends with either a 2 or a 3, while it should have a vertex of weight 4 on (at least) one end.
		
		\item $L=D^{1}_{s',t'}$; $k=4$ and the new vertex has weight 5;
		removing either final vertex in $L$ we obtain a chain that ends with a 3 and ends with a vertex of weight at least 3; this is only possible if $S = [3,2,2,3]$, but a direct inspection shows that this is never the case.
		
		\item $L=D^{2}_{s',t'}$; $k=4$ and the new vertex has weight 5;
		the presence of the two internal vertices of weight 2 in $S$ implies that either $t'=2$ or $s'=2$. 
		If $s'=2$, then we also have $t'=2$, because $L$ has a final weight-5 vertex.
		We conclude $t'=2$ yielding for $L = [5,2^{[s']},4,2,2,3,s'+2]$.
		To obtain $S$, we need to remove either the left-most vertex or the vertex of weight $s'+2$.
		In the latter case, $s'=3$ and $S = [5,2,2,2,4,2,2,3]$, which is not of type $B^3$;
		If, on the other hand, we remove the left-most vertex, we obtain $S = [2^{[s']},4,2,2,3,s'+2] = B^3_{3,2^{[s']}}$.
		Since $\frac{x}{y^*} = [2^{[s']},3] = \frac{2s'+3}{2s'+1}$, we have
		\[
		[S]^- = \frac{4(2s'+3)^2}{4(2s'+3)(2s'+1)-1} = \frac{(4s'+6)^2}{(4s'+4)^2-5},
		\]
		and then
		\[
		\frac{q}{p}=[5,S]^- = \left[5,\frac{(4s'+6)^2}{(4s'+4)^2-5}\right]=\frac{(8s'+13)^{2}}{(4s'+6)^{2}}.
		\]
		
		\item $L=D^{3}_{\bb'}$; again, $k=4$, so we add a vertex of weight 5.
		Suppose we attach this vertex to $c_\ell$;
		we are adding four points in the R-diagram of $S$;
		the four central vertices of the R-diagram of $L$, however, overlap with the center of $S$, but this is impossible, since the center of $L$ is a contained in a single row.
		
		\item $L=E_{\bb'}$; here $k=5$, and we add a vertex of weight 6;
		suppose that we link this vertex to $c_{\ell}$ to obtain $L$ from $S$.
		Since we add four points to the R-diagram of $S$, the centre of $L$ is the second point in the row corresponding to $c_1+1$ (it is the unique point that divides the R-diagram of $L$ into two parts of equal size).
		From this we deduce that $b'_1 = \dots = b'_3 = 2$, and $b'_4 = b_1+1$;
		this, in turn, implies $c'_1 = 5$, and $c_1+1 = b'_1+c'_1 = 7$, so that $b'_4 = b_1 + 1 = 3$, $b'_5 = \dots b'_7 = b_2 = \dots = b_4 = 2$, and therefore $c_2' = 6$.
		The usual inductive argument shows that $L = [6^{[s]},7,2,2,(3,2^{[3]})^{[s]},2] = E_{[5,6^{[s]}]}$.
		Since $\frac{x}{y*} = [6^{[s]},5]^- = \frac{S_{s+1}}{S_s}$, we have
		\[
		\frac{q}{p} = \frac{S_{s+1}^2}{S_{s+1}S_s - 1} = \frac{S_{s+1}^2}{4T_s^2}.
		\]
	\end{itemize}
\item $S=C_{s,t}^{1}$ with $s>0$ and $t\geq 0$. Here we immediately rule out the case when $L$ is of type $A$, since this would force $k=0$.

	\begin{itemize}
		
		\item $L=B^{3}_{\bb'}$; $k=1$, and we add a vertex of weight 2.
		We can immediately rule out the case when $h'=1$ or $\ell' = 1$, as done several times above;
		it follows that the center of $L$ corresponds to the substring $2^{[t]}$ in $S$, so that $t=2$.
		Moreover, by duality, and since $h',\ell' > 1$, the new vertex must be added on the right of $S$;
		since $t+2 = 4$, $L$ must end with exactly two 2s, so that $s+1 = 2$, hence $s=1$;
		but then $L = [4,4,3,2,2,4,2,2]$, which is not of type $B^3$.


		\item $L=C^{3}_{s',t'}$; $k=2$, so we add a vertex of weight 3.
		Since there is an internal 4 in $S$, we deduce that $s'=1$;
		this in turn implies $t'=0$, $L$ must end start or end with a 3.
		Then $L = [3,2,4,3,3,2]$, and $S = [2,4,3,3,2]$, which is $C^1_{1,0}$ read backwards.
		Then
		\[
		\frac{q}{p}=\frac{14^{2}}{9^{2}}.
		\]
		
		\item $L=D^{3}_{\bb'}$; now $k=3$ and the new vertex has weight 4.
		It is clear that the only possibility for the central vertex of weight $5$ in $L$ is for it to correspond to the vertex of weight $s+2$ in in $S$, forcing $s=3$.
		By duality, $t=0$; since the dual of $[3,4,2^{[s]}]$, i.e. the string to the right of the central 5, is $[2,3,2,s+2]$, the duality structure is not respected.

		\item $L=E_{\bb'}$; $k=4$, and we add a vertex of weight 5.
		Since we can immediately rule out the case where $\bb'$ or $\cc'$ has length 1, there is an internal vertex of weight $c'_{1}+b'_{1} \ge 5$ in $L$.
		This cannot correspond to the vertex of weight $t+2$ in $S$, since otherwise the only possibilities for $\bb$ are $\bb = [2,5]$ and $[t,5]$, neither of which yields the correct structure for $S$.
		So the vertex of weight $b'_1+c'_1$ corresponds to the vertex of weight $s+2$ in $S$.
		Just as above, the structure to the sides of this vertex in $S$ does not correspond to complementary legs, as the right-hand side is too long to match the left-hand side.
	\end{itemize}
\item $S=C^{2}_{s,t}$ with $s,t\geq 0$.
	\begin{itemize}
		\item $L=B^{3}_{\bb'}$; $k=1$, so the new vertex has weight 2.
		Since we remove a 2 from $L$, the sequence of two internal 2s survive, and hence $t=2$;
		therefore, $S = [4,2,s+3,2,2,4,2^{[s]}]$.
		Looking at the right side of the string of 2s, $\cc'$ can only be $[3,2^{[s]}]$ or $[3,2^{[s+1]}]$ (depending on whether we add the new vertex attaching it to $c'_{\ell'}$ or not), but in neither case is the string on the left dual to it.
		\item $L=D^{3}_{\bb'}$; now $k=3$, so we add a vertex of weight 4.
		The central vertex of weight 5 in $L$ remains internal in $S$, unless $\bb' = [4]$ or $\cc' = [4]$;
		these cases are easily excluded, so $s+3=5$, and then $s=2$.
		Moreover, since to its side we have two complementary legs, $t=0$ and then $L = [4,2,2,5,4,2,2] = D^{3}_{[2,2,4]}$.
		The associated fraction is
		\[
		\frac{q}{p}=\frac{20^{2}}{11^{2}}.
		\]
		
		\item $L=E_{\bb'}$; now $k=4$, so we add a vertex of weight 5.
		Arguing as we did in the previous case, the vertex of weight $b'_1 + c'_1$ in $L$ remains central in $S$ after removing the extremal, weight-5 vertex.
		So the center of $L$ is one of the points in the row corresponding to the $s+2$-vertex;
		call $x$ the second-to-last point in that row;
		$x$ has $s+t+4$ points to its right, and $s+t+2$ points to its left, so the new vertex must be added to the left of $S$.
		But in $L$ there are $s+t+6$ points to the left of the $x$ of $S$, which means that the center of $L$ needs to be the point immediately to the left of $x$;
		this in turn implies that $s+3 = 4$, so $s=1$;
		but then $L$ starts with a 5 and ends with a single $2$, which contradicts the duality structure.
	\end{itemize}	
\item $S=C^{3}_{s,t}$ with $s,t\geq 0$.
	\begin{itemize}
		 \item $L=B^{3}_{\bb'}$; now $k=1$, so we add a vertex of weight 2.
		 In $L$, even after removing the added 2, we have a configuration of two consecutive internal 2s linked to larger weight vertices.
		 It follows that $t=2$, however, since in $S$ these two 2s are linked to 3s, the configuration is not possible unless $\bb' = \cc' = [2]$, i.e. $S = [3,2,2,3]$, $L = [2,3,2,2,3]$, which is not of type $C^3$.
		 
		 \item $L=D^{3}_{\bb'}$; now $k=3$, and we add a vertex of weight 4.
		 If the 5 in $L$ became final in $S$, we would not have two complementary legs to its sides;
		 it follows that we need to identify the 5 in $L$ with the vertex of weight $s+3$ in $S$, thus $s=2$.
		 The added vertex to $S$ needs to be linked to the vertex with weight $t+3$, since we need a final 2. Then $L = [4,t+3,2,5,3,2^{[t]},3,2,2] = D^3_{[2,t+3,4]}$.
		 Since
		 \[
		 \frac{x}{y^*} = [4,t+3,2]^- = \frac{8t+18}{2t+5},
		 \]
		 we have
		\[
		\frac{q}{p} = \frac{4(8t+18)^2}{4(8t+18)(2t+5)+1} =\frac{(16t+36)^{2}}{(8t+19)^{2}}.
		\]
				
		\item $L=E_{\bb'}$; $k=4$, and we add a vertex of weight 5.
		As in the preceding case, this vertex need to be linked to the vertex with weight $t+3$ in $S$, because of the complementary legs structure;
		it follows that $s=3$ and we have $S = [t+3,2,6,3,2^{[t]},3,2,2,2]$.
		Moreover, the vertex in $L$ of weight $b_{1}'+c_{1}'$ has to be the vertex in $S$ of weight $s+3=6$;
		when comparing it to the string $L$ there are two possibilities, either the $6$ splits as $2+4=b_{1}'+c_{1}'$ or as $4+2=b_{1}'+c_{1}'$.
		In neither case do we have two complementary legs, so we have no case in this family.
	\end{itemize}
\item $S=D^{1}_{s,t}$ with $s,t\geq 0$. Here, and whenever $S = D$, we can already rule out the cases $L = A, B$, because of the value of $k$.
	\begin{itemize}
		\item $L=D^{3}_{\bb'}$; now $k=2$, and we add a vertex of weight 3.
		In $S$ there is no internal vertex of weight 5, so necessarily $\bb'$ has length 1, and we need to remove the vertex labelled with $b_1$;
		then $\bb'=[3]$, and $S = [5,2,2]$, which is not of type $D^1$.

		\item $L=E_{\bb'}$; $k=3$, so we add a vertex of weight 4.
		To obtain $S$ we remove a 4 from $L$; then the other end in $L$ is a 2, but $S$ has two final vertices of weight greater than 2.
	\end{itemize}
\item $S=D^{2}_{s,t}$ with $s,t\geq 0$.
	\begin{itemize}
		\item $L=C^{3}_{s',t'}$; $k=1$, so we add a vertex of weight 2.
		There is an interior 4 in $S$, which henceforth is interior in $L$;
		the only chance is that this vertex is identified with the vertex of weight $s'+3$, so that $s'=1$.
		This implies that the extra vertex in $L$ is the weight-2, right-most vertex in the string $[t'+3,2,4,3,2^{[t']},3,2]$.
		Since in $S$ there are not two 3s separated by a 2-chain, we also need $t'=0$;
		hence the only possibility is $L = [2,3,3,4,2,3]$, which is $C^3_{1,0}$ read backwards, and $S = [3,3,4,2,3] = D^2_{1,0}$ read backwards.
		The associated fraction is
		\[
		\frac{p}{q}=\frac{14^{2}}{11^{2}}.
		\]

		\item $L=D^{3}_{\bb'}$; $k=2$, and we add a vertex of weight 3.
		If neither $\bb'$ nor $\cc'$ has length 1, the string $S$ has an internal 5 and this is not compatible with $S$.
		However, if $\bb' = [b'_1]$ and we can remove a 3, we must have $b_1 = 3$, which is not compatible with the structure of $S$.
		
		\item $L=E_{\bb'}$; now $k=3$, so we add a vertex of weight 4.
		Since $\bb \neq [2]$ by direct inspection, the vertex of weight $b_{1}'+c_{1}'$ stays internal in $S$, and it is and at least 5, unless $c_{2}'=4$ and the length of $c$ is 2.
		In these configurations the set $L$ is two small to be identified with $S$.
	\end{itemize}

\item $S = D^3_{\bb}$.

	\begin{itemize}
		\item $L = C^1_{s',t'}$; now $k=1$, so we add a vertex of weight 2.
		The vertex of weight 5 in $S$ is inner in $L$, so $s'=3$.
		Since the vertex to its right is a $3$, then $t'+2 = 2$, i.e. $t'=0$, and ${\bf c} = [2]$ but this is clearly a contradiction, since the dual of $[2]$ is $[2]$.

		\item $L = C^2_{s',t'}$; again $k=1$, and we add a vertex of weight 2.
		As above, the vertex $5$ in $S$ is inner in $L$, so $s'=2$.
		We need to remove a 2 from $L$;
		if we remove it from the left, $t'=0$, and the two remaining strings to the left and right of $5$ are not dual.
		So we have to remove it from the right. The two strings $\bf b$ and $\bf c$ are ${\bf b} = [2^{[t']}, 4,2,2]$, ${\bf c} = [2,t'+2]$, which forces $t'=0$ (otherwise they both start with a 2), but $[4,2,2]$ and $[2,2]$ are not dual.

		\item $L = C^3_{s',t'}$; $k=1$, and the new vertex has weight 2.
		As in the previous cases, $s'=2$, and the vertex we need to remove is the one far at the right.
		So, $S = [t'+3, 2,5,3,2^{[t']},3,2]$, so ${\bf b} = [3,2^{[t']},3,2]$, ${\bf c} = [2,t'+3]$. Since the dual of $\bf c$ is $[3,2^{[t'+1]}]$, this is never equal to $\bf b$.

		\item $L = D^1_{s',t'}$; $k = 2$, and we add a vertex of weight 3.
		The long string does not have an inner $5$, so this case is empty.

		\item $L = D^2_{s',t'}$; $k = 2$, and we add a vertex of weight 3.
		As above, here there are no inner $5$s.

		\item $L = D^3_{\bb'}$; again $k=2$, and the new vertex has weight 3.
		By symmetry, we can assume that we attach it to the vertex of weight $c_\ell$ in $S$.
		In the R-diagram of $L$, then, we have two more points below of the center of $S$ than above it;
		this means that the central four points in the R-diagram of $L$ are the last three points of the center of $S$, and the first point of the line of $c_1$;
		but the central four points of the R-diagram of $L$ should occupy a single row.

		\item $L = E_{\bb'}$; now $k=3$, and we add a vertex of weight 4;
		by symmetry, we can suppose that it it attached to the vertex of weight $c_\ell$.
		Then, in the R-diagram of $L$, there are three more points below the center of $S$ than above it;
		it follows that the center of $L$ is the right-most point of the center of $L$, which however is the final vertex in its row, which is impossible.
	\end{itemize}

\item $S = E^3_{\bb}$.
Here we rule out all cases when $L$ is of type $A$, $B$, or $C$, since the corresponding value of $k$ would be non-positive.

	\begin{itemize}
		\item $L = D^1_{s',t'}$; $k=1$, so we add a vertex of weight 2;
		but $L$ does not have any final vertex of weight 2.

		\item $L = D^2_{s',t'}$; again, $k=1$, and the new vertex has weight 2.
		The only final vertex that can have weight 2 in $L$ is the one labeled with $s'+2$, so $s'=0$;
		however, removing that vertex we obtain a string of length at least 2 that neither starts nor ends with a 2, which contradicts the assumption that $S$ be of type $E^3$.

		\item $L = D^3_{\bb'}$; $k=1$, and the new vertex has weight 2.
		Without loss of generality, we can assume that the new vertex is attached to $c_\ell$;
		then the center of $S$ has one more point to its right than to its left, which means that the center of $L$ comprises the center of $S$, the point to its left, and the next two points to its right.
		This, in turn, implies that $b_1=2$ and $c_1 = 3$.
		Now, since $c_1 = 3$, $b_1' = b_2 > 2$, which in turn forces $c_1'=2$.
		A straightforward induction yields the string $L = [s+1,5,2^{[s]}] = D^3_{[s+1]}$.
		The associated fraction is
		\[
		\frac{q}{p}=\frac{(2s+2)^{2}}{(2s+1)^{2}}.
		\]

		\item $L = E_{\bb'}$; $k=2$, so we add a vertex of weight 3.
		Again, by symmetry, we can assume we attach it to the right of $S$.
		In terms of R-diagrams, the center of $L$ is the point immediately to the right of the center of $S$, hence the central vertex has weight $b_1 + c_1 = 5 = b_1'+c_1'$.
		More precisely, $(b_1, c_1) = (2, 3)$ and $(b'_1, c'_1) = (3,2)$.
		Using the same argument as in the preceding case, we obtain $L = [3,3^{[s-1]},5,3^{[s-1]},2] = E_{[2,3^{[s]}]}$;
		hence
		\[
		\frac{x}{y^*} = [3^{[s]},2]^- = \frac{F_{2s+3}}{F_{2s+1}},
		\]
		from which
		\[
		\frac{q}{p} = \frac{F_{2s+3}^2}{F_{2s+3}F_{2s+1}-1} = \frac{F_{2s+3}^2}{F_{2s+2}^2}.\qedhere
		\]
	\end{itemize}
\end{itemize}
\end{proof}

\subsection{The lens space cases}

We now prove Theorem~\ref{lensspacesbounding}. The argument is essentially contained in the proof of Theorem~\ref{reduciblebounding}.

\begin{proof}[Proof of Theorem~\ref{lensspacesbounding}]
As noted at the beginning of the section, $S^3_{pq\pm 1}(T_{p,q}) = \pm L(pq\pm1,p^2)$, and by looking at continued fraction expansions we see that this corresponds to a long graph (of any type) beginning with a vertex of weight $k+1$;
removing this vertex we get a short graph of type $E_\bb$ (when the sign is $-$) or of type $A_\bb$ (when the sign is $+$). These situations have been already studied above, and here there is only a slight twist: the continued fraction expansions corresponding to the long legs now compute $\frac{pq\pm 1}{p^2}$, while in the previous analysis they computed $\frac{q^2}{p^2}$; this means that we obtain the following families.

We start with the case of $(pq+1)$-surgery, when the short leg is of type $E$, and then either:
\[
\frac{pq+1}{p^2} = \frac{(2s+2)^2}{(2s+1)^2} \Longrightarrow \frac{q}{p} = \frac{2s+3}{2s+1},
\]
or
\[
\frac{pq+1}{p^2} = \frac{F_{2s+3}^2}{F_{2s+2}^2} \Longrightarrow \frac{q}{p} = \frac{F_{2s+4}}{F_{2s+2}}.
\]
In the case of $(pq-1)$-surgery, when the short leg is of type $A$, then either:
\[
\frac{pq-1}{p^2} = \frac{F_{2s+4}^2}{F_{2s+3}^2} \Longrightarrow \frac{q}{p} = \frac{F_{2s+5}}{F_{2s+3}},
\]
or
\[
\frac{pq-1}{p^2} =  \frac{4T_{s+2}^2}{S_{s+2}^2}\Longrightarrow \frac{q}{p} = \frac{S_{s+3}}{S_{s+2}},
\]
or
\[
\frac{pq-1}{p^2} =  \frac{F_{2s+3}^2}{F_{2s+1}^2} \Longrightarrow \frac{q}{p} = \frac{F_{2s+5}}{F_{2s+1}}.\qedhere
\]
\end{proof}
\bibliographystyle{amsplain}
\bibliography{newaccpoints}

\end{document}